\documentclass[a4paper, 12pt]{amsbook}
\usepackage{amsfonts, amsthm, amssymb, amsmath, stmaryrd}
\usepackage{mathrsfs,array}
\usepackage{xy}
\usepackage{hyperref}
\usepackage{verbatim}
\input xy
\xyoption{all}

\newenvironment{altenumerate}
   {\begin{list}
      {\textup{(\theenumi)} }
      {\usecounter{enumi}
       \setlength{\labelwidth}{0pt}
       \setlength{\labelsep}{2pt}
       \setlength{\leftmargin}{0pt}
       \setlength{\itemsep}{\the\smallskipamount}
       \renewcommand{\theenumi}{\roman{enumi}}
      }}
   {\end{list}}

\renewcommand{\thechapter}{\Roman{chapter}}

\newtheorem{lem}{Lemma}[section]

\newtheorem{definition}[lem]{Definition}

\newtheorem{cor}[lem]{Corollary}
\newtheorem{thm}[lem]{Theorem}
\newtheorem{prop}[lem]{Proposition}

\newtheorem{conj}[lem]{Conjecture}

\theoremstyle{remark}
\newtheorem{rem}[lem]{Remark}

\DeclareMathOperator{\Hom}{Hom}
\DeclareMathOperator{\End}{End}
\DeclareMathOperator{\Ext}{Ext}

\DeclareMathOperator{\Tor}{Tor}

\DeclareMathOperator{\Ker}{ker}
\DeclareMathOperator{\im}{im}
\DeclareMathOperator{\coker}{coker}

\DeclareMathOperator{\Spa}{Spa}
\DeclareMathOperator{\Spec}{Spec}
\DeclareMathOperator{\Spf}{Spf}

\DeclareMathOperator{\Gal}{Gal}
\DeclareMathOperator{\Proj}{Proj}
\DeclareMathOperator{\Lie}{Lie}
\DeclareMathOperator{\Frob}{Frob}
\DeclareMathOperator{\Res}{Res}

\def\ad{\mathrm{ad}}

\def\et{\mathrm{\acute{e}t}}

\def\proet{\mathrm{pro\acute{e}t}}

\def\cont{\mathrm{cont}}

\def\id{\mathrm{id}}
\def\cycl{\mathrm{cycl}}
\def\HT{\mathrm{HT}}
\def\Ha{\mathrm{Ha}}
\def\ord{\mathrm{ord}}
\def\perf{\mathrm{perf}}
\def\can{\mathrm{can}}
\def\univ{\mathrm{univ}}
\def\der{\mathrm{der}}
\def\uast{{\underline{\ast}}}
\def\cl{\mathrm{cl}}
\def\BM{\mathrm{BM}}
\def\BS{\mathrm{BS}}
\def\LT{\mathrm{LT}}
\def\Dr{\mathrm{Dr}}
\def\int{\mathrm{int}}

\newcommand{\Z}{\mathbb{Z}}
\newcommand{\F}{\mathbb{F}}
\newcommand{\Q}{\mathbb{Q}}
\newcommand{\R}{\mathbb{R}}
\newcommand{\A}{\mathbb{A}}

\newcommand{\C}{\mathbb{C}}
\newcommand{\G}{\mathbb{G}}
\newcommand{\OO}{\mathcal{O}}
\newcommand{\II}{\mathcal{I}}
\newcommand{\mm}{\mathfrak{m}}
\newcommand{\tr}{\mathrm{tr}}
\newcommand{\GL}{\mathrm{GL}}

\newcommand{\GSp}{\mathrm{GSp}}
\newcommand{\Sp}{\mathrm{Sp}}
\newcommand{\Fl}{{\mathscr{F}\!\ell}}

\begin{document}
\title{On torsion in the cohomology of locally symmetric varieties}
\author{Peter Scholze}
\begin{abstract}
The main result of this paper is the existence of Galois representations associated with the mod $p$ (or mod $p^m$) cohomology of the locally symmetric spaces for $\GL_n$ over a totally real or CM field, proving conjectures of Ash and others. Following an old suggestion of Clozel, recently realized by Harris-Lan-Taylor-Thorne for characteristic $0$ cohomology classes, one realizes the cohomology of the locally symmetric spaces for $\GL_n$ as a boundary contribution of the cohomology of symplectic or unitary Shimura varieties, so that the key problem is to understand torsion in the cohomology of Shimura varieties.

Thus, we prove new results on the $p$-adic geometry of Shimura varieties (of Hodge type). Namely, the Shimura varieties become perfectoid when passing to the inverse limit over all levels at $p$, and a new period map towards the flag variety exists on them, called the Hodge-Tate period map. It is roughly analogous to the embedding of the hermitian symmetric domain (which is roughly the inverse limit over all levels of the complex points of the Shimura variety) into its compact dual. The Hodge-Tate period map has several favorable properties, the most important being that it commutes with the Hecke operators away from $p$ (for the trivial action of these Hecke operators on the flag variety), and that automorphic vector bundles come via pullback from the flag variety.
\end{abstract}

\date{\today}
\maketitle
\tableofcontents

\chapter{Introduction}

\renewcommand{\thesection}{\thechapter}

This paper deals with $p$-adic questions in the Langlands program. To put things into context, recall the global Langlands (-- Clozel -- Fontaine -- Mazur) conjecture.

\begin{conj}\label{GlobalLanglands} Let $F$ be a number field, $p$ some rational prime, and fix an isomorphism $\C\cong \overline{\Q}_p$. Then for any $n\geq 1$ there is a unique bijection between
\begin{altenumerate}
\item[{\rm (i)}] the set of $L$-algebraic cuspidal automorphic representations of $\GL_n(\A_F)$, and
\item[{\rm (ii)}] the set of (isomorphism classes of) irreducible continuous representations $\Gal(\overline{F}/F)\to \GL_n(\overline{\Q}_p)$ which are almost everywhere unramified, and de Rham at places dividing $p$,
\end{altenumerate}
such that the bijection matches Satake parameters with eigenvalues of Frobenius elements.
\end{conj}

Here, an $L$-algebraic automorphic representation is defined to be one for which the (normalized) infinitesimal character of $\pi_v$ is integral for all infinite places $v$ of $F$. Also,
\[
\A_F = {\prod_v}^\prime F_v
\]
denotes the ad\`eles of $F$, which is the restricted product of the completions $F_v$ of $F$ at all (finite or infinite places) of $F$. It decomposes as the product $\A_F = \A_{F,f}\times (F\otimes_{\Q} \R)$ of the finite ad\`eles $\A_{F,f}$ and $F\otimes_{\Q} \R = \prod_{v|\infty} F_v\cong \R^{n_1}\times \C^{n_2}$, where $n_1$, resp. $n_2$, is the number of real, resp. complex, places of $F$.

For both directions of this conjecture, the strongest available technique is $p$-adic interpolation. This starts with the construction of Galois representations by $p$-adic interpolation, cf. e.g. \cite{WilesPAdicInterpolation}, \cite{TaylorPseudoRepr}, but much more prominently it figures in the proof of modularity theorems, i.e. the converse direction, where it is the only known technique since the pioneering work of Wiles and Taylor-Wiles, \cite{WilesFLT}, \cite{TaylorWiles}.

For these techniques to be meaningful, it is necessary to replace the notion of automorphic forms (which is an analytic one, with $\C$-coefficients) by a notion of $p$-adic automorphic forms, so as to then be able to talk about $p$-adic families of such. The only known general way to achieve this is to look at the singular cohomology groups of the locally symmetric spaces for $\GL_n$ over $F$. Recall that for any (sufficiently small) compact open subgroup $K\subset \GL_n(\A_{F,f})$, these are defined as
\[
X_K = \GL_n(F)\backslash [D\times \GL_n(\A_{F,f})/K]\ ,
\]
where $D = \GL_n(F\otimes_{\Q} \R)/\R_{>0} K_\infty$ is the symmetric space for $\GL_n(F\otimes_{\Q}\R)$, with $K_\infty\subset \GL_n(F\otimes_{\Q} \R)$ a maximal compact subgroup. Then one can look at the singular cohomology groups
\[
H^i(X_K,\C)\ ,
\]
which carry an action by an algebra $\mathbb{T}_K$ of Hecke operators. By a theorem of Franke, \cite{Franke}, all Hecke eigenvalues appearing in $H^i(X_K,\C)$ come (up to a twist) from $L$-algebraic automorphic representations of $\GL_n(\A_F)$. Conversely, allowing suitable coefficient systems, all regular $L$-algebraic cuspidal automorphic representations will show up in the cohomology of $X_K$. Unfortunately, non-regular $L$-algebraic cuspidal automorphic representations will not show up in this way, and it is not currently known how to define any $p$-adic analogues for them, and thus how to use $p$-adic techniques to prove anything about them. The simplest case of this phenomenon is the case of Maass forms on the complex upper half-plane whose eigenvalue of the Laplace operator is $1/4$ (which give rise to $L$-algebraic cuspidal automorphic representations of $\GL_2(\A_\Q)$). In fact, for them, it is not even known that the eigenvalues of the Hecke operators are algebraic, which seems to be a prerequisite to a meaningful formulation of Conjecture \ref{GlobalLanglands}.\footnote{Although of course Deligne proved the Weil conjectures by simply choosing an isomorphism $\C\cong \overline{\Q}_p$, and deducing algebraicity of Frobenius eigenvalues only a posteriori.}

It is now easy to define a $p$-adic, or even integral, analogue of $H^i(X_K,\C)$, namely $H^i(X_K,\Z)$. This discussion also suggests to define a mod-$p$-automorphic form as a system of Hecke eigenvalues appearing in $H^i(X_K,\F_p)$. One may wonder whether a mod-$p$-version of Conjecture \ref{GlobalLanglands} holds true in this case, and it has been suggested that this is true, see e.g. the papers of Ash, \cite{Ash1}, \cite{Ash2}.

\begin{conj}\label{ConjModP} For any system of Hecke eigenvalues appearing in $H^i(X_K,\F_p)$, there is a continuous semisimple representations $\Gal(\overline{F}/F)\to \GL_n(\overline{\F}_p)$ such that Frobenius and Hecke eigenvalues match up.
\end{conj}

There is also a conjectural converse, generalizing Serre's conjecture for $F=\Q$, $n=2$, cf. e.g. \cite{ADP}, \cite{AS}. It is important to note that Conjecture \ref{ConjModP} is not a consequence of Conjecture \ref{GlobalLanglands}, but really is a complementary conjecture. The problem is that $H^i(X_K,\Z)$ has in general a lot of torsion, so that the dimension of $H^i(X_K,\F_p)$ may be much larger than $H^i(X_K,\C)$, and not every system of Hecke eigenvalues in $H^i(X_K,\F_p)$ is related to a system of Hecke eigenvalues in $H^i(X_K,\C)$ (which would then fall into the realm of Conjecture \ref{GlobalLanglands}). In fact, at least with nontrivial coefficient systems, there are precise bounds on the growth of the torsion in $H^i(X_K,\Z)$, showing exponential growth in the case that $n=2$ and $F$ is imaginary-quadratic (while $H^i(X_K,\C)$ stays small), cf. \cite{BergeronVenkatesh}, \cite{Pfaff}. In other words, Conjecture \ref{ConjModP} predicts the existence of many more Galois representations than Conjecture \ref{GlobalLanglands}.

The main aim of this paper is to prove Conjecture \ref{ConjModP} for totally real or CM fields:

\begin{thm}\label{ThmModP} Conjecture \ref{ConjModP} holds true if $F$ is CM, and contains an imaginary-quadratic field. Assuming the work of Arthur, \cite{ArthurBook}, it holds true if $F$ is totally real or CM.
\end{thm}

In fact, we also prove a version for $H^i(X_K,\Z/p^m\Z)$, which in the inverse limit over $m$ gives results on Conjecture \ref{GlobalLanglands}:

\begin{thm} There are Galois representations associated with regular $L$-algebraic cuspidal automorphic representations of $\GL_n(\A_F)$, if $F$ is totally real or CM.
\end{thm}

The second theorem has recently been proved by Harris-Lan-Taylor-Thorne, \cite{HLTT}. For the precise results, we refer the reader to Section \ref{SectionResults}.

In a recent preprint, Calegari and Geraghty, \cite{CalegariGeragthy}, show how results on the existence of Galois representations of the kind we provide may be used to prove modularity results, generalizing the method of Taylor-Wiles to $\GL_n$ over general number fields. This is conditional on their \cite[Conjecture B]{CalegariGeragthy}, which we prove over a totally real or CM field (modulo a nilpotent ideal of bounded nilpotence degree), except that some properties of the constructed Galois representations remain to be verified. Once these extra properties are established, Conjecture \ref{GlobalLanglands} might be within reach for regular $L$-algebraic cuspidal automorphic representations (corresponding to Galois representations with regular Hodge-Tate weights) over totally real or CM fields, at least in special cases or `potentially' as in \cite{BLGHT}.

To prove our results, we follow Harris-Lan-Taylor-Thorne to realize the cohomology of $X_K$ as a boundary contribution of the cohomology of the Shimura varieties attached to symplectic or unitary groups (depending on whether $F$ is totally real or CM). In particular, these Shimura varieties are of Hodge type.

Our main result here is roughly the following. Let $G$ be a group giving rise to a (connected) Shimura variety of Hodge type (thus, we allow $\Sp_{2g}$, not only $\GSp_{2g}$), and let $S_K$, $K\subset G(\A_f)$ be the associated Shimura variety over $\C$.\footnote{We need not worry about fields of definition by fixing an isomorphism $\C\cong \overline{\Q}_p$.} Recall the definition of the (compactly supported) completed cohomology groups for a tame level $K^p\subset G(\A_f^p)$,
\[
\widetilde{H}_{c,K^p}^i = \varprojlim_m \varinjlim_{K_p} H_c^i(S_{K_pK^p},\Z/p^m\Z)\ .
\]
The statement is roughly the following; for a precise version, see Theorem \ref{ThmHeckeAlgebras}.

\begin{thm}\label{InterpolationThm} All Hecke eigenvalues appearing in $\widetilde{H}_{c,K^p}^i$ can be $p$-adically interpolated by Hecke eigenvalues coming from classical cusp forms.
\end{thm}

In fact, only very special cusp forms are necessary, corresponding to cuspidal sections of tensor powers of the natural ample line bundle $\omega_K$ on $S_K$.

Combining this with known results on existence of Galois representations in the case of symplectic or unitary Shimura varieties (using the endoscopic transfer, due to Arthur, \cite{ArthurBook} (resp. Mok, \cite{Mok}, in the unitary case))\footnote{These results are still conditional on the stabilization of the twisted trace formula, but compare the recent work of Waldspurger and Moeglin, \cite{MoeglinWaldspurger}. In the unitary case, there are unconditional results of Shin, \cite{ShinGoldringApp}, which make our results unconditional for a CM field containing an imaginary-quadratic field, cf. Remark \ref{UnconditionalRem}.}, one sees that there are Galois representations for all Hecke eigenvalues appearing in $\widetilde{H}_{c,K^p}^i$ in this case. By looking at the cohomology of the boundary, this will essentially give the desired Galois representations for Theorem \ref{ThmModP}, except that one gets a $2n+1$-, resp. $2n$-, dimensional representation, from which one has to isolate an $n$-dimensional direct summand. This is possible, and done in Section \ref{DivideAndConquer}.

Thus, the key automorphic result of this paper is Theorem \ref{InterpolationThm}. The first key ingredient in its proof is a comparison result from $p$-adic Hodge theory with torsion coefficients proved in \cite{ScholzePAdicHodge}. Here, it is important that this comparison result holds without restriction on the reduction type of the variety -- we need to use it with arbitrarily small level at $p$, so that there will be a lot of ramification in the special fibre. The outcome is roughly that one can compute the compactly supported cohomology groups as the \'etale cohomology groups of the sheaf of cusp forms of infinite level.

Fix a complete and algebraically closed extension $C$ of $\Q_p$, and let $\mathcal{S}_K$ be the adic space over $C$ associated with $S_K$ (via base-change $\C\cong \overline{\Q}_p\hookrightarrow C$). Then the second key ingredient is the following theorem.

\begin{thm} There is a perfectoid space $\mathcal{S}_{K^p}$ over $C$ such that
\[
\mathcal{S}_{K^p}\sim \varprojlim_{K_p} \mathcal{S}_{K_pK^p}\ .
\]
\end{thm}

Thus, the Shimura variety becomes perfectoid as a $p$-adic analytic space when passing to the inverse limit over all levels at $p$. In fact, one needs a version of this result for the minimal compactification, cf. Theorem \ref{PerfShHodge}. One can then use results on perfectoid spaces (notably a version of the almost purity theorem) to show that the \'etale cohomology groups of the sheaf of cusp forms of infinite level can be computed by the Cech complex of an affinoid cover of (the minimal compactification of) $\mathcal{S}_{K^p}$. The outcome of this argument is Theorem \ref{CompAutomForm}, comparing the compactly supported completed cohomology groups with the Cech cohomology of the sheaf of cusp forms of infinite level. Besides the applications to Theorem \ref{InterpolationThm}, this comparison result has direct applications to vanishing results. Namely, the Cech cohomology of any sheaf vanishes above the dimension $d = \dim_\C S_K$ of the space. Thus:

\begin{thm}\label{VanThm} For $i>d$, the compactly supported completed cohomology group $\widetilde{H}_{c,K^p}^i$ vanishes.
\end{thm}

By Poincar\'e duality, this also implies that in small degrees, the (co)homology groups are small, confirming most of \cite[Conjecture 1.5]{CalegariEmerton} for Shimura varieties of Hodge type, cf. Corollary \ref{CalegariEmertonConj}.

Thus, there is a complex, whose terms are cusp forms of infinite level on affinoid subsets, which computes the compactly supported cohomology groups. To finish the proof of Theorem \ref{InterpolationThm}, one has to approximate these cusp forms of infinite level, defined on affinoid subsets, by cusp forms of finite level which are defined on the whole Shimura variety, without messing up the Hecke eigenvalues. The classical situation is that these cusp forms are defined on the ordinary locus, and one multiplies by a power of the Hasse invariant to remove all poles. The crucial property of the Hasse invariant is that it commutes with all Hecke operators away from $p$, so that it does not change the Hecke eigenvalues. Thus, we need an analogue of the Hasse invariant that works on almost arbitrary subsets of the Shimura variety. This is possible using a new period map, which forms the third key ingredient.

\begin{thm} There is a flag variety $\Fl$ with an action by $G$, and a $G(\Q_p)$-equivariant Hodge-Tate period map of adic spaces over $C$,
\[
\pi_\HT: \mathcal{S}_{K^p}\to \Fl\ ,
\]
which commutes with the Hecke operators away from $p$, and such that (some) automorphic vector bundles come via pullback from $\Fl$ along $\pi_\HT$. Moreover, $\pi_\HT$ is affine.
\end{thm}

For a more precise version, we refer to Theorem \ref{PerfShHodge}. In fact, this result is deduced from a more precise version for the Siegel moduli space (by embedding the Shimura variety into the Siegel moduli space, using that it is of Hodge type). In that case, all semisimple automorphic vector bundles come via pullback from $\Fl$, cf. Theorem \ref{ExistenceHT}. For a more detailed description of these geometric results, we refer to the introduction of Chapter \ref{SiegelChapter}. We note that the existence of $\pi_\HT$ is new even for the moduli space of elliptic curves.

In particular, the ample line bundle $\omega_{K^p}$ on $\mathcal{S}_{K^p}$ comes via pullback from $\omega_\Fl$ on $\Fl$. Any section $s\in \omega_\Fl$ pulls back to a section of $\omega_{K^p}$ on $\mathcal{S}_{K^p}$ that commutes with the Hecke operators away from $p$, and thus serves as a substitute for the Hasse invariant. As $\pi_\HT$ is affine, there are enough of these fake-Hasse invariants. In fact, in the precise version of this argument, one ends up with some integral models of the Shimura variety together with an integral model of $\omega$ (constructed in Section \ref{FormalModelSection}), such that the fake-Hasse invariants are integral sections of $\omega$, and are defined at some finite level modulo any power $p^m$ of $p$. Interestingly, these integral models are not at all related to the standard integral models of Shimura varieties: E.g., there is no family of abelian varieties above the special fibre. Perhaps this explains why the existence of these fake-Hasse invariants (or of $\pi_\HT$) was not observed before -- they are only defined at infinite level, and if one wants to approximate them modulo powers of $p$, one has to pass to a strange integral model of the Shimura variety.

Finally, let us give a short description of the content of the different chapters. In Chapter \ref{PrelimChapter}, we collect some results that will be useful later. In particular, we prove a version of Riemann's Hebbarkeitssatz for perfectoid spaces, saying roughly that bounded functions on normal perfectoid spaces extend from complements of Zariski closed subsets. Unfortunately, the results here are not as general as one could hope, and we merely manage to prove exactly what we will need later. In Chapter \ref{SiegelChapter}, which forms the heart of this paper, we prove that the minimal compactification of the Siegel moduli space becomes perfectoid in the inverse limit over all levels at $p$, and that the Hodge-Tate period map exists on it, with its various properties. In Chapter \ref{AutomorphicChapter}, we give the automorphic consequences of this result to Shimura varieties of Hodge type, as sketched above. Finally, in Chapter \ref{GaloisReprChapter}, we deduce our main results on Galois representations.

{\bf Acknowledgments}. The geometric results on Shimura varieties (i.e., that they are perfectoid in the inverse limit, and that the Hodge-Tate period map exists on them), as well as the application to Theorem \ref{VanThm}, were known to the author for some time, and he would like to thank Matthew Emerton, Michael Harris, Michael Rapoport, Richard Taylor and Akshay Venkatesh for useful discussions related to them. The possibility to apply these results to the construction of Galois representations struck the author after listening to talks of Michael Harris and Kai-Wen Lan on their joint results with Richard Taylor and Jack Thorne, and he wants to thank them heartily for the explanation of their method. The author would also like to thank the organizers of the Hausdorff Trimester Program on Geometry and Arithmetic in January -- April 2013 in Bonn, where these talks were given. These results (especially Chapter \ref{SiegelChapter}) were the basis for the ARGOS seminar of the summer term 2013 in Bonn, and the author would like to thank all participants for working through this manuscript, their very careful reading, and their suggestions for improvements. Finally, he would like to thank Bhargav Bhatt, Vincent Pilloni, Michael Spiess, Benoit Stroh, Thomas Zink and in particular the referees for very useful feedback. This work was done while the author was a Clay Research Fellow.

\renewcommand{\thesection}{\thechapter.\arabic{section}}

\chapter{Preliminaries}\label{PrelimChapter}

This chapter provides various foundational statements needed in the main part of the paper.

In Section \ref{FormalModelSection}, we recall how one can construct formal models of rigid spaces starting from suitable affinoid covers. This will be used in the proof of Theorem \ref{ThmHeckeAlgebras} to construct new formal models of Shimura varieties, on which the `fake-Hasse invariants' are defined. In the context of Lemma \ref{ConstructionFormalModel}, these are given by the sections $\bar{s}_i$.

In Section \ref{ClosedEmbeddingsSection}, we define Zariski closed embeddings of perfectoid spaces, and prove various basic properties about this notion. In fact, this notion comes in two flavours, called Zariski closed, and strongly Zariski closed, respectively, and both notions are useful later. The most important property here is that something Zariski closed in a perfectoid space is again perfectoid. This is used later to deduce that Hodge type Shimura varieties are perfectoid at infinite level, once this is known for the Siegel case. On the other hand, it will be important to know that the boundary of the Shimura variety is strongly Zariski closed. Intuitively, this says that the boundary is `infinitely ramified': One extracts lots of $p$-power roots of defining equations of the boundary.

Finally, in Section \ref{HebbarkeitssatzSection}, we prove a version of Riemann's Hebbarkeitssatz for perfectoid spaces, saying roughly that bounded functions on normal perfectoid spaces extend from complements of Zariski closed subsets. \footnote{For a version of Riemann's Hebbarkeitssatz in the setting of usual rigid geometry, see \cite{BartenwerferHebbarkeit}.} This section is, unfortunately, extremely technical, and our results are just as general as needed later. The most important use of the Hebbarkeitssatz in this paper is to show the existence of the Hodge-Tate period map. A priori, we can only construct it away from the boundary, but the Hebbarkeitssatz guarantees that it extends to the boundary. However, there is a second use of the Hebbarkeitssatz in Section \ref{ConclusionSubsection}. Here, the situation is that one wants to show that a certain space is perfectoid, by showing that it is the untilt of an (obviously perfectoid) space in characteristic $p$. This is easy to show away from boundary; to deduce the result, one needs the Hebbarkeitssatz to control the whole space in terms of the complement of the boundary.

\section{Constructing formal models from affinoid covers}\label{FormalModelSection}

Let $K$ be a complete algebraically closed nonarchimedean field with ring of integers $\OO_K$. Choose some nonzero topologically nilpotent element $\varpi\in \OO_K$. We will need the following result on constructing formal models of rigid-analytic varieties.

\begin{lem}\label{ConstructionFormalModel} Let $\mathcal{X}$ be a reduced proper rigid-analytic variety over $K$, considered as an adic space. Let $\mathcal{L}$ be a line bundle on $\mathcal{X}$. Moreover, let $\mathcal{X} = \bigcup_{i\in I} \mathcal{U}_i$ be a cover of $\mathcal{X}$ by finitely many affinoid open subsets $\mathcal{U}_i = \Spa(R_i,R_i^+)$. \footnote{Here, $R_i^+ = R_i^\circ$ is the subset of powerbounded elements.} For $J\subset I$, let $\mathcal{U}_J = \bigcap_{i\in J} \mathcal{U}_i = \Spa(R_J,R_J^+)$. Assume that on each $\mathcal{U}_i$, one has sections
\[
s_j^{(i)}\in H^0(\mathcal{U}_i,\mathcal{L})
\]
for $j\in I$, satisfying the following conditions.
\begin{altenumerate}
\item[{\rm (i)}] For all $i\in I$, $s_i^{(i)}$ is invertible, and
\[
\frac{s_j^{(i)}}{s_i^{(i)}}\in H^0(\mathcal{U}_i,\OO_\mathcal{X}^+)\ .
\]
\item[{\rm (ii)}] For all $i, j\in I$, the subset $\mathcal{U}_{ij}\subset \mathcal{U}_i$ is defined by the condition
\[
|\frac{s_j^{(i)}}{s_i^{(i)}}| = 1\ .
\]
\item[{\rm (iii)}] For all $i_1, i_2, j\in I$,
\[
|\frac{s_j^{(i_1)}}{s_{i_1}^{(i_1)}} - \frac{s_j^{(i_2)}}{s_{i_1}^{(i_1)}}|\leq |\varpi|
\]
on $\mathcal{U}_{i_1i_2}$.
\end{altenumerate}
Then for $J\subset J^\prime$, the map $\Spf R_{J^\prime}^+\to \Spf R_J^+$ is an open embedding of formal schemes, formally of finite type. Gluing them defines a formal scheme $\mathfrak{X}$ over $\OO_K$ with an open cover by $\mathfrak{U}_i = \Spf R_i^+$; define also $\mathfrak{U}_J = \bigcap_{i\in J} \mathfrak{U}_i = \Spf R_J^+$. The generic fibre of $\mathfrak{X}$ is given by $\mathcal{X}$.

Moreover, there is a unique invertible sheaf $\mathfrak{L}$ on $\mathfrak{X}$ with generic fibre $\mathcal{L}$, and such that
\[
s_j^{(i)}\in H^0(\mathfrak{U}_i,\mathfrak{L})\subset H^0(\mathcal{U}_i,\mathcal{L}) = H^0(\mathfrak{U}_i,\mathfrak{L})[\varpi^{-1}]\ ,
\]
with $s_i^{(i)}$ being an invertible section. There are unique sections
\[
\bar{s}_j\in H^0(\mathfrak{X},\mathfrak{L}/\varpi)
\]
such that for all $i\in I$, $\bar{s}_j = s_j^{(i)}\mod \varpi\in H^0(\mathfrak{U}_i,\mathfrak{L}/\varpi)$.

Furthermore, $\mathfrak{X}$ is projective, and $\mathfrak{L}$ is ample.
\end{lem}

\begin{proof} By \cite[Section 6.4.1, Corollary 5]{BoschGuentzerRemmert}, $R_J^+$ is topologically of finite type over $\OO_K$. By assumption (i), there is some $f\in R_J^+$ such that $\mathcal{U}_{J^\prime}\subset \mathcal{U}_J$ is defined by $|f|=1$. One formally checks that this implies that $R_{J^\prime}^+$ is the $\varpi$-adic completion of $R_J^+[f^{-1}]$. In particular, $\Spf R_{J^\prime}^+\to \Spf R_J^+$ is an open embedding. One gets $\mathfrak{X}$ by gluing, and its generic fibre is $\mathcal{X}$. As $\mathcal{X}$ is proper, it follows that $\mathfrak{X}$ is proper, cf. \cite[Remark 1.3.18 (ii)]{Huber}. \footnote{Note that there is only one notion of a proper rigid space, i.e. in \cite[Remark 1.3.19 (iv)]{Huber}, conditions (a) and (b) are always equivalent, not only if the nonarchimedean field is discretely valued. This follows from the main result of \cite{Temkin}.}

To define $\mathfrak{L}$, we want to glue the free sheaves $\mathfrak{L}_i = s_i^{(i)} \OO_{\mathfrak{U}_i}$ of rank $1$ on $\mathfrak{U}_i$. Certainly, $\mathfrak{L}_i$ satisfies the conditions on $\mathfrak{L}|_{\mathfrak{U}_i}$, and is the unique such invertible sheaf on $\mathfrak{U}_i$. To show that they glue, we need to identify $\mathfrak{L}_i|_{\mathfrak{U}_{ij}}$ with $\mathfrak{L}_j|_{\mathfrak{U}_{ij}}$. By (ii), $\mathfrak{L}_i|_{\mathfrak{U}_{ij}}$ is freely generated by $s_j^{(i)}$. Also, by (iii) (applied with $i_1=j$, $i_2=i$), $\mathfrak{L}_j|_{\mathfrak{U}_{ij}}$ is freely generated by $s_j^{(i)}$, giving the desired equality.

To show that there are the sections $\bar{s}_j\in H^0(\mathfrak{X},\mathfrak{L}/\varpi)$, we need to show that for all $i_1$, $i_2$ and $j$,
\[
s_j^{(i_1)}\equiv s_j^{(i_2)}\mod \varpi\in H^0(\mathfrak{U}_{i_1i_2},\mathfrak{L}/\varpi)\ .
\]
Dividing by $s_{i_1}^{(i_1)}$ translates this into condition (iii).

It remains to prove that $\mathfrak{L}$ is ample. For this, it is enough to prove that $\mathfrak{L}/\varpi$ is ample on $\mathfrak{X}\times_{\Spf \OO_K} \Spec \OO_K/\varpi$. Here, the affine complements $\mathfrak{U}_i\times_{\Spf \OO_K} \Spec \OO_K/\varpi$ of the vanishing loci of the sections $\bar{s}_i$ cover, giving the result.
\end{proof}

We will need a complement on this result, concerning ideal sheaves.

\begin{lem}\label{ConstructionFormalModelIdeal} Assume that in the situation of Lemma \ref{ConstructionFormalModel}, one has a coherent ideal sheaf $\mathcal{I}\subset \mathcal{O}_\mathcal{X}$. Then the association
\[
\mathfrak{U}_i\mapsto H^0(\mathcal{U}_i,\mathcal{I}\cap \OO_\mathcal{X}^+)
\]
extends uniquely to a coherent $\OO_\mathfrak{X}$-module $\mathfrak{I}$, with generic fibre $\mathcal{I}$.
\end{lem}

\begin{proof} From \cite[Lemma 1.2 (c), Proposition 1.3]{BoschLuetkebohmert1}, it follows that $H^0(\mathcal{U}_i,\mathcal{I}\cap \OO_\mathcal{X}^+)$ is a coherent $R_i^+$-module. One checks that as $\mathcal{U}_{J^\prime}\subset \mathcal{U}_J$ for $J\subset J^\prime$ is defined by the condition $|f|=1$ for some $f\in R_J^+$, $H^0(\mathcal{U}_{J^\prime},\mathcal{I}\cap \OO_\mathcal{X}^+)$ is given as the $\varpi$-adic completion of $H^0(\mathcal{U}_J,\mathcal{I}\cap \OO_\mathcal{X}^+)[f^{-1}]$. Thus, these modules glue to give the desired coherent $\OO_\mathfrak{X}$-module $\mathfrak{I}$. From the definition, it is clear that the generic fibre of $\mathfrak{I}$ is $\mathcal{I}$.
\end{proof}

\section{Closed Embeddings of perfectoid spaces}\label{ClosedEmbeddingsSection}

Let $K$ be a perfectoid field with tilt $K^\flat$. Fix some element $0\neq \varpi^\flat\in K^\flat$ with $|\varpi^\flat|<1$, and set $\varpi = (\varpi^\flat)^\sharp\in K$. Let $\mathcal{X}=\Spa(R,R^+)$ be an affinoid perfectoid space over $K$.

\begin{definition} A subset $Z\subset |\mathcal{X}|$ is Zariski closed if there is an ideal $I\subset R$ such that
\[
Z = \{x\in \mathcal{X}\mid |f(x)| = 0\ \mathrm{for\ all}\ f\in I\}\ .
\]
\end{definition}

\begin{lem}\label{ZariskiClosed} Assume that $Z\subset \mathcal{X}$ is Zariski closed. There is a universal perfectoid space $\mathcal{Z}$ over $K$ with a map $\mathcal{Z}\to \mathcal{X}$ for which $|\mathcal{Z}|\to |\mathcal{X}|$ factors over $Z$. The space $\mathcal{Z}=\Spa(S,S^+)$ is affinoid perfectoid, the map $R\to S$ has dense image, and the map $|\mathcal{Z}|\to Z$ is a homeomorphism.
\end{lem}

As the proof uses some almost mathematics, let us recall that an $\OO_K$-module $M$ is almost zero if it is killed by the maximal ideal $\mathfrak{m}_K$ of $\OO_K$. The category of almost $\OO_K$-modules, or $\OO_K^a$-modules, is by definition the quotient of the category of $\OO_K$-modules by the thick subcategory of almost zero modules. There are two functors from $\OO_K^a$-modules to $\OO_K$-modules, right and left adjoint to the forgetful functor. The first is
\[
M\mapsto M_\ast = \Hom_{\OO_K^a}(\OO_K^a,M)\ ,
\]
and the second is
\[
M\mapsto M_! = \mathfrak{m}_K\otimes_{\OO_K} M_\ast\ .
\]
The existence of left and right adjoints implies that the forgetful functor $N\mapsto N^a$ commutes with all limits and colimits. For this reason, we are somewhat sloppy in the following on whether we take limits and colimits of actual modules or almost modules, if we are only interested in an almost module in the end. For more discussion of almost mathematics, cf. \cite{GabberRamero} and, for a very brief summary, \cite[Section 4]{ScholzePerfectoid}.

\begin{proof} One can write $Z\subset |\mathcal{X}|$ as an intersection $Z = \bigcap_{Z\subset U} U$ of all rational subsets $U\subset X$ containing $Z$. Indeed, for any $f_1,\ldots,f_k\in I$, one has the rational subset
\[
U_{f_1,\ldots,f_k} = \{x\in \mathcal{X}\mid |f_i(x)|\leq 1\ ,\ i=1,\ldots,k\}\ ,
\]
and $Z$ is the intersection of these subsets. Any $U_{f_1,\ldots,f_k}=\Spa(R_{f_1,\ldots,f_k},R_{f_1,\ldots,f_k}^+)$ is affinoid perfectoid, where
\[
R_{f_1,\ldots,f_k} = R\langle T_1,\ldots,T_k\rangle/\overline{(T_i - f_i)}\ .
\]
In particular, $R\to R_{f_1,\ldots,f_k}$ has dense image. Let $S^+$ be the $\varpi$-adic completion of $\varinjlim R_{f_1,\ldots,f_k}^+$; thus,
\[
S^+/\varpi = \varinjlim R_{f_1,\ldots,f_k}^+/\varpi\ .
\]
It follows that Frobenius induces an almost isomorphism $(S^+/\varpi^{1/p})^a\cong (S^+/\varpi)^a$, so that $(S^+/\varpi)^a$ is a perfectoid $(\OO_K/\varpi)^a$-algebra. Thus, $S^{+a}$ is a perfectoid $\OO_K^a$-algebra, and $S=S^+[\varpi^{-1}]$ is a perfectoid $K$-algebra, cf. \cite[Section 5]{ScholzePerfectoid}. Let $\mathcal{Z} = \Spa(S,S^+)$. All properties are readily deduced.
\end{proof}

\begin{rem} More precisely, for any affinoid $K$-algebra $(T,T^+)$ for which $T^+\subset T$ is bounded, and any map $(R,R^+)\to (T,T^+)$ for which $\Spa(T,T^+)\to \Spa(R,R^+)$ factors over $Z$, there is a unique factorization $(R,R^+)\to (S,S^+)\to (T,T^+)$. This follows directly from the proof, using that $T^+$ is bounded in proving that the map
\[
\varinjlim R_{f_1,\ldots,f_k}^+\to T^+
\]
extends by continuity to the $\varpi$-adic completion $S^+$.
\end{rem}

We will often identify $Z=\mathcal{Z}$, and say that $\mathcal{Z}\to \mathcal{X}$ is a (Zariski) closed embedding.

\begin{rem} We caution the reader that in general, the map $R\to S$ is not surjective. For an example, let $R=K\langle T^{1/p^\infty}\rangle$ for some $K$ of characteristic $0$, and look at the Zariski closed subset defined by $I=(T-1)$.
\end{rem}

\begin{lem}\label{StronglyZariskiClosedCharP} Assume that $K$ is of characteristic $p$, and that $\mathcal{Z}=\Spa(S,S^+)\to \mathcal{X}=\Spa(R,R^+)$ is a closed embedding. Then the map $R^+\to S^+$ is almost surjective. (In particular, $R\to S$ is surjective.)
\end{lem}

\begin{proof} One can reduce to the case that $\mathcal{Z}$ is defined by a single equation $f=0$, for some $f\in R$. One may assume that $f\in R^+$. Consider the $K^{\circ a}/\varpi$-algebra
\[
A = R^{\circ a}/(\varpi,f,f^{1/p},f^{1/p^2},\ldots)\ .
\]
We claim that $A$ is a perfectoid $K^{\circ a}/\varpi$-algebra. To show that it is flat over $K^{\circ a}/\varpi$, it is enough to prove that
\[
R^{\circ a}/(f,f^{1/p},f^{1/p^2},\ldots)
\]
is flat over $K^{\circ a}$, i.e. has no $\varpi$-torsion. Thus, assume some element $g\in R^\circ$ satisfies $\varpi g = f^{1/p^m} h$ for some $m\geq 0$, $h\in R^\circ$. Then we have
\[
\varpi^{1/p^n} g = (\varpi g)^{1/p^n} g^{1-1/p^n} = f^{1/p^{m+n}} h^{1/p^n} g^{1-1/p^n}\in (f,f^{1/p},f^{1/p^2},\ldots)\ .
\]
Thus, $g$ is almost zero in $R^\circ / (f,f^{1/p},f^{1/p^2},\ldots)$, as desired. That Frobenius induces an isomorphism $A/\varpi^{1/p}\cong A$ is clear.

Thus, $A$ lifts uniquely to a perfectoid $K^{\circ a}$-algebra $T^{\circ a}$ for some perfectoid $K$-algebra $T$. The map $R^{\circ a}\to T^{\circ a}$ is surjective by construction. Also, $f$ maps to $0$ in $T$. Clearly, the map $R^{\circ a}/\varpi\to S^{\circ a}/\varpi$ factors over $A$; thus, $R\to S$ factors over $T$. Let $T^+\subset T$ be the integral closure of the image of $R^+$; then $\Spa(T,T^+)\to \Spa(R,R^+)$ factors over $Z$ (as $f$ maps to $0$ in $T$), giving a map $(S,S^+)\to (T,T^+)$ by the universal property. The two maps between $S$ and $T$ are inverse; thus, $S=T$. Almost surjectivity of $R^+\to S^+$ is equivalent to surjectivity of $R^{\circ a}\to S^{\circ a} = T^{\circ a}$, which we have just verified.
\end{proof}

\begin{definition} A map $\mathcal{Z}=\Spa(S,S^+)\to \mathcal{X}=\Spa(R,R^+)$ is strongly Zariski closed if the map $R^+\to S^+$ is almost surjective.
\end{definition}

Of course, something strongly Zariski closed is also Zariski closed (defined by the ideal $I=\ker(R\to S)$).

\begin{lem}\label{StronglyZariskiClosedTilting} A map $\mathcal{Z}=\Spa(S,S^+)\to \mathcal{X}=\Spa(R,R^+)$ is strongly Zariski closed if and only if the map of tilts $\mathcal{Z}^\flat\to \mathcal{X}^\flat$ is strongly Zariski closed.
\end{lem}

\begin{proof} The map $R^+\to S^+$ is almost surjective if and only if $R^+/\varpi\to S^+/\varpi$ is almost surjective. Under tilting, this is the same as the condition that $R^{\flat +}/\varpi^\flat\to S^{\flat +}/\varpi^\flat$ is almost surjective, which is equivalent to $R^{\flat +}\to S^{\flat +}$ being almost surjective.
\end{proof}

By Lemma \ref{StronglyZariskiClosedCharP}, Zariski closed implies strongly Zariski closed in characteristic $p$. Thus, a Zariski closed map in characteristic $0$ is strongly Zariski closed if and only if the tilt is still Zariski closed. For completeness, let us mention the following result that appears in the work of Kedlaya-Liu, \cite{KedlayaLiu1}; we will not need this result in our work.

\begin{lem}[{\cite[Proposition 3.6.9 (c)]{KedlayaLiu1}}] Let $R\to S$ be a surjective map of perfectoid $K$-algebras. Then $R^\circ\to S^\circ$ is almost surjective.
\end{lem}

In other words, for any rings of integral elements $R^+\subset R$, $S^+\subset S$ for which $R^+$ maps into $S^+$, the map $R^+\to S^+$ is almost surjective. Finally, let us observe some statements about pulling back closed immersions.

\begin{lem}\label{ClosedImmersionPullback} Let
\[\xymatrix{
\mathcal{Z}^\prime = \Spa(S^\prime,S^{\prime +})\ar[d]\ar[r] & \mathcal{X}^\prime = \Spa(R^\prime,R^{\prime +})\ar[d]\\
\mathcal{Z} = \Spa(S,S^+)\ar[r] & \mathcal{X} = \Spa(R,R^+)
}\]
be a pullback diagram of affinoid perfectoid spaces (recalling that fibre products always exist, cf. \cite[Proposition 6.18]{ScholzePerfectoid}).
\begin{altenumerate}
\item[{\rm (i)}] If $\mathcal{Z}\to \mathcal{X}$ is Zariski closed, defined by an ideal $I\subset R$, then so is $\mathcal{Z}^\prime\to \mathcal{X}^\prime$, defined by the ideal $IR^\prime\subset R^\prime$.
\item[{\rm (ii)}] If $\mathcal{Z}\to \mathcal{X}$ is strongly Zariski closed, then so is $\mathcal{Z}^\prime\to \mathcal{X}^\prime$. Moreover, if we define $I^+=\ker(R^+\to S^+)$, $I^{\prime +} = \ker(R^{\prime +}\to S^{\prime +})$, then the map
\[
I^+/\varpi^n\otimes_{R^+/\varpi^n} R^{\prime +}/\varpi^n\to I^{\prime +}/\varpi^n
\]
is almost surjective for all $n\geq 0$.
\end{altenumerate}
\end{lem}

\begin{proof} Part (i) is clear from the universal property. For part (ii), observe that (by the proof of \cite[Proposition 6.18]{ScholzePerfectoid}) the map
\[
S^+/\varpi\otimes_{R^+/\varpi} R^{\prime +}/\varpi\to S^{\prime +}/\varpi
\]
is an almost isomorphism. Thus, if $R^+/\varpi\to S^+/\varpi$ is almost surjective, then so is $R^{\prime +}/\varpi\to S^{\prime +}/\varpi$, showing that if $\mathcal{Z}\to \mathcal{X}$ is strongly Zariski closed, then so is $\mathcal{Z}^\prime\to \mathcal{X}^\prime$. For the result about ideals, one reduces to $n=1$. Now, tensoring the almost exact sequence
\[
0\to I^+/\varpi\to R^+/\varpi\to S^+/\varpi\to 0
\]
with $R^{\prime +}/\varpi$ over $R^+/\varpi$ gives the desired almost surjectivity.
\end{proof}

\section{A Hebbarkeitssatz for perfectoid spaces}\label{HebbarkeitssatzSection}

\subsection{The general result}

Let $K$ be a perfectoid field of characteristic $p$, with ring of integers $\OO_K\subset K$, and maximal ideal $\mm_K\subset \OO_K$. Fix some nonzero element $t\in \mm_K$; then $\mm_K = \bigcup_n t^{1/p^n} \OO_K$. Let $(R,R^+)$ be a perfectoid affinoid $K$-algebra, with associated affinoid perfectoid space $\mathcal{X}=\Spa(R,R^+)$. Fix an ideal $I\subset R$, with $I^+=I\cap R^+$. Let $\mathcal{Z}=V(I)\subset \mathcal{X}$ be the associated Zariski closed subset of $\mathcal{X}$.

Recall the following lemma, which holds true for any adic space over a nonarchimedean field.

\begin{lem}\label{StalkOplusmodp} The stalk of $\OO_\mathcal{X}^+/t$ at a point $x\in \mathcal{X}$ with completed residue field $k(x)$ and valuation subring $k(x)^+\subset k(x)$, is given by $k(x)^+/t$.
\end{lem}

In particular, if $\OO_{k(x)}\subset k(x)$ denotes the powerbounded elements (so that $k(x)^+\subset \OO_{k(x)}$ is an almost equality), then the stalk of $(\OO_{\mathcal{X}}^+/t)^a$ is given by $\OO_{k(x)}^a/t$.

\begin{proof} There is a map $\OO_{\mathcal{X},x}^+\to k(x)^+$ with a dense image. To prove the lemma, one has to see that if $I$ denotes the kernel of this map, then $I/t=0$. If $f\in I$, then $f\in \OO_{\mathcal{X}}^+(U)$ for some neighborhood $U$ of $x$, with $f(x)=0$. Then $|f|\leq |t|$ defines a smaller neighborhood $V$ of $x$, on which $f$ becomes divisible by $t$ as an element of $\OO_{\mathcal{X}}^+(V)$.
\end{proof}

\begin{prop}\label{GeneralHebbarkeit} There is a natural isomorphism of almost $\OO_K$-modules
\[
\Hom_{R^+} (I^{+ 1/p^\infty},R^+/t)^a\cong H^0(\mathcal{X}\setminus \mathcal{Z},\OO_\mathcal{X}^+/t)^a\ .
\]
For any point $x\in \mathcal{X}\setminus \mathcal{Z}$, this isomorphism commutes with evaluation $H^0(\mathcal{X}\setminus \mathcal{Z},\OO_\mathcal{X}^+/t)^a\to \OO_{k(x)}^a/t$ at $x$, where $k(x)$ is the completed residue field of $\mathcal{X}$ at $x$:
\[
\Hom_{R^+} (I^{+ 1/p^\infty},R^+/t)^a \to \Hom_{\OO_{k(x)}} (I_{k(x)}^{+ 1/p^\infty},\OO_{k(x)}/t)^a = \OO_{k(x)}^a/t\ .
\]
\end{prop}

\begin{rem} Recall that the global sections of $(\OO_\mathcal{X}^+/t)^a$ are $(R^+/t)^a$ (cf. \cite[Theorem 6.3 (iii), (iv)]{ScholzePerfectoid}). Also, note that if $x\in \mathcal{X}\setminus \mathcal{Z}$, then the image $I_{k(x)}^+ \subset \OO_{k(x)}$ of $I^+$ is not the zero ideal, so that $I_{k(x)}^{+ 1/p^\infty}$ is almost equal to $\OO_{k(x)}$. Finally, the requirement of the lemma pins down the map uniquely, as for any sheaf $\mathcal{F}$ of almost $\OO_K$-modules on a space $Y$, the map $H^0(\mathcal{F})\to \prod_{y\in Y} \mathcal{F}_y$ is injective.
\end{rem}

\begin{proof} Assume first that $I$ is generated by an element $f\in R$; we may assume that $f\in R^+$. In that case, $I^{+ 1/p^\infty}$ is almost equal to $f^{1/p^\infty} R^+ = \bigcup_n f^{1/p^n} R^+$. Indeed, one has to see that the cokernel of the inclusion $f^{1/p^\infty} R^+ \to I^{+ 1/p^\infty}$ is killed by $\mm_K$. Take any element $g\in I^{+ 1/p^\infty}$; then $g\in I^{+ 1/p^m}$ for some $m$, so $g=f^{1/p^m} h$ for some $h\in R$. There is some $n$ such that $t^n h\in R^+$. For all $k\geq 0$, we have
\[
t^{n/p^k} g = t^{n/p^k} g^{1/p^k} g^{1-1/p^k} = f^{1/p^{m+k}} (t^n h)^{1/p^k} g^{1-1/p^k}\in f^{1/p^{m+k}} R^+\ ,
\]
giving the result.

Thus, we have to see that
\[
\Hom_{R^+}(f^{1/p^\infty} R^+,R^+/t)^a\cong H^0(\mathcal{X}\setminus V(f),\OO_\mathcal{X}^+/t)^a\ .
\]
Consider the rational subsets $\mathcal{U}_n = \{x\in \Spa(R,R^+)\mid |f(x)|\geq |t|^n\}\subset \mathcal{X}$; then $\mathcal{X}\setminus V(f) = \bigcup_n \mathcal{U}_n$. Moreover, by \cite[Lemma 6.4 (i)]{ScholzePerfectoid} (and its proof), one has
\[
H^0(\mathcal{U}_n,\OO_\mathcal{X}^+/t)^a\cong R^+/t[u_n^{1/p^\infty}] / (\forall m: u_n^{1/p^m} f^{1/p^m} - t^{n/p^m})^a\ .
\]
Let
\[
S_n = R^+/t[u_n^{1/p^\infty}] / (\forall m: u_n^{1/p^m} f^{1/p^m} - t^{n/p^m})\ ,
\]
and let $S_n^{(k)}\subset S_n$ be the $R^+$-submodule generated by $u_n^i$ for $i\leq 1/p^k$. One gets maps
\[
f^{1/p^k}: S_n^{(k)}\to \im(R^+/t\to S_n)\ ,
\]
as $f^{1/p^k} u_n^i = f^{1/p^k - i} f^i u_n^i = f^{1/p^k - i} t^{ni}$ for all $i\leq 1/p^k$. Also, we know that
\[
H^0(\mathcal{X}\setminus V(f),\OO_\mathcal{X}^+/t)^a = \varprojlim_n H^0(\mathcal{U}_n,\OO_\mathcal{X}^+/t)^a = \varprojlim_n S_n^a\ ,
\]
and direct inspection shows that for fixed $n$ and $k$, the map $S_{n^\prime}\to S_n$ factors over $S_n^{(k)}$ for $n^\prime$ large enough. It follows that
\[
\varprojlim_n S_n = \varprojlim_n S_n^{(k)}
\]
for any $k$, and then also
\[
\varprojlim_n S_n = \varprojlim_{n,k} S_n^{(k)}\ .
\]
Via the maps $f^{1/p^k}: S_n^{(k)}\to \im(R^+/t\to S_n)$, one gets a map of inverse systems (in $n$ and $k$)
\[
S_n^{(k)}\to \im(R^+/t\to S_n)\ ,
\]
where on the right, the transition maps are given by multiplication by $f^{1/p^k - 1/p^{k^\prime}}$. The kernel and cokernel of this map are killed by $f^{1/p^k}$, and thus by $t^{n/p^k} = f^{1/p^k} u_n^{1/p^k}$. Taking the inverse limit over both $n$ and $k$ (the order does not matter, so we may first take it over $k$, and then over $n$), one sees that the two inverse limits are almost the same.

On the other hand, there is a map of inverse systems (in $n$ and $k$),
\[
R^+/t\to \im(R^+/t\to S_n)\ .
\]
Again, transition maps are multiplication by $f^{1/p^k - 1/p^{k^\prime}}$. Clearly, the maps are surjective. Assume $a\in \ker(R^+/t\to S_n)$. Then, for any sufficiently large $m$, one can write
\[
a = (u_n^{1/p^m} f^{1/p^m} - t^{n/p^m})\sum_i a_i u_n^i
\]
in $R^+/t[u_n^{1/p^\infty}]$. Comparing coefficients, we find that
\[
a=-t^{n/p^m} a_0\ ,\ f^{1/p^m} a = - t^{2n/p^m} a_{1/p^m}\ ,\ \ldots\ ,\ f^{\ell/p^m} a = - t^{(\ell+1)n/p^m} a_{\ell/p^m}
\]
for all $\ell\geq 0$. In particular, $f^{\ell/p^m} a\in t^{n\ell/p^m} R^+/t$ for all $\ell,m\geq 0$ (a priori only for $m$ large enough, but this is enough). Assume that $n=p^{n^\prime}$ is a power of $p$ (which is true for a cofinal set of $n$); then, setting $\ell=1$, $m=n^\prime$, we find that $f^{1/n} a = 0\in R^+/t$.

As the transition maps are given by $f^{1/p^k - 1/p^{k^\prime}}$, and we may take the inverse limit over $k$ and $n$ also as the inverse limit over the cofinal set of $(k,n)$ with $n=p^{k+1}$, one finds that the kernel of the map of inverse systems
\[
R^+/t\to \im(R^+/t\to S_n)
\]
is Mittag-Leffler zero. It follows that the inverse limits are the same. Finally, we have a map
\[
\Hom_{R^+}(f^{1/p^\infty} R^+,R^+/t)\to \varprojlim_k R^+/t\,
\]
given by evaluating a homomorphism on the elements $f^{1/p^k}\in f^{1/p^\infty} R^+$. Let $M$ be the direct limit of $R^+$ along multiplication by $f^{1/p^k - 1/p^{k^\prime}}$; then
\[
\varprojlim_k R^+/t = \Hom_{R^+}(M,R^+/t)\ ;
\]
it is enough to see that the surjective map $M\to f^{1/p^\infty} R^+$ (sending $1$ in the $k$-th term $R^+$ to $f^{1/p^k}$) is injective. For this, if $a\in \ker(f^{1/p^k}: R^+\to R^+)$, then by perfectness of $R^+$, also $a^{1/p} f^{1/p^{k+1}} = 0$, so $a f^{1/p^{k+1}} = 0$; it follows that $a$ gets mapped to $0$ in the direct limit $M$.

This handles the case that $I$ is generated by a single element. The general case follows: Filtering $I$ by its finitely generated submodules, we reduce to the case that $I$ is finitely generated. Thus, assume that $I=I_1 + I_2$, where $I_1$ is principal, and $I_2$ is generated by fewer elements. Clearly,
\[
\mathcal{X}\setminus V(I) = (\mathcal{X}\setminus V(I_1))\cup (\mathcal{X}\setminus V(I_2))\ ,
\]
and
\[
(\mathcal{X}\setminus V(I_1))\cap (\mathcal{X}\setminus V(I_2)) = (\mathcal{X}\setminus V(I_1I_2))\ .
\]
By using the sheaf property, one computes $H^0(\mathcal{X}\setminus V(I),\OO_\mathcal{X}^+/t)^a$ in terms of the others. Note that by induction, we may assume that the result is known for $I_1$, $I_2$ and $I_1I_2$. Also,
\[
0\to (I_1I_2)^{+ 1/p^\infty}\to I_1^{+ 1/p^\infty}\oplus I_2^{+ 1/p^\infty}\to (I_1+I_2)^{+ 1/p^\infty}\to 0
\]
is almost exact. Injectivity at the first step is clear. If $(f,g)$ lies in the kernel of the second map, then $f=g\in R^+$, and $f=f^{1/p} g^{(p-1)/p}\in (I_1I_2)^{+ 1/p^\infty}$, showing exactness in the middle. If $h\in (I_1+I_2)^{+ 1/p^\infty}$, then we may write $h=f+g$ for certain $f\in I_1^{1/p^\infty}$, $g\in I_2^{1/p^\infty}$. After multiplying by a power $t^k$ of $t$, $t^kf, t^kg\in R^+$. But then also
\[
t^{k/p^m} h = (t^k h)^{1/p^m} h^{1-1/p^m} = (t^k f)^{1/p^m} h^{1-1/p^m} + (t^k g)^{1/p^m} h^{1-1/p^m}\in I_1^{+ 1/p^\infty} + I_2^{+ 1/p^\infty}\ .
\]
This gives almost exactness of
\[
0\to (I_1I_2)^{+ 1/p^\infty}\to I_1^{+ 1/p^\infty}\oplus I_2^{+ 1/p^\infty}\to (I_1+I_2)^{+ 1/p^\infty}\to 0\ ,
\]
and applying $\Hom_{R^+}(-,R^+/t)^a$ will then give the result.
\end{proof}

For applications, the following lemma is useful.

\begin{lem}\label{PerfectIdealBC} Let $R$ be a perfectoid $K$-algebra, $I\subset R$ an ideal, and $R^\prime$ a perfectoid $K$-algebra, with a map $R\to R^\prime$; let $I^\prime = IR^\prime$. Then
\[
I^{+ 1/p^\infty}\otimes_{R^+} R^{\prime+}\to I^{\prime + 1/p^\infty}
\]
is an almost isomorphism.
\end{lem}

\begin{proof} Writing $I$ as the filtered direct limit of its finitely generated submodules, one reduces to the case that $I$ is finitely generated. Arguing by induction on the minimal number of generators of $I$ as at the end of the proof of the previous proposition, one reduces further to the case that $I$ is principal, generated by some element $0\neq f\in R^+$. In that case, $I^{+ 1/p^\infty}$ is almost the same as $f^{1/p^\infty} R^+$, which is the same as $\varinjlim R^+$, where the transition maps are given by $f^{1/p^k - 1/p^{k+1}}$, cf. the description of $M$ in the previous proof. The same applies for $I^{\prime + 1/p^\infty}$, and the latter description obviously commutes with base-change.
\end{proof}

\subsection{A special case}

There will be a certain situation where we want to apply the Hebbarkeitssatz (and where it takes its usual form saying that anything extends uniquely from $\mathcal{X}\setminus \mathcal{Z}$ to $\mathcal{X}$). Let $A_0$ be normal, integral and of finite type over $\F_p$ and let $0\neq f\in A_0$. Let $K=\F_p((t^{1/p^\infty}))$, let $S=A_0^{1/p^\infty}\hat{\otimes}_{\F_p} K$ be the associated perfectoid $K$-algebra, and let $S^+ = S^\circ = A_0^{1/p^\infty}\hat{\otimes}_{\F_p} \OO_K$. Then $(S,S^+)$ is a perfectoid affinoid $K$-algebra, and let $\mathcal{Y}=\Spa(S,S^+)$.

Inside $\mathcal{Y}$, consider the open subset $\mathcal{X}=\{y\in \mathcal{Y}\mid |f(y)|\geq |t|\}$, and let $(R,R^+) = (\OO_\mathcal{Y}(\mathcal{X}),\OO_\mathcal{Y}^+(\mathcal{X}))$. Note that
\[
R^{+ a}/t \cong (A_0^{1/p^\infty}\otimes_{\F_p} \OO_K/t)[u^{1/p^\infty}]/(\forall m: u^{1/p^m} f^{1/p^m} - t^{1/p^m})^a\ .
\]
Finally, fix an ideal $0\neq I_0\subset A_0$, let $I = I_0 R$, and let $\mathcal{Z}=V(I)\subset \mathcal{X}$ be the associated closed subset of $\mathcal{X}$.

In the application, $\Spec A_0$ will be an open subset of the minimal compactification of the Siegel moduli space, the element $f$ will be the Hasse invariant, and the ideal $I_0$ will be the defining ideal of the boundary.

In this situation, Riemann's Hebbarkeitssatz holds true, at least under a hypothesis on resolution of singularities. In the application, this exists by the theory of the toroidal compactification.

\begin{cor}\label{SpecialHebbarkeit} Assume that $\Spec A_0$ admits a resolution of singularities, i.e. a proper birational map $T\to \Spec A_0$ such that $T$ is smooth over $\F_p$. Then the map
\[
H^0(\mathcal{X},\OO_\mathcal{X}^+/t)^a\to H^0(\mathcal{X}\setminus \mathcal{Z},\OO_\mathcal{X}^+/t)^a
\]
is an isomorphism of almost $\OO_K$-modules.
\end{cor}

\begin{proof} Arguing as at the end of the proof of Proposition \ref{GeneralHebbarkeit}, we may assume that $I$ is generated by one element $0\neq g\in A_0$. We have to show that the map
\[
R^{+ a}/t\to \Hom_{R^+}(g^{1/p^\infty} R^+,R^+/t)^a
\]
is an isomorphism. Note that we may rewrite
\[
R^{+ a}/t = (A_0^{1/p^\infty}\otimes_{\F_p} \F_p[t^{1/p^\infty}]/t)[u^{1/p^\infty}]/(\forall m: u^{1/p^m} f^{1/p^m} - t^{1/p^m})^a\cong A_0^{1/p^\infty}[u^{1/p^\infty}]/(uf)^a\ ;
\]
thus, we may replace $R^+/t$ by $A=A_0^{1/p^\infty}[u^{1/p^\infty}]/(uf)$, with the almost structure given by $(u^{1/p^\infty} f^{1/p^\infty})$. Also,
\[\begin{aligned}
\Hom_{R^+}(g^{1/p^\infty} R^+,R^+/t)^a&\cong \Hom_{R^+/t}(g^{1/p^\infty} R^+/t,R^+/t)^a\\
&\cong \Hom_A(g^{1/p^\infty} A,A)^a\cong \Hom_{A_0^{1/p^\infty}}(g^{1/p^\infty} A_0^{1/p^\infty},A)^a\ .
\end{aligned}\]
In the last step, we use that the kernel of the surjective map $g^{1/p^\infty} A_0^{1/p^\infty}\otimes_{A_0^{1/p^\infty}} A\to g^{1/p^\infty} A$ is almost zero. Given the formula
\[
A = \bigoplus_{0\leq i<1, i\in \Z[1/p]} A_0^{1/p^\infty} \cdot u^i\oplus \bigoplus_{i\geq 1, i\in \Z[1/p]} \left(A_0^{1/p^\infty}/f\right)\cdot u^i\ ,
\]
this reduces to showing that the kernel of
\[
g^{1/p^\infty} A_0^{1/p^\infty}/fg^{1/p^\infty} A_0^{1/p^\infty}\to A_0^{1/p^\infty}/fA_0^{1/p^\infty}
\]
is almost zero with respect to the ideal generated by all $f^{1/p^m}$, $m\geq 0$. But if $a\in g^{1/p^\infty} A_0^{1/p^\infty}$ is of the form $a=fb$ for some $b\in A_0^{1/p^\infty}$, then
\[
f^{1/p^m} a = f^{1/p^m} a^{1-1/p^m} a^{1/p^m} =  f^{1/p^m} f^{1-1/p^m} b^{1-1/p^m} a^{1/p^m} = f a^{1/p^m} b^{1-1/p^m} \in fg^{1/p^\infty} A_0^{1/p^\infty}\ ,
\]
whence the claim.

It remains to see that the map
\[
A\to \Hom_{A_0^{1/p^\infty}}(g^{1/p^\infty} A_0^{1/p^\infty},A)
\]
is almost an isomorphism. Again, using the explicit formula for $A$, and using the basis given by $u^i$, this reduces to the following lemma.
\end{proof}

\begin{lem} Let $A_0$ be normal, integral and of finite type over $\F_p$, such that $\Spec A_0$ admits a resolution of singularities. Let $0\neq f,g\in A_0$. Then the two maps
\[
A_0^{1/p^\infty}\to \Hom_{A_0^{1/p^\infty}}(g^{1/p^\infty} A_0^{1/p^\infty},A_0^{1/p^\infty})\ ,\ A_0^{1/p^\infty}/f\to \Hom_{A_0^{1/p^\infty}}(g^{1/p^\infty} A_0^{1/p^\infty},A_0^{1/p^\infty}/f)
\]
are almost isomorphisms with respect to the ideal generated by all $f^{1/p^m}$, $m\geq 0$.
\end{lem}

\begin{rem} In fact, the first map is an isomorphism, and the second map injective, without assuming resolution of singularities for $\Spec A_0$. Resolution of singularities is only needed to show that the second map is almost surjective. It may be possible to remove the assumption of resolution of singularities by using de Jong's alterations.
\end{rem}

\begin{proof} First, as $A_0$, and thus $A_0^{1/p^\infty}$ is a domain, the map
\[
A_0^{1/p^\infty}\to \Hom_{A_0^{1/p^\infty}}(g^{1/p^\infty} A_0^{1/p^\infty},A_0^{1/p^\infty})
\]
is injective, and the right-hand side injects into $A_0^{1/p^\infty}[g^{-1}]\subset L$, where $L$ is the quotient field of $A_0^{1/p^\infty}$. If $x\in L$ lies in the image of the right-hand side, then $g^{1/p^n} x\in A_0^{1/p^\infty}$ for all $n$. As $A_0$ is normal and noetherian, one can check whether $x\in A_0^{1/p^\infty}$ by looking at rank-$1$-valuations. If $x$ would not lie in $A_0^{1/p^\infty}$, then there would be some rank-$1$-valuation taking absolute value $>1$ on $x$; then for $n$ sufficiently large, also $g^{1/p^n} x$ has absolute value $>1$, which contradicts $g^{1/p^n} x\in A_0^{1/p^\infty}$. Thus,
\[
A_0^{1/p^\infty}\cong \Hom_{A_0^{1/p^\infty}}(g^{1/p^\infty} A_0^{1/p^\infty},A_0^{1/p^\infty})\ .
\]
In particular, it follows that
\[
A_0^{1/p^\infty}/f\hookrightarrow \Hom_{A_0^{1/p^\infty}}(g^{1/p^\infty} A_0^{1/p^\infty},A_0^{1/p^\infty}/f)\ .
\]

Now assume first that $A_0$ is smooth. Then $A_0^{1/p}$ is a flat $A_0$-module; it follows that $A_0^{1/p^\infty}$ is a flat $A_0$-module. First, observe that for any $A_0$-module $M$ and $0\neq x\in M$, there is some $n$ such that $0\neq g^{1/p^n} x\in A_0^{1/p^\infty}\otimes_{A_0} M$. Indeed, assume not; then replacing $M$ by the submodule generated by $x$ (and using flatness of $A_0^{1/p^\infty}$), we may assume that $M=A_0/J$ for some ideal $J\subset A_0$, and that $g^{1/p^n}: M=A_0/J\to A_0^{1/p^\infty}\otimes_{A_0} M=A_0^{1/p^\infty}/JA_0^{1/p^\infty}$ is the zero map for all $n\geq 0$. This implies that $g\in J^{p^n}A_0^{1/p^\infty}\cap A_0 = J^{p^n}$ (by flatness of $A_0^{1/p^\infty}$ over $A_0$) for all $n\geq 0$, so that $g=0$ by the Krull intersection theorem, as $A_0$ is a domain, which is a contradiction.

In particular, we find that
\[
\bigcap_n (f,g^{1-1/p^n})A_0^{1/p^\infty} = (f,g) A_0^{1/p^\infty}\ .
\]
Indeed, an element $x$ of the left-hand side lies in $A_0^{1/p^m}$ for $m$ large enough; we may assume $m=0$ by applying a power of Frobenius. Then $x$ reduces to an element of $M=A_0/(f,g)$ such that for all $n\geq 0$,
\[
0=g^{1/p^n} x\in A_0^{1/p^\infty}\otimes_{A_0} M = A_0^{1/p^\infty}/(f,g)\ .
\]
Therefore, $0=x\in A_0/(f,g)$, i.e. $x\in (f,g) A_0\subset (f,g) A_0^{1/p^\infty}$.

Now recall that
\[
\Hom_{A_0^{1/p^\infty}}(g^{1/p^\infty} A_0^{1/p^\infty},A_0^{1/p^\infty}/f)
\]
can be computed as the inverse limit of $A_0^{1/p^\infty}/f$, where the transition maps (from the $k$-th to the $k^\prime$-th term) are given by multiplication by $g^{1/p^k - 1/p^{k^\prime}}$. Let $M = R^1\varprojlim A_0^{1/p^\infty}$, with the similar transition maps. Then there is an exact sequence
\[
0\to A_0^{1/p^\infty}\buildrel f\over\to A_0^{1/p^\infty}\to \Hom_{A_0^{1/p^\infty}}(g^{1/p^\infty} A_0^{1/p^\infty},A_0^{1/p^\infty}/f)\to M\buildrel f\over\to M\ .
\]
Thus, it remains to see that kernel of $f: M\to M$ is killed by $f^{1/p^m}$ for all $m\geq 0$. Recall that
\[
M = \coker(\prod_{n\geq 0} A_0^{1/p^\infty}\to \prod_{n\geq 0} A_0^{1/p^\infty})\ ,
\]
where the map is given by $(x_0,x_1,\ldots)\mapsto (y_0,y_1,\ldots)$ with $y_k = x_k - g^{1/p^k - 1/p^{k+1}} x_{k+1}$. Thus, take some sequence $(y_0,y_1,\ldots)$, and assume that there is a sequence $(x_0^\prime,x_1^\prime,\ldots)$ with $fy_k = x_k^\prime - g^{1/p^k - 1/p^{k+1}} x_{k+1}^\prime$. We claim that $x_0^\prime\in (f,g)A_0^{1/p^\infty}$. By the above, it is enough to prove that $x_0^\prime\in (f,g^{1-1/p^k})A_0^{1/p^\infty}$ for all $k\geq 0$. But
\[\begin{aligned}
x_0^\prime &= fy_0 + g^{1-1/p} x_1^\prime = fy_0 + g^{1-1/p} fy_1 + g^{1-1/p^2} x_2^\prime = \ldots \\
&= f(y_0+g^{1-1/p} y_1 + \ldots + g^{1-1/p^{k-1}} y_{k-1}) + g^{1-1/p^k} x_k^\prime\in (f,g^{1-1/p^k})A_0^{1/p^\infty}\ ,
\end{aligned}\]
giving the claim. Similarly, $x_k^\prime\in (f,g^{1/p^k})A_0^{1/p^\infty}$ for all $k\geq 0$. Fix some $k_0\geq 0$. We may add $g^{1/p^k} z$ to $x_k^\prime$ for all $k$ for some $z\in A_0^{1/p^\infty}$; thus, we may assume that $x_{k_0}^\prime\in fA_0^{1/p^\infty}$. It follows that
\[
g^{1/p^{k_0} - 1/p^k} x_k^\prime\in fA_0^{1/p^\infty}
\]
for all $k\geq k_0$ (and $x_k^\prime\in fA_0^{1/p^\infty}$ for $k<k_0$). We claim that there is an integer $C\geq 0$ (depending only on $A_0$, $f$ and $g$) such that this implies
\[
x_k^\prime\in f^{1-C/p^{k_0}} A_0^{1/p^\infty}\ .
\]
Indeed, this is equivalent to a divisibility of Cartier divisors $f^{1-C/p^{k_0}} | x_k^\prime$. As $A_0$ is normal, this can be translated into a divisibility of Weil divisors. Let $x_1,\ldots,x_r\in \Spec A_0$ be the generic points of $V(f)$, and $v_1,\ldots,v_r$ the associated rank-$1$-valuations on $A_0^{1/p^\infty}$, normalized by $v(f)=1$. Then, the condition $f^{1-C/p^{k_0}} | x_k^\prime$ is equivalent to $v_i(x_k^\prime)\geq 1-C/p^{k_0}$ for $i=1,\ldots,r$. As $g\neq 0$, there is some $C<\infty$ such that $v_i(g)\leq C$ for $i=1,\ldots,r$. As
\[
g^{1/p^{k_0} - 1/p^k} x_k^\prime\in fA_0^{1/p^\infty}\ ,
\]
we know that
\[
(1/p^{k_0} - 1/p^k) C + v_i(x_k^{\prime})\geq 1\ ,
\]
thus
\[
v_i(x_k^\prime)\geq 1 - C/p^{k_0}\ ,
\]
as desired.

Thus, taking $k_0$ large enough, we can ensure that all $x_k^\prime$ are divisible by $f^{1-1/p^n}$, which shows that $f^{1/p^n}(y_0,y_1,\ldots)=0\in M$, whence the claim. This finishes the proof in the case that $A_0$ is smooth.

In general, take a resolution of singularities $\pi: T\to \Spec A_0$ (which we assumed to exist). It induces a map $\pi^{1/p^\infty}: T^{1/p^\infty}\to \Spec A_0^{1/p^\infty}$. The result in the smooth case implies that
\[
\Hom_{\OO_{T^{1/p^\infty}}}(g^{1/p^\infty} \OO_{T^{1/p^\infty}}, \OO_{T^{1/p^\infty}}/f)\leftarrow \OO_{T^{1/p^\infty}}/f
\]
is an almost isomorphism of sheaves over $T^{1/p^\infty}$. Note that by Zariski's main theorem, $\pi_\ast \OO_T = \OO_{\Spec A_0}$. Moreover, $R^1\pi_\ast \OO_T$ is a coherent $\OO_{\Spec A_0}$-module, so there is some $n$ such that $f^n$ kills all $f$-power torsion in $R^1\pi_\ast \OO_T$. Passing to the perfection, this implies that on $R^1\pi^{1/p^\infty}_\ast \OO_{T^{1/p^\infty}}$, the kernel of multiplication by $f$ is also killed by $f^{1/p^n}$ for all $n\geq 0$. Therefore, the map
\[
\OO_{\Spec A_0^{1/p^\infty}}/f\to \pi^{1/p^\infty}_\ast (\OO_{T^{1/p^\infty}}/f)
\]
is injective, with cokernel almost zero. Also, $g^{1/p^\infty} \OO_{T^{1/p^\infty}} = \pi^{1/p^\infty\ast} (g^{1/p^\infty} \OO_{\Spec A_0^{1/p^\infty}})$, as $g$ is a regular element in $A_0$ and $\OO_T$. Thus, adjunction shows that
\[\begin{aligned}
\Hom_{\OO_{\Spec A_0^{1/p^\infty}}}(&g^{1/p^\infty} \OO_{\Spec A_0^{1/p^\infty}},\OO_{\Spec A_0^{1/p^\infty}}/f)\\
&\to \Hom_{\OO_{\Spec A_0^{1/p^\infty}}}(g^{1/p^\infty} \OO_{\Spec A_0^{1/p^\infty}},\pi^{1/p^\infty}_\ast \OO_{T^{1/p^\infty}}/f)\\
&\to \pi^{1/p^\infty}_\ast \Hom_{\OO_{T^{1/p^\infty}}}(g^{1/p^\infty} \OO_{T^{1/p^\infty}},\OO_{T^{1/p^\infty}}/f)\\
&\leftarrow \pi^{1/p^\infty}_\ast \OO_{T^{1/p^\infty}}/f\leftarrow \OO_{\Spec A_0^{1/p^\infty}}/f
\end{aligned}\]
is a series of almost isomorphisms, finally finishing the proof by taking global sections.
\end{proof}

\subsection{Lifting to (pro-)finite covers}

As the final topic in this section, we will show how to lift a Hebbarkeitssatz to (pro-)finite covers. In the application, we will first prove a Hebbarkeitssatz at level $\Gamma_0(p^\infty)$ using the result from the previous subsection. After that, we need to lift this result to full $\Gamma(p^\infty)$-level. This is the purpose of the results of this subsection.

The following general definition will be useful.

\begin{definition}\label{DefGood} Let $K$ be a perfectoid field (of any characteristic), and let $0\neq t\in K$ with $|p|\leq |t|<1$. A triple $(\mathcal{X},\mathcal{Z},\mathcal{U})$ consisting of an affinoid perfectoid space $\mathcal{X}$ over $K$, a closed subset $\mathcal{Z}\subset \mathcal{X}$ and a quasicompact open subset $\mathcal{U}\subset \mathcal{X}\setminus \mathcal{Z}$ is good if
\[
H^0(\mathcal{X},\OO_{\mathcal{X}}^+/t)^a\cong H^0(\mathcal{X}\setminus \mathcal{Z},\OO_{\mathcal{X}}^+/t)^a\hookrightarrow H^0(\mathcal{U},\OO_{\mathcal{U}}^+/t)^a\ .
\]
\end{definition}

One checks easily that this notion is independent of the choice of $t$, and is compatible with tilting. Moreover, if $(\mathcal{X},\mathcal{Z},\mathcal{U})$ is good, then for any $t\in \OO_K$, possibly zero, one has
\[
H^0(\mathcal{X},\OO_{\mathcal{X}}^+/t)^a\cong H^0(\mathcal{X}\setminus \mathcal{Z},\OO_{\mathcal{X}}^+/t)^a\hookrightarrow H^0(\mathcal{U},\OO_{\mathcal{U}}^+/t)^a\ .
\]
In particular, the case $t=0$ says that bounded functions from $\mathcal{X}\setminus \mathcal{Z}$ extend uniquely to $\mathcal{X}$.

In the application, $\mathcal{X}$ will be an open subset of the minimal compactification, $\mathcal{Z}$ will be the boundary, and $\mathcal{U}$ the locus of good reduction. Knowing that such a triple is good will allow us to verify statements away from the boundary, or even on the locus of good reduction.

Now we go back to our setup, so in particular $K$ is of characteristic $p$. Let $R_0$ be a reduced Tate $K$-algebra topologically of finite type, $\mathcal{X}_0 = \Spa(R_0,R_0^\circ)$ the associated affinoid adic space of finite type over $K$. Let $R$ be the completed perfection of $R_0$, which is a p-finite perfectoid $K$-algebra, and $\mathcal{X}=\Spa(R,R^+)$ with $R^+=R^\circ$ the associated p-finite affinoid perfectoid space over $K$.

Moreover, let $I_0\subset R_0$ be some ideal, $I=I_0R\subset R$, $\mathcal{Z}_0 = V(I_0)\subset \mathcal{X}_0$, and $\mathcal{Z}=V(I)\subset \mathcal{X}$. Finally, fix a quasicompact open subset $\mathcal{U}_0\subset \mathcal{X}_0\setminus \mathcal{Z}_0$, with preimage $\mathcal{U}\subset \mathcal{X}\setminus \mathcal{Z}$.

In the following lemma, we show that the triple $(\mathcal{X},\mathcal{Z},\mathcal{U})$ is good under suitable conditions on $R_0$, $I_0$ and $U_0$.

\begin{lem}\label{GoodTriple} Let $A_0$ be normal, of finite type over $\F_p$, admitting a resolution of singularities, let
\[
R_0 = (A_0\hat{\otimes}_{\F_p} K)\langle u\rangle / (uf - t)
\]
for some $f\in A_0$ which is not a zero-divisor, and let $I_0 = JR_0$ for some ideal $J\subset A_0$ with $V(J)\subset \Spec A_0$ of codimension $\geq 2$. Moreover, let $\mathcal{U}_0 = \{x\in \mathcal{X}_0\mid |g(x)| = 1\ \mathrm{for\ some}\ g\in J\}$. If $K=\F_p((t^{1/p^\infty}))$, the triple $(\mathcal{X},\mathcal{Z},\mathcal{U})$ is good.
\end{lem}

\begin{proof} We may assume that $A_0$ is integral. Corollary \ref{SpecialHebbarkeit} implies that
\[
H^0(\mathcal{X},\OO_\mathcal{X}^+/t)^a\cong H^0(\mathcal{X}\setminus \mathcal{Z},\OO_\mathcal{X}^+/t)^a\ .
\]
Moreover, using notation from the proof of Corollary \ref{SpecialHebbarkeit}, $H^0(\mathcal{X},\OO_\mathcal{X}^+/t)^a = A_0^{1/p^\infty}[u^{1/p^\infty}]/(uf)^a$, and for any $g\in J$, one has
\[
H^0(\mathcal{U}_g,\OO_\mathcal{X}^+/t)^a = A_0^{1/p^\infty}[g^{-1}][u^{1/p^\infty}]/(uf)^a
\]
by localization, where $\mathcal{U}_g = \{x\in \mathcal{X}\mid |g(x)| = 1\}$. Thus, using the basis given by the $u^i$, the result follows from
\[
H^0(\Spec A_0,\OO_{\Spec A_0}) = H^0(\Spec A_0\setminus V(J),\OO_{\Spec A_0})
\]
and
\[
H^0(\Spec A_0,\OO_{\Spec A_0}/f)\hookrightarrow H^0(\Spec A_0\setminus V(J),\OO_{\Spec A_0}/f)\ ,
\]
where the latter holds true because the depth of $\OO_{\Spec A_0}/f$ at any point of $V(J)$ is at least $2-1=1$.
\end{proof}

In the next lemma, we go back to the abstract setup before Lemma \ref{GoodTriple}.

\begin{lem}\label{GoodTripleFinite} Assume that $(\mathcal{X},\mathcal{Z},\mathcal{U})$ is good. Assume moreover that $R_0$ is normal, and that $V(I_0)\subset \Spec R_0$ is of codimension $\geq 2$. Let $R_0^\prime$ be a finite normal $R_0$-algebra which is \'etale outside $V(I_0)$, and such that no irreducible component of $\Spec R_0^\prime$ maps into $V(I_0)$. Let $I_0^\prime = I_0R_0^\prime$, and $\mathcal{U}_0^\prime\subset \mathcal{X}_0^\prime$ the preimage of $\mathcal{U}_0$. Let $R^\prime$, $I^\prime$, $\mathcal{X}^\prime$, $\mathcal{Z}^\prime$, $\mathcal{U}^\prime$ be the associated perfectoid objects.
\begin{altenumerate}
\item[{\rm (i)}] There is a perfect trace pairing
\[
\tr_{R_0^\prime/R_0}: R_0^\prime\otimes_{R_0} R_0^\prime\to R_0\ .
\]
\item[{\rm (ii)}] The trace pairing from (i) induces a trace pairing
\[
\tr_{R^{\prime \circ}/R^\circ}: R^{\prime\circ}\otimes_{R^\circ} R^{\prime\circ}\to R^\circ
\]
which is almost perfect.
\item[{\rm (iii)}] For all open subsets $\mathcal{V}\subset \mathcal{X}$ with preimage $\mathcal{V}^\prime\subset \mathcal{X}^\prime$, the trace pairing induces an isomorphism
\[
H^0(\mathcal{V}^\prime,\OO_{\mathcal{X}^\prime}^+/t)^a\cong \Hom_{R^\circ/t}(R^{\prime\circ}/t,H^0(\mathcal{V},\OO_\mathcal{X}^+/t))^a\ .
\]
\item[{\rm (iv)}] The triple $(\mathcal{X}^\prime,\mathcal{Z}^\prime,\mathcal{U}^\prime)$ is good.
\item[{\rm (v)}] If $\mathcal{X}^\prime\to \mathcal{X}$ is surjective, then the map
\[
H^0(\mathcal{X},\OO_\mathcal{X}^+/t)\to H^0(\mathcal{X}^\prime,\OO_{\mathcal{X}^\prime}^+/t)\cap H^0(\mathcal{U},\OO_\mathcal{X}^+/t)
\]
is an almost isomorphism.
\end{altenumerate}
\end{lem}

\begin{proof}\begin{altenumerate}
\item[{\rm (i)}] There is an isomorphism of locally free $\OO_{\Spec R_0\setminus V(I_0)}$-modules
\[
f_\ast\OO_{\Spec R_0^\prime\setminus V(I_0^\prime)}\to \Hom_{\OO_{\Spec R_0\setminus V(I_0)}}(f_\ast\OO_{\Spec R_0^\prime\setminus V(I_0^\prime)},\OO_{\Spec R_0\setminus V(I_0)})
\]
induced by the trace pairing on $\Spec R_0\setminus V(I_0)$, as the map
\[
f: \Spec R_0^\prime\setminus V(I_0^\prime)\to \Spec R_0\setminus V(I_0)
\]
is finite \'etale. Now take global sections to conclude, using that $R_0$ and $R_0^\prime$ are normal, and $V(I_0)\subset \Spec R_0$, $V(I_0^\prime)\subset \Spec R_0^\prime$ are of codimension $\geq 2$.
\item[{\rm (ii)}] By part (i) and Banach's open mapping theorem, the cokernel of the injective map
\[
R_0^{\prime \circ}\to \Hom_{R_0^\circ}(R_0^{\prime\circ},R_0^\circ)
\]
is killed by $t^N$ for some $N$. Passing to the completed perfection implies that
\[
R^{\prime\circ}\to \Hom_{R^\circ}(R^{\prime\circ},R^\circ)
\]
is almost exact, as desired.
\item[{\rm (iii)}] If $\mathcal{V}=\mathcal{X}$, this follows from part (ii) by reduction modulo $t$. In general, $\mathcal{V}$ is the preimage of some $\mathcal{V}_0\subset \mathcal{X}_0$, which we may assume to be affinoid. One can then use the result for $\mathcal{V}_0$ in place of $\mathcal{X}_0$, noting that
\[\begin{aligned}
\Hom_{R^\circ/t}(R^{\prime\circ}/t,H^0(\mathcal{V},\OO_\mathcal{X}^+/t))^a&=\Hom_{H^0(\mathcal{V},\OO_\mathcal{X}^+/t)}(H^0(\mathcal{V},\OO_\mathcal{X}^+/t)\otimes_{R^\circ/t} R^{\prime\circ}/t,H^0(\mathcal{V},\OO_\mathcal{X}^+/t))^a\\
&=\Hom_{H^0(\mathcal{V},\OO_\mathcal{X}^+/t)}(H^0(\mathcal{V}^\prime,\OO_\mathcal{X}^+/t),H^0(\mathcal{V},\OO_\mathcal{X}^+/t))^a\ ,
\end{aligned}\]
by the formula for fibre products in the category of perfectoid spaces, cf. \cite[Proposition 6.18]{ScholzePerfectoid}.
\item[{\rm (iv)}] This follows directly from part (iii), and the assumption that $(\mathcal{X},\mathcal{Z},\mathcal{U})$ is good.
\item[{\rm (v)}] By surjectivity of $\mathcal{X}^\prime\to \mathcal{X}$, $H^0(\mathcal{X},\OO_{\mathcal{X}}^+/t)^a\hookrightarrow H^0(\mathcal{X}^\prime, \OO_{\mathcal{X}}^+/t)^a$. Assume $h$ is an almost element of $H^0(\mathcal{X}^\prime,\OO_{\mathcal{X}^\prime}^+/t)^a\cap H^0(\mathcal{U},\OO_\mathcal{X}^+/t)^a$. Then, via the trace pairing, $h$ gives rise to a map
\[
(R^{\prime\circ}/t)^a\to (R^\circ/t)^a\ .
\]
We claim that this factors over the (almost surjective) map
\[
\tr_{(R^{\prime\circ}/t)^a/(R^\circ/t)^a}: (R^{\prime\circ}/t)^a\to (R^\circ/t)^a\ .
\]
As $(R^\circ/t)^a\hookrightarrow H^0(\mathcal{U},\OO_\mathcal{X}^+/t)^a$, it suffices to check this after restriction to $\mathcal{U}$; there it follows from the assumption $h\in H^0(\mathcal{U},\OO_\mathcal{X}^+/t)$. This translates into the statement that $h$ is an almost element of $H^0(\mathcal{X},\OO_\mathcal{X}^+/t)^a$, as desired.
\end{altenumerate}
\end{proof}

Finally, assume that one has a filtered inductive system $R_0^{(i)}$, $i\in I$, as in Lemma \ref{GoodTripleFinite}, giving rise to $\mathcal{X}^{(i)}$, $\mathcal{Z}^{(i)}$, $\mathcal{U}^{(i)}$. We assume that all transition maps $\mathcal{X}^{(i)}\to \mathcal{X}^{(j)}$ are surjective. Let $\tilde{\mathcal{X}}$ be the inverse limit of the $\mathcal{X}^{(i)}$ in the category of perfectoid spaces over $K$, with preimage $\tilde{\mathcal{Z}}\subset \tilde{\mathcal{X}}$ of $\mathcal{Z}$, and $\tilde{\mathcal{U}}\subset \tilde{\mathcal{X}}$ of $\mathcal{U}$.

\begin{lem}\label{GoodTripleLimit} In this situation, the triple $(\tilde{\mathcal{X}},\tilde{\mathcal{Z}},\tilde{\mathcal{U}})$ is good.
\end{lem}

\begin{proof} As $\tilde{\mathcal{X}}$ and $\tilde{\mathcal{U}}$ are qcqs, one may pass to the filtered direct limit to conclude from the previous lemma, part (iv), that
\[
H^0(\tilde{\mathcal{X}},\OO_{\tilde{\mathcal{X}}}^+/t)^a\hookrightarrow H^0(\tilde{\mathcal{U}},\OO_{\tilde{\mathcal{X}}}^+/t)^a\ .
\]
Moreover,
\[
H^0(\mathcal{X}^{(i)},\OO_{\mathcal{X}^{(i)}}^+/t)^a\hookrightarrow H^0(\tilde{\mathcal{X}},\OO_{\tilde{\mathcal{X}}}^+/t)^a
\]
for all $i\in I$, as $\tilde{\mathcal{X}}$ surjects onto $\mathcal{X}^{(i)}$. The same injectivity holds on open subsets. Also,
\[
H^0(\mathcal{X}^{(i)},\OO_{\mathcal{X}^{(i)}}^+/t)^a = H^0(\tilde{\mathcal{X}},\OO_{\tilde{\mathcal{X}}}^+/t)^a\cap H^0(\mathcal{U}^{(i)},\OO_{\mathcal{X}^{(i)}}^+/t)^a\ ,
\]
by passing to the filtered direct limit in the previous lemma, part (v). Let $\tilde{R} = H^0(\tilde{\mathcal{X}},\OO_{\tilde{\mathcal{X}}})$, $\tilde{R}^+ = \tilde{R}^\circ$ and $\tilde{I} = I_0\tilde{R}\subset \tilde{R}$. Then the general form of the Hebbarkeitssatz says that
\[
H^0(\tilde{\mathcal{X}}\setminus \tilde{\mathcal{Z}},\OO_{\tilde{\mathcal{X}}}^+/t)^a = \Hom_{\tilde{R}^+}(\tilde{I}^{+ 1/p^\infty},\tilde{R}^+/t)^a = \Hom_{R^+}(I^{+ 1/p^\infty},\tilde{R}^+/t)^a\ ,
\]
using also Lemma \ref{PerfectIdealBC}. The latter injects into
\[
\Hom_{R^+}(I^{+ 1/p^\infty},H^0(\tilde{\mathcal{U}},\OO_{\tilde{\mathcal{X}}}^+/t))^a = H^0(\tilde{\mathcal{U}}\setminus \tilde{\mathcal{Z}},\OO_{\tilde{\mathcal{X}}}^+/t)^a = H^0(\tilde{\mathcal{U}},\OO_{\tilde{\mathcal{X}}}^+/t)^a\ .
\]
The latter is a filtered direct limit. If an almost element $h$ of $H^0(\tilde{\mathcal{X}}\setminus \tilde{\mathcal{Z}},\OO_{\tilde{\mathcal{X}}}^+/t)^a$ is mapped to an almost element of
\[
H^0(\mathcal{U}^{(i)},\OO_{\mathcal{X}^{(i)}}^+/t)^a\subset H^0(\tilde{\mathcal{U}},\OO_{\tilde{\mathcal{X}}}^+/t)^a\ ,
\]
then the map
\[
(I^{+ 1/p^\infty})^a\to (\tilde{R}^+/t)^a
\]
corresponding to $h$ will take values in $(\tilde{R}^+/t)^a\cap H^0(\mathcal{U}^{(i)},\OO_{\mathcal{X}^{(i)}}^+/t)^a = (R^{(i)+}/t)^a$, thus gives rise to an almost element of
\[
\Hom_{R^+}(I^{+ 1/p^\infty},R^{(i)+}/t)^a = H^0(\mathcal{X}^{(i)}\setminus \mathcal{Z}^{(i)},\OO_{\mathcal{X}^{(i)}}^+/t)^a\ .
\]
But $(\mathcal{X}^{(i)},\mathcal{Z}^{(i)},\mathcal{U}^{(i)})$ is good, so the Hebbarkeitssatz holds there, and $h$ extends to $\mathcal{X}^{(i)}$, and thus to $\tilde{\mathcal{X}}$.
\end{proof}

In particular, one can use this to generalize Lemma \ref{GoodTriple} slightly:

\begin{cor}\label{GoodTripleBaseField} The conclusion of Lemma \ref{GoodTriple} holds under the weaker assumption that $K$ is the completion of an algebraic extension of $\F_p((t^{1/p^\infty}))$.
\end{cor}

\begin{proof} For a finite extension, this follows from Lemma \ref{GoodTripleFinite}. Then the general case follows from Lemma \ref{GoodTripleLimit}.
\end{proof}

\chapter{The perfectoid Siegel space}\label{SiegelChapter}

\section{Introduction}

Fix an integer $g\geq 1$, and a prime $p$. Let $(V,\psi)$ be the split symplectic space of dimension $2g$ over $\Q$. In other words, $V=\Q^{2g}$ with symplectic pairing
\[
\psi((a_1,\ldots,a_g,b_1,\ldots,b_g),(a_1^\prime,\ldots,a_g^\prime,b_1^\prime,\ldots,b_g^\prime)) = \sum_{i=1}^g (a_i b_i^\prime - a_i^\prime b_i)\ .
\]
Inside $V$, we fix the self-dual lattice $\Lambda=\Z^{2g}$. Let $\GSp_{2g}/\Z$ be the group of symplectic similitudes of $\Lambda$, and fix a compact open subgroup $K^p\subset \GSp_{2g}(\A_f^p)$ contained in $\{\gamma\in \GSp_{2g}(\hat{\Z}^p)\mid \gamma\equiv 1\mod N\}$ for some integer $N\geq 3$ prime to $p$.

Let $X_{g,K^p}$ over $\Z_{(p)}$ denote the moduli space of principally polarized $g$-dimensional abelian varieties with level-$K^p$-structure. As $g$ and $K^p$ remain fixed throughout, we will write $X=X_{g,K^p}$. The moduli space $X$ can be interpreted as the Shimura variety for the group of symplectic similitudes $\GSp_{2g} = \GSp(V,\psi)$, acting on the Siegel upper half space. Let $\mathrm{Fl}$ over $\Q$ be the associated flag variety, i.e. the space of totally isotropic subspaces $W\subset V$ (of dimension $g$). Over $\mathrm{Fl}$, one has a tautological ample line bundle $\omega_{\mathrm{Fl}} = (\bigwedge^g W)^\ast$.

Moreover, we have the minimal (Baily-Borel-Satake) compactification $X^\ast = X_{g,K^p}^\ast$ over $\Z_{(p)}$, as constructed by Faltings-Chai, \cite{FaltingsChai}. It carries a natural ample line bundle $\omega$, given (on $X_{g,K^p}$) as the determinant of the sheaf of invariant differentials on the universal abelian scheme; in fact, if $g\geq 2$,
\[
X_{g,K^p}^\ast = \Proj \bigoplus_{k\geq 0} H^0(X_{g,K^p},\omega^{\otimes k})\ .
\]
Moreover, for any compact open subgroup $K_p\subset \GSp_{2g}(\Q_p)$, we have $X_{K_p} = X_{g,K_pK^p}$ over $\Q$, which is the moduli space of principally polarized $g$-dimensional abelian varieties with level-$K^p$-structure and level-$K_p$-structure, with a similar compactification $X_{K_p}^\ast = X_{g,K_pK^p}^\ast$. We will be particularly interested in the following level structures.

\begin{definition} In all cases, the blocks are of size $g\times g$.
\[\begin{aligned}
\Gamma_0(p^m) &= \{ \gamma\in \GSp_{2g}(\Z_p)\mid \gamma\equiv \left(\begin{array}{cc} \ast & \ast\\ 0 & \ast\end{array}\right)\mod p^m\ ,\ \det \gamma\equiv 1\mod p^m\}\ ,\\
\Gamma_1(p^m) &= \{ \gamma\in \GSp_{2g}(\Z_p)\mid \gamma\equiv \left(\begin{array}{cc} 1 & \ast\\ 0 & 1\end{array}\right)\mod p^m\}\ ,\\
\Gamma(p^m) &= \{ \gamma\in \GSp_{2g}(\Z_p)\mid \gamma\equiv \left(\begin{array}{cc} 1 & 0\\ 0 & 1\end{array}\right)\mod p^m\}\ .
\end{aligned}\]
\end{definition}

We note that our definition of $\Gamma_0(p^m)$ is slightly nonstandard in that we put the extra condition $\det \gamma\equiv 1\mod p^m$.

Let $\mathcal{X}_{K_p}^\ast$ denote the adic space over $\Spa(\Q_p,\Z_p)$ associated with $X_{K_p}^\ast$, for any $K_p\subset \GSp_{2g}(\Q_p)$. Similarly, let $\Fl$ be the adic space over $\Spa(\Q_p,\Z_p)$ associated with $\mathrm{Fl}$, with ample line bundle $\omega_\Fl$. Let $\Q_p^\cycl$ be the completion of $\Q_p(\mu_{p^\infty})$. Note that $X_{\Gamma_0(p^m)}^\ast$ lives naturally over $\Q(\zeta_{p^m})$ by looking at the symplectic similitude factor. The following theorem summarizes the main result; for a more precise version, we refer to Theorem \ref{ExistenceHT}.

\begin{thm}\label{MainThmSiegel} Fix any $K^p\subset \GSp_{2g}(\A_f^p)$ contained in the level-$N$-congruence subgroup for some $N\geq 3$ prime to $p$.
\begin{altenumerate}
\item[{\rm (i)}] There is a unique (up to unique isomorphism) perfectoid space
\[
\mathcal{X}_{\Gamma(p^\infty)}^\ast = \mathcal{X}_{g,\Gamma(p^\infty),K^p}^\ast
\]
over $\Q_p^\cycl$ with an action of $\GSp_{2g}(\Q_p)$\footnote{Of course, the action does not preserve the structure morphism to $\Spa(\Q_p^\cycl,\Z_p^\cycl)$.}, such that
\[
\mathcal{X}_{\Gamma(p^\infty)}^\ast\sim \varprojlim_{K_p} \mathcal{X}_{K_p}^\ast\ ,
\]
equivariant for the $\GSp_{2g}(\Q_p)$-action. Here, we use $\sim$ in the sense of \cite[Definition 2.4.1]{ScholzeWeinstein}.
\item[{\rm (ii)}] There is a $\GSp_{2g}(\Q_p)$-equivariant Hodge-Tate period map
\[
\pi_\HT: \mathcal{X}_{\Gamma(p^\infty)}^\ast\to \Fl
\]
under which the pullback of $\omega$ from $\mathcal{X}_{K_p}^\ast$ to $\mathcal{X}_{\Gamma(p^\infty)}^\ast$ gets identified with the pullback of $\omega_\Fl$ along $\pi_\HT$. Moreover, $\pi_\HT$ commutes with Hecke operators away from $p$ (when changing $K^p$), for the trivial action of these Hecke operators on $\Fl$.
\item[{\rm (iii)}] There is a basis of open affinoid subsets $U\subset \Fl$ for which the preimage $V = \pi_\HT^{-1}(U)$ is affinoid perfectoid, and the following statements are true. The subset $V$ is the preimage of an affinoid subset $V_m\subset \mathcal{X}_{\Gamma(p^m)}^\ast$ for $m$ sufficiently large, and the map
\[
\varinjlim_m H^0(V_m,\OO_{\mathcal{X}_{\Gamma(p^m)}^\ast})\to H^0(V,\OO_{\mathcal{X}_{\Gamma(p^\infty)}^\ast})
\]
has dense image.
\end{altenumerate}
\end{thm}

These results, including the Hodge-Tate period map, are entirely new even for the modular curve, i.e. $g=1$. Let us explain in this case what $\pi_\HT$ looks like. One may stratify each
\[
\mathcal{X}_K^\ast = \mathcal{X}_K^{\ast\ord}\bigsqcup \mathcal{X}_K^{\mathrm{ss}}
\]
into the ordinary locus $\mathcal{X}_K^{\ast\ord}$ (which we define for this discussion as the closure of the tubular neighborhood of the ordinary locus in the special fibre) and the supersingular locus $\mathcal{X}_K^{\mathrm{ss}}$. Thus, by definition, $\mathcal{X}_K^{\mathrm{ss}}\subset \mathcal{X}_K^\ast$ is an open subset, which can be identified with a finite disjoint union of Lubin-Tate spaces. Passing to the inverse limit, we get a similar decomposition
\[
\mathcal{X}_{K^p}^\ast = \mathcal{X}_{K^p}^{\ast\ord}\bigsqcup \mathcal{X}_{K^p}^{\mathrm{ss}}\ .
\]
On the flag variety $\Fl = \mathbb{P}^1$ for $\GSp_2 = \GL_2$, one has a decomposition
\[
\mathbb{P}^1 = \mathbb{P}^1(\Q_p)\bigsqcup \Omega^2\ ,
\]
where $\Omega^2 = \mathbb{P}^1\setminus \mathbb{P}^1(\Q_p)$ is Drinfeld's upper half-plane. These decompositions correspond, i.e.
\[
\mathcal{X}_{K^p}^{\ast\ord} = \pi_\HT^{-1}(\mathbb{P}^1(\Q_p))\ ,\ \mathcal{X}_{K^p}^{\mathrm{ss}} = \pi_\HT^{-1}(\Omega^2)\ .
\]
Moreover, on the ordinary locus, the Hodge-Tate period map
\[
\pi_\HT: \mathcal{X}_{K^p}^{\ast\ord}\to \mathbb{P}^1(\Q_p)
\]
measures the position of the canonical subgroup. On the supersingular locus, one has the following description of $\pi_\HT$, using the isomorphism $\mathcal{M}_{\LT,\infty}\cong \mathcal{M}_{\Dr,\infty}$ between Lubin-Tate and Drinfeld tower, cf. \cite{FaltingsTwoTowers}, \cite{FarguesTwoTowers}, \cite{ScholzeWeinstein}:
\[
\pi_\HT: \mathcal{X}_{K^p}^{\mathrm{ss}} = \bigsqcup \mathcal{M}_{\LT,\infty}\cong \bigsqcup \mathcal{M}_{\Dr,\infty}\to \Omega^2\ .
\]
Contrary to the classical Gross-Hopkins period map $\mathcal{M}_\LT\to \mathbb{P}^1$ which depends on a trivialization of the Dieudonn\'e module of the supersingular elliptic curve, the Hodge-Tate period map is canonical. It commutes with the Hecke operators away from $p$ (as it depends only on the $p$-divisible group, and not the abelian variety), and extends continuously to the whole modular curve.

Let us give a short summary of the proof. Note that a result very similar in spirit was proved in joint work with Jared Weinstein, \cite{ScholzeWeinstein}, for Rapoport-Zink spaces. Unfortunately, for a number of reasons, it is not possible to use that result to obtain a result for Shimura varieties (although the process in the opposite direction does work). The key problem is that Rapoport-Zink spaces do not cover the whole Shimura variety. For example, in the case of the modular curve, the points of the adic space specializing to a generic point of the special fibre will not be covered by any Rapoport-Zink space. Also, it is entirely impossible to analyze the minimal compactification using Rapoport-Zink spaces.

For this reason, we settle for a different and direct approach. The key idea is that on the ordinary locus, the theory of the canonical subgroup gives a canonical way to extract $p$-power roots in the $\Gamma_0(p^\infty)$-tower. The toy example is that of the $\Gamma_0(p)$-level structure for the modular curve, cf. \cite{DeligneRapoport}. Above the ordinary locus, one has two components, one mapping down isomorphically, and the other mapping down via the Frobenius map. It is the component that maps down via Frobenius that we work with. Going to deeper $\Gamma_0(p^m)$-level, the maps continue to be Frobenius maps, and in the inverse limit, one gets a perfect space. Passing to the tubular neighborhood in characteristic $0$, one has the similar picture, and one will get a perfectoid space in the inverse limit. It is then not difficult to go from $\Gamma_0(p^\infty)$- to $\Gamma(p^\infty)$-level, using the almost purity theorem; only the boundary of the minimal compactification causes some trouble, that can however be overcome.

Note that we work with the anticanonical tower, and not the canonical tower: The $\Gamma_0(p)$-level subgroup is disjoint from the canonical subgroup. It is well-known that any finite level of the canonical tower is overconvergent, cf. e.g. \cite{Lubin}, \cite{Katz}, \cite{AbbesMokrane}, \cite{AndreattaGasbarri}, \cite{Conrad}, \cite{FarguesCanonical}, \cite{GorenKassaei}, \cite{Rabinoff}, \cite{Tian}; however, not the whole canonical tower is overconvergent. By contrast, the whole anticanonical tower is overconvergent. This lets one deduce that on a strict neighborhood of the anticanonical tower, one can get a perfectoid space at $\Gamma_0(p^\infty)$-level, and then also at $\Gamma(p^\infty)$-level.

Observe that the locus of points in $\mathcal{X}_{\Gamma(p^\infty)}^\ast$ which have a perfectoid neighborhood is stable under the $\GSp_{2g}(\Q_p)$-action. Thus, to conclude, it suffices to see that any abelian variety is isogenous to an abelian variety in a given strict neighborhood of the ordinary locus. Although there may be a more direct way to prove this, we deduce it from the Hodge-Tate period map. Recall that the Hodge-Tate filtration of an abelian variety $A$ over a complete and algebraically closed extension $C$ of $\Q_p$ is a short exact sequence
\[
0\to (\Lie A)(1)\to T_p A\otimes_{\Z_p} C\to (\Lie A^\ast)^\ast\to 0\ ,
\]
where $T_p A$ is the $p$-adic Tate module of $A$. Moreover, $(\Lie A)(1)\subset T_p A\otimes_{\Z_p} C$ is a $\Q_p$-rational subspace if and only if (the abelian part of the reduction of) $A$ is ordinary; this follows from the classification of $p$-divisible groups over $\OO_C$, \cite[Theorem B]{ScholzeWeinstein}, but can also easily be proved directly. One deduces that if the Hodge-Tate filtration is close to a $\Q_p$-rational point, then $A$ lies in a small neighborhood of the ordinary locus (and conversely), cf. Lemma \ref{SmallUExists}, \ref{ExistsSmallEpsilon}. As under the action of $\GSp_{2g}(\Q_p)$, any filtration can be mapped to one that is close to any given $\Q_p$-rational point (making use of $U_p$-like operators), one gets the desired result.

In fact, observe that by \cite[Theorem B]{ScholzeWeinstein}, the $C$-valued points of $\Fl$ are in bijection with principally polarized $p$-divisible groups $G$ over $\OO_C$, with a trivialization of their Tate module. Thus, $\pi_\HT$ is, at least on $C$-valued points of the locus of good reduction, the map sending an abelian variety over $\OO_C$ to its associated $p$-divisible group. We warn the reader that this picture is only clean on geometric points; the analogue of \cite[Theorem B]{ScholzeWeinstein} fails over general nonarchimedean fields, or other base rings.

Most subtleties in the argument arise in relation to the minimal compactification. For example, we can prove existence of $\pi_\HT$ a priori only away from the boundary. To extend to the minimal compactification, we use a version of Riemann's Hebbarkeitssatz, saying that any bounded function has removable singularities. This result was proved in Section \ref{HebbarkeitssatzSection}, in the various forms that we will need. In Section \ref{EtaleOrdinarySection}, we prove the main result on a strict neighborhood of the anticanonical tower. As we need some control on the integral structure of the various objects, we found it useful to have a theory of the canonical subgroup that works integrally. As such a theory does not seem to be available in the literature, we give a new proof of existence of the canonical subgroup. The key result is the following. Note that our result is effective, and close to optimal (and works uniformly even for $p=2$).

\begin{lem} Let $R$ be a $p$-adically complete flat $\Z_p^\cycl$-algebra, and let $A/R$ be an abelian variety. Assume that the $\frac{p^m-1}{p-1}$-th power of the Hasse invariant of $A$ divides $p^\epsilon$ for some $\epsilon<\frac 12$. Then there is a unique closed subgroup $C\subset A[p^m]$ such that $C=\Ker F^m\mod p^{1-\epsilon}$.
\end{lem}

Our proof runs roughly as follows. Look at $G = A[p^m] / \Ker F^m$ over $R/p$. By the assumption on the Hasse invariant, the Lie complex of $G$ is killed by $p^\epsilon$. The results of Illusie's thesis, cf. \cite[Section 3]{IllusieDeformationsBT}, imply that there is finite flat group scheme $\tilde{G}$ over $R$ such that $\tilde{G}$ and $G$ agree over $R/p^{1-\epsilon}$. Similarly, the map $A[p^m]\to G$ over $R/p^{1-\epsilon}$ lifts to a map $A[p^m]\to \tilde{G}$ over $R$ that agrees with the original map modulo $R/p^{1-2\epsilon}$. Letting $C=\Ker(A[p^m]\to \tilde{G})$ proves existence (up to a constant); uniqueness is proved similarly. All expected properties of the canonical subgroup are easily proved as well. In fact, it is not necessary to have an abelian variety for this result; a (truncated) $p$-divisible group would be as good.

As regards subtleties related to the minimal compactification, let us mention that we also need a version of (a strong form of) Hartog's extension principle, cf. Lemma \ref{HartogPrinciple}, and a version of Tate's normalized traces, cf. Lemma \ref{TraceDiv}. In general, our approach is to avoid any direct analysis of the boundary. This is mainly due to lazyness on our side, as we did not wish to speak about the toroidal compactification, which is needed for most explicit arguments about the boundary.\footnote{In a recent preprint, Pilloni and Stroh, \cite{PilloniStroh}, give such an explicit description of the boundary.} Instead, for all of our arguments it is enough to know that all geometric fibres (over $\Spec \Z_p$) of the minimal compactification are normal, with boundary of codimension $g$ (which is $\geq 2$ at least if $g\geq 2$; the case $g=1$ is easy to handle directly). However, the price to pay is that one has to prove a rather involved series of lemmas in commutative algebra.

Finally, in Section \ref{ConclusionSection} we construct the Hodge-Tate period map (first topologically, then as a map of adic spaces), and extend the results to the whole Siegel moduli space, finishing the proof of Theorem \ref{MainThmSiegel}.

\section{A strict neighborhood of the anticanonical tower}\label{EtaleOrdinarySection}

\subsection{The canonical subgroup} We need the canonical subgroup. Let us record the following simple proof of existence, which appears to be new. It depends on the following deformation-theoretic result, proved in Illusie's thesis, \cite[Th\'eor\`eme VII.4.2.5]{IllusieThesis2}.

\begin{thm}\label{IllusieDefTheory} Let $A$ be a commutative ring, and $G$, $H$ be flat and finitely presented commutative group schemes over $A$, with a group morphism $u: H\to G$. Let $B_1,B_2\to A$ be two square-zero thickenings with a morphism $B_1\to B_2$ over $A$. Let $J_i\subset B_i$ be the augmentation ideal. Let $\tilde{G}_1$ be a lift of $G$ to $B_1$, and $\tilde{G}_2$ the induced lift to $B_2$. Let $K$ be a cone of the map $\check{\ell}_H\to \check{\ell}_G$ of Lie complexes.
\begin{altenumerate}
\item[{\rm (i)}] For $i=1,2$, there is an obstruction class
\[
o_i\in \Ext^1(H,K\buildrel\mathbb{L}\over\otimes J_i)
\]
which vanishes precisely when there exists a lifting $(\tilde{H}_i,\tilde{u}_i)$ of $(H,u)$ to a flat commutative group scheme $\tilde{H}_i$ over $B_i$, with a morphism $\tilde{u}_i: \tilde{H}_i\to \tilde{G}_i$ lifting $u: H\to G$.
\item[{\rm (ii)}] The obstruction $o_2\in \Ext^1(H,K\buildrel\mathbb{L}\over\otimes J_2)$ is the image of $o_1\in \Ext^1(H,K\buildrel\mathbb{L}\over\otimes J_1)$ under the map $J_1\to J_2$.
\end{altenumerate}
\end{thm}

\begin{proof} Part (i) is exactly \cite[Th\'eor\`eme VII.4.2.5 (i)]{IllusieThesis2} (except for a different convention on the shift in $K$), where the $A$ from loc. cit. is taken to be $\Z$ and the base ring $T=\Z$. Part (ii) follows from \cite[Remarque VII.4.2.6 (i)]{IllusieThesis2}.
\end{proof}

Recall that $\Z_p^\cycl$ contains elements of $p$-adic valuation $\frac a{(p-1)p^n}$ for any integers $a,n\geq 0$. In the following, $p^\epsilon\in \Z_p^\cycl$ denotes any element of valuation $\epsilon$ for any $\epsilon$; we always assume implicitly that $\epsilon$ is of the form $\frac a{(p-1)p^n}$ for some $a,n\geq 0$. In all the following results, $\Z_p^\cycl$ could be replaced by any sufficiently ramified extension of $\Z_p$.

\begin{cor}\label{LiftGroup} Let $R$ be a $p$-adically complete flat $\Z_p^\cycl$-algebra. Let $G$ be a finite locally free commutative group scheme over $R$, and let $C_1\subset G\otimes_R R/p$ be a finite locally free subgroup. Assume that for $H=(G\otimes_R R/p)/C_1$, multiplication by $p^\epsilon$ on the Lie complex $\check{\ell}_H$ is homotopic to $0$, where $0\leq \epsilon<\frac 12$. Then there is a finite locally free subgroup $C\subset G$ over $R$ such that $C\otimes_R R/p^{1-\epsilon}= C_1\otimes_{R/p} R/p^{1-\epsilon}$.
\end{cor}

\begin{proof} In Theorem \ref{IllusieDefTheory}, we take $A=R/p$, $B_1 = R/p^{2-\epsilon}$, and
\[
B_2 = \{(x,y)\in R/p^{2-2\epsilon}\times R/p \mid x = y\in R/p^{1-\epsilon}\}\ .
\]
One has the map $B_1\to B_2$ sending $x$ to $(x,x)$. Both augmentation ideals $J_i\subset B_i$ are isomorphic to $R/p^{1-\epsilon}$, and the transition map is given by multiplication by $p^\epsilon$. Moreover, one has the group scheme $G\otimes_R R/p^{2-\epsilon}$ over $B_1$, and the morphism $C_1\hookrightarrow G\otimes_R R/p$ over $A$, giving all necessary data. From Theorem \ref{IllusieDefTheory} and the assumption that $p^\epsilon$ is homotopic to $0$ on $\check{\ell}_H = K$, it follows that $o_2=0$. In other words, one gets a lift from $A$ to $B_2$. But lifting from $A$ to $B_2$ is equivalent to lifting from $R/p^{1-\epsilon}$ to $R/p^{2-2\epsilon}$. Thus, everything can be lifted to $R/p^{2-2\epsilon}$, preserving the objects over $R/p^{1-\epsilon}$. As $2-2\epsilon>1$ by assumption, continuing will produce the desired subgroup $C\subset G$.
\end{proof}

\begin{rem} The reader happy with larger (but still explicit) constants, but trying to avoid the subtle deformation theory for group schemes in \cite{IllusieThesis2}, may replace the preceding argument by an argument using the more elementary deformation theory for rings in \cite{IllusieThesis1}. In fact, one can first lift the finite locally free scheme $H$ to $R$ by a similar argument, preserving its reduction to $R/p^{1-\epsilon}$. Next, one can deform the multiplication morphism $H\times H\to H$, preserving its reduction to $R/p^{1-2\epsilon}$, as well as the inverse morphism $H\to H$. The multiplication will continue to be commutative and associative, and the inverse will continue to be an inverse, if $\epsilon$ is small enough. This gives a lift of $H$ to a finite locally free commutative group scheme over $R$, agreeing with the original one modulo $p^{1-2\epsilon}$. Next, one can lift the morphism of finite locally free schemes $G\otimes_R R/p\to H$ to a morphism over $R$, agreeing with the original one modulo $p^{1-3\epsilon}$. Again, this will be a group morphism if $\epsilon$ is small enough. Finally, one takes the kernel of the lifted map.
\end{rem}

\begin{lem}\label{LiftHomEqual} Let $R$ be a $p$-adically complete flat $\Z_p^\cycl$-algebra. Let $X/R$ be a scheme such that $\Omega^1_{X/R}$ is killed by $p^\epsilon$, for some $\epsilon\geq 0$. Let $s,t\in X(R)$ be two sections such that $\bar{s} = \bar{t}\in X(R/p^\delta)$ for some $\delta>\epsilon$. Then $s=t$.
\end{lem}

\begin{proof} By standard deformation theory, the different lifts of $\bar{s}$ to $R/p^{2\delta}$ are a principal homogeneous space for
\[
\Hom(\Omega^1_{X/R}\otimes_{\OO_X} R/p^\delta,R/p^\delta)\ ,
\]
where the tensor product is taken along the map $\OO_X\to R/p^\delta$ coming from $\bar{s}$. Similarly, the different lifts of $\bar{s}$ to $R/p^{2\delta-\epsilon}$ are a principal homogeneous space for
\[
\Hom(\Omega^1_{X/R}\otimes_{\OO_X} R/p^\delta,R/p^{\delta-\epsilon})\ ,
\]
and these identifications are compatible with the evident projection $R/p^\delta\to R/p^{\delta-\epsilon}$. As $M=\Omega^1_{X/R}\otimes_{\OO_X} R/p^\delta$ is killed by $p^\epsilon$, any map $M\to R/p^\delta$ has image in $p^{\delta-\epsilon} R/p^\delta$, and thus has trivial image in $\Hom(\Omega^1_{X/R}\otimes_{\OO_X} R/p^\delta,R/p^{\delta-\epsilon})$. It follows that any two lifts of $\bar{s}$ to $R/p^{2\delta}$ induce the same lift to $R/p^{2\delta-\epsilon}$, so that $s,t\in X(R)$ become equal in $X(R/p^{2\delta-\epsilon})$. Continuing gives the result.
\end{proof}

Let us recall the Hasse invariant. Let $S$ be a scheme of characteristic $p$, and let $A\to S$ be an abelian scheme of dimension $g$. Let $A^{(p)}$ be the pullback of $A$ along the Frobenius of $S$. The Verschiebung isogeny $V: A^{(p)}\to A$ induces a map $V^\ast: \omega_{A/S}\to \omega_{A^{(p)}/S}\cong \omega_{A/S}^{\otimes p}$, i.e. a section $\Ha(A/S)\in \omega_{A/S}^{\otimes (p-1)}$, called the Hasse invariant. We recall the following well-known lemma.

\begin{lem} The section $\Ha(A/S)\in \omega_{A/S}^{\otimes (p-1)}$ is invertible if and only if $A$ is ordinary, i.e. for all geometric points $\bar{x}$ of $S$, $A[p](\bar{x})$ has $p^g$ elements.
\end{lem}

\begin{proof} The Hasse invariant is invertible if and only if the Verschiebung $V: A^{(p)}\to A$ is an isomorphism on tangent spaces. This is equivalent to Verschiebung being finite \'etale, which in turn is equivalent to the condition that the kernel $\Ker V$ of $V$ has $p^g$ distinct geometric points above any geometric point $\bar{x}$ of $S$ (as the degree of $V$ is equal to $p^g$). But $VF = p: A\to A$, and $F$ is purely inseparable, so $A[p](\bar{x}) = (\Ker V)(\bar{x})$.
\end{proof}

\begin{cor}\label{CanSubgroup} Let $R$ be a $p$-adically complete flat $\Z_p^\cycl$-algebra, and let $A\to \Spec R$ be an abelian scheme, with reduction $A_1\to \Spec R/p$. Assume that $\Ha(A_1/\Spec (R/p))^{\frac{p^m-1}{p-1}}$ divides $p^\epsilon$ for some $\epsilon<\frac 12$. Then there is a unique closed subgroup $C_m\subset A[p^m]$ (flat over $R$) such that $C_m = \Ker F^m\subset A[p^m]$ modulo $p^{1-\epsilon}$. For any $p$-adically complete flat $\Z_p^\cycl$-algebra $R^\prime$ with a map $R\to R^\prime$, one has
\[
C_m(R^\prime) = \{s\in A[p^m](R^\prime)\mid s\equiv 0\mod p^{(1-\epsilon)/p^m} \}\ .
\]
\end{cor}

\begin{proof} Let $H_1=\Ker(V^m:A_1^{(p^m)}\to A_1)$, which is a finite locally free group scheme over $R/p$. Then one has a short exact sequence
\[
0\to \Ker F^m\to A_1[p^m]\to H_1\to 0\ .
\]
Moreover, the definition of $H_1$ and the fact that the Lie complex transforms short exact sequences into distinguished triangles compute the Lie complex of $H_1$,
\[
\check{\ell}_{H_1} = (\Lie A_1^{(p^m)}\to \Lie A_1)\ .
\]
Using the definition of the Hasse invariant, the determinant of this map is easily computed to be
\[
\Ha(A_1/(\Spec R/p))^{\frac{p^m-1}{p-1}}\in \omega_{A_1/(\Spec R/p)}^{\otimes (p^m-1)}\ ,
\]
by writing it as a composite of $m$ Verschiebung maps, contributing $\Ha(A_1/(\Spec R/p))^{p^i}$, $i=0,\ldots,m-1$. As multiplication by the determinant is null-homotopic (using the adjugate matrix), our assumptions imply that multiplication by $p^\epsilon$ is homotopic to zero on $\check{\ell}_{H_1}$. Thus, existence of $C_m$ follows from Corollary \ref{LiftGroup}, with $G=A[p^m]$.

For uniqueness, it is enough to prove that the final formula holds for any $C_m\subset A[p^m]$ with $C_m = \Ker F^m$ modulo $p^{1-\epsilon}$. We may assume $R^\prime = R$. Certainly, if $s\in C_m(R)$, then $s_{1-\epsilon}\in C_m(R/p^{1-\epsilon})$ lies in the kernel of $F^m$, as $C_m = \Ker F^m$ modulo $p^{1-\epsilon}$. As the action of $F$ is given by the action on $R/p^{1-\epsilon}$, this translates into the condition $s=0\in C_m(R/p^{(1-\epsilon)/p^m})$. Conversely, assume $s\in A[p^m](R)$ reduces to $0$ modulo $p^{(1-\epsilon)/p^m}$. Following the argument in reverse, we see that $s_{1-\epsilon}\in C_m(R/p^{1-\epsilon})\subset A[p^m](R/p^{1-\epsilon})$. Let $H=A[p^m]/C_m$. We see that the image $t\in H(R)$ of $s$ is $0$ modulo $p^{1-\epsilon}$. Also, $H$ and $H_1$ have the same reduction to $R/p^{1-\epsilon}$; in particular, $\Omega_{H/R}^1$ is killed by $p^\epsilon$. By Lemma \ref{LiftHomEqual}, we find that $t=0\in H(R)$, showing that $s\in C_m(R)$, as desired.
\end{proof}

\begin{definition}\label{DefCanSubgroup} Let $R$ be a $p$-adically complete flat $\Z_p^\cycl$-algebra. We say that an abelian scheme $A\to \Spec R$ has a weak canonical subgroup of level $m$ if $\Ha(A_1/\Spec (R/p))^{\frac{p^m-1}{p-1}}$ divides $p^\epsilon$ for some $\epsilon<\frac 12$. In that case, we call $C_m\subset A[p^m]$ the weak canonical subgroup of level $m$, where $C_m$ is the unique closed subgroup such that $C_m = \Ker F^m\mod p^{1-\epsilon}$.

If moreover $\Ha(A_1/\Spec (R/p))^{p^m}$ divides $p^\epsilon$, then we say that $C_m$ is a canonical subgroup.\footnote{For emphasis, we sometimes call it a strong canonical subgroup.}
\end{definition}

The definition of a (strong) canonical subgroup is made to ensure that the following basic properties are true.

\begin{prop}\label{PropCanSubgroup} Let $R$ be a $p$-adically complete flat $\Z_p^\cycl$-algebra, and let $A,B\to \Spec R$ be abelian schemes.
\begin{altenumerate}
\item[{\rm (i)}] If $A$ has a canonical subgroup $C_m\subset A[p^m]$ of level $m$, then it has a canonical subgroup $C_{m^\prime}\subset A[p^{m^\prime}]$ of any level $m^\prime\leq m$, and $C_{m^\prime}\subset C_m$.
\item[{\rm (ii)}] Let $f: A\to B$ be a morphism of abelian schemes. Assume that both $A$ and $B$ have canonical subgroups $C_m\subset A[p^m]$, $D_m\subset B[p^m]$ of level $m$. Then $C_m$ maps into $D_m$.
\item[{\rm (iii)}] Assume that $A$ has a canonical subgroup $C_{m_1}\subset A[p^{m_1}]$ of level $m_1$. Then $B=A/C_{m_1}$ has a canonical subgroup $D_{m_2}\subset B[p^{m_2}]$ of level $m_2$ if and only if $A$ has a canonical subgroup $C_m\subset A[p^m]$ of level $m=m_1+m_2$. In that case, there is a short exact sequence
\[
0\to C_{m_1}\to C_m\to D_{m_2}\to 0\ ,
\]
commuting with $0\to C_{m_1}\to A\to B\to 0$.
\item[{\rm (iv)}] Assume that $A$ has a canonical subgroup $C_m\subset A[p^m]$ of level $m$, and let $\bar{x}$ be a geometric point of $\Spec R[p^{-1}]$. Then $C_m(\bar{x})\cong (\Z/p^m\Z)^g$, where $g$ is the dimension of the abelian variety over $\bar{x}$.
\end{altenumerate}
\end{prop}

\begin{proof}
\begin{altenumerate}
\item[{\rm (i)}] This follows directly from the displayed formula in Corollary \ref{CanSubgroup}.
\item[{\rm (ii)}] This follows directly from the displayed formula in Corollary \ref{CanSubgroup}.
\item[{\rm (iii)}] Observe that
\[
\Ha(B_{1-\epsilon}/\Spec(R/p^{1-\epsilon})) = \Ha(A_{1-\epsilon}/\Spec(R/p^{1-\epsilon}))^{p^{m_1}}\ .
\]
This implies that $B$ has a canonical subgroup of level $m_2$ if and only if $A$ has a canonical subgroup of level $m$. In order to verify the short exact sequence, it suffices to check that $C_m$ maps into $D_{m_2}$. After base change to the global sections of $C_m$, it suffices to show that $C_m(R)$ maps into $D_{m_2}(R)$. Take a section $s\in C_m(R)$. Look at the short exact sequence $0\to D_{m_2}\to B[p^{m_2}]\to H\to 0$. We need to check that $s$ maps to $0$ in $H(R)$. By Lemma \ref{LiftHomEqual}, it is enough to check that $s$ maps to $0$ in $H(R/p^{1-\epsilon})$. But modulo $p^{1-\epsilon}$, we have the short exact sequence
\[
0\to \Ker F_A^{m_1}\to \Ker F_A^m\to \Ker F_B^{m_2}\to 0\ .
\]
\item[{\rm (iv)}] First, we reduce to the case that $R=\OO_K$ is the ring of integers in an algebraically closed complete nonarchimedean field $K$ of mixed characteristic.

The subset of $\Spec R[p^{-1}]$ where the statement is true is open and closed. Assume that there is a point $x\in \Spec R[p^{-1}]$ with a geometric point $\bar{x}$ above $x$ where the statement is not true, and fix a maximal ideal $x^\prime\in \Spec R$ specializing $x$; in particular, $x^\prime$ lies in $\Spec R/p$. Take a valuation $v$ on $R$ with support $x$, and such that the local ring $R_{x^\prime}$ is contained in the valuation subring of $v$. Applying the specialization mapping from \cite[Proposition 2.6]{HuberContVal}, we get a continuous valuation $v^\prime$ on $R$, such that its support $y\in \Spec R$ is still of characteristic $0$, and is a specialization of $x$. It follows that the desired statement is also false at $y$. This gives a map $R\to \OO_K$ to the ring of integers $\OO_K$ of a complete nonarchimedean field $K$ of mixed characteristic. We may assume that $K$ is algebraically complete, and replace $R$ by $\OO_K$.

Assume that $C_m(K)\not\cong (\Z/p^m\Z)^g$. As $C_m(K)\subset A[p^m](K)\cong (\Z/p^m\Z)^{2g}$ and $C_m(K)$ has $p^{mg}$ elements, it follows from a consideration of elementary divisors that $(C_m\cap A[p])(K)$ has more than $p^g$ elements; in particular, there is an element $s\in (C_m\cap A[p])(K)$ such that $s\not\in C_1(K)$. By projection, this gives a nonzero section $t$ of $H=A[p]/C_1$. Take $\epsilon<\frac 12$ as in Definition \ref{DefCanSubgroup}. As $C_m$ is finite, $s$ extends to a section $s\in C_m(\OO_K)$, giving a section $s_{1-\epsilon}\in C_m(\OO_K/p^{1-\epsilon})$; similarly, we have $t\in H(\OO_K)$. But modulo $p^{1-\epsilon}$, $C_m = \Ker F^m$, so it follows that
\[
F^m(s_{1-\epsilon}) = 0\in C_m(\OO_K/p^{1-\epsilon})\ .
\]
The action of $F$ is given by the action on $\OO_K/p^{1-\epsilon}$, so we see that
\[
s_{(1-\epsilon)/p^m} = 0\in C_m(\OO_K/p^{(1-\epsilon)/p^m})\ .
\]
By projection, this gives $t_{(1-\epsilon)/p^m} = 0\in H(\OO_K/p^{(1-\epsilon)/p^m})$. Now we use Lemma \ref{LiftHomEqual}, with $\delta = (1-\epsilon)/p^m$ and $\epsilon^\prime = \epsilon/p^m$ (which works for the group $H$). Note that
\[
\epsilon^\prime = \epsilon/p^m < \delta = (1-\epsilon)/p^m\ .
\]
It follows that $t=0\in H(\OO_K)$, which contradicts $s\not\in C_1(\OO_K)$, as $0\to C_1\to A[p]\to H\to 0$ is exact.
\end{altenumerate}
\end{proof}

Moreover, one has compatibility with duality and products, but we will not need this.

\subsection{Canonical Frobenius lifts} In this section, we will apply repeatedly Hartog's extension principle. Let us first recall what we will refer to as `classical algebraic Hartog' below.

\begin{prop}\label{ClassicalHartog} Let $R$ be normal and noetherian, and let $Z\subset \Spec R$ be a subset everywhere of codimension $\geq 2$. Then
\[
R = H^0(\Spec R\setminus Z,\OO_{\Spec R})\ .
\]
\end{prop}

\begin{proof} By Serre's criterion, the depth of $\OO_{\Spec R,x}$ is at least $2$ for all $x\in Z$. This implies vanishing of local cohomology groups in degrees $\leq 1$ by SGA2 III Exemple 3.4. This, in turn, implies the desired extension by SGA2 I Proposition 2.13.
\end{proof}

In particular, this applies to an open subset $\Spec R$ of the minimal compactification of the Siegel moduli space, and its boundary $Z$, if $g\geq 2$. However, we will need to work with an admissible blow-up of the minimal compactification, corresponding to a strict neighborhood of the ordinary locus. Thus, we will also need a slightly nonstandard version of Hartog's extension principle, given by the following lemma, where in the application $f$ will be a lift of the Hasse invariant.

\begin{lem}\label{HartogPrinciple} Let $R$ be a topologically finitely generated flat $p$-adically complete $\Z_p$-algebra such that $\overline{R} = R/p$ is normal. Let $f\in R$ be an element such that $\overline{f}\in \overline{R}$ is not a zero divisor. Take some $0\leq \epsilon\leq 1$, and consider the algebra
\[
S=(R\hat\otimes_{\Z_p} \Z_p^\cycl)\langle u\rangle/(fu-p^\epsilon)\ .
\]
Then $S$ is a flat $p$-adically complete $\Z_p^\cycl$-algebra, and $u$ is not a zero divisor. Fix a closed subset $Y\subset \Spec \overline{R}$ of codimension $\geq 2$; let $Z\subset \Spf S$ be the preimage of $Y$, and let $U\subset \Spf S$ be the complement of $Z$. Then
\[
H^0(U,\OO_{\Spf S}) = H^0(\Spf S,\OO_{\Spf S}) = S\ .
\]
\end{lem}

\begin{rem} It may be helpful to illustrate how this relates to the classical theorem of Hartog over $\C$. In that case, inside the $2$-dimensional open complex unit disc $D = \{ (z_1,z_2) \mid |z_1| < 1, |z_2| < 1 \}$, consider the open subset
\[
U = \{ (z_1,z_2)\in D \mid 1 - \epsilon < |z_1| < 1\ \mathrm{or}\ |z_2| < \epsilon \}\ .
\]
Then Hartog's theorem states that all holomorphic functions on $U$ extend uniquely to $D$. This is easily seen to be equivalent to the following statement: Let $D^\prime\subset D$ be the locus $|z_2|>\epsilon/2$. Then holomorphic functions extend uniquely from $U\cap D^\prime$ to $D^\prime$.

In the lemma, take $R=\Z_p\langle T_1,T_2\rangle$, $f=T_2$ and $\epsilon = 1$, say. Moreover, take $Y=\{T_1=T_2=0\}\subset \Spec \overline{R}$. Then the generic fibre of $\Spf S$ is the rigid-analytic space of pairs $(t_1,t_2)$ with $|t_1|\leq 1$, and $|p|\leq |t_2|\leq 1$. The generic fibre of $U$ is given by
\[
U_\eta = \{ (t_1,t_2)\in (\Spf S)_\eta\mid |t_1| = 1\ \mathrm{or}\ |t_2| = 1\}\ .
\]
The lemma asserts that holomorphic functions extend uniquely from $U_\eta$ to $(\Spf S)_\eta$ (and the natural integral subalgebras are preserved). The relation to Hartog's principle becomes most clear when one sets $z_1 = t_1$ and $z_2 = \frac p{t_2}$. The analogue of $D^\prime\subset D$ is given by $(\Spf S)_\eta$, and the analogue of $U\cap D^\prime$ is given by $U_\eta$. (In rigid-analytic geometry, one replaces strict inequalities by nonstrict inequalities. Moreover, the inequalities $1-\epsilon < |z_1| < 1$ become contracted to $|z_1| =1$; similarly, $\epsilon/2<|z_2|<\epsilon$ becomes contracted to $|z_2|=p$.)
\end{rem}

\begin{proof} The first assertions are standard. First, we check that
\[
H^0(\Spf S,\OO_{\Spf S})\to H^0(U,\OO_{\Spf S})
\]
is injective. Since $H^0(\Spf S,\OO_{\Spf S})$ is $p$-adically separated, it suffices to prove the same for $\Spec S_\epsilon$, where $S_\epsilon = S/p^\epsilon$ (and analogous notation is used below). Let $W\subset \Spec S_\epsilon$ be the preimage of $V=V(\overline{f})\subset \Spec \overline{R}$; then $W=V\times_{\Spec \F_p} \mathbb{A}_{\Z_p^\cycl/p^\epsilon}^1$ is affine. There is a section $\Spec R_\epsilon\to \Spec S_\epsilon$ given by setting $u=0$. One has, since $S_\epsilon$ is the scheme-theoretical union of the loci $\{u=0\}$ and $\{\overline{f}=0\}$,
\[\begin{aligned}
H^0(U,\OO_{\Spec S_\epsilon}) = \{ (f_1,f_2)\mid &f_1\in H^0(U\cap \Spec R_\epsilon,\OO_{\Spec R_\epsilon}),\\
& f_2\in H^0(U\cap W,\OO_W), f_1 = f_2\in H^0(U\cap V,\OO_V)\otimes_{\F_p} \Z_p^\cycl/p^\epsilon \}\ .
\end{aligned}\]
One has a similar description for $H^0(\Spec S_\epsilon,\OO_{\Spec S_\epsilon})$. As $H^0(U\cap \Spec R_\epsilon,\OO_{\Spec R_\epsilon}) = R_\epsilon$ by the classical algebraic version of Hartog's extension principle (plus $\otimes_{\F_p} \Z_p^\cycl/p^\epsilon$), it is enough to prove that
\[
H^0(W,\OO_W)\to H^0(U\cap W,\OO_W)
\]
is injective. As both $W$ and $U$ come from $V$ and $U\cap V$ via a product with $\mathbb{A}_{\Z_p^\cycl/p^\epsilon}^1$, it is enough to prove that
\[
H^0(V,\OO_V)\to H^0(U\cap V,\OO_V)
\]
is injective. For any point $x\in Y$, the depth of $\OO_V$ at $x$ is $\geq 1$ (as $\overline{R}$ is normal, by Serre's criterion, $\overline{R}$ has depth $\geq 2$ at $x$, thus $\overline{R}/\overline{f}$ has depth $\geq 1$). As $Y$ contains the complement of $U\cap V$ in $V$, this gives the desired statement.

Now take any section $f\in H^0(U,\OO_{\Spf S})$. Let $\hat{S}_u$ be the $u$-adic completion of $S$. Clearly, $S$ injects into $\hat{S}_u$. Moreover, as $u$ divides $p^\epsilon$, the $(p,u)$-adic ring $\hat{S}_u$ is actually $u$-adic, and $f$ induces a section $\hat{f}_u\in H^0(U\cap \Spf \hat{S}_u,\OO_{\Spf \hat{S}_u})$. The special fibre of $\Spf \hat{S}_u$ is given by $\Spec R_\epsilon$. Thus, $U\cap \Spf \hat{S}_u = U\cap \Spec R_\epsilon\subset \Spec R_\epsilon$ is of codimension $\geq 2$, and the classical algebraic version of Hartog's extension principle ensures that
\[
\hat{f}_u\in H^0(\Spf \hat{S}_u,\OO_{\Spf \hat{S}_u}) = \hat{S}_u\ .
\]
It remains to see that $\hat{f}_u\in S$. It suffices to check modulo $p^\epsilon$ (by a successive approximation argument). Thus, $f$ induces a section $f_2\in H^0(U\cap W,\OO_W)$, and we have to check that it extends to $H^0(W,\OO_W)$. But
\[
H^0(U\cap W,\OO_W) = \bigoplus_{i\geq 0} (H^0(U\cap V,\OO_V)\otimes_{\F_p} \Z_p^\cycl/p^\epsilon) u^i\ ,
\]
and we have to check that all coefficients of $u^i$ lie in $H^0(V,\OO_V)\otimes_{\F_p} \Z_p^\cycl/p^\epsilon$. This can be checked after $u$-adic completion, finishing the proof.
\end{proof}

Now let us go back to Shimura varieties. Recall that $X=X_{g,K^p}$ over $\Z_{(p)}$ is the Siegel moduli space. We let $\mathfrak{X}$ be the formal scheme over $\Z_p^\cycl$ which is the $p$-adic completion of $X\otimes_{\Z_{(p)}} \Z_p^\cycl$. Occasionally, we will use that $\mathfrak{X}$ is already defined over $\Z_p$; we let $\mathfrak{X}_{\Z_p}$ denote the $p$-adic completion of $X$, so that $\mathfrak{X} = \mathfrak{X}_{\Z_p}\times_{\Spf \Z_p} \Spf \Z_p^\cycl$. The same applies for the minimal compactification. In general, formal schemes will be denoted by fractal letters.

We let $\mathcal{X}$ be the generic fibre of $\mathfrak{X}$ as an adic space over $\Q_p^\cycl$. Moreover, for any $K_p$ of the form $\Gamma_0(p^m)$, $\Gamma_1(p^m)$ or $\Gamma(p^m)$, we let $X_{K_p}^\ad$ be the adic space associated with the scheme $X_{K_p}\otimes_{\Q(\zeta_{p^m})} \Q_p^\cycl$, using the tautological element $\zeta_{p^m}\in \Q_p^\cycl$ and $\zeta_{p^m}\in \OO_{X_{K_p}}$, given by the symplectic similitude factor. Let $\mathcal{X}_{K_p}\subset X_{K_p}^\ad$ be the preimage of $\mathcal{X}\subset X^\ad$: This is the locus of good reduction. Again, similar notation applies for the minimal compactification. In general, adic spaces will be denoted by calligraphic letters.

We warn the reader that our notation conflicts slightly with the notation from the introduction. Indeed, $\mathcal{X}_{K_p}^\ast$ now denotes an adic space over $\Q_p^\cycl$. It is the base-change of the space $\mathcal{X}_{K_p}^\ast$, $K_p\in \{\Gamma_0(p^m),\Gamma_1(p^m),\Gamma(p^m)\}$, considered in the introduction along $\Q_p(\zeta_{p^m})\hookrightarrow \Q_p^\cycl$. As in the inverse limit over $m$, the difference goes away, we will forget about this difference.\footnote{A better solution would be to associate the spaces over $\Q_p^\cycl$ with $K_p\cap \Sp_{2g}(\Q_p)$ instead.}

Recall that the Hasse invariant defines a section $\Ha\in H^0(X_{\F_p},\omega^{\otimes (p-1)})$. The sheaf $\omega$ extends to the minimal compactification $X^\ast$. If $g\geq 2$, then classical Hartog implies that $\Ha$ extends to $\Ha\in H^0(X_{\F_p}^\ast, \omega^{\otimes (p-1)})$, as the boundary of the minimal compactification is of codimension $g$. For $g=1$, the Hasse invariant extends by direct inspection.

\begin{definition} Let $0\leq \epsilon<1$ such that there exists an element $p^\epsilon\in \Z_p^\cycl$ of $p$-adic valuation $\epsilon$. Let $\mathfrak{X}^\ast(\epsilon)\to \mathfrak{X}^\ast$ over $\Spf \Z_p^\cycl$ be the functor sending any $p$-adically complete flat $\Z_p^\cycl$-algebra $S$ to the set of pairs $(f,u)$, where $f: \Spf S\to \mathfrak{X}^\ast$ is a map, and $u\in H^0(\Spf S,f^\ast \omega^{\otimes (1-p)})$ is a section such that
\[
u\Ha(\bar{f}) = p^\epsilon\in S/p\ ,
\]
where $\bar{f} = f\otimes_{\Z_p} \F_p$, up to the following equivalence. Two pairs $(f,u)$, $(f^\prime,u^\prime)$ are equivalent if $f=f^\prime$ and there exists some $h\in S$ with $u^\prime = u(1+p^{1-\epsilon}h)$.
\end{definition}

The following lemma explains the choice of the equivalence relation: After choosing a lift $\widetilde{\Ha}$ of $\Ha$ locally, one parametrizes $\tilde{u}$ with $\tilde{u} \widetilde{\Ha} = p^\epsilon\in S$. The point of our definition is to make clear that $\mathfrak{X}^\ast(\epsilon)$ is independent of the choice of the local lift.

\begin{lem} The functor $\mathfrak{X}^\ast(\epsilon)$ is representable by a formal scheme which is flat over $\Z_p^\cycl$. Locally over an affine $\Spf (R\hat{\otimes}_{\Z_p} \Z_p^\cycl) \subset \mathfrak{X}^\ast$ (coming via scalar extension from $\Spf R\subset \mathfrak{X}^\ast_{\Z_p}$), choose a lift $\widetilde{\Ha}\in \omega^{\otimes(p-1)}$ of $\Ha\in \omega^{\otimes (p-1)}/p$. Then
\[
\mathfrak{X}^\ast(\epsilon)\times_{\mathfrak{X}^\ast} \Spf (R\hat{\otimes}_{\Z_p} \Z_p^\cycl) = \Spf ((R\hat{\otimes}_{\Z_p} \Z_p^\cycl)\langle u\rangle/(u \widetilde{\Ha} - p^\epsilon) )\ .
\]
\end{lem}

In particular, $\mathfrak{X}^\ast(\epsilon)\to \mathfrak{X}^\ast$ is an admissible blow-up in the sense of Raynaud.

\begin{proof} By Lemma \ref{HartogPrinciple}, the right-hand side is flat, so it suffices to prove the equality. Clearly, the right-hand side represents the functor of pairs $(f,\tilde{u})$ with $\tilde{u}\in H^0(\Spf S,f^\ast \omega^{\otimes (1-p)})$ such that $\tilde{u} \widetilde{\Ha} = p^\epsilon\in S$. Any such pair gives a pair $(f,u)$. We need to show that conversely, for any pair $(f,u)$, there is a unique pair $(f,\tilde{u})$ equivalent to it, with $\tilde{u} \widetilde{\Ha} = p^\epsilon$.

Note that $u \widetilde{\Ha} = p^\epsilon + ph$ for some $h\in H^0(S,\OO_S)$. Thus, $u\widetilde{\Ha} = p^\epsilon (1 + p^{1-\epsilon} h)$, and setting $\tilde{u} = u (1 + p^{1-\epsilon} h)^{-1}$ gives an equivalent $(f,\tilde{u})$ with $\tilde{u}\widetilde{\Ha} = p^\epsilon$.

If $\tilde{u}^\prime = \tilde{u} (1 + p^{1-\epsilon}h^\prime)$ is equivalent to $\tilde{u}$, and also satisfies $\tilde{u}^\prime \widetilde{\Ha} = p^\epsilon$, then
\[
p^\epsilon = \tilde{u}^\prime \widetilde{\Ha} = \tilde{u}\widetilde{\Ha} + p^{1-\epsilon} h^\prime \tilde{u}\widetilde{\Ha} = p^\epsilon + p h^\prime\ .
\]
As we restricted the functor to flat $\Z_p^\cycl$-algebras, it follows that $h^\prime=0$, as desired.
\end{proof}

By pullback, we get formal schemes $\mathfrak{X}(\epsilon)$ and $\mathfrak{A}(\epsilon)$, where $A\to X$ denotes the universal abelian scheme. Note that on generic fibres, $\mathcal{X}(\epsilon)\subset \mathcal{X}$ is the open subset where $|\Ha|\geq |p|^\epsilon$; similarly for $\mathcal{X}^\ast(\epsilon)\subset \mathcal{X}^\ast$ and $\mathcal{A}(\epsilon)\subset \mathcal{A}$.\footnote{Again, it is understood that this condition is independent of the choice of a local lift of $\Ha$.}

For any formal scheme $\mathfrak{Y}$ over $\Z_p^\cycl$ and $a\in \Z_p^\cycl$, let $\mathfrak{Y}/a$ denote $\mathfrak{Y}\otimes_{\Z_p^\cycl} \Z_p^\cycl/a$.

\begin{lem} Let $0\leq \epsilon<1$. There is a natural commutative diagram
\begin{equation}\label{EqFrobModP}\xymatrix{
\mathfrak{A}(p^{-1}\epsilon)/p \ar[rrrr]^{F_{(\mathfrak{A}(p^{-1}\epsilon)/p)/(\Z_p^\cycl/p)}}\ar[d] &&&& (\mathfrak{A}(p^{-1}\epsilon)/p)^{(p)}\ar[rr]^\cong \ar[d] && \mathfrak{A}(\epsilon)/p\ar[d]\\
\mathfrak{X}(p^{-1}\epsilon)/p \ar[rrrr]^{F_{(\mathfrak{X}(p^{-1}\epsilon)/p)/(\Z_p^\cycl/p)}}\ar[d] &&&& (\mathfrak{X}(p^{-1}\epsilon)/p)^{(p)}\ar[rr]^\cong \ar[d] && \mathfrak{X}(\epsilon)/p\ar[d]\\
\mathfrak{X}^\ast(p^{-1}\epsilon)/p\ar[rrrr]^{F_{(\mathfrak{X}^\ast(p^{-1}\epsilon)/p)/(\Z_p^\cycl/p)}} &&&& (\mathfrak{X}^\ast(p^{-1}\epsilon)/p)^{(p)} \ar[rr]^\cong && \mathfrak{X}^\ast(\epsilon)/p
}\end{equation}
Here, $F$ denotes the relative Frobenius map.
\end{lem}

\begin{proof} The diagram lives over a corresponding diagram for $\mathfrak{A}/p\to \mathfrak{X}/p\hookrightarrow \mathfrak{X}^\ast/p$. Then relative to that base diagram, one adds a section $u$ of $\omega^{\otimes (1-p)}$ such that $u \Ha = p^{p^{-1}\epsilon}$ on the left-hand side, and a section $u^\prime$ of $\omega^{\otimes (1-p)}$ such that $u^\prime \Ha = p^\epsilon$ on the right-hand side.\footnote{To check this, choose a local lift $\widetilde{\Ha}$ of $\Ha$. Then one parametrizes $\tilde{u}$ with $\tilde{u}\widetilde{\Ha} = p^\epsilon$ on the right-hand side, over any ring. As we are working modulo $p$, the choice of $\widetilde{\Ha}$ does not matter.} As $\Ha$ gets raised to the $p$-th power under division by the kernel of Frobenius, one can map $u^\prime$ to $u^{(p)}$ considered as a section of $(\omega^{\otimes (1-p)})^{(p)}$ over $(\mathfrak{X}^\ast(p^{-1}\epsilon)/p)^{(p)}$ (which pulls back to $u^p$ on the left-hand side); this gives the desired canonical maps.
\end{proof}

In this section, we prove the following result.

\begin{thm}\label{ExCanFrobLifts} Let $0\leq \epsilon < \frac 12$.
\begin{altenumerate}
\item[{\rm (i)}] There is a unique diagram
\begin{equation}\label{EqLiftFrob}\xymatrix{
\mathfrak{A}(p^{-1}\epsilon) \ar[rr]^{\tilde{F}_{\mathfrak{A}(p^{-1}\epsilon)}}\ar[d] && \mathfrak{A}(\epsilon)\ar[d]\\
\mathfrak{X}(p^{-1}\epsilon) \ar[rr]^{\tilde{F}_{\mathfrak{X}(p^{-1}\epsilon)}}\ar[d] && \mathfrak{X}(\epsilon)\ar[d]\\
\mathfrak{X}^\ast(p^{-1}\epsilon)\ar[rr]^{\tilde{F}_{\mathfrak{X}^\ast(p^{-1}\epsilon)}} && \mathfrak{X}^\ast(\epsilon)
}\end{equation}
that gets identified with \eqref{EqFrobModP} modulo $p^{1-\epsilon}$.
\item[{\rm (ii)}] For any $m\geq 0$, the abelian variety $\mathfrak{A}(p^{-m}\epsilon)\to \mathfrak{X}(p^{-m}\epsilon)$ admits a canonical subgroup $C_m\subset \mathfrak{A}(p^{-m}\epsilon)[p^m]$ of level $m$. This induces a morphism on the generic fibre
\[
\mathcal{X}(p^{-m}\epsilon)\to \mathcal{X}_{\Gamma_0(p^m)}
\]
given by the pair $(\mathcal{A}(p^{-m}\epsilon)/C_m,\mathcal{A}(p^{-m}\epsilon)[p^m]/C_m)$. This morphism extends uniquely to a morphism $\mathcal{X}^\ast(p^{-m}\epsilon)\to \mathcal{X}_{\Gamma_0(p^m)}^\ast$. These morphisms are open immersions. Moreover, for $m\geq 1$, the diagram
\[\xymatrix{
\mathcal{X}^\ast(p^{-m-1}\epsilon)\ar[r]\ar[d]^{(\tilde{F}_{\mathfrak{X}(p^{-m-1}\epsilon)})^\ad_\eta} & \mathcal{X}_{\Gamma_0(p^{m+1})}^\ast\ar[d]\\
\mathcal{X}^\ast(p^{-m}\epsilon)\ar[r] & \mathcal{X}_{\Gamma_0(p^m)}^\ast
}\]
is commutative and cartesian.
\item[{\rm (iii)}] There is a weak canonical subgroup $C\subset \mathfrak{A}(\epsilon)[p]$ of level $p$. Also write $C\subset \mathcal{A}(\epsilon)[p]$ for its generic fibre, and let $\mathcal{X}_{\Gamma_0(p)}(\epsilon)\to \mathcal{X}(\epsilon)$ be the pullback of $\mathcal{X}_{\Gamma_0(p)}\to \mathcal{X}$. Then the diagram
\[\xymatrix{
\mathcal{X}(p^{-1}\epsilon)\ar[r]\ar[d]^{(\tilde{F}_{\mathfrak{X}(p^{-1}\epsilon)})^\ad_\eta} & \mathcal{X}_{\Gamma_0(p)}(\epsilon)\ar[d]\\
\mathcal{X}(\epsilon)\ar[r]^{=} & \mathcal{X}(\epsilon)
}\]
is commutative, and identifies $\mathcal{X}(p^{-1}\epsilon)$ with the open and closed subset $\mathcal{X}_{\Gamma_0(p)}(\epsilon)_a\subset \mathcal{X}_{\Gamma_0(p)}(\epsilon)$ parametrizing those $D\subset \mathcal{A}(\epsilon)[p]$ with $D\cap C = \{0\}$.
\end{altenumerate}
\end{thm}

\begin{rem} The letter `a' stands for anticanonical, indicating that $D$ is a complement of the (weak) canonical subgroup $C$.
\end{rem}

\begin{proof} First, we handle the assertions in the good reduction case. Thus, we are considering $\mathfrak{X}(\epsilon)$. There is a strong canonical subgroup $C\subset \mathfrak{A}(p^{-1}\epsilon)[p]$ of level $1$, by Definition \ref{DefCanSubgroup}: On $\mathfrak{X}(p^{-1}\epsilon)$, the $p$-th power of the Hasse invariant divides $p^\epsilon$, with $\epsilon<\frac 12$. Note that on $\mathfrak{X}(\epsilon)$, one still has a weak canonical subgroup of level $1$. In particular, we get a second abelian variety $\mathfrak{A}(p^{-1}\epsilon)/C$ over $\mathfrak{X}(p^{-1}\epsilon)$. By uniqueness of $C$, $C$ is totally isotropic; in particular, $\mathfrak{A}(p^{-1}\epsilon)/C$ is naturally principally polarized. Also, it carries a level-$K^p$-structure. Thus, it comes via pullback $\mathfrak{X}(p^{-1}\epsilon)\to \mathfrak{X}$. This morphism lifts uniquely to
\[
\tilde{F}_{\mathfrak{X}(p^{-1}\epsilon)}: \mathfrak{X}(p^{-1}\epsilon)\to \mathfrak{X}(\epsilon)
\]
by a calculation of Hasse invariants. By construction, one has a map $\tilde{F}_{\mathfrak{A}(p^{-1}\epsilon)}: \mathfrak{A}(p^{-1}\epsilon)\to \mathfrak{A}(\epsilon)$ above this map of formal schemes. Moreover, by definition of $C$, these maps reduce to the relative Frobenius maps modulo $p^{1-\epsilon}$. This constructs the maps in part (i). Uniqueness is immediate from uniqueness of the canonical subgroup.

Let us observe that it follows that $\tilde{F}_{\mathfrak{X}(p^{-1}\epsilon)}$ and $\tilde{F}_{\mathfrak{A}(p^{-1}\epsilon)}$ are finite and locally free of degree $g(g+1)/2$. For part (ii), the existence of the canonical subgroup $C_m\subset A(p^{-m}\epsilon)[p^m]$ follows from Definition \ref{DefCanSubgroup} and Corollary \ref{CanSubgroup}. That it induces a morphism
\[
\mathcal{X}(p^{-m}\epsilon)\to \mathcal{X}_{\Gamma_0(p^m)}
\]
follows from Proposition \ref{PropCanSubgroup} (iv) (using that $C_m$ is totally isotropic, by uniqueness). Commutativity of the diagram (in parts (ii) and (iii)) follows from Proposition \ref{PropCanSubgroup} (iii).

Next, observe that the composite
\[
\mathcal{X}(p^{-m}\epsilon)\to \mathcal{X}_{\Gamma_0(p^m)}\to \mathcal{X}\ ,
\]
where the latter map sends a pair $(A,D)$ to $A/D$ (with its canonical principal polarization, and level-$N$-structure), is just the forgetful map $\mathcal{X}(p^{-m}\epsilon)\to \mathcal{X}$. Indeed, the maps send $\mathcal{A}(p^{-m}\epsilon)$ to $(\mathcal{A}(p^{-m}\epsilon)/C_m,\mathcal{A}(p^{-m}\epsilon)[p^m]/C_m)$, and then to
\[
(\mathcal{A}(p^{-m}\epsilon)/C_m)/(\mathcal{A}(p^{-m}\epsilon)[p^m]/C_m) = \mathcal{A}(p^{-m}\epsilon)/\mathcal{A}(p^{-m}\epsilon)[p^m]\cong \mathcal{A}(p^{-m}\epsilon)\ .
\]
Therefore the composite map
\[
\mathcal{X}(p^{-m}\epsilon)\to \mathcal{X}_{\Gamma_0(p^m)}\to \mathcal{X}\ ,
\]
is an open embedding; moreover, the second map is \'etale. It follows that the first map is an open embedding, as desired. Now it follows that the diagram in part (ii) is cartesian on the good reduction locus: Both vertical maps are finite \'etale of degree $p^{g(g+1)/2}$ (using that $m\geq 1$). The same argument works in part (iii) as soon as we have checked that $\mathcal{X}(p^{-1}\epsilon)$ maps into $\mathcal{X}_{\Gamma_0(p)}(\epsilon)_a$.

Thus, take some $\Spf R\subset \mathfrak{X}(p^{-1}\epsilon)$ over which one has an abelian scheme $A_R\to \Spec R$. By assumption, the Hasse invariant divides $p^{p^{-1}\epsilon}$. This gives a (strong) canonical subgroup $C_0\subset A_R[p]$ of level $1$, and  $A^\prime_R = A_R / C_0$ has a weak canonical subgroup $C\subset A^\prime_R[p]$. We have to see that
\[
C\cap (A_R[p]/C_0) = \{0\}\ .
\]
Take a section $s\in A_R[p](S)$, for some $p$-adically complete $p$-torsion free $R$-algebra $S$. If $s$ maps into $C$, then $s$ modulo $p^{1-\epsilon}$ lies in the kernel of Frobenius on $A^\prime_R[p]$, thus $s$ modulo $p^{(1-\epsilon)/p}$ is $0$ in $A^\prime_R[p]$. This means that $s$ modulo $p^{(1-\epsilon)/p}$ lies in $C_0$. Let $H=A_R[p]/C_0$. Then $s$ gives a section $t\in H(S)$, with $t=0$ modulo $p^{(1-\epsilon)/p}$. Moreover, as the Hasse invariant of $A_R$ kills $\Omega^1_{H/S}$, and the Hasse invariant of $A_R$ divides $p^{\epsilon/p}$, one can use Lemma \ref{LiftHomEqual} with $\delta = (1-\epsilon)/p$ and $\epsilon^\prime = \epsilon/p$ to conclude that $t=0\in H(R)$. This finally shows that $C\cap (A_R[p]/C_0) = \{0\}$, as desired.

Now we can extend to the minimal compactification. The case $g=1$ is easy and left to the reader. (It may be reduced to the case $g>1$ by embedding the modular curve into the Siegel $3$-fold via $E\mapsto E\times E$, but one can also argue directly.) If $g\geq 2$, we use our version of Hartog's extension principle. Indeed, Lemma \ref{HartogPrinciple}, applied with $R$ the sections of an affine subset of $\mathfrak{X}^\ast$ and $f=\Ha$ (which is not a zero divisor as the ordinary locus is dense) implies that the maps $\tilde{F}_{\mathfrak{X}(p^{-1}\epsilon)}$ extend uniquely to $\tilde{F}_{\mathfrak{X}^\ast(p^{-1}\epsilon)}$. One gets the commutative diagram in (i), and it reduces to \eqref{EqFrobModP} (using that restriction of functions from $\Spf S$ to $U$ in Lemma \ref{HartogPrinciple} is injective even on the special fibre).

Essentially the same argument proves that the maps to $\mathcal{X}_{\Gamma_0(p^m)}^\ast$ extend: For this, use that if in the situation of Lemma \ref{HartogPrinciple}, one has a finite normal $\mathcal{Y}\to (\Spf S)^\ad_\eta$ and a section $U^\ad_\eta\to \mathcal{Y}$ which is an open embedding, then it extends uniquely to an open embedding $(\Spf S)^\ad_\eta\to \mathcal{Y}$. Indeed, extension is automatic by Lemma \ref{HartogPrinciple} (as $\mathcal{Y}$ is affinoid), and it has to be an open embedding as the section $(\Spf S)^\ad_\eta\to \mathcal{Y}$ is finite and generically an open embedding, thus an open and closed embedding as $\mathcal{Y}$ is normal. The diagram in part (ii) is commutative and cartesian, by using Hartog's principle once more.
\end{proof}

For any $K_p$, let $\mathcal{X}_{K_p}(\epsilon)\subset \mathcal{X}_{K_p}$ be the preimage of $\mathcal{X}(\epsilon)\subset \mathcal{X}$. Similarly, define $\mathcal{X}^\ast_{K_p}(\epsilon)$.

For $m\geq 1$, we define $\mathcal{X}_{\Gamma_0(p^m)}(\epsilon)_a\subset \mathcal{X}_{\Gamma_0(p^m)}$ as the image of $\mathcal{X}(p^{-m}\epsilon)$, and similarly for $\mathcal{X}^\ast$. Observe that $\mathcal{X}_{\Gamma_0(p^m)}(\epsilon)_a\subset \mathcal{X}_{\Gamma_0(p^m)}(\epsilon)$ is open and closed, and is the locus where the universal totally isotropic subgroup $D\subset \mathcal{A}(\epsilon)[p^m]$ satisfies $D[p]\cap C = \{0\}$, for $C\subset \mathcal{A}(\epsilon)[p]$ the weak canonical subgroup, cf. Theorem \ref{ExCanFrobLifts} (ii), (iii).

In fact, also on the minimal compactification, $\mathcal{X}_{\Gamma_0(p^m)}^\ast(\epsilon)_a\subset \mathcal{X}_{\Gamma_0(p^m)}^\ast(\epsilon)$ is open and closed: Open by Theorem \ref{ExCanFrobLifts}, and closed because $\mathcal{X}_{\Gamma_0(p^m)}^\ast(\epsilon)_a\cong \mathcal{X}^\ast(p^{-m}\epsilon)\to \mathcal{X}^\ast(\epsilon)$ is finite.

Thus, we get a tower
\[
\ldots\to \mathcal{X}_{\Gamma_0(p^{m+1})}^\ast(\epsilon)_a\to \mathcal{X}_{\Gamma_0(p^m)}^\ast(\epsilon)_a\to\ldots\to \mathcal{X}_{\Gamma_0(p)}^\ast(\epsilon)_a\ ,
\]
which is the pullback of the tower
\[
\ldots\to \mathcal{X}_{\Gamma_0(p^{m+1})}^\ast\to \mathcal{X}_{\Gamma_0(p^m)}^\ast\to\ldots\to \mathcal{X}_{\Gamma_0(p)}^\ast
\]
along the open embedding $\mathcal{X}_{\Gamma_0(p)}^\ast(\epsilon)_a\subset \mathcal{X}_{\Gamma_0(p)}^\ast$. Moreover, we have integral models for the first tower, such that the transition maps identify with the relative Frobenius maps modulo $p^{1-\epsilon}$. Also, we have the abelian schemes $\mathcal{A}_{\Gamma_0(p^m)}(\epsilon)_a\to \mathcal{X}_{\Gamma_0(p^m)}(\epsilon)_a$ by pullback, and the similar situation there.

Let us state one last result in this subsection.

\begin{lem}\label{MinCompAffinoid} Fix some $0\leq \epsilon<\frac 12$. Then for $m$ sufficiently large, $\mathcal{X}^\ast_{\Gamma_0(p^m)}(\epsilon)_a$ is affinoid.
\end{lem}

\begin{proof} There is some integer $m$ such that $H^i(X^\ast,\omega^{\otimes p^m(p-1)}) = 0$ for all $i>0$. In that case, one can find a global lift $\widetilde{\Ha}^{p^m}$ of $\Ha^{p^m}$. The condition $|\Ha|\geq |p|^{p^{-m}\epsilon}$ is equivalent to $|\widetilde{\Ha}^{p^m}|\geq |p|^\epsilon$. As $\widetilde{\Ha}^{p^m}$ is a section of an ample line bundle, this condition defines an affinoid space $\mathcal{X}^\ast(p^{-m}\epsilon)\cong \mathcal{X}^\ast_{\Gamma_0(p^m)}(\epsilon)_a$.
\end{proof}

\subsection{Tilting} Fix an element $t\in (\Z_p^\cycl)^\flat$ such that $|t|=|t^\sharp|=|p|$; one can do this in such a way that $t$ admits a $p-1$-th root. In that case, one gets an identification $(\Z_p^\cycl)^\flat = \F_p[[t^{1/(p-1)p^\infty}]]$. Let $\mathfrak{X}^\prime$ be the formal scheme over $\F_p[[t^{1/(p-1)p^\infty}]]$ given by the $t$-adic completion of $X_{g,K^p}\otimes_{\Z_{(p)}} \F_p[[t^{1/(p-1)p^\infty}]]$. We denote by $\mathcal{X}^\prime$ over $\F_p((t^{1/(p-1)p^\infty}))$ the generic fibre of $\mathfrak{X}^\prime$. The same applies for $\mathfrak{X}^{\prime\ast}$ and $\mathfrak{A}^\prime$, with generic fibres $\mathcal{X}^{\prime\ast}$ and $\mathcal{A}^\prime$.

In characteristic $p$, one can pass to perfections.

\begin{definition}\begin{altenumerate}
\item[{\rm (i)}] Let $\mathfrak{Y}$ be a flat $t$-adic formal scheme over $\F_p[[t^{1/(p-1)p^\infty}]]$. Let $\Phi: \mathfrak{Y}\to \mathfrak{Y}$ denote the relative Frobenius map. The inverse limit
\[
\varprojlim_\Phi \mathfrak{Y} = \mathfrak{Y}^\perf
\]
is representable by a perfect flat $t$-adic formal scheme over $\F_p[[t^{1/(p-1)p^\infty}]]$. Locally,
\[
(\Spf R)^\perf = \Spf (R^\perf)\ ,
\]
where $R^\perf$ is the $t$-adic completion of $\varinjlim_\Phi R$.
\item[{\rm (ii)}] Let $\mathcal{Y}$ be an adic space over $\F_p((t^{1/(p-1)p^\infty}))$. There is a unique perfectoid space $\mathcal{Y}^\perf$ over $\F_p((t^{1/(p-1)p^\infty}))$ such that
\[
\mathcal{Y}^\perf\sim \varprojlim_\Phi \mathcal{Y}\ ,
\]
where we use $\sim$ in the sense of \cite[Definition 2.4.1]{ScholzeWeinstein}. Locally,
\[
\Spa(R,R^+)^\perf = \Spa(R^\perf,R^{\perf +})\ ,
\]
where $R^{\perf +}$ is the $t$-adic completion of $\varinjlim_\Phi R^+$, and $R^\perf = R^{\perf +}[t^{-1}]$.
\end{altenumerate}
\end{definition}

One checks directly that the two operations are compatible, i.e. $(\mathfrak{Y}^\perf)^\ad_\eta = (\mathfrak{Y}^\ad_\eta)^\perf$. We get perfectoid spaces $\mathcal{X}^{\prime\perf}$, $\mathcal{X}^{\prime\ast\perf}$ and $\mathcal{A}^{\prime\perf}$ over $\F_p((t^{1/(p-1)p^\infty}))$.

\begin{cor}\label{TiltingStrictNbhd} Let $0\leq \epsilon<\frac 12$. There are unique perfectoid spaces $\mathcal{X}_{\Gamma_0(p^\infty)}(\epsilon)_a$, $\mathcal{X}_{\Gamma_0(p^\infty)}^\ast(\epsilon)_a$ and $\mathcal{A}_{\Gamma_0(p^\infty)}(\epsilon)_a$ over $\Q_p^\cycl$ such that
\[
\mathcal{X}_{\Gamma_0(p^\infty)}(\epsilon)_a\sim \varprojlim_m \mathcal{X}_{\Gamma_0(p^m)}(\epsilon)_a\ ,
\]
and similarly for $\mathcal{X}_{\Gamma_0(p^\infty)}^\ast(\epsilon)_a$ and $\mathcal{A}_{\Gamma_0(p^\infty)}(\epsilon)_a$. Moreover, the tilt $\mathcal{X}_{\Gamma_0(p^\infty)}^\ast(\epsilon)_a^\flat$ identifies naturally with the open subset $\mathcal{X}^{\prime\ast\perf}(\epsilon)\subset \mathcal{X}^{\prime\ast\perf}$ where $|\Ha|\geq |t|^\epsilon$. Similarly, $\mathcal{A}_{\Gamma_0(p^\infty)}(\epsilon)_a^\flat$ gets identified with the open subset $\mathcal{A}^{\prime\perf}(\epsilon)\subset \mathcal{A}^{\prime\perf}$ where $|\Ha|\geq |t|^\epsilon$.
\end{cor}

\begin{proof} We give only the proof in the case of $\mathcal{X}$; the other statements are entirely analogous. Note that $\mathcal{X}_{\Gamma_0(p^m)}(\epsilon)_a\cong \mathcal{X}(p^{-m}\epsilon)$ has the integral model $\mathfrak{X}(p^{-m}\epsilon)$. On the tower of the $\mathfrak{X}(p^{-m}\epsilon)$, the transition maps agree with the relative Frobenius map modulo $p^{1-\epsilon}$. Define
\[
\mathfrak{X}_{\Gamma_0(p^\infty)}(\epsilon)_a = \varprojlim \mathfrak{X}(p^{-m}\epsilon)\ ,
\]
where the inverse limit is taken in the category of formal schemes over $\Z_p^\cycl$. Over an affine subset $\Spf R_{m_0}\subset \mathfrak{X}(p^{-m_0}\epsilon)$ with preimages $\Spf R_m\subset \mathfrak{X}(p^{-m}\epsilon)$ for $m\geq m_0$, we get a corresponding open affine subset $\Spf R_\infty\subset \mathfrak{X}_{\Gamma_0(p^\infty)}(\epsilon)_a$, where $R_\infty$ is the $p$-adic completion of $\varinjlim_m R_m$. In particular, $R_\infty$ is flat over $\Z_p^\cycl$. Moreover, using that the transition maps $R_m/p^{1-\epsilon}\to R_{m+1}/p^{1-\epsilon}$ agree with the relative Frobenius map, we find that (absolute) Frobenius induces an isomorphism
\[
R_\infty/p^{(1-\epsilon)/p} = \varinjlim_m R_{m+1}/p^{(1-\epsilon)/p}\cong \varinjlim_m R_m/p^{1-\epsilon} = R_\infty/p^{1-\epsilon}\ .
\]
Thus, by \cite[Definition 5.1 (ii)]{ScholzePerfectoid}, $R_\infty^a$ is a perfectoid $\Z_p^{\cycl a}$-algebra; in particular, $R_\infty[p^{-1}]$ is a perfectoid $\Q_p^\cycl$-algebra (cf. \cite[Lemma 5.6]{ScholzePerfectoid}). Thus, the generic fibre of $\mathfrak{X}_{\Gamma_0(p^\infty)}(\epsilon)_a$ is a perfectoid space $\mathcal{X}_{\Gamma_0(p^\infty)}(\epsilon)_a$ over $\Q_p^\cycl$, with
\[
\mathcal{X}_{\Gamma_0(p^\infty)}(\epsilon)_a\sim \varprojlim_m \mathcal{X}_{\Gamma_0(p^m)}(\epsilon)_a\ ,
\]
cf. \cite[Definition 2.4.1, Proposition 2.4.2]{ScholzeWeinstein}; for uniqueness, cf. \cite[Proposition 2.4.5]{ScholzeWeinstein}.

Now we analyze the tilt. We may define a characteristic $p$-analogue $\mathfrak{X}^{\prime\ast}(\epsilon)$ of $\mathfrak{X}^\ast(\epsilon)$, which relatively over $\mathfrak{X}^{\prime\ast}$ parametrizes sections $u\in \omega^{\otimes (1-p)}$ such that $u\Ha = t^\epsilon$.

Obviously, there are transition maps $\mathfrak{X}^{\prime\ast}(p^{-1}\epsilon)\to \mathfrak{X}^{\prime\ast}(\epsilon)$ given by the relative Frobenius map. Moreover, the inverse limit $\varprojlim_m \mathfrak{X}^{\prime\ast}(p^{-m}\epsilon)$ is representable by a perfect flat formal scheme over $\F_p[[t^{1/(p-1)p^\infty}]]$, which is naturally the same as $\mathfrak{X}^{\prime\ast}(\epsilon)^\perf$. Its generic fibre is thus a perfectoid space over $\F_p((t^{1/(p-1)p^\infty}))$, that is identified with the open subset of $\mathcal{X}^{\prime\ast\perf}$ where $|\Ha|\geq |t|^\epsilon$.

On the other hand, by Theorem \ref{ExCanFrobLifts} (i), one has a canonical identification
\[
\mathfrak{X}^{\prime\ast}(p^{-m}\epsilon) / t^{1-\epsilon} = \mathfrak{X}^\ast(p^{-m}\epsilon)/p^{1-\epsilon}\ ,
\]
compatible with transition maps. Thus, for an open affine $\Spf R_{m_0}\subset \mathfrak{X}^\ast(p^{-m}\epsilon)$ with preimages $\Spf R_m$, one gets similar open affine subsets $\Spf S_m\subset \mathfrak{X}^{\prime\ast}(p^{-m}\epsilon)$, with $S_m / t^{1-\epsilon} = R_m / p^{1-\epsilon}$. Let $R_\infty$ be the $p$-adic completion of $\varinjlim_m R_m$ as above, and $S_\infty$ the $t$-adic completion of $\varinjlim_m S_m$. Then $\Spf R_\infty\subset \mathfrak{X}_{\Gamma_0(p^\infty)}(\epsilon)_a$ and $\Spf S_\infty\subset \mathfrak{X}^{\prime\ast}(\epsilon)^\perf$ give corresponding open subsets, and
\[
R_\infty / p^{1-\epsilon} = \varinjlim_m R_m / p^{1-\epsilon} = \varinjlim_m S_m / t^{1-\epsilon} = S_\infty / t^{1-\epsilon}\ .
\]
From \cite[Theorem 5.2]{ScholzePerfectoid}, it follows that $R_\infty[p^{-1}]$ and $S_\infty[t^{-1}]$ are tilts, as desired.
\end{proof}

\begin{cor}\label{BoundaryStronglyClosedAnticanTower} The space $\mathcal{X}_{\Gamma_0(p^\infty)}^\ast(\epsilon)_a$ is affinoid perfectoid, and the boundary
\[
\mathcal{Z}_{\Gamma_0(p^\infty)}(\epsilon)_a\subset \mathcal{X}_{\Gamma_0(p^\infty)}^\ast(\epsilon)_a
\]
is strongly Zariski closed.
\end{cor}

\begin{proof} It suffices to check the same assertions for the tilts in characteristic $p$ (cf. Lemma \ref{StronglyZariskiClosedTilting}). In characteristic $p$, the open subset $\mathcal{X}^{\prime\ast}(\epsilon)\subset \mathcal{X}^{\prime\ast}$ given by $|\Ha|\geq |t|^\epsilon$ is affinoid, and the boundary $\mathcal{Z}^\prime(\epsilon)\subset \mathcal{X}^{\prime\ast}(\epsilon)$ is Zariski closed. Passing to the perfection, one gets affinoid perfectoid spaces, and a Zariski closed embedding, which is strongly Zariski closed by Lemma \ref{StronglyZariskiClosedCharP}.
\end{proof}

\subsection{Tate's normalized traces} We need Tate's normalized traces to relate the situation at $\Gamma_0(p^\infty)$-level to the situation at some finite $\Gamma_0(p^m)$-level. More precisely, we will use them to extend Hartog's extension principle to $\Gamma_0(p^\infty)$-level, and finite covers thereof.

\begin{lem}\label{TraceDiv} Let $R$ be a $p$-adically complete flat $\Z_p$-algebra. Let $Y_1,\ldots,Y_n\in R$ and take $P_1,\ldots,P_n\in R\langle X_1,\ldots,X_n\rangle$ such that $P_1,\ldots,P_n$ are topologically nilpotent. Let
\[
S=R\langle X_1,\ldots,X_n\rangle / (X_1^p - Y_1 - P_1,\ldots,X_n^p - Y_n - P_n)\ .
\]
\begin{altenumerate}
\item[{\rm (i)}] The ring $S$ is a finite free $R$-module, with basis given by $X_1^{i_1}\cdots X_n^{i_n}$, where $0\leq i_1,\ldots,i_n\leq p-1$.
\item[{\rm (ii)}] Let $I_i\subset R$ be the ideal generated by the coefficients of $P_i$, and let $I=(p,I_1,\ldots,I_n)\subset R$. Then $\tr_{S/R}(S)\subset I^n$.
\end{altenumerate}
\end{lem}

\begin{proof}\begin{altenumerate}
\item[{\rm (i)}] Note that $I$ is finitely generated. By assumption $I^N$ is contained in $pR$ for $N$ large, so $R$ is $I$-adically complete. Modulo $I$, the assertion is clear; moreover, the presentation gives a regular embedding of $\Spec S/I$ into affine $n$-space over $\Spec R/I$. Thus, the Koszul complex $C_1$ for $R/I[X_1,\ldots,X_n]$ and the functions $f_i = X_i^p - Y_i - P_i$ is acyclic in nonzero degrees, and its cohomology in degree $0$ is $S/I$, which is finite free over $R/I$. In particular, $C_1$ is a perfect complex of $R/I$-modules. Looking at the Koszul complex $C_k$ for $R/I^k[X_1,\ldots,X_n]$ and the functions $f_i$, one has
\[
C_k\otimes^{\mathbb{L}}_{R/I^k} R/I = C_1\ ,
\]
which is a perfect complex of $R/I$-modules. It follows that $C_k$ is a perfect complex of $R/I^k$-modules (cf. e.g. \cite[Tag 07LU]{StacksProject}). Moreover, $C_1$ is acyclic in nonzero degrees, which implies that $C_k$ is also acyclic in nonzero degrees, e.g. by writing it as a successive extension of
\[
C_k\otimes^{\mathbb{L}}_{R/I^k} I^j/I^{j+1}\ .
\]
Thus, $C_k$ is quasi-isomorphic to a finite projective $R/I^k$-module in degree $0$, which is finite free modulo $I$ with the desired basis; thus, $C_k$ itself is quasi-isomorphic to a finite free $R/I^k$-module with the desired basis. As $S/I^k$ is the cohomology of $C_k$ in degree $0$, we get the result modulo $I^k$, and then in the inverse limit over $k$ for $S$ itself.

\item[{\rm (ii)}] We make some preliminary reductions. First, we may assume that each $Y_i=W_i^p$ is a $p$-th power; this amounts to a faithfully flat base change. Replacing $X_i$ by $X_i - W_i + 1$, we may assume in fact that $Y_i = 1$. In that case, all $X_i$ are invertible.

Next, it is enough to show $\tr_{S/R}(X_1)\in I^n$. Indeed, it is enough to prove the statement for the basis $X_1^{i_1}\cdots X_n^{i_n}$. If all $i_k=0$, the result is clear. Otherwise, set $X_1^\prime = X_i^{i_1}\cdots X_n^{i_n}$, and choose an invertible $n\times n$-matrix over $\F_p$ with first row $i_1,\ldots,i_n$: This gives elements $X_2^\prime,\ldots,X_n^\prime$ such that $X_1,\ldots,X_n$ may be expressed in terms of $X_1^\prime,\ldots,X_n^\prime$ and $X_1^{\pm p},\ldots,X_n^{\pm p}$ (here, we use that all $X_i$ are invertible). This means that $X_1^\prime,\ldots,X_n^\prime$ generate $S/I$ over $R/I$, thus they generate $S$ over $R$. Moreover, the equations are of the similar form, as one sees after reduction modulo $I$.

Replacing all $X_i$ by $X_i-1$, one may then assume that all $Y_i=0$ instead. One sees easily that one still has to prove $\tr_{S/R}(X_1)\in I^n$. We may also assume that all $P_i$ have only monomials of degree $\leq p-1$ in all $X_i$'s. Finally, we can reduce to the case
\[
R=\Z_p[[a_{i,i_1,\ldots,i_n}]]\ ,
\]
where $1\leq i\leq n$ and $0\leq i_1,\ldots,i_n\leq p-1$, and
\[
P_i = \sum_{i_1,\ldots,i_n} a_{i,i_1,\ldots,i_n} X_1^{i_1}\cdots X_n^{i_n}\ .
\]
(In that case, the $P_i$ are not topologically nilpotent for the $p$-adic topology, but the conclusion of part (i) is still satisfied, and it is enough to prove the analogue of (ii) in this case, as all other cases arise via base change.)

Now,
\[
\bigcap_{0\leq k\leq n-1,\{j_1,\ldots,j_k\}\subset \{2,\ldots,n\}} (p^k I_1 + I_{j_1} + \ldots + I_{j_k})\subset I^n\ :
\]
Assume $x\in R$ is in the left-hand side, but not in the right-hand side. Thus, there exists a monomial
\[
a_{i^{(1)},i_1^{(1)},\ldots,i_n^{(1)}} \cdots a_{i^{(m)},i_1^{(m)},\ldots,i_n^{(m)}}
\]
whose coefficient in $x$ is not divisible by $p^{n-m}$ in $R$ ($m\geq 0$). Enumerate the $j$'s between $2$ and $n$ that do not occur as an $i^{(\cdot)}$ as $j_1,\ldots,j_k$. Thus $k\geq n-m-1$, and $k\geq n-m$ if $1$ is among the $i^{(\cdot)}$'s. Now, using that $x\in p^kI_1 + I_{j_1} + \ldots + I_{j_k}$, we see that $1$ has to be among the $i^{(\cdot)}$'s, and that the desired coefficient is divisible by $p^k$. As $k\geq n-m$, we get the desired contradiction.

We claim that for all $k\geq 0$ and $\{j_1,\ldots,j_k\}\subset \{2,\ldots,n\}$, $\tr_{S/R}(X_1)\in p^k I_1 + I_{j_1} + \ldots + I_{j_k}$. For this, we may assume that $\{j_1,\ldots,j_k\} = \{n-k+1,\ldots,n\}$ (by symmetry), and then divide by the ideal $I_{n-k+1}+\ldots+I_n$; thus, we may assume $P_{n-k+1} = \ldots = P_n = 0$. In that case, all $X_{n-k+1},\ldots,X_n$ are nilpotent; let $\bar{S} = S/(X_{n-k+1},\ldots,X_n)$. One finds that $\tr_{S/R}(X_1) = p^k \tr_{\bar{S}/R}(X_1)$, so we may assume that $k=0$. In that case, we have to prove $\tr_{S/R}(X_1)\in I_1$. We can compute the trace by using the basis $X_1^{i_1}\cdots X_n^{i_n}$, $0\leq i_1,\ldots,i_n\leq p-1$. If $i_1<p-1$, then multiplication by $X_1$ maps this to a different basis element, so that it does not contribute to the trace. If $i_1=p-1$, then
\[
X_1\cdot (X_1^{i_1}\cdots X_n^{i_n}) = P_1 X_2^{i_2}\cdots X_n^{i_n}\ ,
\]
which contributes an element of $I_1$ to the trace.
\end{altenumerate}
\end{proof}

\begin{cor}\label{CorTraceDiv} Let $R$ be a $p$-adically complete $\Z_p$-algebra topologically of finite type, formally smooth of dimension $n$, and let $f\in R$ such that $\overline{f}\in \overline{R} = R/p$ is not a zero divisor. For $0\leq \epsilon<1$, define
\[
R_\epsilon = (R\hat{\otimes}_{\Z_p} \Z_p^\cycl)\langle u_\epsilon\rangle / (fu_\epsilon-p^\epsilon)\ .
\]
Let $\varphi: R_\epsilon\to R_{\epsilon/p}$ be a map of $\Z_p^\cycl$-algebras such that $\varphi\mod p^{1-\epsilon}$ is given by the map
\[
(\overline{R}\otimes_{\F_p} \Z_p^\cycl/p^{1-\epsilon})[u_\epsilon]/(fu_\epsilon - p^\epsilon)\to (\overline{R}\otimes_{\F_p} \Z_p^\cycl/p^{1-\epsilon})[u_{\epsilon/p}]/(fu_{\epsilon/p} - p^{\epsilon/p})
\]
which is the Frobenius on $\overline{R}$, and sends $u_\epsilon$ to $u_{\epsilon/p}^p$. We assume that $\epsilon<\frac 12$.
\begin{altenumerate}
\item[{\rm (i)}] The map $\varphi[\frac 1p]: R_\epsilon[\frac 1p]\to R_{\epsilon/p} [\frac 1p]$ is finite and flat.
\item[{\rm (ii)}] The trace map
\[
\tr_{R_{\epsilon/p} [\frac 1p]/R_\epsilon[\frac 1p]}: R_{\epsilon/p}[\tfrac 1p]\to R_\epsilon[\tfrac 1p]
\]
maps $R_{\epsilon/p}$ into $p^{n-(2n+1)\epsilon} R_\epsilon$.
\end{altenumerate}
\end{cor}

\begin{proof} All assertions are local on $\Spf R$. Thus, we may assume that there is an \'etale map $\Spf R\to \Spf \Z_p\langle Y_1,\ldots,Y_n\rangle$. In particular, regarding $\overline{R}$ as an $\overline{R}$-module via Frobenius, it is free with basis given by $Y_1^{i_1}\cdots Y_n^{i_n}$, $0\leq i_1,\ldots,i_n\leq p-1$.

\begin{altenumerate}
\item[{\rm (i)}] Consider the $\Z_p^\cycl$-algebra
\[
R_{\epsilon/p}^\prime = (R\hat{\otimes}_{\Z_p} \Z_p^\cycl)\langle v_\epsilon\rangle/(f^p v_\epsilon - p^\epsilon)\ .
\]
There is a map $\tau: R_{\epsilon/p}^\prime\to R_{\epsilon/p}$ by mapping $v_\epsilon$ to $u_{\epsilon/p}^p$. After inverting $p$, $\tau$ becomes an isomorphism, the inverse being given by mapping $u_{\epsilon/p}$ to $p^{-(p-1)\epsilon/p} f^{p-1} v_\epsilon$. As $R_{\epsilon/p}^\prime$ is $p$-torsion free by Lemma \ref{HartogPrinciple}, it follows that $\tau$ is injective. As $p^{(p-i)\epsilon/p} u_{\epsilon/p}^i = f^{p-i} v_\epsilon$ for $i=1,\ldots,p-1$, the cokernel of $\tau$ is killed by $p^{(p-1)\epsilon/p}$; in particular, by $p^\epsilon$.

We claim that $\varphi: R_\epsilon\to R_{\epsilon/p}$ factors over a map $\psi: R_\epsilon\to R_{\epsilon/p}^\prime$. As $\epsilon<1-\epsilon$, this can be checked after reduction modulo $p^{1-\epsilon}$. By assumption, $\varphi\mod p^{1-\epsilon}$ factors as a composite
\[\begin{aligned}
(\overline{R}\otimes_{\F_p} \Z_p^\cycl/p^{1-\epsilon})[u_\epsilon]/(fu_\epsilon - p^\epsilon)&\to (\overline{R}\otimes_{\F_p} \Z_p^\cycl/p^{1-\epsilon})[v_\epsilon]/(f^p v_\epsilon - p^\epsilon)\\
&\buildrel \tau\over \to (\overline{R}\otimes_{\F_p} \Z_p^\cycl/p^{1-\epsilon})[u_{\epsilon/p}]/(fu_{\epsilon/p} - p^{\epsilon/p})\ ,
\end{aligned}\]
where the first map is the Frobenius on $\overline{R}$ and sends $u_\epsilon$ to $v_\epsilon$, and the second map is $\tau\mod p^{1-\epsilon}$. This gives the desired factorization.

Moreover, the kernel of $\tau\mod p^{1-\epsilon}$ is killed by $p^\epsilon$, as the kernel comes from a $\Tor_1$-term with coefficients in the cokernel of $\tau$. It follows that $\psi\mod p^{1-2\epsilon}$ agrees with the map
\[
(\overline{R}\otimes_{\F_p} \Z_p^\cycl/p^{1-2\epsilon})[u_\epsilon]/(fu_\epsilon - p^\epsilon)\to (\overline{R}\otimes_{\F_p} \Z_p^\cycl/p^{1-2\epsilon})[v_\epsilon]/(f^p v_\epsilon - p^\epsilon)
\]
which is the Frobenius on $\overline{R}$ and sends $u_\epsilon$ to $v_\epsilon$. This is finite free with basis $Y_1^{i_1}\cdots Y_n^{i_n}$, $0\leq i_1,\ldots,i_n\leq p-1$. It follows that the same is true for $\psi$, as desired.

\item[{\rm (ii)}] It suffices to show that
\[
\tr_{R_{\epsilon/p}^\prime/R_\epsilon}: R_{\epsilon/p}^\prime\to R_\epsilon
\]
has image contained in $p^{n-2n\epsilon} R_\epsilon$. But we can write
\[
R_{\epsilon/p}^\prime = R_\epsilon\langle X_1,\ldots,X_n\rangle/(X_1^p - Y_1 - P_1,\ldots,X_n^p - Y_n - P_n)\ ,
\]
where $P_1,\ldots,P_n\in p^{1-2\epsilon} R_\epsilon\langle X_1,\ldots,X_n\rangle$. Applying Lemma \ref{TraceDiv} gives the result.
\end{altenumerate}
\end{proof}

\begin{cor}\label{CorTateTraces} Fix $0\leq \epsilon<\frac 12$, and consider the formal scheme
\[
\mathfrak{X}_{\Gamma_0(p^\infty)}(\epsilon)_a = \varprojlim_m \mathfrak{X}(p^{-m}\epsilon)
\]
over $\Z_p^\cycl$. Fix some $m\geq 0$. For $m^\prime\geq m$, the maps
\[
1/p^{(m^\prime-m)g(g+1)/2} \tr: \OO_{\mathfrak{X}(p^{-m^\prime}\epsilon)}[p^{-1}]\to \OO_{\mathfrak{X}(p^{-m}\epsilon)}[p^{-1}]
\]
are compatible for varying $m^\prime$, and give a map
\[
\overline{\tr}_m: \varinjlim_{m^\prime} \OO_{\mathfrak{X}(p^{-m^\prime}\epsilon)}[p^{-1}]\to \OO_{\mathfrak{X}(p^{-m}\epsilon)}[p^{-1}]\ .
\]
The image of $\varinjlim_{m^\prime} \OO_{\mathfrak{X}(p^{-m^\prime}\epsilon)}$ is contained in $p^{-C_m} \OO_{\mathfrak{X}(p^{-m}\epsilon)}$ for some constant $C_m$, with $C_m\to 0$ as $m\to \infty$. Thus, $\overline{\tr}_m$ extends by continuity to a map
\[
\overline{\tr}_m: \OO_{\mathfrak{X}_{\Gamma_0(p^\infty)}(\epsilon)_a}[p^{-1}]\to \OO_{\mathfrak{X}(p^{-m}\epsilon)}[p^{-1}]\ ,
\]
called Tate's normalized trace. Moreover, for any $x\in \OO_{\mathfrak{X}_{\Gamma_0(p^\infty)}(\epsilon)_a}[p^{-1}]$,
\[
x = \lim_{m\to\infty} \overline{\tr}_m(x)\ .
\]
\end{cor}

\begin{proof} We only need to prove existence of $C_m$, with $C_m\to 0$ as $m\to \infty$. Observe that by Theorem \ref{ExCanFrobLifts} (i), the transition maps
\[
\tilde{F}_{\mathfrak{X}(p^{-m-1}\epsilon)}: \mathfrak{X}(p^{-m-1}\epsilon)\to \mathfrak{X}(p^{-m}\epsilon)
\]
are of the type considered in Corollary \ref{CorTraceDiv} for $\epsilon^\prime = p^{-m}\epsilon$, with $n=g(g+1)/2$. It follows that
\[
1/p^{g(g+1)/2} \tr: \OO_{\mathfrak{X}(p^{-m-1}\epsilon)}[p^{-1}]\to \OO_{\mathfrak{X}(p^{-m}\epsilon)}[p^{-1}]
\]
maps $\OO_{\mathfrak{X}(p^{-m-1}\epsilon)}$ into $p^{-(g^2+g+1)\epsilon/p^m} \OO_{\mathfrak{X}(p^{-m}\epsilon)}$. As the sum of $(g^2+g+1)\epsilon/p^{m^\prime}$ over all $m^\prime\geq m$ exists, one gets the existence of $C_m$; moreover, $C_m\to 0$ as $m\to \infty$.
\end{proof}

\subsection{Conclusion}\label{ConclusionSubsection} Recall that we have proved that a strict neighborhood of the anticanonical locus becomes perfectoid at $\Gamma_0(p^\infty)$-level. Our goal in this section is to extend this result to full $\Gamma(p^\infty)$-level. This is done in two steps: From $\Gamma_0(p^\infty)$-level to $\Gamma_1(p^\infty)$-level, and from $\Gamma_1(p^\infty)$-level to $\Gamma(p^\infty)$-level. The second part is easy, and follows from almost purity, as there is no ramification at the boundary.

More critical is the transition from $\Gamma_0(p^\infty)$-level to $\Gamma_1(p^\infty)$-level. The issue is that it is very hard to understand what happens at the boundary. Our strategy is to first guess what the tilt of the space is, and then prove that our guess is correct. Away from the boundary, it is clear which finite \'etale cover to take. In characteristic $p$, one can build a candidate by taking the perfection of the normalization. One can take the untilt of this space, and we want to compare this with the spaces in characteristic $0$. Away from the boundary, this can be done. To extend to the whole space, we need two ingredients: The Hebbarkeitssatz for the candidate space in characteristic $p$, and Hartog's extension principle for the space in characteristic $0$.

Assume that $g\geq 2$ until further notice. We start by proving the version of Hartog's extension principle that we will need. This follows from a combination of the earlier version of Hartog (which is a statement at finite level) with Tate's normalized traces.

\begin{lem}\label{HardHartog} Let $\mathcal{Y}_m^\ast\to \mathcal{X}_{\Gamma_0(p^m)}^\ast(\epsilon)_a$ be finite, \'etale away from the boundary, and assume that $\mathcal{Y}_m^\ast$ is normal, and that no irreducible component of $\mathcal{Y}_m^\ast$ maps into the boundary. In particular, $\mathcal{Y}_m\to \mathcal{X}_{\Gamma_0(p^m)}(\epsilon)_a$ is finite \'etale, where $\mathcal{Y}_m\subset \mathcal{Y}_m^\ast$ is the preimage of $\mathcal{X}_{\Gamma_0(p^m)}(\epsilon)_a\subset \mathcal{X}_{\Gamma_0(p^m)}^\ast(\epsilon)_a$. For $m^\prime\geq m$, let $\mathcal{Y}_{m^\prime}^\ast\to \mathcal{X}_{\Gamma_0(p^{m^\prime})}^\ast(\epsilon)_a$ be the normalization of the pullback, with $\mathcal{Y}_{m^\prime}\subset \mathcal{Y}_{m^\prime}^\ast$. Let $\mathcal{Y}_\infty$ be the pullback of $\mathcal{Y}_m$ to $\mathcal{X}_{\Gamma_0(p^\infty)}(\epsilon)_a$, which exists as $\mathcal{Y}_m\to \mathcal{X}_{\Gamma_0(p^m)}(\epsilon)_a$ is finite \'etale.

Observe that as $\mathcal{X}_{\Gamma_0(p^{m^\prime})}^\ast(\epsilon)_a$ is affinoid for $m^\prime$ sufficiently large (cf. Lemma \ref{MinCompAffinoid}), all $\mathcal{Y}_{m^\prime}^\ast = \Spa(S_{m^\prime},S_{m^\prime}^+)$ (with $S_{m^\prime}^+ = S_{m^\prime}^\circ$) are affinoid for $m^\prime$ sufficiently large.
\begin{altenumerate}
\item[{\rm (i)}] For all $m^\prime$ sufficiently large,
\[
S_{m^\prime}^+ = H^0(\mathcal{Y}_{m^\prime},\OO_{\mathcal{Y}_{m^\prime}}^+)\ .
\]
\item[{\rm (ii)}] The map
\[
\varinjlim_{m^\prime} S_{m^\prime}^+\to H^0(\mathcal{Y}_\infty,\OO_{\mathcal{Y}_\infty}^+)
\]
is injective with dense image. Moreover, there are canonical continuous retractions
\[
H^0(\mathcal{Y}_\infty,\OO_{\mathcal{Y}_\infty})\to S_{m^\prime}\ .
\]
\item[{\rm (iii)}] Assume that $S_\infty = H^0(\mathcal{Y}_\infty,\OO_{\mathcal{Y}_\infty})$ is a perfectoid $\Q_p^\cycl$-algebra; define
\[
\mathcal{Y}_\infty^\ast = \Spa (S_\infty,S_\infty^+)\ ,
\]
where $S_\infty^+ = S_\infty^\circ$. Then $\mathcal{Y}_\infty^\ast$ is an affinoid perfectoid space over $\Q_p^\cycl$, $\mathcal{Y}_\infty^\ast\sim \varprojlim_{m^\prime} \mathcal{Y}_{m^\prime}^\ast$, and $S_\infty^+$ is the $p$-adic completion of $\varinjlim_{m^\prime} S_{m^\prime}^+$.
\end{altenumerate}
\end{lem}

\begin{proof}\begin{altenumerate}
\item[{\rm (i)}] We may assume $m = m^\prime$, so that Lemma \ref{MinCompAffinoid} applies. Let
\[
S = S_m\ ,\ R=H^0(\mathcal{X}_{\Gamma_0(p^m)}^\ast(\epsilon)_a,\OO_{\mathcal{X}_{\Gamma_0(p^m)}^\ast(\epsilon)_a})\ .
\]
Then $S$ is a finite $R$-module, and $R$ and $S$ are normal and noetherian. Let $Z\subset \Spec R$ denote the boundary, which is of codimension $\geq 2$, with preimage $Z^\prime\subset \Spec S$, again of codimension $\geq 2$ (by the assumption on irreducible components). Thus, $S=H^0(\Spec S\setminus Z^\prime,\OO_{\Spec S})$, and $R=H^0(\Spec R\setminus Z,\OO_{\Spec R})$. Away from $Z$, the map is finite \'etale, so that one has a trace map $\tr_{S/R}: S\to R$ (a priori only on the structure sheaf away from the boundary, but then by taking global sections on $S$). Moreover, the trace pairing
\[
S\otimes_R S\to R: s_1\otimes s_2\mapsto \tr_{S/R}(s_1s_2)
\]
induces an isomorphism $S\to \Hom_R(S,R)$: If $s_1\in S$ is in the kernel, it still lies in the kernel of the pairing away from the boundary. There, it is perfect (as the map is finite \'etale), thus $s_1$ vanishes away from the boundary, thus is $0$. Similarly, given an element of $\Hom_R(S,R)$, it comes from a unique element of $S$ away from the boundary, thus from an element of $S$, as $S=H^0(\Spec S\setminus Z^\prime,\OO_{\Spec S})$.

Arguing as in the proof of Lemma \ref{GoodTripleFinite} (iii) (i.e., repeating the argument after pullback to affinoid open subsets of $\mathcal{X}_{\Gamma_0(p^m)}^\ast(\epsilon)_a$), we see that for all open subsets $\mathcal{U}\subset \mathcal{X}_{\Gamma_0(p^m)}^\ast(\epsilon)_a$ with preimage $\mathcal{V}\subset \mathcal{Y}_m^\ast$, the trace pairing gives an isomorphism
\[
H^0(\mathcal{V},\OO_{\mathcal{Y}_m^\ast})\cong \Hom_R(S,H^0(\mathcal{U},\OO_{\mathcal{X}_{\Gamma_0(p^m)}^\ast(\epsilon)_a}))\ .
\]
Thus, the desired statement follows from
\[
H^0(\mathcal{X}_{\Gamma_0(p^m)}(\epsilon)_a,\OO_{\mathcal{X}_{\Gamma_0(p^m)}^\ast(\epsilon)_a}) = H^0(\mathcal{X}_{\Gamma_0(p^m)}^\ast(\epsilon)_a,\OO_{\mathcal{X}_{\Gamma_0(p^m)}^\ast(\epsilon)_a})\ ,
\]
which is a consequence of Lemma \ref{HartogPrinciple}.
\item[{\rm (ii)}] Use Tate's normalized traces (Corollary \ref{CorTateTraces}) (and part (i)) to produce the retractions (proving injectivity). Moreover, Tate's normalized traces for varying $m^\prime$ converge to the element one started with, giving the density.
\item[{\rm (iii)}] This is immediate from (ii).
\end{altenumerate}
\end{proof}

First, we deal with the case of adjoining a $\Gamma_1(p^m)$-level structure. Assume first that $g\geq 2$; the case $g=1$ can be handled directly, cf. below. Note that on the tower $\mathcal{X}_{\Gamma_0(p^m)}(\epsilon)_a$, we have the tautological abelian variety $\mathcal{A}_{\Gamma_0(p^m)}^t(\epsilon)_a$ (which are related to each other by pullback), as well as the abelian varieties $\mathcal{A}_{\Gamma_0(p^m)}(\epsilon)_a = \mathcal{A}(p^{-m}\epsilon)$ over $\mathcal{X}_{\Gamma_0(p^m)}(\epsilon)_a\cong \mathcal{X}(p^{-m}\epsilon)$. They are related by an isogeny
\[
\mathcal{A}_{\Gamma_0(p^m)}(\epsilon)_a\to \mathcal{A}_{\Gamma_0(p^m)}^t(\epsilon)_a\ ,
\]
whose kernel is the canonical subgroup $C_m\subset \mathcal{A}_{\Gamma_0(p^m)}(\epsilon)_a[p^m]$ of level $m$. One gets an induced subgroup
\[
D_m = \mathcal{A}_{\Gamma_0(p^m)}(\epsilon)_a[p^m] / C_m\subset \mathcal{A}_{\Gamma_0(p^m)}^t(\epsilon)_a\ .
\]
Let $D_{m\Gamma_0(p^{m^\prime})}$ be the pullback of $D_m$ to $\mathcal{X}_{\Gamma_0(p^{m^\prime})}(\epsilon)_a$ for $m^\prime\geq m$. One has $D_{m\Gamma_0(p^{m^\prime})} = D_{m^\prime}[p^m]$. Also, the $D_m$ give the $\Gamma_0(p^m)$-level structure.

Let $D_{m\Gamma_0(p^\infty)}$ denote the pullback of $D_m$ to $\mathcal{X}_{\Gamma_0(p^\infty)}(\epsilon)_a$; as $D_m\to \mathcal{X}_{\Gamma_0(p^m)}(\epsilon)_a$ is finite \'etale, $D_{m\Gamma_0(p^\infty)}$ is a perfectoid space.

\begin{lem} The map
\[
\mathcal{A}_{\Gamma_0(p^\infty)}(\epsilon)_a[p^m]\to D_{m\Gamma_0(p^\infty)}
\]
is an isomorphism of perfectoid spaces.
\end{lem}

\begin{proof} Let $(R,R^+)$ be a perfectoid affinoid $\Q_p^\cycl$-algebra. Then
\[
\mathcal{A}_{\Gamma_0(p^\infty)}(\epsilon)_a[p^m](R,R^+) = \varprojlim_{m^\prime} \mathcal{A}_{\Gamma_0(p^{m^\prime})}(\epsilon)_a[p^m](R,R^+)\ .
\]
The transition map
\[
\mathcal{A}_{\Gamma_0(p^{m^\prime+m})}(\epsilon)_a[p^m]\to \mathcal{A}_{\Gamma_0(p^{m^\prime})}(\epsilon)_a[p^m]
\]
kills the canonical subgroup $C_m$, so that it factors as
\[
\mathcal{A}_{\Gamma_0(p^{m^\prime+m})}(\epsilon)_a[p^m]\to \mathcal{A}_{\Gamma_0(p^{m^\prime+m})}(\epsilon)_a[p^m]/C_m = D_{m\Gamma_0(p^{m^\prime+m})}\to \mathcal{A}_{\Gamma_0(p^{m^\prime})}(\epsilon)_a[p^m]\ .
\]
This shows that the projective limit is the same as the projective limit
\[
D_{m\Gamma_0(p^\infty)}(R,R^+) = \varprojlim_{m^\prime} D_{m\Gamma_0(p^{m^\prime})}(R,R^+)\ .
\]
\end{proof}

Let $D_m^\prime\to \mathcal{X}^\prime(\epsilon)\subset \mathcal{X}^\prime$ denote the quotient $\mathcal{A}^\prime(\epsilon)[p^m]/C_m^\prime$, where $C_m^\prime$ denotes the canonical subgroup on the ordinary locus in characteristic $p$. Note that all abelian varieties over $\F_p((t^{1/(p-1)p^\infty}))$ parametrized by $\mathcal{X}^\prime(\epsilon)$ are ordinary, as the Hasse invariant divides $t^\epsilon$, and thus is invertible.\footnote{Of course, the abelian varieties need not have good ordinary reduction.}

\begin{lem}\label{TiltingAntiCanSubgroup} The tilt of $D_{m\Gamma_0(p^\infty)}$ identifies canonically with the perfection of $D_m^\prime$.
\end{lem}

\begin{proof} As the kernel of Frobenius (i.e., $C_m^\prime$) gets killed under perfection, we have
\[
(D_m^\prime)^\perf = \mathcal{A}^\prime(\epsilon)[p^m]^\perf\ .
\]
By Corollary \ref{TiltingStrictNbhd}, the right-hand side is the tilt of $\mathcal{A}_{\Gamma_0(p^\infty)}(\epsilon)_a[p^m]$, so applying the previous lemma finishes the proof.
\end{proof}

Let $X^{\ord\ast}\subset X^\ast\otimes_{\Z_{(p)}} \F_p$ be the locus where the Hasse invariant is invertible; thus, $X^{\ord\ast}$ is affine over $\F_p$. Let $X^\ord\subset X\otimes_{\Z_{(p)}} \F_p$ be the preimage, which is the ordinary locus, and let $D_m^\ord\to X^\ord$ be the quotient of the $p^m$-torsion of the universal abelian variety by the canonical subgroup. Now, let $X_{\Gamma_1(p^m)}^\ord\to X^\ord$ parametrize isomorphisms $D_m^\ord\cong (\Z/p^m\Z)^g$. Then $X_{\Gamma_1(p^m)}^\ord\to X^\ord$ is a finite map of schemes over $\F_p$. Recall that we are assuming $g\geq 2$; thus, we find that setting
\[
X_{\Gamma_1(p^m)}^{\ord\ast} = \Spec H^0(X_{\Gamma_1(p^m)}^\ord,\OO_{X_{\Gamma_1(p^m)}^\ord})\ ,
\]
the map $X_{\Gamma_1(p^m)}^{\ord\ast}\to X^{\ord\ast}$ is a finite map of affine schemes over $\F_p$, such that $X_{\Gamma_1(p^m)}^\ord$ is the preimage of $X^\ord$. Also, $X_{\Gamma_1(p^m)}^{\ord\ast}$ is normal.

Let $\mathcal{X}_{\Gamma_1(p^m)}^{\prime\ast}(\epsilon)$ be the open locus of the adic space associated with $X_{\Gamma_1(p^m)}^{\ord\ast}\otimes \F_p((t^{1/(p-1)p^\infty}))$ where $|\Ha|\geq |t|^\epsilon$. Then
\[
\mathcal{X}_{\Gamma_1(p^m)}^{\prime\ast}(\epsilon)\to \mathcal{X}^{\prime\ast}(\epsilon)
\]
is finite, and \'etale away from the boundary. In particular, the base-change $\mathcal{X}_{\Gamma_1(p^m)}^\prime(\epsilon)\to \mathcal{X}^\prime(\epsilon)\subset \mathcal{X}^{\prime\ast}(\epsilon)$ is finite \'etale, parametrizing isomorphisms $D_m^\prime\cong (\Z/p^m\Z)^g$.

Let $\mathcal{Z}^{\prime\ast}(\epsilon)\subset \mathcal{X}^{\prime\ast}(\epsilon)$ denote the boundary, with pullback $\mathcal{Z}^{\prime\ast}_{\Gamma_1(p^m)}(\epsilon)\subset \mathcal{X}^{\prime\ast}_{\Gamma_1(p^m)}(\epsilon)$.

\begin{lem}\label{BaseTripleIsGood} The triple $(\mathcal{X}^{\prime\ast}(\epsilon)^\perf,\mathcal{Z}^{\prime\ast}(\epsilon)^\perf,\mathcal{X}^\prime(\epsilon)^\perf)$ is good, cf. Definition \ref{DefGood}, i.e.
\[
H^0(\mathcal{X}^{\prime\ast}(\epsilon)^\perf,\OO^\circ/t)^a\cong H^0(\mathcal{X}^{\prime\ast}(\epsilon)^\perf\setminus \mathcal{Z}^{\prime\ast}(\epsilon)^\perf,\OO^\circ/t)^a\hookrightarrow H^0(\mathcal{X}^\prime(\epsilon)^\perf,\OO^\circ/t)^a\ .
\]
\end{lem}

\begin{proof} Recall that $\mathcal{X}^{\prime\ast}(\epsilon)$ is the generic fibre of the formal scheme
\[
\mathfrak{X}^{\prime\ast}(\epsilon)\to \mathfrak{X}^{\prime\ast}
\]
parametrizing $u$ with $u\Ha = t^\epsilon$. It is enough to prove that for any open affine formal subscheme $\mathfrak{U}\subset \mathfrak{X}^{\prime\ast}$, the corresponding triple one gets by pullback is good. This follows from Lemma \ref{GoodTriple} (with $t$ replaced by $t^\epsilon$), cf. also Corollary \ref{GoodTripleBaseField}. Observe that $X^\ast\otimes_{\Z_{(p)}} \F_p$ admits a resolution of singularities, given by the toroidal compactification, cf. \cite{FaltingsChai}.
\end{proof}

\begin{cor}\label{FiniteTripleIsGood} The triple $(\mathcal{X}^{\prime\ast}_{\Gamma_1(p^m)}(\epsilon)^\perf,\mathcal{Z}^{\prime\ast}_{\Gamma_1(p^m)}(\epsilon)^\perf,\mathcal{X}^\prime_{\Gamma_1(p^m)}(\epsilon)^\perf)$ is good.
\end{cor}

\begin{proof} This follows from the previous lemma and Lemma \ref{GoodTripleFinite}.
\end{proof}

Now fix $m\geq 1$, and consider $\mathcal{Y}_m^\ast = \mathcal{X}_{\Gamma_1(p^m)}^\ast(\epsilon)_a\to \mathcal{X}_{\Gamma_0(p^m)}^\ast(\epsilon)_a$. We use notation as in Lemma \ref{HardHartog}.

\begin{lem} The tilt of $\mathcal{Y}_\infty$ identifies with $\mathcal{X}^\prime_{\Gamma_1(p^m)}(\epsilon)^\perf$.
\end{lem}

\begin{proof} As $\mathcal{Y}_m\to \mathcal{X}_{\Gamma_0(p^m)}(\epsilon)_a$ is finite \'etale, $\mathcal{Y}_\infty\to \mathcal{X}_{\Gamma_0(p^\infty)}(\epsilon)_a$ is finite \'etale, and parametrizes isomorphisms $D_{m\Gamma_0(p^\infty)}\cong (\Z/p^m\Z)^g$. Using Lemma \ref{TiltingAntiCanSubgroup}, one sees that the tilt will parametrize isomorphisms $(D_m^\prime)^\perf\cong (\Z/p^m\Z)^g$. This moduli problem is given by $\mathcal{X}^\prime_{\Gamma_1(p^m)}(\epsilon)^\perf\to \mathcal{X}^\prime(\epsilon)^\perf$.
\end{proof}

Note that $\mathcal{Y}_m^\ast\setminus \partial\to \mathcal{X}_{\Gamma_0(p^m)}^\ast(\epsilon)_a\setminus \partial$ is finite \'etale, where $\partial$ denotes the boundary of any of the spaces involved. By pullback (and abuse of notation -- $\mathcal{Y}_\infty^\ast$ is not defined yet), we get a perfectoid space $\mathcal{Y}_\infty^\ast\setminus \partial\to \mathcal{X}_{\Gamma_0(p^\infty)}^\ast(\epsilon)_a\setminus \partial$.

\begin{lem}\label{TiltingFiniteAwayFromBoundary} The tilt of $\mathcal{Y}_\infty^\ast\setminus \partial$ identifies with $\mathcal{X}^{\prime\ast}_{\Gamma_1(p^m)}(\epsilon)^\perf\setminus \partial$.
\end{lem}

\begin{proof} Let $\mathcal{X}^{\prime\ast}_{\Gamma_1(p^m)}(\epsilon)^\perf = \Spa(T,T^+)$, and let $(U,U^+)$ be the untilt of $(T,T^+)$. By the previous lemma, we get a map
\[
U^+\to H^0(\mathcal{Y}_\infty,\OO_{\mathcal{Y}_\infty}^+) = S_\infty^+\ ,
\]
and by Lemma \ref{HardHartog}, the right-hand side is the $p$-adic completion of $\varinjlim_m S_m^+$. From the latter statement, it follows that there is a map of adic spaces in the sense of \cite[Definition 2.1.5]{ScholzeWeinstein} $\mathcal{Y}_m^\ast\setminus \partial\to \Spa(S_\infty,S_\infty^+)$. Combining, we get a map $\mathcal{Y}_m^\ast\setminus \partial\to \Spa(U,U^+)$. After restricting to the complement of the boundary, both spaces are perfectoid, and finite \'etale over $\mathcal{X}_{\Gamma_0(p^\infty)}^\ast(\epsilon)_a\setminus \partial$. Thus, using the previous lemma and Lemma \ref{BaseTripleIsGood}, the result follows from the next lemma.
\end{proof}

\begin{lem} Let $K$ be a perfectoid field, $\mathcal{X}$, $\mathcal{Y}_1$, $\mathcal{Y}_2$ be perfectoid spaces over $K$, $\mathcal{Y}_1,\mathcal{Y}_2\to \mathcal{X}$ two finite \'etale maps, and $f:\mathcal{Y}_1\to \mathcal{Y}_2$ a map over $\mathcal{X}$. Let $\mathcal{U}\subset \mathcal{X}$ be an open subset such that $H^0(\mathcal{X},\OO_\mathcal{X})\hookrightarrow H^0(\mathcal{U},\OO_\mathcal{U})$. Assume that $f|_\mathcal{U}$ is an isomorphism. Then $f$ is an isomorphism.
\end{lem}

\begin{proof} The locus of $\mathcal{X}$ above which $f$ is an isomorphism is open and closed: As the maps are finite \'etale, this reduces to the classical algebraic case. Thus, if $f$ is not an isomorphism, there is a nontrivial idempotent $e\in H^0(\mathcal{X},\OO_\mathcal{X})$ which is equal to $1$ on the locus where $f$ is an isomorphism. In particular, $e|_\mathcal{U} = 1$. But as $H^0(\mathcal{X},\OO_\mathcal{X})\hookrightarrow H^0(\mathcal{U},\OO_\mathcal{U})$, $e=1$, as desired.
\end{proof}

\begin{lem}\label{TiltingFinite} The ring $S_\infty = H^0(\mathcal{Y}_\infty,\OO_{\mathcal{Y}_\infty})$ is perfectoid, and the tilt of $\mathcal{Y}_\infty^\ast = \Spa(S_\infty,S_\infty^+)$ identifies with $\mathcal{X}^{\prime\ast}_{\Gamma_1(p^m)}(\epsilon)^\perf$.
\end{lem}

\begin{proof} Recall that in the proof of Lemma \ref{TiltingFiniteAwayFromBoundary}, we constructed a map $U^+\to S_\infty^+$; we need to show that it is an isomorphism. From $S_\infty^+=H^0(\mathcal{Y}_\infty,\OO_{\mathcal{Y}_\infty}^+)$, we know that
\[
S_\infty^+/p\hookrightarrow H^0(\mathcal{Y}_\infty,\OO_{\mathcal{Y}_\infty}^+/p)\ .
\]
Using Corollary \ref{FiniteTripleIsGood},
\[
(U^+/p)^a = H^0(\mathcal{X}^{\prime\ast}_{\Gamma_1(p^m)}(\epsilon)^\perf,\OO^+/t)^a\hookrightarrow H^0(\mathcal{X}^\prime_{\Gamma_1(p^m)}(\epsilon)^\perf,\OO^+/t)^a = H^0(\mathcal{Y}_\infty,\OO^+/p)^a\ ,
\]
so $U^+/p\to S_\infty^+/p$ is almost injective, and the map $U^+\to S_\infty^+$ is injective. To prove surjectivity, observe that there is a map
\[\begin{aligned}
(S_\infty^+/p)^a\hookrightarrow H^0(\mathcal{Y}_\infty^\ast\setminus \partial,\OO^+/p)^a&\cong H^0(\mathcal{X}^{\prime\ast}_{\Gamma_1(p^m)}(\epsilon)^\perf\setminus \partial,\OO^+/t)^a\\
&\cong H^0(\mathcal{X}^{\prime\ast}_{\Gamma_1(p^m)}(\epsilon)^\perf,\OO^+/t)^a = (U^+/p)^a
\end{aligned}\]
by Lemma \ref{TiltingFiniteAwayFromBoundary} and Corollary \ref{FiniteTripleIsGood}. This gives almost surjectivity, thus $S_\infty = U$, and then also $S_\infty^+ = S_\infty^\circ = U^\circ = U^+$.
\end{proof}

Summarizing the discussion, we have proved the following.

\begin{prop} For any $m\geq 1$, there exists a unique perfectoid space $\mathcal{X}_{\Gamma_1(p^m)\cap \Gamma_0(p^\infty)}^\ast(\epsilon)_a$ over $\Q_p^\cycl$ such that
\[
\mathcal{X}_{\Gamma_1(p^m)\cap \Gamma_0(p^\infty)}^\ast(\epsilon)_a\sim \varprojlim_{m^\prime} \mathcal{X}_{\Gamma_1(p^m)\cap \Gamma_0(p^{m^\prime})}^\ast(\epsilon)_a\ .
\]
Moreover, $\mathcal{X}_{\Gamma_1(p^m)\cap \Gamma_0(p^\infty)}^\ast(\epsilon)_a$ and all $\mathcal{X}_{\Gamma_1(p^m)\cap \Gamma_0(p^{m^\prime})}^\ast(\epsilon)_a$ for $m^\prime$ sufficiently large are affinoid, and
\[
\varinjlim_{m^\prime} H^0(\mathcal{X}_{\Gamma_1(p^m)\cap \Gamma_0(p^{m^\prime})}^\ast(\epsilon)_a,\OO)\to H^0(\mathcal{X}_{\Gamma_1(p^m)\cap \Gamma_0(p^\infty)}^\ast(\epsilon)_a,\OO)
\]
has dense image.

Let $\mathcal{Z}_{\Gamma_1(p^m)\cap \Gamma_0(p^\infty)}(\epsilon)_a\subset \mathcal{X}_{\Gamma_1(p^m)\cap \Gamma_0(p^\infty)}^\ast(\epsilon)_a$ denote the boundary, and $\mathcal{X}_{\Gamma_1(p^m)\cap \Gamma_0(p^\infty)}(\epsilon)_a$ the preimage of $\mathcal{X}_{\Gamma_0(p)}(\epsilon)_a\subset \mathcal{X}_{\Gamma_0(p)}^\ast(\epsilon)_a$. Then the triple
\[
(\mathcal{X}_{\Gamma_1(p^m)\cap \Gamma_0(p^\infty)}^\ast(\epsilon)_a,\mathcal{Z}_{\Gamma_1(p^m)\cap \Gamma_0(p^\infty)}(\epsilon)_a,\mathcal{X}_{\Gamma_1(p^m)\cap \Gamma_0(p^\infty)}(\epsilon)_a)
\]
is good.$\hfill \Box$
\end{prop}

In fact, the proposition is also true for $g=1$. In that case, $\mathcal{X}_{\Gamma_1(p^m)}^\ast(\epsilon)_a\to \mathcal{X}_{\Gamma_0(p^m)}^\ast(\epsilon)_a$ is finite \'etale, and the boundary is contained in the ordinary locus. Note that Lemma \ref{GoodTriple} holds true if the codimension of the boundary is $1$ when the boundary $V(J)$ does not meet $V(f)$. Also, Lemma \ref{GoodTripleFinite} holds true if the codimension of the boundary is $1$ when the map is finite \'etale. Certainly, one can pull back the finite \'etale map $\mathcal{X}_{\Gamma_1(p^m)}^\ast(\epsilon)_a\to \mathcal{X}_{\Gamma_0(p^m)}^\ast(\epsilon)_a$ to get $\mathcal{X}_{\Gamma_1(p^m)\cap \Gamma_0(p^\infty)}^\ast(\epsilon)_a\to \mathcal{X}_{\Gamma_0(p^\infty)}^\ast(\epsilon)_a$, and arrives at all desired properties.

Passing to the inverse limit over $m$ and using Lemma \ref{GoodTripleLimit}, we get the following proposition.

\begin{prop} There is a unique perfectoid space $\mathcal{X}_{\Gamma_1(p^\infty)}^\ast(\epsilon)_a$ over $\Q_p^\cycl$ such that
\[
\mathcal{X}_{\Gamma_1(p^\infty)}^\ast(\epsilon)_a\sim \varprojlim_m \mathcal{X}_{\Gamma_1(p^m)}^\ast(\epsilon)_a\ .
\]
Moreover, $\mathcal{X}_{\Gamma_1(p^\infty)}^\ast(\epsilon)_a$ and all $\mathcal{X}_{\Gamma_1(p^m)}^\ast(\epsilon)_a$ for $m$ sufficiently large are affinoid, and
\[
\varinjlim_m H^0(\mathcal{X}_{\Gamma_1(p^m)}^\ast(\epsilon)_a,\OO)\to H^0(\mathcal{X}_{\Gamma_1(p^\infty)}^\ast(\epsilon)_a,\OO)
\]
has dense image.

Let $\mathcal{Z}_{\Gamma_1(p^\infty)}(\epsilon)_a\subset \mathcal{X}_{\Gamma_1(p^\infty)}^\ast(\epsilon)_a$ denote the boundary, and $\mathcal{X}_{\Gamma_1(p^\infty)}(\epsilon)_a$ the preimage of $\mathcal{X}_{\Gamma_0(p)}(\epsilon)_a\subset \mathcal{X}_{\Gamma_0(p)}^\ast(\epsilon)_a$. Then the triple
\[
(\mathcal{X}_{\Gamma_1(p^\infty)}^\ast(\epsilon)_a,\mathcal{Z}_{\Gamma_1(p^\infty)}(\epsilon)_a,\mathcal{X}_{\Gamma_1(p^\infty)}(\epsilon)_a)
\]
is good.$\hfill \Box$
\end{prop}

The case of $\Gamma(p^m)$-level structures is now easy, using \cite[Theorem 7.9 (iii)]{ScholzePerfectoid}, and the following lemma.

\begin{lem}\label{GammaVsGamma1FiniteEtale} For any $m\geq 1$, the map
\[
\mathcal{X}_{\Gamma(p^m)}^\ast(\epsilon)_a\to \mathcal{X}_{\Gamma_1(p^m)}^\ast(\epsilon)_a
\]
is finite \'etale.
\end{lem}

\begin{proof} We leave the case $g=1$ to the reader. Thus, assume $g\geq 2$. First, we check the assertion in the case $\epsilon=0$, i.e., on the ordinary locus. In that case, we claim that it decomposes as
\[
\mathcal{X}_{\Gamma(p^m)}^\ast(0)_a\cong \bigsqcup_{\Gamma_1(p^m)/\Gamma(p^m)} \mathcal{X}_{\Gamma_1(p^m)}^\ast(0)_a\to \mathcal{X}_{\Gamma_1(p^m)}^\ast(0)_a\ .
\]
By Hartog's principle (cf. Lemma \ref{HardHartog}), it suffices to check that
\[
\mathcal{X}_{\Gamma(p^m)}(0)_a\cong \bigsqcup_{\Gamma_1(p^m)/\Gamma(p^m)} \mathcal{X}_{\Gamma_1(p^m)}(0)_a\to \mathcal{X}_{\Gamma_1(p^m)}(0)_a\ .
\]
The left-hand side parametrizes abelian varieties $A$ with good ordinary reduction, with a symplectic isomorphism $\alpha: A[p^m]\cong (\Z/p^m\Z)^{2g}$ such that $D_1 = (\alpha\mod p)^{-1}(\F_p^g\oplus 0^g)\subset A[p]$ satisfies $D_1\cap C_1 = \{0\}$, where $C_1\subset A[p]$ is the canonical subgroup of level $1$. Similarly, $\mathcal{X}_{\Gamma_1(p^m)}(0)$ parametrizes abelian varieties $A$ with good ordinary reduction, together with a totally isotropic subgroup $D_m\subset A[p^m]$ and an isomorphism $\alpha_0: D_m\cong (\Z/p^m\Z)^g$, such that $D_1 = D_m[p]$ satisfies $D_1\cap C_1 = \{0\}$.

Note that $A$ has good ordinary reduction, and thus a canonical subgroup $C_m\subset A[p^m]$. Moreover, $C_m\oplus D_m = A[p^m]$, as follows from $C_1\oplus D_1 = A[p]$. The map of functors is given by $\alpha\mapsto (D_m,\alpha_0)$, where $D_m = \alpha^{-1}((\Z/p^m\Z)^g\oplus 0^g)$, and $\alpha_0 = \alpha|_{D_m}$. But $\alpha$ also gives rise to a totally isotropic subspace $\Sigma = \alpha(C_m)\subset (\Z/p^m\Z)^{2g}$, with
\[
\Sigma\oplus ((\Z/p^m\Z)^g\oplus 0^g) = (\Z/p^m\Z)^{2g}\ .
\]
One checks that $\Gamma_1(p^m)/\Gamma(p^m)$ acts simply transitively on the set of such $\Sigma$, and that the datum of $(D_m,\alpha_0,\Sigma)$ is equivalent to $\alpha$. This finishes the proof in case $\epsilon=0$.

In general, there is a description of the boundary strata, and the induced map, of $\mathcal{X}_{\Gamma(p^m)}^\ast\to \mathcal{X}_{\Gamma_1(p^m)}^\ast$ in terms of lower-dimensional Siegel moduli spaces. In particular, above any locally closed stratum meeting $\mathcal{X}_{\Gamma_1(p^m)}^\ast(0)_a$, the map is finite \'etale, as it is so generically. As any locally closed stratum that meets $\mathcal{X}_{\Gamma_1(p^m)}^\ast(\epsilon)_a$ will also meet $\mathcal{X}_{\Gamma_1(p^m)}^\ast(0)_a$, we get the conclusion.
\end{proof}

Using Lemma \ref{GoodTripleFinite} and Lemma \ref{GoodTripleLimit} once more, we get the following theorem.

\begin{thm}\label{ExOnStrictNbhd} There is a unique perfectoid space $\mathcal{X}_{\Gamma(p^\infty)}^\ast(\epsilon)_a$ over $\Q_p^\cycl$ such that
\[
\mathcal{X}_{\Gamma(p^\infty)}^\ast(\epsilon)_a\sim \varprojlim_m \mathcal{X}_{\Gamma(p^m)}^\ast(\epsilon)_a\ .
\]
Moreover, $\mathcal{X}_{\Gamma(p^\infty)}^\ast(\epsilon)_a$ and all $\mathcal{X}_{\Gamma(p^m)}^\ast(\epsilon)_a$ for $m$ sufficiently large are affinoid, and
\[
\varinjlim_m H^0(\mathcal{X}_{\Gamma(p^m)}^\ast(\epsilon)_a,\OO)\to H^0(\mathcal{X}_{\Gamma(p^\infty)}^\ast(\epsilon)_a,\OO)
\]
has dense image.

Let $\mathcal{Z}_{\Gamma(p^\infty)}(\epsilon)_a\subset \mathcal{X}_{\Gamma(p^\infty)}^\ast(\epsilon)_a$ denote the boundary, and $\mathcal{X}_{\Gamma(p^\infty)}(\epsilon)_a$ the preimage of $\mathcal{X}_{\Gamma_0(p)}(\epsilon)_a\subset \mathcal{X}_{\Gamma_0(p)}^\ast(\epsilon)_a$. Then the triple
\[
(\mathcal{X}_{\Gamma(p^\infty)}^\ast(\epsilon)_a,\mathcal{Z}_{\Gamma(p^\infty)}(\epsilon)_a,\mathcal{X}_{\Gamma(p^\infty)}(\epsilon)_a)
\]
is good.$\hfill \Box$
\end{thm}

Summarizing our efforts so far, we have proved that a strict (and explicit) neighborhood of the ordinary locus in the minimal compactification becomes affinoid perfectoid in the inverse limit, and that Riemann's Hebbarkeitssatz holds true with respect to the boundary. We will now extend these results to the whole Shimura variety by using the $\GSp_{2g}(\Q_p)$-action.

\section{The Hodge-Tate period map}\label{ConclusionSection}

The next task is to extend the result of the previous section to all of $\mathcal{X}_{\Gamma(p^\infty)}^\ast$, and to construct the Hodge-Tate period map
\[
\pi_\HT: \mathcal{X}_{\Gamma(p^\infty)}^\ast\to \Fl\ .
\]
In fact, the two tasks will go hand in hand. As we are always working over $\Q_p^\cycl$, we can ignore all Tate twists in the following.

\subsection{On topological spaces} We need a version of \cite[Proposition 4.15]{ScholzeSurvey} for the case of bad reduction.

\begin{prop}\label{CompAbPDivHT} Let $C$ be an algebraically closed and complete extension of $\Q_p$ with ring of integers $\OO_C$. Let $A/C$ be an abelian variety with connected N\'eron model $G/\OO_C$. Let $\hat{G}$ be the $p$-adic completion of $G$ (as a formal scheme over $\Spf \OO_C$); then there is an extension
\[
0\to \hat{T}\to \hat{G}\to \hat{B}\to 0\ ,
\]
where $T$ is a split torus over $\OO_C$, and $B$ is an abelian variety over $\OO_C$. Thus, $\hat{G}[p^\infty]$ defines a $p$-divisible group over $\OO_C$, which has a Hodge-Tate filtration
\[
0\to \Lie \hat{G}\otimes_{\OO_C} C(1)\to T_p \hat{G}\otimes_{\Z_p} C\to (\Lie \hat{G}[p^\infty]^\ast)^\ast\otimes_{\OO_C} C\to 0\ .
\]
Also, $A$ has its Hodge-Tate filtration
\[
0\to \Lie A(1)\to T_p A\otimes_{\Z_p} C\to (\Lie A^\ast)^\ast\to 0\ .
\]
The diagram
\[\xymatrix{
\Lie \hat{G}\otimes_{\OO_C} C(1)\ar[d]^{=} \ar[r]& T_p \hat{G}\otimes_{\Z_p} C\ar@{^(->}[d]\\
\Lie A(1)\ar[r]& T_p A\otimes_{\Z_p} C
}\]
commutes.
\end{prop}

\begin{proof} The proof is identical to the proof of \cite[Proposition 4.15]{ScholzeSurvey}.
\end{proof}

In the situation of the proposition, we need a comparison of Hasse invariants.

\begin{lem}\label{CompHasseInv} In the situation of Proposition \ref{CompAbPDivHT}, assume that $A$ comes from a point $x\in X(C) = X_{g,K^p}(C)$. By properness of $X_{g,K^p}^\ast$, it extends to a point $x\in X_{g,K^p}^\ast(\OO_C)$.
\begin{altenumerate}
\item[{\rm (i)}] The pullback $x^\ast \omega_{X_{g,K^p}^\ast}$ is canonically isomorphic to $\omega_G$.
\item[{\rm (ii)}] Let $\bar{x}\in X_{g,K^p}^\ast(\OO_C/p)$ be the reduction modulo $p$ of $x$. Then there is an equality
\[
\Ha(\bar{x}) = \Ha(B\otimes_{\OO_C} \OO_C/p)\otimes (\omega_T^\can)^{p-1}\in \omega_G^{\otimes (p-1)}/p\cong \omega_B^{\otimes (p-1)}/p\otimes \omega_T^{\otimes (p-1)}/p\ .
\]
Here, $\pm \omega_T^\can\in \omega_T$ denotes the canonical differential, given by $d\log(Z_1)\wedge \ldots\wedge d\log(Z_m)$ on the split torus $T\cong \Spec \OO_C[Z_1^{\pm 1},\ldots,Z_m^{\pm 1}]$.\footnote{The sign ambiguity goes away when taking the $p-1$-th power if $p\neq 2$; if $p=2$, then it goes away modulo $p$.}
\end{altenumerate}
\end{lem}

\begin{proof} Let $f: \overline{X}_{g,K^p}\to X_{g,K^p}^\ast$ be a (smooth projective) toroidal compactification, as constructed in \cite{FaltingsChai}. Over $\overline{X}_{g,K^p}$, one has a family of semiabelian varieties $G^\univ\to \overline{X}_{g,K^p}$. In particular, one has the invertible sheaf $\omega_{G^\univ}$ over $\overline{X}_{g,K^p}$, and by construction of $X_{g,K^p}^\ast$, $\omega_{G^\univ} = f^\ast \omega_{X_{g,K^p}^\ast}$. Pulling back to $x$ gives part (i).

For part (ii), observe that one can define an element $\Ha^\prime\in H^0(\overline{X}_{g,K^p}\otimes_{\Z_{(p)}} \F_p,\omega_{G^\univ}^{\otimes (p-1)})$ as follows. The Verschiebung map $V: (G^\univ)^{(p)}\to G^\univ$ in characteristic $p$ induces a map $\omega_{G^\univ}\to \omega_{(G^\univ)^{(p)}}\cong \omega_{G^\univ}^{\otimes p}$, i.e. a section $\Ha^\prime\in H^0(\overline{X}_{g,K^p}\otimes_{\Z_{(p)}} \F_p,\omega_{G^\univ}^{\otimes (p-1)})$, as desired. Clearly, $\Ha = \Ha^\prime$ on $X_{g,K^p}\otimes_{\Z_{(p)}} \F_p$; it follows that $\Ha^\prime$ is the pullback of $\Ha$ to $\overline{X}_{g,K^p}\otimes_{\Z_{(p)}} \F_p$.

Pulling back to $x$ reduces part (ii) to a direct verification.
\end{proof}

Look at the spectral topological spaces
\[
|\mathcal{X}_{\Gamma(p^\infty)}^\ast| = \varprojlim_m |\mathcal{X}_{\Gamma(p^m)}^\ast|\ ,\ |\mathcal{Z}_{\Gamma(p^\infty)}| = \varprojlim_m |\mathcal{Z}_{\Gamma(p^m)}|\ ,\ |\mathcal{X}_{\Gamma(p^\infty)}| = \varprojlim_m |\mathcal{X}_{\Gamma(p^m)}|\ .
\]
There is a continuous action of $\GSp_{2g}(\Q_p)$ on these spaces.

\begin{rem} For any complete nonarchimedean field extension $K$ of $\Q_p^\cycl$ with an open and bounded valuation subring $K^+\subset K$, we define
\[
\mathcal{X}_{\Gamma(p^\infty)}^\ast(K,K^+) = \varprojlim_m \mathcal{X}_{\Gamma(p^m)}^\ast(K,K^+)\ ,
\]
and similarly for the other spaces. For any $(K,K^+)$, one gets a map
\[
\mathcal{X}_{\Gamma(p^\infty)}^\ast(K,K^+)\to |\mathcal{X}_{\Gamma(p^\infty)}^\ast|\ .
\]
One checks easily that one has a bijection
\[
|\mathcal{X}_{\Gamma(p^\infty)}^\ast| = \varinjlim_{(K,K^+)} \mathcal{X}_{\Gamma(p^\infty)}^\ast(K,K^+)\ .
\]
Note that the direct limit on the right-hand side is not filtered; however, any point comes from a unique minimal $(K,K^+)$.
\end{rem}

\begin{lem}\label{PiHTTop} There is a $\GSp_{2g}(\Q_p)$-equivariant continuous map
\[
|\pi_\HT|: |\mathcal{X}_{\Gamma(p^\infty)}^\ast|\setminus |\mathcal{Z}_{\Gamma(p^\infty)}|\to |\Fl|\ ,
\]
sending a point $x\in (\mathcal{X}_{\Gamma(p^\infty)}^\ast\setminus \mathcal{Z}_{\Gamma(p^\infty)})(K,K^+)$, corresponding to a principally polarized abelian variety $A/K$ and a symplectic isomorphism $\alpha: T_p A\cong \Z_p^{2g}$, to the Hodge-Tate filtration $\Lie A\subset K^{2g}$.
\end{lem}

\begin{proof} One can check from Proposition \ref{CompAbPDivHT} that the Hodge-Tate filtration, a priori defined over $C=\hat{\bar{K}}$, is already $K$-rational, as this is true by definition for the Hodge-Tate filtration of $p$-divisible groups, cf. \cite[Ch. 2, App. C]{FarguesTwoTowers}. Thus, one gets a map
\[
|\mathcal{X}_{\Gamma(p^\infty)}^\ast|\setminus |\mathcal{Z}_{\Gamma(p^\infty)}| = \varinjlim_{(K,K^+)} (\mathcal{X}_{\Gamma(p^\infty)}^\ast\setminus \mathcal{Z}_{\Gamma(p^\infty)})(K,K^+)\to \varinjlim_{(K,K^+)} \Fl(K,K^+) = |\Fl|\ .
\]
The $\GSp_{2g}(\Q_p)$-equivariance is clear.

For continuity, we argue as follows. Consider the smooth adic space $S = \mathcal{X}^\ast\setminus \mathcal{Z}$, with the universal abelian variety $g: A_S\to S$. Then $g$ is a proper smooth morphism of smooth adic spaces. Applying \cite[Theorem 1.3]{ScholzePAdicHodge}, we see that the map
\[
(R^1 g_\ast \Z/p^n\Z)\otimes_{\Z/p^n\Z} \OO_S^+/p^n\to R^1 g_\ast \OO_{A_S}^+/p^n
\]
is an almost isomorphism for all $n\geq 1$; by the $5$-lemma, this reduces to the case $n=1$. Going to the pro-\'etale site, passing to the inverse limit over $n$ and inverting $p$, we find an isomorphism of sheaves on $S_\proet$,
\[
R^1 g_\ast \widehat{\Q}_p\otimes_{\widehat{\Q}_p} \hat{\OO}_S\cong R^1 g_\ast \hat{\OO}_{A_S}\ .
\]
In particular, we get a map
\[
(R^1 g_\ast \OO_{A_S})\otimes_{\OO_S} \hat{\OO}_S\to R^1 g_\ast \hat{\OO}_{A_S} = R^1 g_\ast \widehat{\Q}_p\otimes_{\widehat{\Q}_p} \hat{\OO}_S\ .
\]
Note that $R^1 g_\ast \OO_{A_S}$ is a finite locally free $\OO_S$-module given by $\Lie A_S$ (using the principal polarization on $A_S$ to identify $A_S$ with its dual). Locally, there is a pro-finite \'etale cover $\tilde{U}\to U\subset S$ such that $\tilde{U}$ is affinoid perfectoid. Let
\[
\tilde{U}_\infty = \tilde{U}\times_S (\mathcal{X}_{\Gamma(p^\infty)}^\ast\setminus \mathcal{Z}_{\Gamma(p^\infty)})\ ;
\]
as $\mathcal{X}_{\Gamma(p^\infty)}^\ast\setminus \mathcal{Z}_{\Gamma(p^\infty)}\to S$ is pro-finite \'etale, $\tilde{U}_\infty$ exists, and is affinoid perfectoid over $\Q_p^\cycl$.  Evaluating the map
\[
(R^1 g_\ast \OO_{A_S})\otimes_{\OO_S} \hat{\OO}_S\to R^1 g_\ast \hat{\OO}_{A_S} = R^1 g_\ast \widehat{\Q}_p\otimes_{\widehat{\Q}_p} \hat{\OO}_S
\]
on $\tilde{U}_\infty\in S_\proet$, we get a map
\[
(\Lie A_S)\otimes_{\OO_S} \OO_{\tilde{U}_\infty}\to \OO_{\tilde{U}_\infty}^{2g}\ ,
\]
using the tautological trivialization of $R^1 g_\ast \widehat{\Z}_p$ over $\tilde{U}_\infty$. At all geometric points of $\tilde{U}_\infty$, this identifies with the Hodge-Tate filtration as defined in the statement of the lemma, using \cite[Proposition 4.10]{ScholzeSurvey}. In particular, $(\Lie A_S)\otimes_{\OO_S} \OO_{\tilde{U}_\infty}\subset \OO_{\tilde{U}_\infty}^{2g}$ is totally isotropic, and defines a map of adic spaces
\[
\tilde{U}_\infty\to \Fl\ .
\]
By checking on points, we see that the continuous map $|\tilde{U}_\infty|\to |\Fl|$ factors over
\[
|U|\times_{|S|} (|\mathcal{X}_{\Gamma(p^\infty)}^\ast|\setminus |\mathcal{Z}_{\Gamma(p^\infty)}|)\ ,
\]
and agrees with the map of sets defined previously. As the map
\[
|\tilde{U}_\infty|\to |U|\times_{|S|} (|\mathcal{X}_{\Gamma(p^\infty)}^\ast|\setminus |\mathcal{Z}_{\Gamma(p^\infty)}|)
\]
is the realization on topological spaces of a pro-\'etale and surjective map in $S_\proet$, and pro-\'etale maps in $S_\proet$ are open, a subset $V\subset |U|\times_{|S|} (|\mathcal{X}_{\Gamma(p^\infty)}^\ast|\setminus |\mathcal{Z}_{\Gamma(p^\infty)}|)$ is open if and only if its preimage in $|\tilde{U}_\infty|$ is open. The result follows.
\end{proof}

\begin{definition}\begin{altenumerate}
\item[{\rm (i)}] A subset $U\subset |\mathcal{X}_{\Gamma(p^\infty)}^\ast|$ is affinoid perfectoid if it is the preimage of some affinoid $U_m=\Spa(R_m,R_m^+)\subset |\mathcal{X}_{\Gamma(p^m)}^\ast|$ for all sufficiently large $m$, and $(R_\infty,R_\infty^+)$ is an affinoid perfectoid $\Q_p^\cycl$-algebra, where $R_\infty^+$ is the $p$-adic completion of $\varinjlim_m R_m^+$, and $R_\infty = R_\infty^+[p^{-1}]$.

\item[{\rm (ii)}] A subset $U\subset |\mathcal{X}_{\Gamma(p^\infty)}^\ast|$ is perfectoid if it can be covered by affinoid perfectoid subsets.
\end{altenumerate}
\end{definition}

By Theorem \ref{ExOnStrictNbhd}, $\mathcal{X}_{\Gamma(p^\infty)}^\ast(\epsilon)_a$ is affinoid perfectoid. Also, the condition of being affinoid perfectoid is stable under the action of $\GSp_{2g}(\Q_p)$. Moreover, any perfectoid subset of $|\mathcal{X}_{\Gamma(p^\infty)}^\ast|$ has a natural structure as a perfectoid space over $\Q_p^\cycl$, by gluing the spaces $\Spa(R_\infty,R_\infty^+)$ on affinoid perfectoid subsets. Our goal is to show that $|\mathcal{X}_{\Gamma(p^\infty)}^\ast|$ is perfectoid.

For $\epsilon<1$, recall that $\mathcal{X}^\ast(\epsilon)\subset \mathcal{X}^\ast$ denotes the locus where $|\Ha|\geq |p|^\epsilon$ (observing that this is independent of the lift of $\Ha$). Let $|\mathcal{X}_{\Gamma(p^\infty)}^\ast(\epsilon)|\subset |\mathcal{X}_{\Gamma(p^\infty)}^\ast|$ denote the preimage. Similar notation applies for $\mathcal{Z}$ and $\mathcal{X}\subset \mathcal{X}^\ast$.

Note that for $\epsilon=0$, one gets the tubular neighborhood of the ordinary locus in the special fibre.

\begin{lem}\label{PreimageRationalPoints} The preimage of $\Fl(\Q_p)\subset |\Fl|$ under $|\pi_\HT|$ is given by the closure of $|\mathcal{X}_{\Gamma(p^\infty)}^\ast(0)|\setminus |\mathcal{Z}_{\Gamma(p^\infty)}(0)|$.
\end{lem}

Note that $|\mathcal{X}_{\Gamma(p^\infty)}^\ast(0)|\setminus |\mathcal{Z}_{\Gamma(p^\infty)}(0)|$ is a retro-compact open subset of the locally spectral space $|\mathcal{X}_{\Gamma(p^\infty)}^\ast|\setminus |\mathcal{Z}_{\Gamma(p^\infty)}|$ (i.e., the intersection with any quasi-compact open is quasi-compact). In this case, the closure is exactly the set of specializations.

\begin{proof} Let $C$ be an algebraically closed complete extension of $\Q_p$, with an open and bounded valuation subring $C^+\subset C$, and take a $(C,C^+)$-valued point $x$ of $\mathcal{X}_{\Gamma(p^\infty)}^\ast\setminus \mathcal{Z}_{\Gamma(p^\infty)}$. It admits the unique rank-$1$-generalization $\tilde{x}$ given as the corresponding $(C,\OO_C)$-valued point, and $x$ lies in the closure of $\mathcal{X}_{\Gamma(p^\infty)}^\ast(0)\setminus \mathcal{Z}_{\Gamma(p^\infty)}(0)$ if and only if $\tilde{x}$ lies in $\mathcal{X}_{\Gamma(p^\infty)}^\ast(0)\setminus \mathcal{Z}_{\Gamma(p^\infty)}(0)$ itself.  Also, by continuity, $x$ maps into $\Fl(\Q_p)$ if and only if $\tilde{x}$ maps into $\Fl(\Q_p)$. Thus, we may assume that $x=\tilde{x}$ is a rank-$1$-point, with values in $(C,\OO_C)$.

The point $x$ corresponds to a principally polarized abelian variety $A/C$ with trivialization of its Tate module. Let $G/\OO_C$ be the N\'eron model, and use notation as in Proposition \ref{CompAbPDivHT}. By Lemma \ref{CompHasseInv}, the point $x$ lies in $\mathcal{X}_{\Gamma(p^\infty)}^\ast(0)\setminus \mathcal{Z}_{\Gamma(p^\infty)}(0)$, i.e. the Hasse invariant is invertible, if and only if $B$ is ordinary. By Proposition \ref{CompAbPDivHT}, $x$ maps into $\Fl(\Q_p)$ if and only if
\[
\Lie \hat{G}\otimes_{\OO_C} C\subset T_p \hat{G}\otimes_{\Z_p} C
\]
is a $\Q_p$-rational subspace. This, in turn, is equivalent to
\[
\Lie B\otimes_{\OO_C} C\subset T_p B\otimes_{\Z_p} C
\]
being a $\Q_p$-rational subspace. Also, $B$ is ordinary if and only if $B[p^\infty]\cong (\Q_p/\Z_p)^g\times \mu_{p^\infty}^g$. One checks directly that in this case, the Hodge-Tate filtration is $\Q_p$-rational (and measures the position of the canonical subgroup). Conversely, all $\Q_p$-rational totally isotropic subspaces $W\subset C^{2g}$ are in one $\GSp_{2g}(\Z_p)$-orbit. By the classification result for $p$-divisible groups over $\OO_C$, \cite[Theorem B]{ScholzeWeinstein}, it follows that if the Hodge-Tate filtration is $\Q_p$-rational, then $B[p^\infty]\cong (\Q_p/\Z_p)^g\times \mu_{p^\infty}^g$. This finishes the proof.
\end{proof}

\begin{rem} Here is a more direct argument for the final step, not refering to \cite{ScholzeWeinstein}, which was suggested by the referee. It is enough to prove the following assertion. Let $C$ be a complete algebraically closed extension of $\Q_p$, and let $G$ over $\OO_C$ be a $p$-divisible group. Then the kernel of
\[
\alpha_G: T_p G\to \Lie G^\ast
\]
is given by $T_p(G^{\mathrm{mult}})$, where $G^{\mathrm{mult}}\subset G$ denotes the maximal multiplicative subgroup. Indeed, if the Hodge-Tate filtration is $\Q_p$-rational, this kernel is $g$-dimensional, so the multiplicative part is of dimension $g$, which is equivalent to the abelian variety $B$ being ordinary.

To prove this, one may split off the multiplicative part, so as to assume that $G^{\mathrm{mult}} = 0$. In this case, $G^\ast$ is a formal group. Then, for an element $x\in T_p G$ corresponding to a morphism of $p$-divisible group $\Q_p/\Z_p\to G$, one takes the dual map $x^\ast: G^\ast\to \mu_{p^\infty}$. Then $\alpha_G(x)$ is defined as the induced map on Lie algebras. As $G^\ast$ is formal, it follows that if the induced map on Lie algebras is $0$, then $x^\ast=0$, so that $x=0$, proving the desired injectivity.
\end{rem}

The following lemma compares the condition that an abelian variety is close to being ordinary, with the condition that the associated Hodge-Tate periods are close to $\Q_p$-rational (cf. also Lemma \ref{ExistsSmallEpsilon}). This is one of the technical key results of this paper, and is ultimately the reason that it was enough to understand some strict neighborhood of the anticanonical tower.

\begin{lem}\label{SmallUExists} Fix some $0<\epsilon<1$. There is an open subset $U\subset \Fl$ containing $\Fl(\Q_p)$ such that
\[
|\pi_\HT|^{-1}(U)\subset |\mathcal{X}_{\Gamma(p^\infty)}^\ast(\epsilon)|\setminus |\mathcal{Z}_{\Gamma(p^\infty)}(\epsilon)|\ .
\]
\end{lem}

\begin{proof} We argue by induction on $g$. For $g=0$, there is nothing to show. We have to show that we can find some $U$ such that for any algebraically closed and complete extension $C$ of $\Q_p$ with a principally polarized $g$-dimensional abelian variety $A/C$ and a symplectic isomorphism $\alpha: T_p A\cong \Z_p^{2g}$ for which $|\pi_\HT|(A)\in U$, one has $|\Ha|\geq |p|^\epsilon$.

If $A$ has bad reduction, then using Proposition \ref{CompAbPDivHT} and Lemma \ref{CompHasseInv}, the result reduces by induction to the case already handled. Thus, assume that one has an abelian variety $A/\OO_C$. In particular, we have a point $x\in |\mathcal{X}_{\Gamma(p^\infty)}|$. The map
\[
|\pi_\HT|: |\mathcal{X}_{\Gamma(p^\infty)}|\to |\Fl|
\]
is continuous. One has
\[
\bigcap_{U\supset \Fl(\Q_p)} |\pi_\HT|^{-1}(U) = |\pi_\HT|^{-1}(\Fl(\Q_p)) = \overline{|\mathcal{X}_{\Gamma(p^\infty)}(0)|}\subset |\mathcal{X}_{\Gamma(p^\infty)}(\epsilon)|\ .
\]
The complement $|\mathcal{X}_{\Gamma(p^\infty)}|\setminus |\mathcal{X}_{\Gamma(p^\infty)}(\epsilon)|$ is quasicompact for the constructible topology. Thus, there is some $U\supset \Fl(\Q_p)$ with
\[
|\pi_\HT|^{-1}(U)\subset |\mathcal{X}_{\Gamma(p^\infty)}(\epsilon)|\ ,
\]
as desired.
\end{proof}

Before we continue, let us recall some facts about the geometry of $\Fl$. There is the Pl\"ucker embedding $\Fl\hookrightarrow \mathbb{P}^{\binom{2g}{g}-1}$. For any subset $J\subset \{1,\ldots,2g\}$ of cardinality $g$, let $s_J$ denote the corresponding homogeneous coordinate on projective space, and let $\Fl_J\subset \Fl$ denote the open affinoid subset where $|s_{J^\prime}|\leq |s_J|$ for all $J^\prime$. The action of $\GSp_{2g}(\Z_p)$ permutes the $\Fl_J$ transitively. As an example that will be important later, $\Fl_{\{g+1,\ldots,2g\}}(\Q_p)\subset \Fl(\Q_p) = \mathrm{Fl}(\Z_p)$ parametrizes those totally isotropic direct summands $M\subset \Z_p^{2g}$ with $(M/p)\cap (\F_p^g\oplus 0^g) = \{0\}$; equivalently, $M\oplus (\Z_p^g\oplus 0^g)\buildrel\cong\over\to \Z_p^{2g}$.

\begin{lem}\label{UTranslatesCover} For any open subset $U\subset \Fl$ containing a $\Q_p$-rational point, $\GSp_{2g}(\Q_p)\cdot U = \Fl$.
\end{lem}

\begin{proof} We may assume that $U = \GSp_{2g}(\Q_p)\cdot U$. By assumption, $\Fl(\Q_p)\subset U$. It suffices to see that $\Fl_{\{1,\ldots,g\}}\subset U$. The point $x\in \Fl(\Q_p)$ defined by $\Q_p^g\oplus 0^g\subset \Q_p^{2g}$ lies in $U$. The action of the diagonal element $\gamma = (1,\ldots,1,p,\ldots,p)\in \GSp_{2g}(\Q_p)$ has the property that $\gamma^n(y)\to x$ for $n\to\infty$ for all $y\in \Fl_{\{1,\ldots,g\}}$. By quasicompacity, there is some $n$ such that $\gamma^n(\Fl_{\{1,\ldots,g\}})\subset U$, i.e. $\Fl_{\{1,\ldots,g\}}\subset \gamma^{-n}(U) = U$, as desired.
\end{proof}

\begin{lem}\label{FiniteCoverAwayFromBoundary} Take any $0<\epsilon<1$. There are finitely many $\gamma_1,\ldots,\gamma_k\in \GSp_{2g}(\Q_p)$ such that
\[
|\mathcal{X}_{\Gamma(p^\infty)}^\ast|\setminus |\mathcal{Z}_{\Gamma(p^\infty)}| = \bigcup_{i=1}^k \gamma_i \cdot \left(|\mathcal{X}_{\Gamma(p^\infty)}^\ast(\epsilon)|\setminus |\mathcal{Z}_{\Gamma(p^\infty)}(\epsilon)|\right)\ .
\]
\end{lem}

\begin{proof} Take $U$ as in Lemma \ref{SmallUExists}. By Lemma \ref{UTranslatesCover} and quasicompacity of $\Fl$, there are finitely many $\gamma_1,\ldots,\gamma_k\in \GSp_{2g}(\Q_p)$ such that $\Fl = \bigcup_{i=1}^k \gamma_i\cdot U$. Taking the preimage of this equality under $|\pi_\HT|$ gives the lemma.
\end{proof}

\begin{lem}\label{FiniteCoverMinComp} With $0<\epsilon<1$ and $\gamma_1,\ldots,\gamma_k\in \GSp_{2g}(\Q_p)$ as in Lemma \ref{FiniteCoverAwayFromBoundary}, one has
\[
|\mathcal{X}_{\Gamma(p^\infty)}^\ast| = \bigcup_{i=1}^k \gamma_i \cdot |\mathcal{X}_{\Gamma(p^\infty)}^\ast(\epsilon)|\ .
\]
\end{lem}

\begin{proof} Let $V\subset |\mathcal{X}_{\Gamma(p^\infty)}^\ast|$ denote the right-hand side. Thus, $V$ is a quasicompact open subset containing $|\mathcal{X}_{\Gamma(p^\infty)}^\ast|\setminus |\mathcal{Z}_{\Gamma(p^\infty)}|$. By quasicompacity, $V$ is the preimage of some $V_m\subset \mathcal{X}_{\Gamma(p^m)}^\ast$, containing $\mathcal{X}_{\Gamma(p^m)}^\ast\setminus \mathcal{Z}_{\Gamma(p^m)}$. To prove $V_m=\mathcal{X}_{\Gamma(p^m)}^\ast$, it suffices to see that they have the same classical points. Thus, assume $x\not\in V_m$ is a classical point of $\mathcal{X}_{\Gamma(p^m)}^\ast$. Then $x=\bigcap_{x\in U} U$ is the intersection of all open neighborhoods $U\subset \mathcal{X}_{\Gamma(p^m)}^\ast$. As $V_m$ is quasicompact for the constructible topology, it follows that $U\subset \mathcal{X}_{\Gamma(p^m)}^\ast\setminus V_m$ for some open neighborhood $U$ of $x$. In particular, $U\subset \mathcal{Z}_{\Gamma(p^m)}$. This is impossible, as $U$ is open (so that e.g. $\dim U > \dim \mathcal{Z}_{\Gamma(p^m)}$).
\end{proof}

\subsection{On adic spaces}

\begin{cor}\label{Existence} There exists a perfectoid space $\mathcal{X}_{\Gamma(p^\infty)}^\ast$ over $\Q_p^\cycl$ such that
\[
\mathcal{X}_{\Gamma(p^\infty)}^\ast\sim \varprojlim_m \mathcal{X}_{\Gamma(p^m)}^\ast\ .
\]
It is covered by finitely many $\GSp_{2g}(\Q_p)$-translates of $\mathcal{X}_{\Gamma(p^\infty)}^\ast(\epsilon)_a$, for any $0<\epsilon<\frac 12$.
\end{cor}

\begin{proof} Choose any $0<\epsilon<\frac 12$ and use Lemma \ref{FiniteCoverMinComp} and Theorem \ref{ExOnStrictNbhd}. Note that
\[
\mathcal{X}_{\Gamma(p^\infty)}^\ast(\epsilon) = \GSp_{2g}(\Z_p)\cdot \mathcal{X}_{\Gamma(p^\infty)}^\ast(\epsilon)_a\ .
\]
\end{proof}

Let $\mathcal{Z}_{\Gamma(p^\infty)}\subset \mathcal{X}_{\Gamma(p^\infty)}^\ast$ denote the boundary, which has an induced structure as a perfectoid space.

\begin{cor} There is a unique map of adic spaces over $\Q_p$
\[
\pi_\HT: \mathcal{X}_{\Gamma(p^\infty)}^\ast\setminus \mathcal{Z}_{\Gamma(p^\infty)}\to \Fl
\]
which realizes $|\pi_\HT|$ on topological spaces.
\end{cor}

\begin{proof} Uniqueness is clear. For existence, we argue as in the proof of Lemma \ref{PiHTTop}, using affinoid perfectoid subsets of $\mathcal{X}_{\Gamma(p^\infty)}^\ast\setminus \mathcal{Z}_{\Gamma(p^\infty)}$ in place of $\tilde{U}_\infty$.
\end{proof}

\begin{lem}\label{ImageOrdinaryEtale} The preimage of $\Fl_{\{g+1,\ldots,2g\}}(\Q_p)$ is given by the closure of $\mathcal{X}_{\Gamma(p^\infty)}^\ast(0)_a\setminus \mathcal{Z}_{\Gamma(p^\infty)}(0)_a$.
\end{lem}

\begin{proof} By Lemma \ref{PreimageRationalPoints}, the preimage is contained in the closure of $\mathcal{X}_{\Gamma(p^\infty)}^\ast(0)\setminus \mathcal{Z}_{\Gamma(p^\infty)}(0)$. Also, it is enough to argue with rank-$1$-points, and we have to see that a rank-$1$-point $x$ of $\mathcal{X}_{\Gamma(p^\infty)}^\ast(0)\setminus \mathcal{Z}_{\Gamma(p^\infty)}(0)$ is mapped into $\Fl_{\{g+1,\ldots,2g\}}(\Q_p)$ if and only if $x\in \mathcal{X}_{\Gamma(p^\infty)}^\ast(0)_a$. On $\mathcal{X}_{\Gamma(p^\infty)}(0)$, we can argue as follows. The point $x$ corresponds to an abelian variety $A/\OO_K$ with ordinary reduction, with a symplectic isomorphism $\alpha: T_p A\cong \Z_p^{2g}$. The abelian variety $A$ thus has its canonical subgroup $C\subset T_p A$, as well as $D = \alpha^{-1}(\Z_p^g\oplus 0^g)\subset T_p A$. We have $x\in \mathcal{X}_{\Gamma(p^\infty)}(0)_a$ if and only if $C/p \oplus D/p\cong A[p]$, or equivalently $C\oplus D\cong T_p A$, or also $\alpha(C)\oplus (\Z_p^g\oplus 0^g)\cong \Z_p^{2g}$. Also, the Hodge-Tate filtration is given by $\alpha(C)\otimes_{\Z_p} K\subset K^{2g}$. Thus, the result follows from the observation that $\Fl_{\{g+1,\ldots,2g\}}(\Q_p)$ is the set of those totally isotropic direct summands $M\subset \Z_p^{2g}$ which satisfy $M\oplus (\Z_p^g\oplus 0^g)\buildrel\cong\over\to \Z_p^{2g}$.

To extend to $\mathcal{X}_{\Gamma(p^\infty)}^\ast(0)\setminus \mathcal{Z}_{\Gamma(p^\infty)}(0)$, use that by Theorem \ref{ExOnStrictNbhd}, the triple
\[
(\mathcal{X}_{\Gamma(p^\infty)}^\ast(0)_a,\mathcal{Z}_{\Gamma(p^\infty)}(0)_a,\mathcal{X}_{\Gamma(p^\infty)}(0)_a)
\]
is good. Take any point $x\in \Fl(\Q_p)\setminus \Fl_{\{g+1,\ldots,2g\}}(\Q_p)$, and fix an open affinoid neighborhood $U\subset \Fl$ of $x$ with $U\cap \Fl_{\{g+1,\ldots,2g\}} = \emptyset$. Then $\Fl(\Q_p) = U(\Q_p)\bigsqcup (\Fl(\Q_p)\setminus U(\Q_p))$ is a decomposition into open and closed subsets. Taking the preimage under $\pi_\HT$ of $U$ gives an open and closed subset of $\mathcal{X}_{\Gamma(p^\infty)}^\ast(0)_a\setminus \mathcal{Z}_{\Gamma(p^\infty)}(0)_a$. Because the displayed triple is good, any open and closed subset of $\mathcal{X}_{\Gamma(p^\infty)}^\ast(0)_a\setminus \mathcal{Z}_{\Gamma(p^\infty)}(0)_a$ extends to an open and closed subset of $\mathcal{X}_{\Gamma(p^\infty)}^\ast(0)_a$; let $V$ be the open and closed subset corresponding to $U$. Then intersecting $V$ with the displayed triple gives another good triple. Assume that $V$ is nonempty. As $V$ gives rise to a good triple, it follows that $V\cap \mathcal{X}_{\Gamma(p^\infty)}(0)_a$ is nonempty. But elements of this intersection map under $\pi_\HT$ into $U(\Q_p)\cap \Fl_{\{g+1,\ldots,2g\}}(\Q_p) = \emptyset$, contradiction.

Thus, $\mathcal{X}_{\Gamma(p^\infty)}^\ast(0)_a\setminus \mathcal{Z}_{\Gamma(p^\infty)}(0)_a$ maps into $\Fl_{\{g+1,\ldots,2g\}}(\Q_p)$. Assume that some point
\[
x\in (\mathcal{X}_{\Gamma(p^\infty)}^\ast(0)\setminus \mathcal{X}_{\Gamma(p^\infty)}^\ast(0)_a)\setminus \mathcal{Z}_{\Gamma(p^\infty)}(0)
\]
maps into $\Fl_{\{g+1,\ldots,2g\}}(\Q_p)$. Applying an element $\gamma\in \GSp_{2g}(\Z_p)$, one can arrange that $\gamma x\in \mathcal{X}_{\Gamma(p^\infty)}^\ast(0)_a$. The subset
\[
\mathcal{X}_{\Gamma(p^\infty)}^\ast(0)_a\setminus \gamma\mathcal{X}_{\Gamma(p^\infty)}^\ast(0)_a\subset \mathcal{X}_{\Gamma(p^\infty)}^\ast(0)_a
\]
is open and closed, and thus gives rise to a good triple. By the argument above, if the element $\gamma x$ of this set maps to some element $y\in \gamma \Fl_{\{g+1,\ldots,2g\}}(\Q_p)$, then there is some element $x^\prime\in \mathcal{X}_{\Gamma(p^\infty)}(0)_a\setminus \gamma\mathcal{X}_{\Gamma(p^\infty)}(0)_a$ with $\pi_\HT(x^\prime)\in \gamma \Fl_{\{g+1,\ldots,2g\}}(\Q_p)$. Thus, $\gamma^{-1}x^\prime\in \mathcal{X}_{\Gamma(p^\infty)}(0)\setminus \mathcal{X}_{\Gamma(p^\infty)}(0)_a$ with $\pi_\HT(\gamma^{-1}x^\prime)\in \Fl_{\{g+1,\ldots,2g\}}(\Q_p)$. This contradicts what we proved about the good reduction locus.
\end{proof}

\begin{lem}\label{ExistsSmallEpsilon} For any open subset $U\subset\Fl$ containing $\Fl(\Q_p)$, there is some $\epsilon>0$ such that
\[
\mathcal{X}_{\Gamma(p^\infty)}^\ast(\epsilon)\setminus \mathcal{Z}_{\Gamma(p^\infty)}(\epsilon)\subset \pi_\HT^{-1}(U)\ .
\]
\end{lem}

\begin{proof} The proof is identical to the proof of Lemma \ref{SmallUExists}, reversing the quantification of $\epsilon$ and $U$. Note that one can a priori assume that $U$ is $\GSp_{2g}(\Z_p)$-invariant, as such open subsets are cofinal; this facilitates the induction argument.
\end{proof}

\begin{lem}\label{ExistsSmallEpsilon2} There exists some $0<\epsilon<\frac 12$ such that
\[
\mathcal{X}_{\Gamma(p^\infty)}^\ast(\epsilon)_a\setminus \mathcal{Z}_{\Gamma(p^\infty)}(\epsilon)_a\subset \pi_\HT^{-1}(\Fl_{\{g+1,\ldots,2g\}})\ .
\]
\end{lem}

\begin{proof} Fix some $U\subset \Fl$ containing $\Fl(\Q_p)$ such that $U\cap \Fl_{\{g+1,\ldots,2g\}}\subset U$ is open and closed; let $U^\prime\subset U$ be the open and closed complement. By Lemma \ref{ExistsSmallEpsilon}, we may assume that $\mathcal{X}_{\Gamma(p^\infty)}^\ast(\epsilon)_a\setminus \mathcal{Z}_{\Gamma(p^\infty)}(\epsilon)_a$ maps into $U$. The open and closed preimage of $U^\prime$ gives rise to an open and closed subset $V_\epsilon\subset \mathcal{X}_{\Gamma(p^\infty)}^\ast(\epsilon)_a$ by the goodness part of Theorem \ref{ExOnStrictNbhd}. By Lemma \ref{ImageOrdinaryEtale}, the intersection of $V_\epsilon$ over all $\epsilon>0$ is empty. As all $V_\epsilon$ are spectral spaces, thus quasicompact for the constructible topology, it follows that $V_\epsilon = \emptyset$ for some $\epsilon>0$. Thus, $\pi_\HT$ maps $\mathcal{X}_{\Gamma(p^\infty)}^\ast(\epsilon)_a\setminus \mathcal{Z}_{\Gamma(p^\infty)}(\epsilon)_a$ into $U\cap \Fl_{\{g+1,\ldots,2g\}}\subset \Fl_{\{g+1,\ldots,2g\}}$, as desired.
\end{proof}

\begin{cor}\label{CorHTExists} There is a unique map of adic spaces
\[
\pi_\HT: \mathcal{X}_{\Gamma(p^\infty)}^\ast\to \Fl
\]
extending $\pi_\HT$ on $\mathcal{X}_{\Gamma(p^\infty)}^\ast\setminus \mathcal{Z}_{\Gamma(p^\infty)}$.
\end{cor}

\begin{proof} One checks easily that for any open subset $U\subset \mathcal{X}_{\Gamma(p^\infty)}^\ast$, there is at most one extension of $\pi_\HT$ from $U\setminus (\mathcal{Z}_{\Gamma(p^\infty)}\cap U)$ to $U$. Indeed, we may assume that $U$ is affinoid perfectoid. Given two functions $f,g$ on $U$ with $f=g$ on $U\setminus (\mathcal{Z}_{\Gamma(p^\infty)}\cap U)$, the subset $|f-g|\geq |p|^n$ is an open subset of $U$ contained in the boundary; thus, it is empty. Therefore, $|f-g|<|p|^n$ for all $n$, i.e. $|f-g|=0$. As $U$ is affinoid perfectoid, this implies $f=g$.

To prove existence, we can now work locally. Clearly, the locus of existence of $\pi_\HT$ is $\GSp_{2g}(\Q_p)$-equivariant. By Corollary \ref{Existence}, it suffices to prove that $\pi_\HT$ extends from
\[
\mathcal{X}_{\Gamma(p^\infty)}^\ast(\epsilon)_a\setminus \mathcal{Z}_{\Gamma(p^\infty)}(\epsilon)_a
\]
to $\mathcal{X}_{\Gamma(p^\infty)}^\ast(\epsilon)_a$ for some $\epsilon>0$. Using Lemma \ref{ExistsSmallEpsilon2}, we may assume that the image of $\mathcal{X}_{\Gamma(p^\infty)}^\ast(\epsilon)_a\setminus \mathcal{Z}_{\Gamma(p^\infty)}(\epsilon)_a$ is contained in the open affinoid subset $\Fl_{\{g+1,\ldots,2g\}}\subset\Fl$. Every function on $\Fl_{\{g+1,\ldots,2g\}}$ pulls back to a bounded function on $\mathcal{X}_{\Gamma(p^\infty)}^\ast(\epsilon)_a\setminus \mathcal{Z}_{\Gamma(p^\infty)}(\epsilon)_a$, and thus extends uniquely to $\mathcal{X}_{\Gamma(p^\infty)}^\ast(\epsilon)_a$ by the goodness part of Theorem \ref{ExOnStrictNbhd}. This proves extension of $\pi_\HT$, as desired.
\end{proof}

\subsection{Conclusion} Finally, we can assemble everything and prove the main theorem.

\begin{thm}\label{ExistenceHT} For any tame level $K^p\subset \GSp_{2g}(\A_f^p)$ contained in $\{\gamma\in \GSp_{2g}(\hat{\Z}^p)\mid \gamma\equiv 1\mod N\}$ for some $N\geq 3$ prime to $p$, there exists a perfectoid space $\mathcal{X}_{\Gamma(p^\infty),K^p}^\ast$ over $\Q_p^\cycl$ such that
\[
\mathcal{X}_{\Gamma(p^\infty),K^p}^\ast\sim \varprojlim_m \mathcal{X}_{\Gamma(p^m),K^p}^\ast\ .
\]
Moreover, there is a $\GSp_{2g}(\Q_p)$-equivariant Hodge-Tate period map (of adic spaces over $\Q_p$)
\[
\pi_\HT: \mathcal{X}_{\Gamma(p^\infty),K^p}^\ast\to \Fl\ .
\]
Let $\mathcal{Z}_{\Gamma(p^\infty),K^p}\subset \mathcal{X}_{\Gamma(p^\infty),K^p}^\ast$ denote the boundary. One has the following results.
\begin{altenumerate}
\item[{\rm (i)}] For any subset $J\subset \{1,\ldots,2g\}$ of cardinality $g$, the preimage $\mathcal{V}_J=\Spa(R_{J,\infty},R_{J,\infty}^+)\subset \mathcal{X}_{\Gamma(p^\infty),K^p}^\ast$ of $\Fl_J\subset \Fl$ is affinoid perfectoid. Moreover, $\mathcal{V}_J$ is the preimage of some affinoid $\mathcal{V}_{J,m}=\Spa(R_{J,m},R_{J,m}^+)\subset \mathcal{X}_{\Gamma(p^m),K^p}^\ast$ for all sufficiently large $m$, and $R_{J,\infty}^+$ is the $p$-adic completion of $\varinjlim_m R_{J,m}^+$.
\item[{\rm (ii)}] The subspace $\mathcal{Z}_{\Gamma(p^\infty),K^p}\cap \mathcal{V}_J\subset \mathcal{V}_J$ is strongly Zariski closed.
\item[{\rm (iii)}] For any $(K^p)^\prime\subset K^p$, the diagram
\[\xymatrix{
\mathcal{X}_{\Gamma(p^\infty),(K^p)^\prime}^\ast\ar[rr]^{\pi_\HT}\ar[dr] & & \Fl \\
& \mathcal{X}_{\Gamma(p^\infty),K^p}^\ast\ar[ur]^{\pi_\HT} &
}\]
commutes.
\item[{\rm (iv)}] For any $\gamma\in \GSp_{2g}(\A_f^p)$ such that $\gamma^{-1} K^p \gamma$ is contained in $\{\gamma\in \GSp_{2g}(\hat{\Z}^p)\mid \gamma\equiv 1\mod N\}$ for some $N\geq 3$ prime to $p$, the diagram
\[\xymatrix{
\mathcal{X}_{\Gamma(p^\infty),K^p}^\ast\ar[dd]^\gamma \ar[dr]^{\pi_\HT} & \\
 & \Fl \\
\mathcal{X}_{\Gamma(p^\infty),\gamma^{-1} K^p \gamma}^\ast\ar[ur]^{\pi_\HT} & 
}\]
commutes.
\item[{\rm (v)}] Let $W_\Fl\subset \OO_{\Fl}^{2g}$ denote the universal totally isotropic subspace. Over $\mathcal{X}_{\Gamma(p^\infty),K^p}^\ast\setminus \mathcal{Z}_{\Gamma(p^\infty),K^p}$, one has the locally free module $\Lie A_{K^p}$ given by the Lie algebra of the tautological abelian variety. There is a natural $\GSp_{2g}(\Q_p)$-equivariant isomorphism
\[
\Lie A_{K^p}\cong (\pi_\HT^\ast W_\Fl)|_{\mathcal{X}_{\Gamma(p^\infty),K^p}^\ast\setminus \mathcal{Z}_{\Gamma(p^\infty),K^p}}\ .
\]
It satisfies the obvious analogue of (iii) and (iv).
\item[{\rm (vi)}] Let $\omega_\Fl = (\bigwedge^g W_\Fl)^\ast$ be the natural ample line bundle on $\Fl$. Over $\mathcal{X}_{\Gamma(p^\infty),K^p}^\ast$, one has the natural line bundle $\omega_{K^p}$ (via pullback from any finite level). There is a natural $\GSp_{2g}(\Q_p)$-equivariant isomorphism
\[
\omega_{K^p}\cong \pi_\HT^\ast \omega_\Fl\ ,
\]
extending the isomorphism one gets from (v) by taking the dual of the top exterior power. Moreover, it satisfies the obvious analogue of (iii) and (iv).
\end{altenumerate}
\end{thm}

\begin{proof} We have established existence of $\mathcal{X}_{\Gamma(p^\infty),K^p}^\ast$ and $\pi_\HT$.
\begin{altenumerate}
\item[{\rm (i)}] First, observe that one has the following versions of Lemma \ref{PreimageRationalPoints} and Lemma \ref{ImageOrdinaryEtale}.

\begin{lem} The preimage of $\Fl(\Q_p)\subset \Fl$ under $\pi_\HT$ is given by the closure of $\mathcal{X}_{\Gamma(p^\infty)}^\ast(0)$.
\end{lem}

\begin{lem} The preimage of $\Fl_{\{g+1,\ldots,2g\}}(\Q_p)\subset \Fl$ under $\pi_\HT$ is given by the closure of $\mathcal{X}_{\Gamma(p^\infty)}^\ast(0)_a$.
\end{lem}

For the first, note that for any open $U\subset \Fl$ containing $\Fl(\Q_p)$, $\pi_\HT^{-1}(U)\subset \mathcal{X}_{\Gamma(p^\infty)}^\ast$ is a quasicompact open containing $\mathcal{X}_{\Gamma(p^\infty)}^\ast(\epsilon)\setminus \mathcal{Z}_{\Gamma(p^\infty)}(\epsilon)$ for some $\epsilon>0$ by Lemma \ref{ExistsSmallEpsilon}; thus, $\pi_\HT^{-1}(U)$ contains $\mathcal{X}_{\Gamma(p^\infty)}^\ast(\epsilon)$. In particular, $\pi_\HT^{-1}(\Fl(\Q_p))$ contains the closure of $\mathcal{X}_{\Gamma(p^\infty)}^\ast(0)$. The converse is clear by continuity. The second lemma follows from the first and the proof of Lemma \ref{ImageOrdinaryEtale}.

Let $x\in \Fl(\Q_p)$ correspond to $0^g\oplus \Q_p^g\subset \Q_p^{2g}$. Then by the second lemma, $\pi_\HT^{-1}(x)\subset \mathcal{X}_{\Gamma(p^\infty)}^\ast(\epsilon)_a$ for any $\epsilon>0$. Thus, there is some open neighborhood $U\subset \Fl$ of $x$ with $\pi_\HT^{-1}(U)\subset \mathcal{X}_{\Gamma(p^\infty)}^\ast(\epsilon)_a$. On the other hand, we may choose $\epsilon>0$ such that
\[
\mathcal{X}_{\Gamma(p^\infty)}^\ast(\epsilon)_a\subset \pi_\HT^{-1}(\Fl_{\{g+1,\ldots,2g\}})\ .
\]
Let $\gamma\in \GSp_{2g}(\Q_p)$ be the diagonal element $(p,\ldots,p,1,\ldots,1)$. Then $\gamma^n(\Fl_{\{g+1,\ldots,2g\}})\subset U$ for $n$ sufficiently large. Moreover, $\gamma^n(\Fl_{\{g+1,\ldots,2g\}})\subset \Fl_{\{g+1,\ldots,2g\}}$ is a rational subset. It follows that $\pi_\HT^{-1}(\gamma^n(\Fl_{\{g+1,\ldots,2g\}}))\subset \mathcal{X}_{\Gamma(p^\infty)}^\ast(\epsilon)_a$ is a rational subset.

The analogue of the conditions in (i) is satisfied for $\mathcal{X}_{\Gamma(p^\infty)}^\ast(\epsilon)_a$ by Theorem \ref{ExOnStrictNbhd}. By \cite[Proposition 2.22 (ii)]{ScholzeSurvey}, the properties are stable under passage to rational subsets, giving the result for $\gamma^n(\Fl_{\{g+1,\ldots,2g\}})$. However, the desired property is also stable under the $\GSp_{2g}(\Q_p)$-action, giving it for $\Fl_{\{g+1,\ldots,2g\}}$ itself, and then for all $\Fl_J$.

\item[{\rm (ii)}] This follows from the constructions in the proof of (i), Corollary \ref{BoundaryStronglyClosedAnticanTower} and Lemma \ref{ClosedImmersionPullback}.

\item[{\rm (iii)}] Clear by construction.

\item[{\rm (iv)}] It suffices to check on geometric points outside the boundary, cf. proof of uniqueness in Corollary \ref{CorHTExists}. Thus, the result follows from Proposition \ref{CompAbPDivHT}, comparing the Hodge-Tate filtration of the abelian variety with the Hodge-Tate filtration of the $p$-divisible group (which depends only on the abelian variety up to prime-to-$p$-isogeny).

\item[{\rm (v)}] The isomorphism comes directly from the construction of the Hodge-Tate period map, cf. Lemma \ref{PiHTTop}. The commutativity in (iii) is clear, while the commutativity in (iv) can again be checked on geometric points, where it follows from Proposition \ref{CompAbPDivHT}.

\item[{\rm (vi)}] The only nontrivial point is to show that the isomorphism extends to $\mathcal{X}_{\Gamma(p^\infty)}^\ast$; all commutativity statements will then follow by continuity from the commutativity in (v). Both $\omega_{K^p}$ and $\pi_\HT^\ast \omega_{\Fl}$ have natural $\OO^+$-structures $\omega_{K^p}^+$ resp. $\pi_\HT^\ast \omega_\Fl^+$; i.e. sheaves of $\OO_{\mathcal{X}_{\Gamma(p^\infty)}^\ast}^+$-modules which are locally free of rank $1$, and give rise to $\omega_{K^p}$ resp. $\pi_\HT^\ast \omega_\Fl$ after inverting $p$. For $\omega_\Fl^+$, this follows from the existence of the natural integral model of the flag variety over $\Z_p$. For $\omega_{K^p}^+$, one gets it via pullback from the integral model $\mathfrak{X}^\ast$ of $\mathcal{X}^\ast$. We claim that the isomorphism $\omega_{K^p}\cong \pi_\HT^\ast \omega_\Fl$ over $\mathcal{X}_{\Gamma(p^\infty)}^\ast\setminus \mathcal{Z}_{\Gamma(p^\infty)}$ is bounded with respect to these integral structures, i.e. there is some constant $C$ (depending only on $g$) such that
\[
p^C \omega_{K^p}^+\subset \pi_\HT^\ast \omega_\Fl^+ \subset p^{-C} \omega_{K^p}^+
\]
as sheaves over $\mathcal{X}_{\Gamma(p^\infty)}^\ast\setminus \mathcal{Z}_{\Gamma(p^\infty)}$. This follows from Proposition \ref{CompAbPDivHT} and \cite[Th\'eor\`eme II.1.1]{FarguesTwoTowers}: These results show that in fact, one map is defined integrally, and has an integral inverse up to $p^{g/(p-1)}$, except that the latter theorem was only proved there for $p\neq 2$.

Here is an alternative argument to get the desired boundedness. Argue by induction on $g$. The locus of good reduction is quasicompact, so necessarily the isomorphism is bounded there. By Proposition \ref{CompAbPDivHT}, the Hodge-Tate period map near the boundary can be described in terms of the Hodge-Tate period map for smaller genus; thus, the isomorphism is bounded there by induction.

Let $j: \mathcal{X}_{\Gamma(p^\infty)}^\ast\setminus \mathcal{Z}_{\Gamma(p^\infty)}\to \mathcal{X}_{\Gamma(p^\infty)}^\ast$ denote the inclusion. Then we have inclusions
\[
\omega_{K^p}\hookrightarrow j_\ast j^\ast \omega_{K^p}\cong j_\ast j^\ast \pi_\HT^\ast \omega_\Fl\hookleftarrow \pi_\HT^\ast \omega_\Fl\ .
\]
We claim that $\omega_{K^p}$ and $\pi_\HT^\ast \omega_\Fl$ agree as subsheaves of $j_\ast j^\ast \omega_{K^p}$. First, we check that $\omega_{K^p}\subset \pi_\HT^\ast \omega_\Fl$. By Corollary \ref{Existence}, this can be checked after pullback to $\mathcal{X}_{\Gamma(p^\infty)}^\ast(\epsilon)_a$, for any given $\epsilon$. If $\epsilon$ is small enough, then by Lemma \ref{ExistsSmallEpsilon2}, $\pi_\HT^\ast \omega_\Fl^+$ is trivial over $\mathcal{X}_{\Gamma(p^\infty)}^\ast(\epsilon)_a$, as $\omega_\Fl^+$ is trivial over $\Fl_{\{g+1,\ldots,2g\}}$. As $\mathcal{X}_{\Gamma(p^\infty)}^\ast(\epsilon)_a$ is affinoid and $\omega_{K^p}$ is locally free (of rank $1$), $\omega_{K^p}$ restricted to $\mathcal{X}_{\Gamma(p^\infty)}^\ast(\epsilon)_a$ is generated by its global sections. Thus, to check the inclusion $\omega_{K^p}\subset \pi_\HT^\ast \omega_\Fl$ over $\mathcal{X}_{\Gamma(p^\infty)}^\ast(\epsilon)_a$, it is enough to check that there is an inclusion
\[
\omega_{K^p}(\mathcal{X}_{\Gamma(p^\infty)}^\ast(\epsilon)_a)\subset (\pi_\HT^\ast \omega_\Fl)(\mathcal{X}_{\Gamma(p^\infty)}^\ast(\epsilon)_a)\ .
\]
But any section of the left-hand side is bounded with respect to the integral structure $\omega_{K^p}^+$, thus by the above also bounded with respect to the integral structure $\pi_\HT^\ast \omega_\Fl^+$. As $\pi_\HT^\ast \omega_\Fl^+$ is isomorphic to $\OO^+$ over $\mathcal{X}_{\Gamma(p^\infty)}^\ast(\epsilon)_a$, the desired inclusion follows from the goodness part of Theorem \ref{ExOnStrictNbhd}, which shows that
\[
(\pi_\HT^\ast \omega_\Fl^+)(\mathcal{X}_{\Gamma(p^\infty)}^\ast(\epsilon)_a\setminus \mathcal{Z}_{\Gamma(p^\infty)}(\epsilon)_a) = (\pi_\HT^\ast \omega_\Fl^+)(\mathcal{X}_{\Gamma(p^\infty)}^\ast(\epsilon)_a)\ .
\]

In particular, we get a map of line bundles
\[
\alpha: \omega_{K^p}\to \pi_\HT^\ast \omega_\Fl
\]
defined on all of $\mathcal{X}_{\Gamma(p^\infty)}^\ast$. We claim that it is an isomorphism. Let $U=\Spa(R,R^+)\subset \mathcal{X}_{\Gamma(p^\infty)}^\ast$ be any affinoid subset over which $\omega_{K^p}^+$ and $\pi_\HT^\ast \omega_\Fl^+$ become trivial; thus $\omega_{K^p}^+(U)\cong R^+ \cdot f_1$, $\omega_{K^p}(U)\cong R\cdot f_1$, $\pi_\HT^\ast \omega_\Fl^+(U)\cong R^+\cdot f_2$ and $\pi_\HT^\ast \omega_\Fl(U)\cong R\cdot f_2$ for certain generators $f_1, f_2$. Under the map $\alpha$, $\alpha(f_1) = h f_2$ for a function $h\in R$. The boundedness of the isomorphism away from the boundary says that $|p|^C\leq |h(x)|$ for all $x\in U\setminus \partial$, where $\partial$ denotes the boundary. The open subset $|h|\leq |p|^{C+1}$ is an open subset of $U$ which does not meet the boundary; thus, it is empty. It follows that $h$ is bounded away from $0$, and therefore invertible. This shows that $\alpha$ is an isomorphism over $U$, as desired.
\end{altenumerate}
\end{proof}

\chapter{$p$-adic automorphic forms}\label{AutomorphicChapter}

Let $G$ be a reductive group over $\Q$. Although there is a tremendous amount of activity surrounding `$p$-adic automorphic forms', a general definition is missing. There are essentially two approaches to defining such spaces. The first works only under special hypothesis on $G$, namely that there is a Shimura variety associated with $G$. More precisely, we will consider the following setup (slightly different from the usual setup). For convenience, assume that $G$ has simply connected derived group $G_\der$, and that there is a $G(\R)$-conjugacy class $D$ of homomorphisms $u: U(1)\to G^\ad_\R$ for which $\ad u(-1)$ is a Cartan involution, and $\mu = u_\C: \G_m\to G^\ad_\C$ is minuscule. In particular, $G$ has a compact inner form, and $G(\R)$ is connected: As $G_\der$ is simply connected, $G_\der(\R)$ is connected, and $(G/G_\der)(\R)$ is a compact, thus connected, torus. In this situation, $D\cong G(\R)/K_\infty$ carries the structure of a hermitian symmetric domain, where the stabilizer $K_\infty$ of any chosen $u$ is a maximal compact subgroup. Moreover, for any (sufficiently small) compact open subgroup $K\subset G(\A_f)$, the quotient
\[
X_K = G(\Q)\backslash [D\times G(\A_f)/K]
\]
is a complex manifold, which by the theorem of Baily-Borel, \cite{BailyBorel}, has a unique structure as an algebraic variety over $\C$. By a theorem of Faltings, \cite{FaltingsFieldOfDefinition}, it is canonically defined over $\bar{\Q}$, and one might yet further descend to a canonical model over a number field (depending on $K$ in this generality, however). For the purpose of this paper, it is however not necessary to worry about fields of definition. Fix a prime $p$, an isomorphism $\C\cong \bar{\Q}_p$, as well as a complete algebraically closed extension $C$ of $\bar{\Q}_p$; then, via base-change, we may get corresponding algebraic varieties over $C$.

In fact, we will be interested in the minimal (Satake-Baily-Borel) compactifications
\[
X_K^\ast = G(\Q)\backslash [D^\ast\times G(\A_f)/K]\ ,
\]
where $D^\ast\supset D$ is the Satake compactification. These carry a natural structure as projective normal algebraic varieties over $\C$. By base-change, we get algebraic varieties over $C$, and we let $\mathcal{X}_K^\ast$ be the associated adic space over $\Spa(C,\OO_C)$. Moreover, one can define a natural ample line bundle $\omega_K$ on $X_K^\ast$, and sections of
\[
H^0(X_K^\ast,\omega_K^{\otimes k})
\]
are certainly complex automorphic forms, for any $k\geq 0$. Denote by $\omega_K$ also the associated line bundle on $\mathcal{X}_K^\ast$; then
\[
H^0(\mathcal{X}_K^\ast,\omega_K^{\otimes k})
\]
forms a space of $p$-adic automorphic forms, in the sense that it is a vector space over the $p$-adic field $C$, and that it bears a direct relationship to complex automorphic forms (so that e.g. Hecke eigenvalues match up). More general spaces of $p$-adic automorphic forms can be defined by looking at (overconvergent) sections of $\omega_K^{\otimes k}$ (or other automorphic vector bundles) on affinoid subsets of $\mathcal{X}_K^\ast$, such as the ordinary locus.

For general groups $G$ (not having a compact inner form), no such definition of $p$-adic automorphic forms is possible. In fact, only the holomorphic (instead of merely real-analytic) automorphic forms will occur even for those $G$ which give rise to a Shimura variety; for general $G$, there are no `holomorphic' automorphic forms. It was suggested by Calegari and Emerton, \cite{CalegariEmerton}, to consider the `completed cohomology groups' as a working model for the space of $p$-adic automorphic forms. Let us recall the definition, for any compact open subgroup $K^p\subset G(\A_f^p)$ (referred to as a tame level):
\[
\widetilde{H}^i_{K^p}(\Z/p^n\Z) = \varinjlim_{K_p} H^i(X_{K_pK^p},\Z/p^n\Z)\ ,
\]
as well as
\[
\widetilde{H}^i_{K^p}(\Z_p) = \varprojlim_n \widetilde{H}^i_{K^p}(\Z/p^n\Z) = \varprojlim_n \varinjlim_{K_p} H^i(X_{K_pK^p},\Z/p^n\Z)\ .
\]
Here, $X_K$ denotes the locally symmetric space associated with $G$ and $K\subset G(\A_f)$ (which exists for any reductive group $G$, and agrees with the $X_K$ defined previously if $G$ satisfies the above hypothesis). For any (sufficiently small) $K_p\subset G(\Q_p)$, one has a map
\[
H^i(X_{K_pK^p},\Q_p)\to \widetilde{H}^i_{K^p}(\Z_p)[p^{-1}]\ .
\]
By a theorem of Franke, \cite{Franke}, all Hecke eigenvalues appearing in
\[
H^i(X_{K_pK^p},\C) = H^i(X_{K_pK^p},\Q_p)\otimes_{\Q_p} \C
\]
come from automorphic forms on $G$ (possibly non-holomorphic!). Thus, by the global Langlands conjectures, one expects to have $p$-adic Galois representations associated with these Hecke eigenvalues. However, the space $\widetilde{H}^i_{K^p}(\Z_p)[p^{-1}]$ is in general much bigger than $\varinjlim_{K_p} H^i(X_{K_pK^p},\Q_p)$, because torsion in the cohomology for the individual $X_{K_pK^p}$ may build up in the inverse limit to torsion-free $\Z_p$-modules. Then Calegari and Emerton conjecture that although the completed cohomology group has no apparent relation to classical automorphic forms, there should still be $p$-adic Galois representation associated with them. In fact, this should hold already on the integral level for $\widetilde{H}^i_{K^p}(\Z_p)$, and thus equivalently for all $\widetilde{H}^i_{K^p}(\Z/p^n\Z)$. In the following, we will usually work at torsion level with $\widetilde{H}^i_{K^p}(\Z/p^n\Z)$, as some technical issues go away.

We remark that this second approach works uniformly for all reductive groups $G$, and that the corresponding space of $p$-adic automorphic forms is (in general) strictly larger than what one can get from classical automorphic forms by $p$-adic interpolation, i.e. there are genuinely new $p$-adic phenomena. We will prove however that if $G$ has an associated Shimura variety (of Hodge type), then one can get all Hecke eigenvalues in the completed cohomology groups $\widetilde{H}^i_{K^p}(\Z_p)$ via $p$-adic interpolation from Hecke eigenvalues appearing in (the cuspidal subspace of) $H^0(\mathcal{X}_{K_pK^p}^\ast,\omega_{K_pK^p}^{\otimes k})$ for some $K_p\subset G(\Q_p)$ and $k\geq 0$.

\section{Perfectoid Shimura varieties of Hodge type}

In this section, we assume that the pair $(G,D)$ is of Hodge type, i.e. admits a closed embedding $(G,D)\hookrightarrow (\Sp_{2g},D_{\Sp_{2g}})$ into the split symplectic group $\Sp_{2g}$, with $D_{\Sp_{2g}}$ given by the Siegel upper-half space.\footnote{Sometimes, symplectic groups for general symplectic $\Q$-vector spaces are allowed in the definition; however, by Zarhin's trick, the corresponding notions are equivalent.} We fix such an embedding; all constructions to follow will depend (at least a priori) on this choice.

To lighten notation, write $(G^\prime,D^\prime) = (\Sp_{2g},D_{\Sp_{2g}})$. We continue to denote by $X_K$ and $X_K^\ast$, $K\subset G(\A_f)$, the locally symmetric varieties associated with $G$, and we denote by $Y_{K^\prime}$, $Y_{K^\prime}^\ast$, $K^\prime\subset G^\prime(\A_f)$, the locally symmetric varieties associated with $G^\prime = \Sp_{2g}$. There are natural finite maps
\[
X_K\to Y_{K^\prime}\ ,\ X_K^\ast\to Y_{K^\prime}^\ast
\]
for any compact open subgroup $K^\prime\subset G^\prime(\A_f)$ with $K=K^\prime\cap G(\A_f)$. By \cite[Proposition 1.15]{DeligneShimura}, for any $K\subset G(\A_f)$, there is some $K^\prime\subset G^\prime(\A_f)$ with $K=K^\prime\cap G(\A_f)$ such that the map $X_K\to Y_{K^\prime}$ is a closed embedding. Unfortunately, it is not known to the author whether the analogous result holds true for the minimal compactification. We define $X_K^\ast\to X_K^{\uast}$ as the universal finite map over which $X_K^\ast\to Y_{K^\prime}^\ast$ factors for all $K^\prime$ with $K=K^\prime\cap G(\A_f)$. As everything is of finite type, $X_K^{\uast}$ is the scheme-theoretic image of $X_K^\ast$ in $Y_{K^\prime}^\ast$ for any sufficiently small $K^\prime$ with $K=K^\prime\cap G(\A_f)$. Note that one still has an action of $G(\A_f)$ on the tower of the $X_K^\uast$.

Let $\mathcal{X}_K^\uast$ be the adic space over $C$ associated with $X_K^\uast$. We continue to denote by $\Fl$ the adic space over $C$ which is the flag variety of totally isotropic subspaces of $C^{2g}$ (i.e., the flag variety associated with $(G^\prime,D^\prime)$\footnote{There is also a flag variety $\Fl_G\subset \Fl$ for $(G,D)$, and one may conjecture that the Hodge-Tate period map defined below factors over $\Fl_G$. We do not address this question here.}). Let $\omega_K$ be the ample line bundle on $\mathcal{X}_K^\uast$ given via pullback from the ample line bundle $\omega_{K^\prime}$ on $Y_{K^\prime}^\ast$ (given by the dual of the determinant of the Lie algebra of the universal abelian variety on $Y_{K^\prime}$); also recall that we have $\omega_\Fl$. We get the following version of Theorem \ref{ExistenceHT}.

\begin{thm}\label{PerfShHodge} For any tame level $K^p\subset G(\A_f^p)$ contained in the level-$N$-subgroup $\{\gamma\in G^\prime(\hat{\Z}^p)\mid \gamma\equiv 1\mod N\}$ of $G^\prime$ for some $N\geq 3$ prime to $p$, there exists a perfectoid space $\mathcal{X}_{K^p}^\uast$ over $C$ such that
\[
\mathcal{X}_{K^p}^\uast\sim \varprojlim_{K_p} \mathcal{X}_{K_pK^p}^\uast\ .
\]
Moreover, there is a $G(\Q_p)$-equivariant Hodge-Tate period map
\[
\pi_\HT: \mathcal{X}_{K^p}^\uast\to \Fl\ .
\]
Let $\mathcal{Z}_{K^p}\subset \mathcal{X}_{K^p}^\uast$ denote the boundary. One has the following results.
\begin{altenumerate}
\item[{\rm (i)}] For any subset $J\subset \{1,\ldots,2g\}$ of cardinality $g$, the preimage $\mathcal{V}_J=\Spa(R_{J,\infty},R_{J,\infty}^+)\subset \mathcal{X}_{K^p}^\uast$ of $\Fl_J\subset \Fl$ is affinoid perfectoid. Moreover, $\mathcal{V}_J$ is the preimage of some affinoid $\mathcal{V}_{J,K_p}=\Spa(R_{J,K_p},R_{J,K_p}^+)\subset \mathcal{X}_{K_pK^p}^\uast$ for all sufficiently small $K_p$, and $R_{J,\infty}^+$ is the $p$-adic completion of $\varinjlim_{K_p} R_{J,K_p}^+$.
\item[{\rm (ii)}] The subset $\mathcal{Z}_{K^p}\cap \mathcal{V}_J\subset \mathcal{V}_J$ is strongly Zariski closed.
\item[{\rm (iii)}] For any $(K^p)^\prime\subset K^p$, the diagram
\[\xymatrix{
\mathcal{X}_{(K^p)^\prime}^\uast\ar[rr]^{\pi_\HT}\ar[dr] & & \Fl \\
& \mathcal{X}_{K^p}^\uast\ar[ur]^{\pi_\HT} &
}\]
commutes.
\item[{\rm (iv)}] For any $\gamma\in G(\A_f^p)$ such that $\gamma^{-1} K^p \gamma$ is contained in the level-$N$-subgroup of $G^\prime$ for some $N\geq 3$ prime to $p$, the diagram
\[\xymatrix{
\mathcal{X}_{K^p}^\uast\ar[dd]^\gamma \ar[dr]^{\pi_\HT} & \\
 & \Fl \\
\mathcal{X}_{\gamma^{-1} K^p \gamma}^\uast\ar[ur]^{\pi_\HT} & 
}\]
commutes.
\item[{\rm (v)}] Over $\mathcal{X}_{K^p}^\uast$, one has the natural line bundle $\omega_{K^p}$ (via pullback from any finite level). There is a natural $G(\Q_p)$-equivariant isomorphism
\[
\omega_{K^p}\cong \pi_\HT^\ast \omega_\Fl\ .
\]
Moreover, it satisfies the obvious analogue of (iii) and (iv).
\end{altenumerate}
\end{thm}

\begin{proof} First, observe that Theorem \ref{ExistenceHT} implies the theorem in case $G=\Sp_{2g}$ by tensoring with $C$ over $\Q_p^\cycl$, and passing to a connected component.

Next, we prove existence of $\mathcal{X}_{K^p}^\uast$. Take any $J\subset \{1,\ldots,2g\}$ of cardinality $g$, and let $Z(J)\subset Z$ be the inverse image of $\Fl_J$ under $\pi_\HT$, for any space $Z$ mapping via $\pi_\HT$ to $\Fl$. Then $\mathcal{Y}_{K^{p\prime}}^\ast(J) = \Spa(S_{K^{p\prime}},S_{K^{p\prime}}^+)$ is affinoid perfectoid by Theorem \ref{ExistenceHT}, for any $K^{p\prime}\subset G^\prime(\A_f^p)$ contained in the level-$N$-subgroup for some $N\geq 3$ prime to $p$. It follows that
\[
\mathcal{Y}_{K^p}^\ast(J) = \varprojlim_{K^p\subset K^{p\prime}\subset G(\A_f^p)} \mathcal{Y}_{K^{p\prime}}^\ast(J) = \Spa(S_{K^p},S_{K^p}^+)
\]
is affinoid perfectoid, with $S_{K^p}^+$ being the $p$-adic completion of $\varinjlim_{K^p\subset K^{p\prime}} S_{K^{p\prime}}^+$. Next, 
\[
(\mathcal{Y}_{K^p}^\ast\times_{\mathcal{Y}_{K_p^\prime K^{p\prime}}^\ast} \mathcal{X}_{K_pK^p}^\uast)(J)\subset \mathcal{Y}_{K^p}^\ast(J)
\]
is defined by some ideal $I\subset S_{K^p}$, for any sufficiently small $K_p^\prime K^{p\prime}\subset G^\prime(\A_f)$ with $K_pK^p = K_p^\prime K^{p\prime}\cap G(\A_f)$. From Lemma \ref{ZariskiClosed}, it follows that
\[
(\mathcal{Y}_{K^p}^\ast\times_{\mathcal{Y}_{K_p^\prime K^{p\prime}}^\ast} \mathcal{X}_{K_pK^p}^\uast)(J) = \Spa(R_{K^p,K_p^\prime K^{p\prime}},R_{K^p,K_p^\prime K^{p\prime}}^+)
\]
is affinoid perfectoid again, and that the map $S_{K^p}\to R_{K^p,K_p^\prime K^{p\prime}}$ has dense image. Then, finally,
\[
\mathcal{X}_{K^p}^\uast(J) = \varprojlim_{K_p^\prime,K^p\subset K^{p\prime}} (\mathcal{Y}_{K^p}^\ast\times_{\mathcal{Y}_{K_p^\prime K^{p\prime}}^\ast} \mathcal{X}_{K_pK^p}^\uast)(J) = \Spa(R_{K^p},R_{K^p}^+)
\]
is affinoid perfectoid, and $R_{K^p}^+$ is the $p$-adic completion of $\varinjlim_{K_p^\prime,K^p\subset K^{p\prime}} R_{K^p,K_p^\prime K^{p\prime}}^+$. This verifies existence of $\mathcal{X}_{K^p}^\uast$ over $\pi_\HT^{-1}(\Fl_J)$, and by varying $J$, we get the result.

Going through the argument, and using part (i) for $G^\prime$, it is easy to deduce part (i) for $G$. The boundary of $\mathcal{X}_{K^p}^\uast(J)$ is the pullback of the boundary of $\mathcal{Y}_{K^{p\prime}}^\ast(J)$, for $K^{p\prime}\subset G^\prime(\A_f^p)$ sufficiently small with $K^p = K^{p\prime}\cap G(\A_f^p)$. Thus, part (ii) follows from Lemma \ref{ClosedImmersionPullback}. All other properties are deduced directly via pullback from $G^\prime$.
\end{proof}

\section{Completed cohomology vs. $p$-adic automorphic forms}

We continue to assume that $(G,D)$ is of Hodge type, and fix the embedding $(G,D)\hookrightarrow (G^\prime,D^\prime) = (\Sp_{2g},D_{\Sp_{2g}})$. Recall the compactly supported completed cohomology groups
\[
\widetilde{H}^i_{c,K^p}(\Z/p^n\Z) = \varinjlim_{K_p} H^i_c(X_{K_pK^p},\Z/p^n\Z)\ .
\]
As usual, we assume that $K^p$ is contained in the level-$N$-subgroup of $G^\prime(\A_f^p)$ for some $N\geq 3$ prime to $p$.

Let $\II_{\mathcal{X}_{K^p}^\uast}\subset \OO_{\mathcal{X}_{K^p}^\uast}$ be the ideal sheaf of the boundary, $\II_{\mathcal{X}_{K^p}^\uast}^+ = \II_{\mathcal{X}_{K^p}^\uast}\cap \OO_{\mathcal{X}_{K^p}^\uast}^+$.

\begin{thm}\label{CompAutomForm} There is natural isomorphism of almost-$\OO_C$-modules
\[
\widetilde{H}^i_{c,K^p}(\Z/p^n\Z)\otimes_{\Z/p^n\Z} \OO_C^a/p^n\cong H^i(\mathcal{X}_{K^p}^\uast,\II_{\mathcal{X}_{K^p}^\uast}^{+a}/p^n)\ ,
\]
where the cohomology group on the right-hand side is computed on the topological space $\mathcal{X}_{K^p}^\uast$. Moreover, for $K_1^p\subset K_2^p$, the diagrams
\[\xymatrix{
\widetilde{H}^i_{c,K_2^p}(\Z/p^n\Z)\otimes_{\Z/p^n\Z} \OO_C^a/p^n\ar[d]\ar[rrr]^\cong & & & H^i(\mathcal{X}_{K_2^p}^\uast,\II_{\mathcal{X}_{K_2^p}^\uast}^{+a}/p^n)\ar[d]  \\
\widetilde{H}^i_{c,K_1^p}(\Z/p^n\Z)\otimes_{\Z/p^n\Z} \OO_C^a/p^n\ar[rrr]^\cong & & & H^i(\mathcal{X}_{K_1^p}^\uast,\II_{\mathcal{X}_{K_1^p}^\uast}^{+a}/p^n)
}\]
\[\xymatrix{
 \widetilde{H}^i_{c,K_1^p}(\Z/p^n\Z)\otimes_{\Z/p^n\Z} \OO_C^a/p^n\ar[rrr]^\cong\ar[d]^\tr & & & H^i(\mathcal{X}_{K_1^p}^\uast,\II_{\mathcal{X}_{K_1^p}^\uast}^{+a}/p^n)\ar[d]^\tr \\
 \widetilde{H}^i_{c,K_2^p}(\Z/p^n\Z)\otimes_{\Z/p^n\Z} \OO_C^a/p^n\ar[rrr]^\cong & & & H^i(\mathcal{X}_{K_2^p}^\uast,\II_{\mathcal{X}_{K_2^p}^\uast}^{+a}/p^n)
}\]
commute, where the definition of the trace maps is recalled below.
\end{thm}

We note that the right-hand side is the cohomology of the sheaf of $p$-adic cusp forms modulo $p^n$ of infinite level.

\begin{proof} Let $j_K: \mathcal{X}_K^\uast\setminus \mathcal{Z}_K\hookrightarrow \mathcal{X}_K^\uast$ be the open embedding, where $\mathcal{Z}_K$ denotes the boundary of $\mathcal{X}_K^\uast$. By the various comparison results between complex and algebraic, resp. algebraic and adic, singular and \'etale cohomology, we have
\[
H^i_c(X_{K_pK^p},\Z/p^n\Z) = H^i_\et (\mathcal{X}_{K_pK^p}^\uast, j_{K_pK^p!} \Z/p^n\Z)\ .
\]
Now we use \cite[Theorem 3.13]{ScholzeSurvey} to write
\[
H^i_\et(\mathcal{X}_{K_pK^p}^\uast, j_{K_pK^p!} \Z/p^n\Z)\otimes_{\Z/p^n\Z} \OO_C^a/p^n = H^i_\et(\mathcal{X}_{K_pK^p}^\uast, j_{K_pK^p!} \OO_{\mathcal{X}_{K_pK^p}^\uast\setminus \mathcal{Z}_{K_pK^p}}^{+a}/p^n)\ .
\]
Passing to the inverse limit over $K_p$ and using \cite[Corollary 7.18]{ScholzePerfectoid}, one gets
\[
\widetilde{H}^i_{c,K^p}(\Z/p^n\Z)\otimes_{\Z/p^n\Z} \OO_C^a/p^n\cong H^i_\et(\mathcal{X}_{K^p}^\uast,\varinjlim_{K_p} j_{K_pK^p!} \OO_{\mathcal{X}_{K_pK^p}^\uast\setminus \mathcal{Z}_{K_pK^p}}^{+a}/p^n)\ .
\]
But
\[
\varinjlim_{K_p} j_{K_pK^p!} \OO_{\mathcal{X}_{K_pK^p}^\uast\setminus \mathcal{Z}_{K_pK^p}}^{+a}/p^n = j_{K^p!} \OO_{\mathcal{X}_{K^p}^\uast\setminus \mathcal{Z}_{K^p}}^{+a}/p^n\ ,
\]
and there is a short exact sequence
\[
0\to j_{K^p!} \OO_{\mathcal{X}_{K^p}^\uast\setminus \mathcal{Z}_{K^p}}^{+a}/p^n\to \OO_{\mathcal{X}_{K^p}^\uast}^{+a}/p^n\to \OO_\partial^{+a}/p^n\to 0\ ,
\]
where $\OO_\partial$ is (the pushforward of) the structure sheaf of the boundary. By \cite[Propositions 6.14, 7.13]{ScholzePerfectoid}, analytic and \'etale cohomology of $\OO^{+a}/p^n$ and $\OO_\partial^{+a}/p^n$ agree: On affinoid subsets, both vanish in positive degrees (also noting that the intersection of an open affinoid subset with the boundary is an open affinoid subset of the boundary by Lemma \ref{ZariskiClosed}). Thus,
\[
\widetilde{H}^i_{c,K^p}(\Z/p^n\Z)\otimes_{\Z/p^n\Z} \OO_C^a/p^n\cong H^i(\mathcal{X}_{K^p}^\uast,j_{K^p!} \OO_{\mathcal{X}_{K^p}^\uast\setminus \mathcal{Z}_{K^p}}^{+a}/p^n)\ .
\]
Moreover, as the boundary is strongly Zariski closed by Theorem \ref{PerfShHodge} (ii), one also an exact sequence of sheaves on the topological space $\mathcal{X}_{K^p}^\uast$,
\[
0\to \II_{\mathcal{X}_{K^p}^\uast}^{+a}/p^n\to \OO_{\mathcal{X}_{K^p}^\uast}^{+a}/p^n\to \OO_\partial^{+a}/p^n\to 0\ ,
\]
so that
\[
j_{K^p!} \OO_{\mathcal{X}_{K^p}^\uast\setminus \mathcal{Z}_{K^p}}^{+a}/p^n = \II_{\mathcal{X}_{K^p}^\uast}^{+a}/p^n\ ,
\]
and we arrive at the desired isomorphism.\footnote{This argument, which appears also in \cite{ScholzeSurvey}, shows that one should think of $\OO^+/p^n$ and related sheaves as being like an 'algebraic topology local system', and not as being like a coherent sheaf.}

The commutativity of the first diagram is immediate from functoriality. Also, the definition of the first trace map is standard (and its various definitions in the complex, algebraic, and $p$-adic worlds are compatible). Let $j_{K^p}: \mathcal{X}_{K^p}^\uast\setminus \mathcal{Z}_{K^p}\to \mathcal{X}_{K^p}^\uast$ be the open embedding. To define the second trace map, it is enough to define a trace map
\[
\tr_{K_1^p/K_2^p}: \pi_{K_1^p/K_2^p\ast} \OO_{\mathcal{X}_{K_1^p}^\uast\setminus \mathcal{Z}_{K_1^p}}^{+a}/p^n\to \OO_{\mathcal{X}_{K_2^p}^\uast\setminus \mathcal{Z}_{K_2^p}}^{+a}/p^n\ ,
\]
where $\pi_{K_1^p/K_2^p}: \mathcal{X}_{K_1^p}^\uast\setminus \mathcal{Z}_{K_1^p}\to \mathcal{X}_{K_2^p}^\uast\setminus \mathcal{Z}_{K_2^p}$ denotes the finite \'etale projection. Locally, this projection has the form $\Spa(B,B^+)\to \Spa(A,A^+)$, where $A$ is a perfectoid $C$-algebra, $A^+\subset A^\circ$ is open and integrally closed, $B$ is a finite \'etale $A$-algebra, and $B^+\subset B$ is the integral closure of $A^+$. From the almost purity theorem, \cite[Theorem 7.9 (iii)]{ScholzePerfectoid}, it follows that $B^{+a}/p^n$ is a finite \'etale $A^{+a}/p^n$-algebra. In particular, there is a trace map $B^{+a}/p^n\to A^{+a}/p^n$ (cf. \cite[Definition 4.14]{ScholzePerfectoid}), as desired.

In order to prove that the second diagram commutes, it is enough to prove that the diagram
\[\xymatrix{
\pi_{K_1^p/K_2^p\ast} \OO_C^a/p^n\ar[d]\ar[rrr]^\tr & & & \OO_C^a/p^n\ar[d] \\
\pi_{K_1^p/K_2^p\ast} \OO_{\mathcal{X}_{K_1^p}^\uast\setminus \mathcal{Z}_{K_1^p}}^{+a}/p^n\ar[rrr]^\tr & & & \OO_{\mathcal{X}_{K_2^p}^\uast\setminus \mathcal{Z}_{K_2^p}}^{+a}/p^n
}\]
of \'etale sheaves on $\mathcal{X}_{K_2^p}^\uast\setminus \mathcal{Z}_{K_2^p}$ commutes (as the diagram in the statement of the theorem comes about by applying $H^i_\et(\mathcal{X}_{K_2^p}^\uast,j_{K_2^p!}-)$ to this diagram). As this can be checked \'etale locally, one can reduce to the case where the morphism $\pi_{K_1^p/K_2^p}$ is a disjoint union of copies of the base, where it is trivial.
\end{proof}

As a first application, we get a vanishing result for (compactly supported) completed cohomology. Recall that the (usual or compactly supported) cohomology groups of $X_K$ are nonzero in the range $[0,2d]$, where $d = \dim_\C X_K$. The following corollary shows that upon taking the direct limit over all levels $K_p$ at $p$, complete cancellation occurs in degrees $i>d$.

\begin{cor}\label{CorVanishing} The cohomology group $\widetilde{H}^i_{c,K^p}(\Z/p^n\Z)$ (and thus $\widetilde{H}^i_{c,K^p}(\Z_p)$) vanishes for $i>d$.
\end{cor}

\begin{proof} We may reduce to the case $n=1$ by long exact sequences. It is enough to prove that $\widetilde{H}_{c,K^p}^i(\F_p)\otimes_{\F_p} \OO_C/p$ is almost zero for $i>d$: Indeed, for a nontrivial $\F_p$-vector space $V$, $V\otimes_{\F_p} \OO_C/p$ is flat over $\OO_C/p$ and nonzero. Thus, if it is killed by the maximal ideal of $\OO_C$, then it is $0$.

By the previous theorem, it suffices to prove that more generally, for any sheaf $F$ of abelian groups on $\mathcal{X}_{K^p}^\uast$, $H^i(\mathcal{X}_{K^p}^\uast,F)=0$ for $i>d$.

Recall that $S=\mathcal{X}_{K^p}^\uast$ is a spectral space; we call the minimal $i$ such that $H^{i+1}(S,F)=0$ for all abelian sheaves $F$ on $S$ the cohomological dimension of $S$. Thus, we claim that the cohomological dimension of $\mathcal{X}_{K^p}^\uast$ is at most $d$. Observe that if $S=\varprojlim S_j$ is a cofiltered inverse limit of spectral spaces $S_j$ of cohomological dimension $\leq d$ along spectral transition maps, then $S$ has cohomological dimension $\leq d$. Indeed, any $F$ can be written as a filtered direct limit of constructible sheaves, constructible sheaves come via pullback from some $S_j$, and one computes cohomology on $S$ as a direct limit of cohomology groups over $S_j$ for increasing $j$.

As $|\mathcal{X}_{K^p}^\uast|\cong \varprojlim_{K_p} |\mathcal{X}_{K_pK^p}^\uast|$, it is enough to prove that $\mathcal{X}_{K_pK^p}^\uast$ has cohomological dimension $\leq d$. For this, we could either cite \cite[Proposition 2.5.8]{deJongvanderPut}, or write $\mathcal{X}_{K_pK^p}^\uast$ as the inverse limit of the topological spaces underlying all possible formal models (each of which is of dimension $\leq d$), and use Grothendieck's bound for noetherian spectral spaces.
\end{proof}

The following corollary implies a good part of \cite[Conjecture 1.5]{CalegariEmerton} in the case considered here. We use freely notation from \cite{CalegariEmerton}. The tame level $K^p$ is fixed, and all modules are taken with $\Z_p$-coefficients.

\begin{cor}\label{CalegariEmertonConj} For $i>d$, $\widetilde{H}_i^\BM=0$, and $\widetilde{H}_d^\BM$ is $p$-torsion free. For $i<d$, the codimension (as a module over the Iwasawa algebra) of $\widetilde{H}_i$ is $\geq d-i$.
\end{cor}

If the $X_K$ are compact, this implies all of \cite[Conjecture 1.5]{CalegariEmerton}, except for non-strict instead of strict inequalities on the codimensions. (Note that here, $l_0 = d$, $q_0 = 0$. Also observe \cite[Theorem 1.4]{CalegariEmerton}.)

\begin{proof} The first two assertions follow from the previous corollary and \cite[Theorem 1.1 (iii)]{CalegariEmerton}. Assume that the last statement was not satisfied; among all codimensions of $\widetilde{H}_i$ which violate this inequality, choose the maximal one, $c$. Among all $i<d$ for which this codimension is achieved, choose the minimal one. Thus, the codimension of $\widetilde{H}_i$ is $c<d-i$, but the codimension of $\widetilde{H}_k$ for $k<i$ is greater than $c$. The results of \cite{Venjakob} imply that if $X$ is of codimension $c$, then $E^j(X)=0$ for $j<c$, $E^c(X)$ is of codimension (exactly) $c$, and $E^j(X)$ is of codimension $\geq j$ for $j>c$.

Now look at the Poincar\'e duality spectral sequence \cite[Section 1.3]{CalegariEmerton}:
\[
E_2^{jk} = E^j(\widetilde{H}_k)\Rightarrow \widetilde{H}_{2d-j-k}^\BM\ .
\]
For $j+k<d$, the limit term $\widetilde{H}_{2d-j-k}^\BM$ vanishes. We look at the diagonal $j+k = i+c<d$. In that case, there is a contribution of codimension $c$, $E^c(\widetilde{H}_i)$. For $k<i$, any term $E^j(\widetilde{H}_k)$ is of codimension at least the codimension of $\widetilde{H}_k$, i.e. of codimension $\geq c+1$. For $j<c$, but $j+k<d$, all terms $E^j(\widetilde{H}_k)$ are zero. If not, the codimension of $\widetilde{H}_k$ is $\leq j<c$ and $j+k<d$, which contradicts our choice of $c$.

It follows that all groups that might potentially cancel the contribution of $E^c(\widetilde{H}_i)$ are of codimension $>c$; as by \cite{Venjakob}, the notion of codimension is well-behaved under short exact sequences, it follows that a subquotient of $E^c(\widetilde{H}_i)$ of codimension $c$ survives the spectral sequence. It would contribute to $\widetilde{H}_{2d-j-k}^\BM$ with $j+k= i+c<d$, contradiction.
\end{proof}

\section{Hecke algebras}

We keep the assumption that $(G,D)$ is of Hodge type, with a fixed embedding $(G,D)\hookrightarrow (G^\prime,D^\prime) = (\Sp_{2g},D_{\Sp_{2g}})$. Moreover, fix some compact open subgroup $K^p\subset G(\A_f^p)$ contained in the level-$N$-subgroup of $G^\prime(\A_f^p)$ for some $N\geq 3$ prime to $p$.

Let
\[
\mathbb{T} = \mathbb{T}_{K^p} = \Z_p[G(\A_f^p)//K^p]
\]
be the abstract Hecke algebra of $K^p$-biinvariant compactly supported functions on $G(\A_f^p)$, where the Haar measure gives $K^p$ measure $1$. In this section, we prove the following result, which says roughly that all Hecke eigenvalues appearing in $\widetilde{H}^i_{c,K^p}(\Z_p)$ come via $p$-adic interpolation from Hecke eigenvalues in $H^0(\mathcal{X}_{K_pK^p}^\ast,\omega_{K_pK^p}^{\otimes k}\otimes \II)$, where $\II$ is the ideal sheaf of the boundary, and $k$ is sufficiently divisible.

\begin{thm}\label{ThmHeckeAlgebras} Fix some integer $m\geq 1$. Let $\mathbb{T}_\cl = \mathbb{T}_{\cl,m}$ denote $\mathbb{T}$ equipped with the weakest topology for which all the maps
\[
\mathbb{T}\to \End_C (H^0(\mathcal{X}_{K_pK^p}^\ast,\omega_{K_pK^p}^{\otimes mk}\otimes \II))
\]
are continuous, for varying $k\geq 1$ and $K_p\subset G(\Q_p)$, where the right-hand side is a finite-dimensional $C$-vector space endowed with the $p$-adic topology.\footnote{Here, $\cl$ stands for classical. Also note that $\mathbb{T}_\cl$ may not be separated; one might replace it by its separated quotient without altering anything that follows.} Then the map
\[
\mathbb{T}_\cl = \mathbb{T}\to \End_{\Z/p^n\Z}(\widetilde{H}^i_{c,K^p}(\Z/p^n\Z))
\]
is continuous, where the right-hand side is endowed with the weakest topology which makes
\[
\End_{\Z/p^n\Z}(\widetilde{H}^i_{c,K^p}(\Z/p^n\Z))\times \widetilde{H}^i_{c,K^p}(\Z/p^n\Z)\to \widetilde{H}^i_{c,K^p}(\Z/p^n\Z)
\]
continuous, where $\widetilde{H}^i_{c,K^p}(\Z/p^n\Z)$ has the discrete topology.
\end{thm}

Before giving the proof, we recall the definition of the action of $\mathbb{T}$ on $H^0(\mathcal{X}_{K_pK^p}^\ast,\omega_{K_pK^p}^{\otimes k}\otimes \II)$. As usual, this boils down to defining trace maps. For this, take two sufficiently small levels $K_1\subset K_2\subset G(\A_f)$, and look at the map
\[
\pi_{K_1/K_2}: \mathcal{X}_{K_1}^\ast\to \mathcal{X}_{K_2}^\ast\ .
\]
This is locally of the form $\Spa(B,B^+)\to \Spa(A,A^+)$, where $A$ is normal, $A^+\subset A^\circ$ is open and integrally closed (and thus normal itself), $B$ is a finite normal and generically \'etale $A$-algebra, and $B^+\subset B$ is the integral closure of $A^+$. In particular, $B^+$ is also a finite normal and generically \'etale $A^+$-algebra. Recall the following lemma.

\begin{lem}\label{TraceMapsNormal} Let $R$ be normal, and let $S$ be a finite and generically \'etale $R$-algebra; i.e., for some non-zero divisor $f\in R$, $S[f^{-1}]$ is a finite \'etale $R[f^{-1}]$-algebra. Then the trace map
\[
\tr_{S[f^{-1}]/R[f^{-1}]}: S[f^{-1}]\to R[f^{-1}]
\]
maps $S$ into $R$. Moreover, for any integrally closed ideal $I\subset R$ with integral closure $J\subset S$, $\tr(J)\subset I$.
\end{lem}

\begin{proof} For an element $x\in R[f^{-1}]$, the condition $x\in R$ can be checked at valuations of $R$. Thus, one can assume that $R=K^+$ is the ring of integers of a field $K$ equipped with some valuation $v: K\to \Gamma\cup \{\infty\}$. We may assume that $K$ is algebraically closed. Also, one may replace $S$, a finite and generically \'etale $K^+$-algebra, by its normalization in $S\otimes_{K^+} K$. In that case, $S$ is a finite product of copies of $K^+$, and the claim is clear.

The condition $x\in I$ can also be checked using valuations, so the same argument works in that case.
\end{proof}

In particular, we get trace maps
\[\begin{aligned}
\tr &: \pi_{K_1/K_2\ast} \OO_{\mathcal{X}_{K_1}^\ast}\to \OO_{\mathcal{X}_{K_2}^\ast}\ ,\\
\tr &: \pi_{K_1/K_2\ast} \II_{\mathcal{X}_{K_1}^\ast}\to \II_{\mathcal{X}_{K_2}^\ast}\ ,\\
\tr &: \pi_{K_1/K_2\ast} \OO_{\mathcal{X}_{K_1}^\ast}^+\to \OO_{\mathcal{X}_{K_2}^\ast}^+\ ,\\
\tr &: \pi_{K_1/K_2\ast} \II_{\mathcal{X}_{K_1}^\ast}^+\to \II_{\mathcal{X}_{K_2}^\ast}^+\ ,
\end{aligned}\]
where $\II_{\mathcal{X}_K^\ast}\subset \OO_{\mathcal{X}_K^\ast}$ is the ideal sheaf of the boundary, and $\II^+ = \II\cap \OO^+$. In particular, by tensoring the trace map for $\II$ with a tensor power of the line bundle $\omega_{K_2}$, we get a trace map
\[
\tr: \pi_{K_1/K_2\ast} (\omega_{K_1}^{\otimes k}\otimes \II_{\mathcal{X}_{K_1}^\ast}) = \pi_{K_1/K_2\ast} (\pi_{K_1/K_2}^\ast\omega_{K_2}^{\otimes k}\otimes \II_{\mathcal{X}_{K_1}^\ast}) = \omega_{K_2}^{\otimes k}\otimes \pi_{K_1/K_2\ast} \II_{\mathcal{X}_{K_1}^\ast}\to \omega_{K_2}^{\otimes k}\otimes \II_{\mathcal{X}_{K_2}^\ast}\ ,
\]
giving the desired trace map by taking global sections.

We will need the following comparison of trace maps. It says in particular that as far as cusp forms of infinite level are concerned, there is no difference between $\mathcal{X}_K^\ast$ and $\mathcal{X}_K^\uast$.

\begin{lem}\label{CompAlgTraceMaps} Fix a subset $J\subset \{1,\ldots,2g\}$ of cardinality $g$, and let
\[
\mathcal{X}_{K^p}^\uast(J)\subset \mathcal{X}_{K^p}^\uast
\]
be the preimage of $\Fl_J\subset \Fl$ under the Hodge-Tate period map $\pi_\HT: \mathcal{X}_{K^p}^\uast\to \Fl$. Recall that by Theorem \ref{PerfShHodge} (i), $\mathcal{X}_{K^p}^\uast(J) = \Spa(R_{K^p},R_{K^p}^+)$ is affinoid perfectoid, and the preimage of $\Spa(R_{K_pK^p},R_{K_pK^p}^+)=\mathcal{X}_{K_pK^p}^\uast(J)\subset \mathcal{X}_{K_pK^p}^\uast$ for $K_p$ small enough, with $R_{K^p}^+$ the $p$-adic completion of $\varinjlim R_{K_pK^p}^+$.

Let $\Spa(\tilde{R}_{K_pK^p},\tilde{R}_{K_pK^p}^+)=\mathcal{X}_{K_pK^p}^\ast(J)\subset \mathcal{X}_{K_pK^p}^\ast$ be the preimage of $\mathcal{X}_{K_pK^p}^\uast(J)$. Moreover, let $I_{K_pK^p}^+\subset R_{K_pK^p}^+$ be the ideal of functions vanishing along the boundary, and let $\tilde{I}_{K_pK^p}^+\subset \tilde{R}_{K_pK^p}^+$, $I_{K^p}^+\subset R_{K^p}^+$ be defined similarly. Then, for all $n\geq 0$,
\[
I_{K^p}^+/p^n = \varinjlim_{K_p} I_{K_pK^p}^+/p^n = \varinjlim_{K_p} \tilde{I}_{K_pK^p}^+/p^n\ .
\]
For $K_1^p\subset K_2^p$ and any $K_p$, the diagram of almost $\OO_C$-modules
\[\xymatrix{
\tilde{I}_{K_pK_1^p}^{+a}/p^n\ar[rrr]^{(\tr\mod p^n)^a}\ar[d] & & & \tilde{I}_{K_pK_2^p}^{+a}/p^n\ar[d]\\
I_{K_1^p}^{+a}/p^n\ar[rrr]^\tr & & & I_{K_2^p}^{+a}/p^n
}\]
commutes, where the trace map on the lower line is as defined in the proof of Theorem \ref{CompAutomForm}.
\end{lem}

\begin{proof} From Theorem \ref{PerfShHodge} (i), we know that
\[
R_{K^p}^+/p^n = \varinjlim_{K_p} R_{K_pK^p}^+/p^n\ .
\]
In particular, the map $\varinjlim_{K_p} I_{K_pK^p}^+/p^n\to I_{K^p}^+/p^n$ is injective. To prove that it is surjective, it is enough to prove that it is almost surjective: Indeed, if $f\in I_{K^p}^+/p^n$ is such that $p^\epsilon f = g_{K_p}$ for some $g_{K_p}\in I_{K_pK^p}^+$ and $0<\epsilon<\frac n2$, then (by considering valuations, using surjectivity of $\mathcal{X}_{K^p}^\uast(J)\to \mathcal{X}_{K_pK^p}^\uast(J)$), $g_{K_p} = p^\epsilon f_{K_p}$ for some $f_{K_p}\in I_{K_pK^p}^+$, and $f_{K_p}\equiv f\mod p^{n-\epsilon}$. Choosing a lift $\tilde{f}\in I_{K^p}^+$ of $f$ and repeating the argument with $f^\prime = (\tilde{f} - f_{K_p})/p^{n-\epsilon}$ gives the claim.

Recall from the construction there exists a pullback diagram of affinoid perfectoid spaces
\[\xymatrix{
\mathcal{Z}_{K^p}(J)\ar[r]\ar[d] & \mathcal{X}_{K^p}^\uast(J)\ar[d]\\
\mathcal{Z}_{\Gamma_0(p^\infty)}(\epsilon)_a\ar[r] & \mathcal{Y}_{\Gamma_0(p^\infty)}^\ast(\epsilon)_a
}\]
where $\mathcal{Y}_{\Gamma_0(p^\infty)}^\ast(\epsilon)_a$ denotes the inverse limit of the anticanonical $\Gamma_0(p^\infty)$-tower in the Siegel moduli space as in Corollary \ref{TiltingStrictNbhd}, with boundary $\mathcal{Z}_{\Gamma_0(p^\infty)}(\epsilon)_a$. By Corollary \ref{BoundaryStronglyClosedAnticanTower} and Lemma \ref{ClosedImmersionPullback} (ii), we are reduced to showing that the ideal in the global sections of $\OO^+$ defining $\mathcal{Z}_{\Gamma_0(p^\infty)}(\epsilon)_a\subset \mathcal{Y}_{\Gamma_0(p^\infty)}^\ast(\epsilon)_a$ is almost generated by functions in $\OO^+$ coming from finite level and vanishing along the boundary. This follows from Tate's normalized traces, cf. Corollary \ref{CorTateTraces} (observing that by Lemma \ref{TraceMapsNormal}, Tate's normalized traces of a function vanishing along the boundary will still vanish along the boundary).

Next, we claim that there is a unique map
\[
\tilde{I}_{K_pK^p}^+\to I_{K^p}^+
\]
commuting with evaluation at points outside the boundary (where there is no difference between $\mathcal{X}_K^\uast$ and $\mathcal{X}_K^\ast$). If it exists, it follows by consideration of valuations that it is injective, with $\tilde{I}_{K_pK^p}^+/p^n\hookrightarrow I_{K^p}^+/p^n$; thus, the composite map
\[
\varinjlim I_{K_pK^p}^+/p^n\to \varinjlim \tilde{I}_{K_pK^p}^+/p^n\to I_{K^p}^+/p^n
\]
is an isomorphism, and the second map injective; i.e., both maps are isomorphisms. To prove existence of $\tilde{I}_{K_pK^p}^+\to I_{K^p}^+$, note that for any $n$, there are maps
\[
\tilde{I}_{K_pK^p}^+/p^n\to H^0(\mathcal{X}_{K_pK^p}^\uast(J),j_{K_pK^p!} \OO^+/p^n)\to H^0(\mathcal{X}_{K^p}^\uast(J),j_{K^p!} \OO^+/p^n) = H^0(\mathcal{X}_{K^p}^\uast(J),\II^+/p^n)\ .
\]
Here, $j_K: \mathcal{X}_K^\uast\setminus \mathcal{Z}_K\to \mathcal{X}_K^\uast$ denotes the open embedding. Using Theorem \ref{PerfShHodge} (ii), we see that for any $n$, the map $I_{K^p}^+/p^n\to H^0(\mathcal{X}_{K^p}^\uast(J),\II^+/p^n)$ is almost an isomorphism; in the inverse limit over $n$, it becomes an isomorphism. Thus, in the inverse limit over $n$, we get the desired map $\tilde{I}_{K_pK^p}^+\to I_{K^p}^+$.

Finally, we need to check commutativity of the diagram of trace maps. It is enough to prove commutativity in the inverse limit over $n$, and then after inverting $p$. The commutativity can be checked after restricting the functions to the complement of the boundary; there, both trace maps are given by trace maps for finite \'etale algebras, giving the result.
\end{proof}

\begin{proof}{\it (of Theorem \ref{ThmHeckeAlgebras}.)} By Theorem \ref{CompAutomForm}, there is a $\mathbb{T}$-equivariant isomorphism
\[
\widetilde{H}^i_{c,K^p}(\Z/p^n\Z)\otimes_{\Z/p^n\Z} \OO_C^a/p^n\cong H^i(\mathcal{X}_{K^p}^\uast,\II^{+a}/p^n)\ .
\]
Also, the map
\[
\Hom_{\Z/p^n\Z}(M,N)\to \Hom_{\OO_C/p^n}(M\otimes_{\Z/p^n\Z} \mathfrak{m}_C/p^n,N\otimes_{\Z/p^n\Z} \mathfrak{m}_C/p^n)
\]
is injective for any $\Z/p^n\Z$-modules $M$, $N$. One may split up $N$ using short exact sequences to reduce to the case $pN=0$. In that case, one reduces further to $n=1$. But for $\F_p$-vector spaces, the result is clear.

In particular, it is enough to prove that
\[
\mathbb{T}_\cl=\mathbb{T}\to \End_{\OO_C^a/p^n}(H^i(\mathcal{X}_{K^p}^\uast,\II^{+a}/p^n))
\]
is continuous, where for an $\OO_C^a/p^n$-module $M$, we endow $M_!$ with the discrete topology, and $\End_{\OO_C^a/p^n}(M)$ with the weakest topology making
\[
\End_{\OO_C^a/p^n}(M)\times M_!\to M_!
\]
continous. We remark that $M\mapsto M_!$ is an exact functor commuting with all colimits. In particular, if $M$ is a colimit of $\mathbb{T}$-modules $M_i$ on which $\mathbb{T}_\cl = \mathbb{T}$ acts continuously, then $\mathbb{T}_\cl = \mathbb{T}$ acts continuously on $M$. Moreover, if $M=N^a$ for an actual $\mathbb{T}\otimes \OO_C/p^n$-module $N$, on which $\mathbb{T}_\cl = \mathbb{T}$ acts continuously, then it also acts continuously on $M$, as $M_! = \mathfrak{m}_C\otimes_{\OO_C} N$. Also, if $\mathbb{T}_\cl = \mathbb{T}$ acts continuously on $M$, it acts continuously on any subquotient.

Now we use the Hodge-Tate period map
\[
\pi_\HT: \mathcal{X}_{K^p}^\uast\to \Fl\hookrightarrow \mathbb{P}^{\binom{2g}g - 1}\ ,
\]
using the Pl\"ucker embedding. Let $N=\binom{2g}g$. The ample line bundle $\omega_\Fl$ on $\Fl$ is the pullback of $\OO(1)$ on $\mathbb{P}^{N - 1}$. Fix the standard sections $s_1,\ldots,s_N\in H^0(\mathbb{P}^{N - 1},\OO(1))$. For $i=1,\ldots,N$, let $\mathcal{U}_i\subset \mathbb{P}^{N - 1}$ be the open affinoid subset where $|s_j|\leq |s_i|$ for all $j=1,\ldots,N$. For $J\subset \{1,\ldots,N\}$, let $\mathcal{U}_J = \bigcap_{i\in J} \mathcal{U}_i$. Observe that $\mathcal{U}_{ij}\subset \mathcal{U}_i$ is given by the condition $|\frac{s_j}{s_i}| = 1$, where $\frac{s_j}{s_i}\in H^0(\mathcal{U}_i,\OO_{\mathcal{U}_i}^+)$.

Let $\mathcal{V}_i = \pi_\HT^{-1}(\mathcal{U}_i)\subset \mathcal{X}_{K^p}^\uast$; by Theorem \ref{PerfShHodge} (i), this is affinoid perfectoid, $\mathcal{V}_i = \Spa(R_i,R_i^+)$. Similarly, one has the $\mathcal{V}_J=\Spa(R_J,R_J^+)\subset \mathcal{X}_{K^p}^\uast$ for $\emptyset\neq J\subset \{1,\ldots,N\}$. By Theorem \ref{PerfShHodge} (ii) (and the observation that this property is stable under passage to affinoid subsets, cf. Lemma \ref{ClosedImmersionPullback}), one can compute
\[
H^i(\mathcal{X}_{K^p}^\uast,\II^{+a}/p^n)
\]
by the Cech complex of almost $\OO_C/p^n$-modules with terms
\[
H^0(\mathcal{V}_J,\II^+/p^n)^a = \{f\in R_J^+/p^n\mid f=0\ \mathrm{on}\ \mathcal{Z}_{K^p}\}^a\ .
\]
As $\pi_\HT$ is equivariant for the Hecke operators away from $p$ (Theorem \ref{PerfShHodge} (iii), (iv)), all $\mathcal{V}_J$ are stable under the action of the Hecke operators away from $p$. Thus, $\mathbb{T}$ acts on each term of the Cech complex individually. We conclude that it is enough to prove that
\[
\mathbb{T}_\cl=\mathbb{T}\to \End_{\OO_C^a/p^n}(H^0(\mathcal{V}_J,\II^+/p^n)^a)
\]
is continuous, for all $\emptyset\neq J\subset \{1,\ldots,N\}$.

For all $K_p\subset G(\Q_p)$ sufficiently small, all $\mathcal{V}_i$ come via pullback from open affinoid subsets $\mathcal{V}_{iK_p}\subset \mathcal{X}_{K_pK^p}^\uast$. By Theorem \ref{PerfShHodge} (i) (tensored with a line bundle), the map
\[
\varinjlim_{K_p} H^0(\mathcal{V}_{iK_p},\omega_{K_pK^p})\to H^0(\mathcal{V}_i,\omega_{K^p})
\]
has dense image. Therefore, making $K_p$ smaller, one can assume that there are sections
\[
s_j^{(i)}\in H^0(\mathcal{V}_{iK_p},\omega_{K_pK^p})
\]
satisfying the conditions of Lemma \ref{ConstructionFormalModel}, and such that
\[
|\frac{s_j - s_j^{(i)}}{s_i^{(i)}}|\leq |p^n|
\]
on $\mathcal{V}_i$ for $j=1,\ldots,N$. One gets a formal model $\mathfrak{X}_{K_pK^p}^\uast$ of $\mathcal{X}_{K_pK^p}^\uast$ with an open cover by $\mathfrak{V}_{iK_p}$.\footnote{This formal model is extremely strange, and not at all related to one of the standard integral models of Shimura varieties. For example, the Newton stratification is not induced from a stratification of the special fibre of $\mathfrak{X}_{K_pK^p}^\uast$. More specifically, there is a finite set of points in the special fibre such that all ordinary points of the generic fibre specialize to one of those points; yet, there are also many non-ordinary points in the tubular neighborhood of these points.} In fact, we can also take the preimages $\tilde{\mathcal{V}}_{iK_p}\subset \mathcal{X}_{K_pK^p}^\ast$ of $\mathcal{V}_{iK_p}\subset \mathcal{X}_{K_pK^p}^\uast$; pulling back the sections $s_j^{(i)}$ will put us into the situation of Lemma \ref{ConstructionFormalModel}, thus constructing a formal model $\mathfrak{X}_{K_pK^p}^\ast$ of $\mathcal{X}_{K_pK^p}^\ast$. It comes equipped with an ample line bundle $\omega_{K_pK^p}^\int$, as well as the ideal sheaf $\mathfrak{I}\subset \OO_{\mathfrak{X}_{K_pK^p}^\ast}$, constructed via Lemma \ref{ConstructionFormalModelIdeal} from the ideal sheaf $\II\subset \OO_{\mathcal{X}_{K_pK^p}^\ast}$.

One checks directly that all of these objects are independent of the choice of the $s_j^{(i)}$ approximating $s_j$ on $\mathcal{V}_i$. In particular, $G(\A_f^p)$ still acts on the tower of the $\mathfrak{X}_{K_pK^p}^\ast$ with the invertible sheaf $\omega_{K_pK^p}^\int$. Also, the sections
\[
\bar{s}_j\in H^0(\mathfrak{X}_{K_pK^p}^\ast,\omega_{K_pK^p}^\int/p^n)
\]
are independent of any choice. They commute with the action of $G(\A_f^p)$, and will serve as a substitute for the Hasse invariant.

By Lemma \ref{CompAlgTraceMaps}, we have for $i\in \{1,\ldots,N\}$ (corresponding to a subset of $\{1,\ldots,2g\}$ of cardinality $g$) a $\mathbb{T}$-equivariant equality
\[
H^0(\mathcal{V}_i,\II^{+a}/p^n) = \varinjlim_{K_p} H^0(\tilde{\mathfrak{V}}_{iK_p},\mathfrak{I}/p^n)^a\ .
\]
In fact, the same holds true for any subset $\emptyset\neq J\subset \{1,\ldots,N\}$: Fix some $i\in J$, look at the previous equality, and invert the sections $\bar{s}_j/\bar{s}_i\in H^0(\tilde{\mathfrak{V}}_{iK_p},\OO)/p^n$ (which commute with the $\mathbb{T}$-action) for all $j\in J$. We see that it is enough to prove that for any $J\subset \{1,\ldots,N\}$ and sufficiently small $K_p$, the map
\[
\mathbb{T}_\cl = \mathbb{T}\to \End_{\OO_C/p^n} (H^0(\tilde{\mathfrak{V}}_{JK_p},\mathfrak{I}/p^n))
\]
is continuous.

For $i\in J$, we have the sections $\bar{s}_i\in H^0(\mathfrak{X}_{K_pK^p}^\ast,\omega_{K_pK^p}^\int/p^n)$; let $\bar{s}_J = \prod_{i\in J} \bar{s}_i$. As $\mathfrak{V}_{JK_p}\subset \mathfrak{X}_{K_pK^p}^\ast$ is the locus where $\bar{s}_i$ is invertible for all $i\in J$, it is also the locus where $\bar{s}_J$ is invertible. It follows that
\[
H^0(\mathfrak{V}_{JK_p},\mathfrak{I}/p^n) = \varinjlim_{\times \bar{s}_J} H^0(\mathfrak{X}_{K_pK^p}^\ast,(\omega_{K_pK^p}^\int)^{\otimes k|J|}\otimes \mathfrak{I}/p^n)\ ,
\]
where all maps are $\mathbb{T}$-equivariant, because $\bar{s}_J$ commutes with the action of $\mathbb{T}$. It remains to prove that for $k$ sufficiently divisible, the action of $\mathbb{T}_\cl = \mathbb{T}$ on
\[
H^0(\mathfrak{X}_{K_pK^p}^\ast,(\omega_{K_pK^p}^\int)^{\otimes k}\otimes \mathfrak{I}/p^n)
\]
is continuous. By Lemma \ref{ConstructionFormalModel}, $\omega_{K_pK^p}^\int$ is ample. Thus, for $k$ sufficiently divisible,
\[
H^1(\mathfrak{X}_{K_pK^p}^\ast,(\omega_{K_pK^p}^\int)^{\otimes k}\otimes \mathfrak{I}/p^n)=0\ ;
\]
it follows that for those $k$, we have
\[
H^0(\mathfrak{X}_{K_pK^p}^\ast,(\omega_{K_pK^p}^\int)^{\otimes k}\otimes \mathfrak{I}/p^n) = H^0(\mathfrak{X}_{K_pK^p}^\ast,(\omega_{K_pK^p}^\int)^{\otimes k}\otimes \mathfrak{I})/p^n\ .
\]
Thus, we are reduced to showing that the action of $\mathbb{T}_\cl = \mathbb{T}$ on
\[
H^0(\mathfrak{X}_{K_pK^p}^\ast,(\omega_{K_pK^p}^\int)^{\otimes k}\otimes \mathfrak{I})
\]
is continuous. But this group is $p$-torsion free, so it is enough to know that the action of $\mathbb{T}_\cl = \mathbb{T}$ on
\[
H^0(\mathfrak{X}_{K_pK^p}^\ast,(\omega_{K_pK^p}^\int)^{\otimes k}\otimes \mathfrak{I})[p^{-1}] = H^0(\mathcal{X}_{K_pK^p}^\ast,\omega_{K_pK^p}^{\otimes k}\otimes \II)
\]
is continuous, which holds true by assumption.
\end{proof}

\chapter{Galois representations}\label{GaloisReprChapter}

\section{Recollections}

We recall some results from the literature in the specific case that we will need. We specialize our group $G$ further. Fix a field $F$ which is either totally real or CM. In the totally real case, let $G=\Res_{F/\Q} \Sp_{2n}$. If $F$ is CM, let $F^+\subset F$ be the maximal totally real subfield, let $U/F^+$ (a form of $\GL_{2n}$) be the quasisplit unitary group with respect to the extension $F/F^+$, and take $G=\Res_{F^+/\Q} U$. In both cases, we take the standard conjugacy class $D$ of $u:U(1)\to G^\ad_\R$; observe that in all cases, $(G,D)$ is of Hodge type. Also, $G$ admits $\Res_{F/\Q} \GL_n$ as a maximal Levi. Write $F^+=F$ if $F$ is totally real, $G_0 = \Sp_{2n}/F^+$ if $F$ is totally real and $G_0=U/F^+$ if $F$ is CM; then, in all cases, $G = \Res_{F^+/\Q} G_0$. Also, $G$ is a twisted endoscopic group of $H=\Res_{F/\Q} H_0$, where $H_0 = \GL_h/F$, with $h=2n+1$ if $F$ is totally real, and $h=2n$ if $F$ is CM. Fix the standard embedding $\eta: {}^L G_0\hookrightarrow {}^L \Res_{F/F^+} H_0$ of $L$-groups (over $F^+$).

We need the existence of the associated endoscopic transfer, due to Arthur, \cite{ArthurBook} (resp. Mok, \cite{Mok}, in the unitary case). These results are still conditional on the stabilization of the twisted trace formula. However, in the unitary case, there are unconditional results of Shin, \cite{ShinGoldringApp}.

In fact, the representations we shall be interested in have a specific type at infinity.

\begin{prop}\label{PropInfType} Consider $\mathcal{G} = \Sp_{2n}/\R$, resp. $\mathcal{G} = U(n,n)/\R$, with maximal compact subgroup $K\subset \mathcal{G}(\R)$. Fix the standard Borel $\mathcal{B}\subset \mathcal{G}$, with torus $\mathcal{T}$, and identify $X^\ast(\mathcal{T}_\C) = \Z^n$, resp. $X^\ast(\mathcal{T}_\C) = \Z^{2n}$, in the usual way (up to the relevant Weyl groups). Let $\chi: K\to \C^\times$ be the character given by
\[
\chi: K\cong U(n)\to \C^\times: g_0\mapsto \det(g_0)\ ,\ \mathrm{resp.}
\]
\[
\chi: K\cong U(n)\times U(n)\to \C^\times: (g_1,g_2)\mapsto \det(g_1) \det(g_2)^{-1}\ .
\]
For $k>n$, resp. $k\geq n$, there is a unique discrete series representation $\pi_k$ of $\mathcal{G}$ with minimal $K$-type $\chi^{\otimes k}$, and it has infinitesimal character
\[
(k-1,k-2,\ldots,k-n)\in X^\ast(\mathcal{T}_\C)_\R = \R^n\ ,\ \mathrm{resp.}
\]
\[
(k-\tfrac 12,k-\tfrac 32,\ldots,k-n+\tfrac 12,n-\tfrac 12 - k, \ldots, \tfrac 32 - k,\tfrac 12 - k)\in X^\ast(\mathcal{T}_\C)_\R = \R^{2n}\ .
\]
\end{prop}

\begin{proof} Fix the standard maximal compact torus $\mathcal{T}_c\subset K$, and identify $X^\ast(\mathcal{T}_\C) = X^\ast(\mathcal{T}_{c,\C})$. Let $\delta_{nc}$ denote the half-sum of the noncompact roots, and $\delta_c$ the half-sum of the compact roots. Let $\Lambda_k\in X^\ast(\mathcal{T}_{c,\C})$ denote the restriction of $\chi^{\otimes k}$ to $\mathcal{T}_c$. Then
\[\begin{aligned}
2\delta_{nc} &= (n+1,n+1,\ldots,n+1), 2\delta_c = (n-1,n-3,\ldots,3-n,1-n),\\
\Lambda_k &= (k,k,\ldots,k)\in X^\ast(\mathcal{T}_{c,\C})\cong \Z^n\ ,\mathrm{resp.}
\end{aligned}\]
\[\begin{aligned}
2\delta_{nc} &= (n,n,\ldots,n,-n,\ldots,-n,-n), 2\delta_c = (n-1,\ldots,3-n,1-n,n-1,n-3,\ldots,1-n),\\
\Lambda_k &= (k,k,\ldots,k,-k,\ldots,-k,-k)\in X^\ast(\mathcal{T}_{c,\C})\cong \Z^{2n}\ .
\end{aligned}\]
Let $\lambda\in X^\ast(\mathcal{T}_{c,\C})\otimes_{\Z} \R$ denote a Harish-Chandra parameter of a discrete series representation. Then the associated minimal $K$-type has highest weight
\[
\Lambda = \lambda + \delta_{nc} - \delta_c\ .
\]
Thus, the minimal $K$-type $\Lambda_k$ determines $\lambda$ uniquely as
\[
\lambda = (k-1,k-2,\ldots,k-n)\in X^\ast(\mathcal{T}_\C)_\R = \R^n\ ,\ \mathrm{resp.}
\]
\[
\lambda = (k-\tfrac 12,k-\tfrac 32,\ldots,k-n+\tfrac 12,n-\tfrac 12 - k, \ldots, \tfrac 32 - k,\tfrac 12 - k)\in X^\ast(\mathcal{T}_\C)_\R = \R^{2n}\ .
\]
If $k>n$, resp. $k\geq n$, this parameter is dominant, and does not lie on any wall. Thus, in that case there is a discrete series representation with that parameter.
\end{proof}

\begin{thm}[\cite{ArthurBook}, \cite{Mok}]\label{EndoscopicTransfer} Let $\pi$ be a cuspidal automorphic representation of $G_0$, and fix an integer $k>n$, resp. $k\geq n$. We assume that for $v|\infty$, $\pi_v\cong \pi_k$ is the discrete series representation from Proposition \ref{PropInfType}.

There exist cuspidal automorphic representations $\Pi_1,\ldots,\Pi_m$ of $\GL_{n_1}/F,\ldots,\GL_{n_m}/F$, integers $\ell_1,\ldots,\ell_m\geq 1$, with  $n_1\ell_1+\ldots+n_m\ell_m = h$, such that
\[
(\Pi_1^\vee)^c\cong \Pi_1, \ldots , (\Pi_m^\vee)^c\cong \Pi_m\ ,
\]
where $c: F\to F$ denotes complex conjugation if $F$ is CM, and $c=\id: F\to F$ if $F$ is totally real, and such that the following condition is satified. For all finite places $v$ of $F$ lying over a place $v^+$ of $F^+$ for which $\pi_{v^+}$ is unramified, all $\Pi_{iv}$ are unramified, and
\[
\eta_{v^+\ast} \varphi_{\pi_{v^+}} = \bigoplus_{i=1}^m \left(\varphi_{\Pi_{iv}} |\cdot|^{(1-\ell_i)/2}\oplus \varphi_{\Pi_{iv}} |\cdot|^{(3-\ell_i)/2}\ldots\oplus \varphi_{\Pi_{iv}} |\cdot|^{(\ell_i-i)/2}\right)\ .
\]
Here, $\Pi_i$ is written as the restricted tensor product of $\Pi_{iv}$ over all places $v$ of $F$,
\[
\varphi_{\pi_{v^+}}: W_{F^+_{v^+}}\to {}^L G_{0,v^+}
\]
is the unramified $L$-parameter of $\pi_{v^+}$, and
\[
\eta_{v^+}: {}^L G_{0,v^+}\to {}^L \Res_{F_v/F^+_{v^+}} \GL_{h,v}
\]
is the $v^+$-component of $\eta$. Thus, $\eta_{v^+\ast} (\varphi_{\pi_{v^+}})$ is a map
\[
W_{F^+_{v^+}}\to {}^L \Res_{F_v/F^+_{v^+}} \GL_{h,v}\ ,
\]
or equivalently a map $W_{F_v}\to {}^L \GL_{h,v}$, i.e. an $h$-dimensional representation of $W_{F_v}$. On the right-hand side,
\[
\varphi_{\Pi_{iv}}: W_{F_v}\to {}^L \GL_{h,v}
\]
denotes the unramified $L$-parameter of $\Pi_{iv}$.

Moreover, for $v|\infty$, the representation $\Pi_{iv} |\cdot|^{(h-\ell_i)/2}$ is regular $L$-algebraic.\footnote{Here, we use the definition of $L$-algebraic automorphic representations from \cite{BuzzardGee}. Regularity means that the infinitesimal character is regular.}
\end{thm}

\begin{rem} As regards the last statement, one checks more precisely from Proposition \ref{PropInfType} and the compatibility of the global and local endoscopic transfer that for fixed $v|\infty$, the infinitesimal characters of the representations $\Pi_{iv}|\cdot|^{(2j-\ell_i-1)/2}$ for $i=1,...,m$, $j=1,\ldots,\ell_i$, (considered as (multi-)sets of real numbers) combine to
\[
(k-1,k-2,\ldots,k-n,0,n-k,n-1-k,\ldots,1-k)
\]
if $F$ is totally real, resp.
\[
(k-\tfrac 12,k-\tfrac 32,\ldots,k-n+\tfrac 12,n-\tfrac 12 - k,\ldots,\tfrac 32 - k,\tfrac 12 - k)
\]
if $F$ is CM.
\end{rem}

We combine this theorem with the existence of Galois representation (\cite{Clozel}, \cite{Kottwitz}, \cite{HarrisTaylor}, \cite{Shin}, \cite{ChenevierHarris}; the precise statement we need is stated as \cite[Theorems 1.1, 1.2]{BLGHT}).\footnote{Curiously, the Galois representations we need are the ones that are hardest to construct: They are regular, but non-Shin-regular, and not of finite slope at $p$.} Recall that we fixed $\iota: \C\cong \overline{\Q}_p$.

\begin{thm}\label{ExGalRepr} Let $\Pi$ be a cuspidal automorphic representation of $\Res_{F/\Q} \GL_n$ such that $(\Pi^\vee)^c\cong \Pi$ and $\Pi|\cdot|^{k/2}$ is regular $L$-algebraic for some integer $k$. Then there exists a continuous semisimple representation
\[
\sigma_\Pi: \Gal(\overline{F}/F)\to \GL_n(\overline{\Q}_p)
\]
such that $(\sigma_\Pi^\vee)^c\cong \sigma_\Pi \chi_p^k$, where $\chi_p$ is the $p$-adic cyclotomic character,\footnote{Note that $\chi_p^{-1}$ is the Galois representation attached to the absolute value $|\cdot|:\Q^\times\backslash \A_\Q^\times\to \R_{>0}$, as we normalize local class-field theory by matching up geometric Frobenius elements with uniformizers.} with the following property. For all finite places $v$ of $F$ for which $\Pi_v$ is unramified, $\sigma_\Pi$ is unramified at $v$, and
\[
\varphi_{\Pi_v} = \sigma_\Pi|_{W_{F_v}}|\cdot|^{k/2}
\]
up to semisimplification (i.e., Frobenius-semisimplification).
\end{thm}

\begin{rem} In the language of \cite{BuzzardGee}, this is the Galois representations attached to the $L$-algebraic cuspidal automorphic representation $\Pi|\cdot|^{k/2}$. To apply the cited result (which is in terms of $C$-algebraic representations), observe that $\Pi^\prime = \Pi|\cdot|^{(k+1-n)/2}$ is regular $C$-algebraic, and satisfies $(\Pi^{\prime\vee})^c \cong \Pi^\prime\otimes \chi$, where $\chi = |\cdot|^{n+1-k}$ is a character that comes via pullback from $\Q$.
\end{rem}

This discussion leads to the following corollary. Fix a sufficiently small level $K\subset G(\A_f) = G_0(\A_{F^+,f})$ of the form $K=K_{S^+}K^{S^+}$ for a finite set $S^+$ of finite places of $F^+$ containing all places dividing $p$, and all places over which $F$ ramifies. Here $K_{S^+}\subset G_0(\A_{F^+,S^+})$ and $K^{S^+}\subset G_0(\A_{F^+,f}^{S^+})$ are compact open, and $K^{S^+}$ is a product of hyperspecial compact open subgroups $K_v\subset G_0(F^+_v)$ at all finite places $v\not\in S^+$. Let
\[
\mathbb{T} = \mathbb{T}_{K^{S^+}} = \bigotimes_{v\not\in S^+} \mathbb{T}_v
\]
be the abstract Hecke algebra, where $v$ runs through finite places outside $S^+$, and
\[
\mathbb{T}_v = \Z_p[G_0(F^+_v)//K_v]
\]
is the spherical Hecke algebra. Before going on, let us recall the description of these Hecke algebras, and define some elements in these algebras.

\begin{lem}\label{DescrHeckeAlg1} Fix a place $v\not\in S^+$. Let $q_v$ be the cardinality of the residue field of $F^+$ at $v$. Let $q_v^{1/2}\in \overline{\Z}_p$ denote the image of the positive square root in $\mathbb{C}$ under the chosen isomorphism $\mathbb{C}\cong \overline{\Q}_p$.
\begin{altenumerate}
\item[{\rm (i)}] Assume $F$ is totally real. Then the Satake transform gives a canonical isomorphism
\[
\mathbb{T}_v[q_v^{1/2}]\cong \Z_p[q_v^{1/2}][X_1^{\pm 1},\ldots,X_n^{\pm 1}]^{S_n\ltimes (\Z/2\Z)^n}\ .
\]
The unramified endoscopic transfer from $G_0(F^+_v)$ to $\GL_{2n+1}(F^+_v)$ is dual to the map
\[
\Z_p[q_v^{1/2}][Y_1^{\pm 1},\ldots,Y_{2n+1}^{\pm 1}]^{S_{2n+1}}\to \Z_p[q_v^{1/2}][X_1^{\pm 1},\ldots,X_n^{\pm 1}]^{S_n\ltimes (\Z/2\Z)^n}
\]
sending the set $\{Y_1,\ldots,Y_{2n+1}\}$ to $\{X_1^{\pm 1},\ldots,X_n^{\pm 1},1\}$. Let
\[
T_i\in \Z_p[Y_1^{\pm 1},\ldots,Y_{2n+1}^{\pm 1}]^{S_{2n+1}}
\]
be the $i$-th elementary symmetric polynomial in the $Y_j$'s for $i=1,\ldots,2n+1$, and let $T_{i,v}\in \mathbb{T}_v[q_v^{1/2}]$ be its image in $\mathbb{T}_v[q_v^{1/2}]$. Then $T_{i,v}\in \mathbb{T}_v$.
\item[{\rm (ii)}] Assume $F$ is CM, and $v$ splits in $F$; fix a lift $\tilde{v}$ of $v$. Then $G_0(F^+_v)\cong H_0(F_{\tilde{v}})\cong \GL_{2n}(F_{\tilde{v}})$, and
\[
\mathbb{T}_v[q_v^{1/2}]\cong \Z_p[q_v^{1/2}][X_1^{\pm 1},\ldots,X_{2n}^{\pm 1}]^{S_{2n}}\ .
\]
The unramified endoscopic transfer is the identity map. Let $T_{i,\tilde{v}}\in \mathbb{T}_v[q_v^{1/2}]$ be the $i$-th elementary symmetric polynomial in $X_1,\ldots,X_{2n}$ for $i=1,\ldots,2n$. Then $q_v^{i/2} T_{i,\tilde{v}}\in \mathbb{T}_v$.\footnote{These elements depend on the choice of $\tilde{v}$; the other choice replaces all $X_i$ by $X_i^{-1}$.}
\item[{\rm (iii)}] Assume $F$ is CM, and $v$ inert (in particular, unramified) in $F$. Then the Satake transform gives a canonical isomorphism
\[
\mathbb{T}_v[q_v^{1/2}]\cong \Z_p[q_v^{1/2}][X_1^{\pm 1},\ldots,X_n^{\pm 1}]^{S_n\ltimes (\Z/2\Z)^n}\ .
\]
The unramified endoscopic transfer from $G_0(F^+_v)$ to $\GL_{2n}(F_v)$ is dual to the map
\[
\Z_p[q_v^{1/2}][Y_1^{\pm 1},\ldots,Y_{2n}^{\pm 1}]^{S_{2n}}\to \Z_p[q_v^{1/2}][X_1^{\pm 1},\ldots,X_n^{\pm 1}]^{S_n\ltimes (\Z/2\Z)^n}
\]
sending the set $\{Y_1,\ldots,Y_{2n}\}$ to $\{X_1^{\pm 1},\ldots,X_n^{\pm 1}\}$. Again, we let $T_{i,v}\in \mathbb{T}_v[q_v^{1/2}]$ denote the image of the $i$-th elementary symmetric polynomial in the $Y_j$'s for $i=1,\ldots,2n$. Then $T_{i,v}\in \mathbb{T}_v$.
\end{altenumerate}
\end{lem}

\begin{proof} Everything is standard. Note that all occuring groups are unramified over $F^+_v$, thus one can extend them to reductive group schemes over the ring of integers. As $v$ does not divide $p$, one can then define a unique Haar measure with values in $\Z_p$ on the unipotent radical of the Borel which gives the integral points measure $1$. The normalized Satake transform is defined over $\Z_p[q_v^{1/2}]$. The final rationality statements are easily verified.
\end{proof}

Now $\mathbb{T}$ acts on the $C$-vector spaces of cusp forms
\[
H^0(X_K^\ast,\omega_K^{\otimes k}\otimes \II)\otimes_{\C} C = H^0(\mathcal{X}_K^\ast,\omega_K^{\otimes k}\otimes \II)\ ,
\]
where $\omega$ is the automorphic line bundle coming via pullback from the standard automorphic ample line bundle on the Siegel moduli space, and $\II$ denotes the ideal sheaf of the boundary (on either space).

Let $S$ be the finite set of finite places $v$ of $F$ which map to a place $v^+\in S^+$ of $F^+$, and let $G_{F,S}$ be the Galois group of the maximal extension of $F$ which is unramified outside $S$.

\begin{cor}\label{CorExGalRepr0} Fix $k>n$, resp. $k\geq n$ (if $F$ is totally real, resp. CM). Let $\delta=0$ if $F$ is totally real, and $\delta=1$ if $F$ is CM. Let $\mathbb{T}_{K,k}$ be the image of $\mathbb{T}$ in
\[
\End_C(H^0(X_K^\ast,\omega_K^{\otimes k}\otimes \II)\otimes_{\C} C)\ .
\]
Then $\mathbb{T}_{K,k}$ is flat over $\Z_p$, and $\mathbb{T}_{K,k}[p^{-1}]$ is finite \'etale over $\Q_p$. For any $x\in (\Spec \mathbb{T}_{K,k})(\overline{\Q}_p)$, there is a continuous semisimple representation
\[
\sigma_x: G_{F,S}\to \GL_h(\overline{\Q}_p)
\]
such that $(\sigma_x^\vee)^c\cong \sigma_x\chi_p^\delta$, and such that for any finite place $v\not\in S$ of $F$, the characteristic polynomial\footnote{We adopt a nonstandard convention on characteristic polynomials.} of the geometric Frobenius $\Frob_v\in G_{F,S}$ is given by
\[
\det(1-X\Frob_v|\sigma_x) = 1 - (q_v^{\delta/2}T_{1,v})(x) X + (q_v^{2\delta/2}T_{2,v})(x) X^2 - \ldots + (-1)^h (q_v^{h \delta/2}T_{h,v})(x) X^h\ ,
\]
where $q_v$ is the cardinality of the residue field of $F$ at $v$.
\end{cor}

\begin{proof} Flatness of $\mathbb{T}_{K,k}$ is clear. The image of $\mathbb{T}\otimes_{\Z_p} \C$ in
\[
\End_\C(H^0(X_K^\ast,\omega_K^{\otimes k}\otimes \II))
\]
is a product of copies of $\C$, because the Petersson inner product defines a positive-definite hermitian form on $H^0(X_K^\ast,\omega_K^{\otimes k}\otimes \II)$ for which the adjoint of a Hecke operator is another Hecke operator. This implies by descent that $\mathbb{T}_{K,k}[p^{-1}]$ is finite \'etale over $\Q_p$.

Now, given $x$, there exists a cuspidal automorphic representation $\pi$ of $G_0/F^+$ such that for all finite places $v\not\in S^+$, $\pi_v$ is unramified, with Satake parameter the map
\[
\mathbb{T}_v\otimes_{\Z_p} \C\to \C
\]
induced by $x$ (and the fixed isomorphism $\C\cong \overline{\Q}_p$), and such that for $v|\infty$, $\pi_v$ is a discrete series representation with given lowest weight as described in Proposition \ref{PropInfType}, i.e. $\pi_v\cong \pi_k$. Thus, by Theorem \ref{EndoscopicTransfer}, we get cuspidal automorphic representations $\Pi_1,\ldots,\Pi_m$ of $\GL_{n_1}/F,\ldots,\GL_{n_m}/F$, and integers $\ell_1,\ldots,\ell_m$, with the properties stated there. By Theorem \ref{ExGalRepr}, there exist Galois representations $\sigma_i$ attached to the regular $L$-algebraic cuspidal automorphic representations $\Pi_i|\cdot|^{(\delta+1-\ell_i)/2}$. We set
\[
\sigma_x = \bigoplus_{i=1}^m \left(\sigma_i\oplus \sigma_i \chi_p^{-1}\oplus\ldots \sigma_i\chi_p^{\otimes (1-\ell_i)}\right)\ ,
\]
where $\chi_p$ is the $p$-adic cyclotomic character. The desired statement is a direct consequence.
\end{proof}

Recall Chenevier's notion of a determinant, \cite{ChenevierDet}, which we use as a strengthening of the notion of pseudorepresentations as introduced by Taylor, \cite{TaylorPseudoRepr}. Roughly, the difference is that a pseudorepresentation is `something that looks like the trace of a representation', whereas a determinant is `something that looks like the characteristic polynomials of a representation'.

\begin{definition} Let $A$ be a (topological) ring, and $G$ a (topological) group. A $d$-dimensional determinant is an $A$-polynomial law $D: A[G]\to A$ which is multiplicative and homogeneous of degree $d$. For any $g\in G$, we call $D(1-Xg)\in A[X]$ the characteristic polynomial of $g$. Moreover, $D$ is said to be continuous if the map $G\to A[X]$, $g\mapsto D(1-Xg)$, is continuous.\footnote{Cf. \cite[2.30]{ChenevierDet}.}
\end{definition}

\begin{rem} We continue to use our nonstandard definition of the characteristic polynomial. Also, slightly more generally, for any $A$-algebra $B$, we call a multiplicative $A$-polynomial law $A[G]\to B$ homogeneous of degree $d$ a determinant of dimension $d$ (with values in $B$). In fact, this is equivalent to a multiplicative $B$-polynomial law $B[G]\to B$ homogeneous of degree $d$, i.e. a determinant over $B$.
\end{rem}

Recall that an $A$-polynomial law between two $A$-modules $M$ and $N$ is simply a natural transformation $M\otimes_A B\to N\otimes_A B$ on the category of $A$-algebras $B$. Multiplicativity means that $D$ commutes with the multiplication morphisms, and homogeneity of degree $d$ means that $D(bx) = b^d D(x)$ for all $b\in B$, $x\in B[G]$. Equivalently, by multiplicativity of $D$, $D(b) = b^d$ for all $b\in B$. Note that if $\rho: G\to \GL_d(A)$ is a (continuous) representation, then $D = \det\circ \rho: A[G]\to M_d(A)\to A$ defines a (continuous) $d$-dimensional determinant. The two notions of characteristic polynomials obviously agree. Also recall that (by Amitsur's formula) the collection of characteristic polynomials determines the determinant, cf. \cite[Lemma 1.12 (ii)]{ChenevierDet}.

Now we go back to Galois representations. Keep the notation from Corollary \ref{CorExGalRepr0}.

\begin{cor}\label{CorExGalRepr} There is a unique continuous $h$-dimensional determinant $D$ of $G_{F,S}$ with values in $\mathbb{T}_{K,k}$, such that
\[
D(1-X\Frob_v) = 1 - q_v^{\delta/2}T_{1,v} X + q_v^{2\delta/2}T_{2,v} X^2 - \ldots + (-1)^h q_v^{h \delta/2}T_{h,v} X^h\ .
\]
for all finite places $v\not\in S$ of $F$.
\end{cor}

\begin{proof} This follows from \cite[Example 2.32]{ChenevierDet}.
\end{proof}

\begin{cor}\label{CorExGalRepr2} Let $\mathbb{T}_\cl$ be as defined in Theorem \ref{ThmHeckeAlgebras}. Then, for any continuous quotient $\mathbb{T}_\cl\to A$ with $A$ discrete, there is a unique continuous $h$-dimensional determinant $D$ of $G_{F,S}$ with values in $A$, such that
\[
D(1-X\Frob_v) = 1 - q_v^{\delta/2}T_{1,v} X + q_v^{2\delta/2}T_{2,v} X^2 - \ldots + (-1)^h q_v^{h \delta/2}T_{h,v} X^h\ .
\]
for all finite places $v\not\in S$ of $F$.
\end{cor}

\begin{proof} If $I_1, I_2\subset \mathbb{T}$ are two ideals such that there exist determinants with values in $\mathbb{T}/I_1$ and $\mathbb{T}/I_2$, then there exists a determinant with values in $\mathbb{T}/(I_1\cap I_2)$, by \cite[Example 2.32]{ChenevierDet}. As $I=\ker(\mathbb{T}_\cl\to A)$ is open, it contains a finite intersection of ideals $I_{K,k} = \ker(\mathbb{T}_\cl\to \mathbb{T}_{K,k})$, by the definition of the topology on $\mathbb{T}_\cl$. The result follows.
\end{proof}

\section{The cohomology of the boundary}

Our primary interest in the specific groups $G$ is that they contain $M=\Res_{F/\Q} \GL_n$ as a maximal Levi. This implies that the cohomology of the locally symmetric spaces associated with $M$ contributes to the cohomology of $X_K$. In this section, we recall the relevant results. Fix a parabolic $P\subset G$ with Levi $M$.

\begin{definition} For a compact open subgroup $K_M\subset M(\A_f)$, let
\[
X^M_{K_M} = M(\Q)\backslash [(M(\R) / \R_{>0} K_\infty)\times M(\A_f)/K_M]
\]
denote the locally symmetric space associated with $M$. Here, $K_\infty\subset M(\R)$ is a maximal compact subgroup, and $\R_{>0}\subset M(\R)$ are the scalar matrices with positive entries. Similarly, for a compact open subgroup $K_P\subset P(\A_f)$, let
\[
X^P_{K_P} = P(\Q)\backslash [(P(\R) / \R_{>0} K_\infty)\times P(\A_f)/K_P]\ .
\]
\end{definition}

\begin{lem}\label{ParabolicVsLevi} For a compact open subgroup $K_P\subset P(\A_f)$, the image $K_P^M\subset M(\A_f)$ of $K_P$ in $M(\A_f)$ is compact and open. There is a natural projection
\[
X^P_{X_P}\to X^M_{X_P^M}
\]
which is a bundle with fibres $(S^1)^k$, where $k$ is the dimension of the unipotent radical of $P$.
\end{lem}

We are using the specific nature of $U$ here: One sees by inspection that $U$ is commutative, which makes the map an actual torus bundle.

\begin{proof} The projection $P(\A_f)\to M(\A_f)$ is open, so that $K_P^M$ is open, and certainly compact. Let $U=\ker(P\to M)$ be the unipotent radical of $P$, $K_P^U = K_P\cap U$. Then the fibre of
\[
(P(\R) / \R_{>0} K_\infty)\times P(\A_f)/K_P\to (M(\R) / \R_{>0} K_\infty)\times M(\A_f)/K_P^M
\]
is given by $U(\R)\times U(\A_f)/K_P^U$. One deduces that the fibres of
\[
X^P_{K_P} = P(\Q)\backslash [(P(\R) / \R_{>0} K_\infty)\times P(\A_f)/K_P]\to X^M_{K_P^M} = M(\Q)\backslash [(M(\R) / \R_{>0} K_\infty)\times M(\A_f)/K_P^M]
\]
are given by
\[
U(\Q)\backslash (U(\R)\times U(\A_f)/K_P^U)\cong (U(\Q)\cap K_P^U)\backslash U(\R)\ .
\]
The subgroup $U(\Q)\cap K_P^U\subset U(\R)$ is a lattice, thus the quotient is isomorphic to $(S^1)^k$, where $k=\dim U$.
\end{proof}

Let $X_K^\BS$ be the Borel-Serre compactification of $X_K$, cf. \cite{BorelSerre}. Recall that we assume that $K$ is sufficiently small. Then $X_K^\BS$ is a compactification of $X_K$ as a real manifold with corners, and the inclusion $X_K\hookrightarrow X_K^\BS$ is a homotopy-equivalence. Thus, there is a long exact sequence
\[
\ldots\to H_c^i(X_K,\Z/p^m\Z)\to H^i(X_K,\Z/p^m\Z)\to H^i(X_K^\BS\setminus X_K,\Z/p^m\Z)\to \ldots\ .
\]
Moreover, if one looks at the compact open subgroup $K^P = K\cap P(\A_f)\subset P(\A_f)$, one has an open embedding
\[
X^P_{K^P}\hookrightarrow X_K^\BS\setminus X_K\ .
\]
In particular, there are natural maps
\[
H_c^i(X^P_{K^P},\Z/p^m\Z)\to H^i(X_K^\BS\setminus X_K,\Z/p^m\Z)\to H^i(X^P_{K^P},\Z/p^m\Z)\ .
\]

Recall that we have fixed a finite set of places $S^+$ of $F^+$ containing all places above $p$ and all places above which $F/F^+$ is ramified, and that $K=K_{S^+}K^{S^+}$, where $K^{S^+}\subset G_0(\A_{F^+,f}^{S^+})$ is a product of hyperspecial maximal compact subgroups. Then similarly $K^P = K^P_{S^+} K^{P,S^+}$, where $K^{P,S^+}\subset P_0(\A_{F^+,f}^{S^+})$; here, $P_0\subset G_0$ is the parabolic subgroup with $P=\Res_{F^+/\Q} P_0$. Its Levi $M_0$ with $M=\Res_{F^+/\Q} M_0$ is given by $M_0 = \Res_{F/F^+} \GL_n$. Let $K^M\subset M(\A_f)$ be the image of $K^P$; then $K^M = K^M_{S^+} K^{M,S^+}$, where $K^{M,S^+}\subset M_0(\A_{F^+,f}^S)$ is a product of hyperspecial maximal compact subgroups.

In the following, we assume that the image $K^M$ of $K^P$ in $M(\A_f)$ agrees with $K^P\cap M(\A_f)$. Given any sufficiently small compact open subgroup $K^M_0$ of $M(\A_f)$, one can easily build a compact open subgroup $K\subset G(\A_f)$ with this property and realizing $K^M=K^M_0$, e.g. by multiplying $K^M_0$ with small compact open subgroups of $U(\A_f)$ and the opposite unipotent radical.

Consider the (unramified) Hecke algebras
\[\begin{aligned}
\mathbb{T}_{K^{S^+}} &= \Z_p[G_0(\A_{F^+,f}^{S^+})//K^{S^+}]\ ,\\
\mathbb{T}_{K^{P,S^+}} &= \Z_p[P_0(\A_{F^+,f}^{S^+})//K^{P,S^+}]\ ,\\
\mathbb{T}_{K^{M,S^+}} &= \Z_p[M_0(\A_{F^+,f}^{S^+})//K^{M,S^+}]\ .
\end{aligned}\]
Then restriction of functions defines a map $\mathbb{T}_{K^{S^+}}\to \mathbb{T}_{K^{P,S^+}}$, and integration along unipotent fibres defines a map $\mathbb{T}_{K^{P,S^+}}\to \mathbb{T}_{K^{M,S^+}}$. The composite is the (unnormalized) Satake transform
\[
\mathbb{T}_{K^{S^+}}\to \mathbb{T}_{K^{M,S^+}}\ .
\]

Recall also that we assumed that $K$ is sufficiently small. In particular, all congruence subgroups are torsion-free, and the quotients defining the Borel-Serre compactification are by discontinuous free group actions. It follows that $\mathbb{T}_{K^{S^+}}$ acts on $H^i(X_K^\BS\setminus X_K,\Z/p^m\Z)$, giving a map of $\Z_p$-algebras
\[
\mathbb{T}_{K^{S^+}}\to \End_{\Z/p^m\Z}(H^i(X_K^\BS\setminus X_K,\Z/p^m\Z))\ .
\]
Also, $\mathbb{T}_{K^{P,S^+}}$ acts on both $H_c^i(X^P_{K^P},\Z/p^m\Z)$ and $H^i(X^P_{K^P},\Z/p^m\Z)$. By letting it act on one of them, one gets a map of $\Z_p$-modules
\[
\mathbb{T}_{K^{P,S^+}}\to \Hom_{\Z/p^m\Z}(H_c^i(X^P_{K^P},\Z/p^m\Z),H^i(X^P_{K^P},\Z/p^m\Z))\ ;
\]
the map does not depend on whether one lets $\mathbb{T}_{K^{P,S^+}}$ acts on $H_c^i$ or $H^i$. Similarly, one has a map of $\Z_p$-modules
\[
\mathbb{T}_{K^{M,S^+}}\to \Hom_{\Z/p^m\Z}(H_c^i(X^M_{K^M},\Z/p^m\Z),H^i(X^M_{K^M},\Z/p^m\Z))\ .
\]
In this last case, define the interior cohomology
\[
H^i_!(X^M_{K^M},\Z/p^m\Z) = \im ( H_c^i(X^M_{K^M},\Z/p^m\Z)\to H^i(X^M_{K^M},\Z/p^m\Z) )\ .
\]
Then one has a map of $\Z_p$-algebras
\[
\mathbb{T}_{K^{M,S^+}}\to \End_{\Z/p^m\Z}(H^i_!(X^M_{K^M},\Z/p^m\Z))\ .
\]
The kernel of $\mathbb{T}_{K^{M,S^+}}\to \Hom_{\Z/p^m\Z}(H_c^i(X^M_{K^M},\Z/p^m\Z),H^i(X^M_{K^M},\Z/p^m\Z))$ agrees with the kernel of $\mathbb{T}_{K^{M,S^+}}\to \End_{\Z/p^m\Z}(H^i_!(X^M_{K^M},\Z/p^m\Z))$; in particular, it is an ideal.

Finally, observe that there is a commutative diagram
\[\xymatrix{
& H^i(X_K^\BS\setminus X_K,\Z/p^m\Z)\ar[dr] & \\
H_c^i(X^P_{K^P},\Z/p^m\Z)\ar[ur]\ar[rr] & & H^i(X^P_{K^P},\Z/p^m\Z)\ar[d] \\
H_c^i(X^M_{K^M},\Z/p^m\Z)\ar[u]\ar[rr] & & H^i(X^M_{K^M},\Z/p^m\Z)
}\]
of $\Z_p$-modules. The only non-tautological map is the map
\[
H^i(X^P_{K^P},\Z/p^m\Z)\to H^i(X^M_{K^M},\Z/p^m\Z)\ .
\]
This map is pullback along the embedding $X^M_{K^M}\to K^P_{K^P}$ (cf. definition of both spaces), using that $K^M = K^P\cap M(\A_f)$; this forms a section of the projection $X^P_{K^P}\to X^M_{K^M}$, implying commutativity of the diagram.

In particular, one gets natural maps of $\Z_p$-modules
\[\begin{aligned}
\End_{\Z/p^m\Z}(H^i(X_K^\BS\setminus X_K,\Z/p^m\Z))&\to \Hom_{\Z/p^m\Z}(H_c^i(X^P_{K^P},\Z/p^m\Z),H^i(X^P_{K^P},\Z/p^m\Z))\\
&\to \Hom_{\Z/p^m\Z}(H_c^i(X^M_{K^M},\Z/p^m\Z),H^i(X^M_{K^M},\Z/p^m\Z))\ .
\end{aligned}\]
The following lemma is an easy verification from the definitions.

\begin{lem} The diagram of $\Z_p$-modules
\[\xymatrix{
\mathbb{T}_{K^{S^+}}\ar[d]\ar[r] & \End_{\Z/p^m\Z}(H^i(X_K^\BS\setminus X_K,\Z/p^m\Z))\ar[d]\\
\mathbb{T}_{K^{P,S^+}}\ar[d]\ar[r] & \Hom_{\Z/p^m\Z}(H_c^i(X^P_{K^P},\Z/p^m\Z),H^i(X^P_{K^P},\Z/p^m\Z))\ar[d]\\
\mathbb{T}_{K^{M,S^+}}\ar[r] & \Hom_{\Z/p^m\Z}(H_c^i(X^M_{K^M},\Z/p^m\Z),H^i(X^M_{K^M},\Z/p^m\Z))
}\]
commutes.$\hfill \Box$
\end{lem}

\begin{cor}\label{CompHeckeAlgParabInd} Let $\overline{\mathbb{T}}_{K^{S^+}}$ be the image of $\mathbb{T}_{K^{S^+}}$ in $\End_{\Z/p^m\Z}(H^i(X_K^\BS\setminus X_K,\Z/p^m\Z))$, and let $\overline{\mathbb{T}}_{K^{M,S^+}}$ be the image of $\mathbb{T}_{K^{M,S^+}}$ in $\End_{\Z/p^m\Z}(H^i_!(X^M_{K^M},\Z/p^m\Z))$. Then there is a commutative diagram
\[\xymatrix{
\mathbb{T}_{K^{S^+}}\ar[r]\ar[d] & \overline{\mathbb{T}}_{K^{S^+}}\ar[d]\\
\mathbb{T}_{K^{M,S^+}}\ar[r] & \overline{\mathbb{T}}_{K^{M,S^+}}
}\]
of $\Z_p$-algebras.
\end{cor}

\begin{proof} We need to check that the kernel of $\mathbb{T}_{K^{S^+}}\to \overline{\mathbb{T}}_{K^{S^+}}$ maps trivially into $\overline{\mathbb{T}}_{K^{M,S^+}}$ via $\mathbb{T}_{K^{M,S^+}}$. This follows from the previous lemma, recalling that the kernels of
$\mathbb{T}_{K^{M,S^+}}\to \Hom_{\Z/p^m\Z}(H_c^i(X^M_{K^M},\Z/p^m\Z),H^i(X^M_{K^M},\Z/p^m\Z))$ and $\mathbb{T}_{K^{M,S^+}}\to \End_{\Z/p^m\Z}(H^i_!(X^M_{K^M},\Z/p^m\Z))$ agree.
\end{proof}

In order to deduce the correct corollaries, we make the maps on Hecke algebras explicit again. For a place $v^+\not\in S^+$ of $F^+$, let
\[
\mathbb{T}_{K^M_{v^+}} = \Z_p[M_0(F^+_{v^+})//K^M_{v^+}]\ ,
\]
where $K^{M,S^+} = \prod K^M_{v^+}$, $K^M_{v^+}\subset M_0(F^+_{v^+})$ a hyperspecial maximal compact subgroup.

\begin{lem}\label{DescrHeckeAlg2} Fix a place $v\not\in S$ of $F$ lying above a place $v^+\not\in S^+$ of $F^+$. Let $q_{v^+}$ be the cardinality of the residue field of $F^+$ at $v^+$, and let $q_v$ be the cardinality of the residue field of $F$ at $v$. Let $q_{v^+}^{1/2},q_v^{1/2}\in \overline{\Z}_p$ denote the image of the positive square root in $\mathbb{C}$ under the chosen isomorphism $\mathbb{C}\cong \overline{\Q}_p$.
\begin{altenumerate}
\item[{\rm (i)}] Assume $F$ is totally real (so that $v=v^+$). Then the unnormalized Satake transform is the map
\[
\mathbb{T}_v[q_v^{1/2}]\cong \Z_p[q_v^{1/2}][X_1^{\pm 1},\ldots,X_n^{\pm 1}]^{S_n\ltimes (\Z/2\Z)^n}\to \mathbb{T}_{K^M_v}[q_v^{1/2}]\cong \Z_p[q_v^{1/2}][X_1^{\pm 1},\ldots,X_n^{\pm 1}]^{S_n}
\]
sending the set $\{X_1^{\pm 1},\ldots,X_n^{\pm 1}\}$ to $\{(q_v^{(n+1)/2} X_1)^{\pm 1},\ldots,(q_v^{(n+1)/2} X_n)^{\pm 1}\}$. Recall the elements $T_{i,v}\in \mathbb{T}_v$ from Lemma \ref{DescrHeckeAlg1}. Let $T_{i,v}^M\in \mathbb{T}_{K^M_v}[q_v^{1/2}]$ be the $i$-th elementary symmetric polynomial in $X_1,\ldots,X_n$. Then $q_v^{i(n+1)/2} T_{i,v}^M\in \mathbb{T}_{K^M_v}$ and
\[\begin{aligned}
&1 - T_{1,v} X + T_{2,v} X^2 - \ldots - T_{2n+1,v} X^{2n+1}\\
\mapsto &(1 - X) (1 - q_v^{(n+1)/2}T_{1,v}^M X + q_v^{2 (n+1)/2}T_{2,v}^M X^2 - \ldots + (-1)^n q_v^{n(n+1)/2} T_{n,v}^M X^n)\\
\times&(1 - q_v^{-(n+1)/2}\frac{T_{n-1,v}^M}{T_{n,v}^M} X + q_v^{-2(n+1)/2}\frac{T_{n-2,v}^M}{T_{n,v}^M} X^2 - \ldots + (-1)^n q_v^{-n(n+1)/2}\frac{1}{T_{n,v}^M} X^n)\ .
\end{aligned}\]
\item[{\rm (ii)}] Assume $F$ is CM, and $v^+$ splits in $F$, with $\bar{v}$ the complex conjugate place of $F$; then $q_{v^+} = q_v$. The unnormalized Satake transform is the map
\[
\mathbb{T}_{v^+}[q_v^{1/2}]\cong \Z_p[q_v^{1/2}][X_1^{\pm 1},\ldots,X_{2n}^{\pm 1}]^{S_{2n}}\to \mathbb{T}_{K^M_{v^+}}[q_v^{1/2}]\cong \Z_p[q_v^{1/2}][X_1^{\pm 1},\ldots,X_n^{\pm 1},Y_1^{\pm 1},\ldots,Y_n^{\pm 1}]^{S_n\times S_n}
\]
sending $\{X_1,\ldots,X_{2n}\}$ to $\{q_v^{n/2}X_1,\ldots,q_v^{n/2} X_n,q_v^{-n/2} Y_1,\ldots,q_v^{-n/2} Y_n\}$. Recall the elements $q_v^{i/2} T_{i,v}\in \mathbb{T}_{v^+}$ from Lemma \ref{DescrHeckeAlg1}. Let $T_{i,v}^M\in \mathbb{T}_{K^M_{v^+}}[q_v^{1/2}]$ be the $i$-th elementary symmetric polynomial in $X_1,\ldots,X_n$. Then $q_v^{i(n+1)/2} T_{i,v}^M\in \mathbb{T}_{K^M_{v^+}}$. Moreover, $T_{i,\bar{v}}^M\in \mathbb{T}_{K^M_{v^+}}[q_v^{1/2}]$ is the $i$-th elementary symmetric polynomial in $Y_1^{-1},\ldots,Y_n^{-1}$, and
\[\begin{aligned}
&1 - q_v^{1/2} T_{1,v} X + q_v T_{2,v} X^2 - \ldots - q_v^n T_{2n,v} X^{2n}\\
\mapsto &(1 - q_v^{(n+1)/2}T_{1,v}^M X + q_v^{2 (n+1)/2}T_{2,v}^M X^2 - \ldots + (-1)^n q_v^{n(n+1)/2} T_{n,v}^M X^n)\\
\times&(1 - q_v^{-(n-1)/2}\frac{T_{n-1,\bar{v}}^M}{T_{n,\bar{v}}^M} X + q_v^{-2(n-1)/2}\frac{T_{n-2,\bar{v}}^M}{T_{n,\bar{v}}^M} X^2 - \ldots + (-1)^n q_v^{-n(n-1)/2}\frac{1}{T_{n,\bar{v}}^M} X^n)\ .
\end{aligned}\]
\item[{\rm (iii)}] Assume $F$ is CM, and $v$ inert (in particular, unramified) in $F$; then $q_v = q_{v^+}^2$, and $q_v^{1/2} = q_{v^+}\in \Z_p$. The unnormalized Satake transform is the map
\[
\mathbb{T}_{v^+}[q_{v^+}^{1/2}]\cong \Z_p[q_{v^+}^{1/2}][X_1^{\pm 1},\ldots,X_n^{\pm 1}]^{S_n\ltimes (\Z/2\Z)^n}\to \mathbb{T}_{K^M_{v^+}}[q_{v^+}^{1/2}]\cong \Z_p[q_{v^+}^{1/2}][X_1^{\pm 1},\ldots,X_n^{\pm 1}]^{S_n}
\]
sending $\{X_1^{\pm 1},\ldots,X_n^{\pm 1}\}$ to $\{(q_v^{n/2} X_1)^{\pm 1},\ldots,(q_v^{n/2} X_n)^{\pm 1}\}$. Recall the elements $T_{i,v}\in \mathbb{T}_v$ from Lemma \ref{DescrHeckeAlg1}. Let $T_{i,v}^M\in \mathbb{T}_{K^M_{v^+}}[q_{v^+}^{1/2}]$ be the $i$-th elementary symmetric polynomial in $X_1,\ldots,X_n$. Then $T_{i,v}^M\in \mathbb{T}_{K^M_{v^+}}$, and
\[\begin{aligned}
&1 - q_v^{1/2} T_{1,v} X + q_v T_{2,v} X^2 - \ldots - q_v^n T_{2n,v} X^{2n}\\
\mapsto &(1 - q_v^{(n+1)/2}T_{1,v}^M X + q_v^{2 (n+1)/2}T_{2,v}^M X^2 - \ldots + (-1)^n q_v^{n(n+1)/2} T_{n,v}^M X^n)\\
\times&(1 - q_v^{-(n-1)/2}\frac{T_{n-1,v}^M}{T_{n,v}^M} X + q_v^{-2(n-1)/2}\frac{T_{n-2,v}^M}{T_{n,v}^M} X^2 - \ldots + (-1)^n q_v^{-n(n-1)/2}\frac{1}{T_{n,v}^M} X^n)\ .
\end{aligned}\]
\end{altenumerate}
\end{lem}

\begin{proof} This is an easy computation, left to the reader.
\end{proof}

To organize this information, the following definitions are useful. Using that $M=\Res_{F/\Q} \GL_n$, one has
\[
\mathbb{T}_{K^{M,S^+}} = \bigotimes_{v\not\in S} \mathbb{T}_v^M\ ,
\]
with
\[
\mathbb{T}_v^M[q_v^{1/2}]\cong \Z_p[q_v^{1/2}][X_1^{\pm 1},\ldots,X_n^{\pm 1}]^{S_n}\ ,
\]
where $q_v$ is the cardinality of the residue field at $v$. One has the $i$-th elementary symmetric polynomial $T_{i,v}^M\in \mathbb{T}_v^M[q_v^{1/2}]$ in the $X_1,\ldots,X_n$, with $q_v^{i(n+1)/2} T_{i,v}^M\in \mathbb{T}_v^M$. Define the polynomials
\[\begin{aligned}
P_v(X) &= 1 - q_v^{(n+1)/2}T_{1,v}^M X + q_v^{2(n+1)/2}T_{2,v}^M X^2 - \ldots + (-1)^n q_v^{n(n+1)/2} T_{n,v}^M X^n\ ,\\
P_v^\vee(X) &= 1 - q_v^{-(n+1)/2}\frac{T_{n-1,v}^M}{T_{n,v}^M} X + q_v^{-2(n+1)/2}\frac{T_{n-2,v}^M}{T_{n,v}^M} X^2 - \ldots + (-1)^n q_v^{-n(n+1)/2}\frac{1}{T_{n,v}^M} X^n
\end{aligned}\]
in $\mathbb{T}_v^M[X]$. Note that $P_v^\vee(X)$ is the polynomial with constant coefficient $1$ whose zeroes are the inverses of the zeroes of $P_v(X)$. Moreover, if $F$ is totally real, define
\[
\tilde{P}_v(X) = (1 - X) P_v(X) P_v^\vee(X)\in \mathbb{T}_v^M[X] = \mathbb{T}_{K^M_v}[X]\ ;
\]
if $F$ is CM, define
\[
\tilde{P}_v(X) = P_v(X) P_{v^c}^\vee(q_vX)\in \mathbb{T}_{K^M_{v^+}}[X]\ ,
\]
where $v^c$ is the complex conjugate place, and $v^+$ the place of $F^+$ below $v$.

\begin{cor}\label{CorExGalRepr3} Let $d$ be the complex dimension of $X_K$. There exists an ideal
\[
I\subset \overline{\mathbb{T}}_{K^{M,S^+}} = \im(\mathbb{T}_{K^{M,S^+}}\to \End_{\Z/p^m\Z}(H^i_!(X^M_{K^M},\Z/p^m\Z)))
\]
with $I^{2(d+1)} = 0$, such that there exists a continuous $h$-dimensional determinant $D$ of $G_{F,S}$ with values in $A=\overline{\mathbb{T}}_{K^{M,S^+}}/I$, satisfying
\[
D(1-X\Frob_v) = \tilde{P}_v(X)
\]
for all finite places $v\not\in S$.
\end{cor}

Here and in the following, we do not strive to give the best bound on the nilpotence degree.

\begin{proof} By Corollary \ref{CompHeckeAlgParabInd} and the computations of Lemma \ref{DescrHeckeAlg2}, it is enough to prove the similar result for
\[
\overline{\mathbb{T}}_{K^{S^+}} = \im(\mathbb{T}_{K^{S^+}}\to \End_{\Z/p^m\Z}(H^i(X_K^\BS\setminus X_K,\Z/p^m\Z)))\ .
\]
Using the long exact sequence
\[
\ldots\to H^i(X_K,\Z/p^m\Z)\to H^i(X_K^\BS\setminus X_K,\Z/p^m\Z)\to H_c^{i+1}(X_K,\Z/p^m\Z)\ ,
\]
it is enough to prove that in the Hecke algebras
\[\begin{aligned}
\mathbb{T}_{K^{S^+}}(H^i(X_K,\Z/p^m\Z)) &= \im(\mathbb{T}_{K^{S^+}}\to \End_{\Z/p^m\Z}(H^i(X_K,\Z/p^m\Z))\ ,\\
\mathbb{T}_{K^{S^+}}(H^i_c(X_K,\Z/p^m\Z)) &= \im(\mathbb{T}_{K^{S^+}}\to \End_{\Z/p^m\Z}(H^i_c(X_K,\Z/p^m\Z))\ ,
\end{aligned}\]
there are ideals $J_1$, $J_2$ whose $d+1$-th powers are $0$, such that there are determinants modulo $J_1$, $J_2$. Indeed, one will then get a determinant modulo $J_1\cap J_2$, and elements of $(J_1\cap J_2)^{d+1}\subset J_1^{d+1}\cap J_2^{d+1}$ will induce endomorphisms of $H^i(X_K^\BS\setminus X_K,\Z/p^m\Z)$ that act trivially on the associated graded of a two-step filtration; thus, $(J_1\cap J_2)^{2(d+1)}$ will give trivial endomorphisms of $H^i(X_K^\BS\setminus X_K,\Z/p^m\Z)$.

By Poincar\'e duality, one reduces further to the case of $\mathbb{T}_{K^{S^+}}(H^i_c(X_K,\Z/p^m\Z))$. We may assume that there is a rational prime $\ell\neq p$, $\ell\geq 3$, such that all places of $F^+$ above $\ell$ are in $S^+$ (by adjoining them to $S^+$, without changing $K$); the desired result follows by varying $\ell$. Thus, we may fix a normal compact open subgroup $K_pK^p_{S^+}\subset K_{S^+}$ for which $K^p_{S^+}$ is contained in the level-$\ell$-subgroup of $G^\prime(\A_f)$. Note that then also $K_{S^+}^p\subset K_{S^+}$ is a closed normal subgroup. One has the Hochschild-Serre spectral sequence
\[
H^i_\cont(K_{S^+}/K_{S^+}^p,\widetilde{H}^j_{c,K_{S^+}^pK^{S^+}}(\Z/p^m\Z))\Rightarrow H^{i+j}_c(X_K,\Z/p^m\Z)\ ,
\]
equivariant for the $\mathbb{T}_{K^{S^+}}$-action.\footnote{For the equivariance, reduce to a finite cover with group $K_{S^+} / K_p^\prime K_{S^+}^p$, passing to a colimit afterwards. For a finite cover, equivariance follows from compatibility of trace maps with base change.} Let $\mathbb{T}_{K^{S^+},\cl}$ be the topological ring $\mathbb{T}_{K^{S^+}}$ defined as in Theorem \ref{ThmHeckeAlgebras}, for $m$ large enough. Then, by Theorem \ref{ThmHeckeAlgebras}, $\mathbb{T}_{K^{S^+},\cl}$ acts continuously on the $E_2$-term of this spectral sequence. In particular, it acts continuously on the $E_\infty$-term, so that there is a filtration of $H^i_c(X_K,\Z/p^m\Z)$ by at most $d+1$ terms (cf. Corollary \ref{CorVanishing}), such that the associated action on the graded quotients is continuous. Let $J\subset \mathbb{T}_{K^{S^+}}(H^i_c(X_K,\Z/p^m\Z))$ be the ideal of elements acting trivially on the associated graded quotients. Then $J^{d+1} = 0$, and we are reduced to showing that there is a determinant modulo $J$.

But now $A = \mathbb{T}_{K^{S^+}}(H^i_c(X_K,\Z/p^m\Z))/J$ is a discrete quotient of $\mathbb{T}_{K^{S^+},\cl}$, so one gets the desired determinant from Corollary \ref{CorExGalRepr2}.
\end{proof}

Let us rephrase this corollary in more intrinsic terms, changing notation slightly.

\begin{cor}\label{CorExGalRepr4} Let $F$ be a totally real or CM field, with totally real subfield $F^+\subset F$. Fix an integer $n\geq 1$. If $F$ is totally real, define $h=2n+1$ and $d=[F:\Q](n^2+n)/2$; if $F$ is CM, define $h=2n$ and $d=[F^+:\Q]n^2$. Let $S$ be a finite set of finite places of $F$ invariant under complex conjugation, containing all places above $p$, and all places which are ramified above $F^+$. Let $K\subset \GL_n(\A_{F,f})$ be a sufficiently small\footnote{This can always be ensured by making it smaller at one place $v\in S$ not dividing $p$.} compact open subgroup of the form $K=K_SK^S$, where $K_S\subset \GL_n(\A_{F,S})$ is compact open and torsion-free, and $K^S=\prod_{v\not\in S} \GL_n(\OO_{F_v})\subset \GL_n(\A_{F,f}^S)$. Let
\[
\mathbb{T}_{F,S} = \bigotimes_{v\not\in S} \mathbb{T}_v = \bigotimes_{v\not\in S} \Z_p[\GL_n(F_v)//\GL_n(\OO_{F_v})]
\]
be the abstract Hecke algebra, and
\[
\mathbb{T}_{F,S}(K,i,m) = \im(\mathbb{T}_{F,S}\to \End_{\Z/p^m\Z}(H^i_!(X_K,\Z/p^m\Z)))\ .
\]
Here,
\[
X_K = \GL_n(F)\backslash [(\GL_n(F\otimes_{\Q} \R) / \R_{>0} K_\infty)\times \GL_n(\A_{F,f})/K]
\]
denotes the locally symmetric space associated with $\GL_n/F$, where $K_\infty\subset \GL_n(F\otimes_{\Q} \R)$ is a maximal compact subgroup. Then there is an ideal $I\subset \mathbb{T}_{F,S}(K,i,m)$ with $I^{2(d+1)}=0$, for which there is a continuous $h$-dimensional determinant $\tilde{D}$ of $G_{F,S}$ with values in $\mathbb{T}_{F,S}(K,i,m)/I$, satisfying
\[
\tilde{D}(1-X\Frob_v) = \tilde{P}_v(X)
\]
for all places $v\not\in S$.\footnote{For the definition of $\tilde{P}_v(X)$, cf. the paragraph before Corollary \ref{CorExGalRepr3}.}
\end{cor}

\begin{proof} This is Corollary \ref{CorExGalRepr3}, noting that any $K$ as in the statement can be realized as a $K^M$ in the notation of Corollary \ref{CorExGalRepr3}.
\end{proof}

\section{Divide and conquer}\label{DivideAndConquer}

Let the situation be as in Corollary \ref{CorExGalRepr4}. Thus, $F$ is a number field, which is totally real or CM, $p$ is a prime number, $n\geq 1$ some integer, and $S$ if a finite set of finite places of $F$ invariant under complex conjugation, containing all places dividing $p$ and all places which are ramified over $F^+$. Moreover, fix a sufficiently small $K=K_SK^S\subset \GL_n(\A_{F,f})$ such that $K^S = \prod_{v\not\in S} \GL_n(\OO_{F_v})\subset \GL_n(\A_{F,f}^S)$. Let
\[
\mathbb{T}_{F,S}(K,i,m) = \im(\mathbb{T}_{F,S}\to \End_{\Z/p^m\Z}(H^i_!(X_K,\Z/p^m\Z)))\ .
\]
In this section, we prove the following theorem.

\begin{thm}\label{ThmExGalRepr2} There exists an ideal $J\subset \mathbb{T}_{F,S}(K,i,m)$, $J^{4(d+1)} = 0$, such that there is a continuous $n$-dimensional determinant $D$ of $G_{F,S}$ with values in $\mathbb{T}_{F,S}(K,i,m)/J$ satisfying
\[
D(1-X\Frob_v) = P_v(X)
\]
for all places $v\not\in S$.
\end{thm}

\begin{proof} Note that $A_0 = \mathbb{T}_{F,S}(K,i,m)$ is a finite ring. Let $A=A_0\otimes_{\Z_p} W(\overline{\F}_p)$. It suffices to prove that there is a determinant (with the stated properties) with values in $A/J$ for some ideal $J\subset A$ with $J^{4(d+1)} = 0$.

By Corollary \ref{CorExGalRepr4}, there exists a determinant $\tilde{D}_1$ with values in $A/I_1$, $I_1^{2(d+1)} = 0$, satisfying
\[
\tilde{D}_1(1-X\Frob_v) = \tilde{P}_v(X)
\]
for all $v\not\in S$. We will use this result for many cyclotomic twists, roughly following an idea used in \cite{HLTT}. Let $\chi: G_{F,S}\to W(\overline{\F}_p)^\times$ be any continuous character of odd order prime to $p$. Define
\[
P_{v,\chi}(X) = P_v(X/\chi(\Frob_v))\in A[X]\ .
\]
Let $P_{v,\chi}^\vee(X) = P_v^\vee(\chi(\Frob_v)X)$ be the polynomial with constant coefficient $1$ whose zeroes are the inverses of the zeroes of $P_{v,\chi}(X)$. Define
\[
\tilde{P}_{v,\chi}(X) = (1-X)P_{v,\chi}(X)P_{v,\chi}^\vee(X)
\]
in case $F$ is totally real, and
\[
\tilde{P}_{v,\chi}(X) = P_{v,\chi}(X) P_{v^c,\chi}^\vee(q_vX)
\]
in case $F$ is CM. We claim that there is a determinant $\tilde{D}_\chi$ with values in $A/I_\chi$, $I_\chi^{2(d+1)} = 0$, satisfying
\[
\tilde{D}_\chi(1-X\Frob_v) = \tilde{P}_{v,\chi}(X)
\]
for all $v\not\in S$.

Indeed, the character $\chi: G_{F,S}\to W(\overline{\F}_p)^\times$ corresponds by class-field theory to a continuous character $\psi: F^\times\backslash \A_F^\times\to W(\overline{\F}_p)^\times$. As it is odd, it is trivial at all archimedean primes, i.e. factors through a character $\psi: F^\times\backslash \A_{F,f}^\times\to W(\overline{\F}_p)^\times$. Also, it is unramified away from $S$, and its order is prime to $p$. Thus, one can find a normal compact open subgroup $K^\prime = K_S^\prime K^S\subset K=K_SK^S$ of index $[K:K^\prime]$ prime to $p$, such that $\psi$ is trivial on $\det(K^\prime)$. Because $[K:K^\prime]$ is prime to $p$, the map
\[
H^i_!(X_K,\Z/p^m\Z)\to H^i_!(X_{K^\prime},\Z/p^m\Z)
\]
is split injective. The $\psi$-isotypic component
\[
H^0(X_{K^\prime},W(\overline{\F}_p))[\psi]
\]
is $1$-dimensional (as $\pi_0(X_{K^\prime})\cong F^\times\backslash \A_{F,f}^\times / \det K^{\prime}$). The cup-product gives a map
\[
H^i_!(X_K,\Z/p^m\Z)\otimes_{\Z_p} H^0(X_{K^\prime},W(\overline{\F}_p))[\psi]\to H^i_!(X_{K^\prime},W(\overline{\F}_p)/p^m)\ ,
\]
which is injective, as cup-product with $H^0(X_{K^\prime},W(\overline{\F}_p))[\psi^{-1}]$ maps this back to
\[
H^i_!(X_K,W(\overline{\F}_p)/p^m)\subset H^i_!(X_{K^\prime},W(\overline{\F}_p)/p^m)\ .
\]
Applying Corollary \ref{CorExGalRepr4} to the Hecke algebra corresponding to
\[
H^i_!(X_K,\Z/p^m\Z)\otimes_{\Z_p} H^0(X_{K^\prime},W(\overline{\F}_p))[\psi]\subset H^i_!(X_{K^\prime},W(\overline{\F}_p)/p^m)
\]
will produce the desired determinant $\tilde{D}_\chi$ with values in $(\mathbb{T}_{F,S}(K,i,m)\otimes_{\Z_p} W(\overline{\F}_p))/I_\chi$ for some ideal $I_\chi$ with $I_\chi^{2(d+1)} = 0$.

Our first aim is to prove that there is an ideal $I_0\subset A_0$ with $I_0^{4(d+1)}=0$ such that there is a continuous function $G_{F,S}\to A_0/I_0[X] : g\mapsto P_g$ with $P_{\Frob_v} = P_v$ for all $v\not\in S$. This will be done in several steps. Let $\overline{A}$ be the reduced quotient of $A$, which is a finite product of copies of $\overline{\F}_p$. Let $\overline{P}_v(X)\in \overline{A}[X]$ be the image of $P_v(X)$.

\begin{lem}\label{Det1OpenSubgroup} There is a finite extension $L_0/F$, Galois over $F^+$ (thus over $F$), such that for all places $v\not\in S$ which split in $L_0$, $\overline{P}_v(X) = \overline{P}_{v^c}(X) = (1-X)^n$, and $q_v\equiv 1\mod p$ in case $F$ is CM.
\end{lem}

\begin{proof} Look at the continuous determinant
\[
\overline{\tilde{D}}_1: \overline{\F}_p[G_{F,S}]\to \overline{\F}_p
\]
It factors over $\Gal(L_0^\prime/F)$ for some finite Galois extension $L_0^\prime/F$ (unramified outside $S$), which we may assume to be Galois over $F^+$. It follows that for all $v\not\in S$ which split in $L_0^\prime$,
\[
(1-X)^{2n+1} = \overline{\tilde{D}}_1(1-X\Frob_v) = (1-X) \overline{P}_v(X) \overline{P}_v^\vee(X)
\]
if $F$ is totally real. In particular, $\overline{P}_v(X)$ divides $(1-X)^{2n+1}$, has constant coefficient $1$ and is of degree $n$, thus $\overline{P}_v(X) = (1-X)^n$; we may take $L_0 = L_0^\prime$.

If $F$ is CM, then for all $v\not\in S$ which split in $L_0^\prime$,
\[
(1-X)^{2n} = \overline{\tilde{D}}_1(1-X\Frob_v) = \overline{P}_v(X) \overline{P}_{v^c}^\vee(q_vX)\ .
\]
Again, one sees that $\overline{P}_v(X) = (1-X)^n$. If one takes $L_0 = L_0^\prime(\zeta_p)$, where $\zeta_p$ is a primitive $p$-th root of unity, then for all $v$ which split in $L_0$, one has $q_v\equiv 1\mod p$, so that one gets also $\overline{P}_{v^c}(X) = (1-X)^n$.
\end{proof}

Next, we define $I_0\subset A_0$ with $I_0^{4(d+1)}=0$. For any odd rational prime $\ell\neq p$, let $F^{\cycl_\ell}/F$ be the cyclotomic $\Z_\ell$-extension. Let $S_\ell = S\cup \{v|\ell\}$ (and similarly for any set of rational primes). For any character $\chi: \Z_\ell\to W(\overline{\F}_p)^\times$, we have an ideal $I_\chi\subset A$, with $I_\chi^{2(d+1)}=0$, such that there is a continuous determinant $\tilde{D}_\chi$ of $G_{F,S_\ell}$ with values in $A/I_\chi$. The ideal $\tilde{I}_\chi = I_1 + I_\chi\subset A$ satisfies $\tilde{I}_\chi^{4(d+1)} = 0$. The intersection $\tilde{I}_{0,\chi} = \tilde{I}_\chi\cap A_0\subset A_0$ is an ideal of the finite ring $A_0$. Thus, there is some $I_{0,\ell}\subset A_0$ with $I_{0,\ell}^{4(d+1)} = 0$ such that for infinitely many $\chi$, $\tilde{I}_{0,\chi} = I_{0,\ell}$. Finally, there is some $I_0\subset A_0$ with $I_0^{4(d+1)} = 0$ such that $I_0 = I_{0,\ell}$ for infinitely many $\ell$.

Now fix any two sufficiently large distinct rational primes $\ell, \ell^\prime\neq p$ with $I_{0,\ell} = I_{0,\ell^\prime} = I_0$ such that $[L_0:F]$ is not divisible by $\ell$ and $\ell^\prime$. Let $L^{\infty_\ell}$ be the maximal pro-$p$-extension of $L_0\cdot F^{\cycl_\ell}$ which is unramified outside $S_\ell$. Thus, there is a quotient $G_{F,S_\ell}\to \Gal(L^{\infty_\ell}/F)$.

\begin{lem}\label{Det1ProP} For any character $\chi: \Gal(F^{\cycl_\ell}/F)\cong \Z_\ell\to W(\overline{\F}_p)^\times$, the determinant $\tilde{D}_\chi$ of $G_{F,S_\ell}$ factors over $\Gal(L^{\infty_\ell}/F)$.
\end{lem}

\begin{proof} Over $\overline{A}$, the determinant $\overline{\tilde{D}}_\chi$ corresponds to a continuous semisimple representation
\[
\overline{\pi}_\chi: G_{F,S_\ell}\to \GL_h(\overline{A})\ ,
\]
by \cite[Theorem 2.12]{ChenevierDet}. For all $g\in \ker(G_{F,S_\ell}\to \Gal(L_0\cdot F^{\cycl_\ell}/F))$, it satisfies
\[
\det(1-Xg|\overline{\pi}_\chi) = (1-X)^h
\]
by Lemma \ref{Det1OpenSubgroup}. It follows that these elements $g$ are mapped to elements of $p$-power order, so that $\overline{\pi}_\chi$ factors over $\Gal(L^{\infty_\ell}/F)$. Now apply \cite[Lemma 3.8]{ChenevierDet}.
\end{proof}

\begin{lem}\label{LemDetLocConst} There is a unique continuous function
\[
g\mapsto P_g: \Gal(L^{\infty_\ell}/F)\setminus \Gal(L^{\infty_\ell}/F^{\cycl_\ell})\to A_0/I_0[X]
\]
such that $P_{\Frob_v} = P_v$ for all $v\not\in S_\ell$.
\end{lem}

\begin{proof} Fix any $g\in \Gal(L^{\infty_\ell}/F)\setminus \Gal(L^{\infty_\ell}/F^{\cycl_\ell})$ and fix a character
\[
\chi: \Gal(F^{\cycl_\ell}/F)\cong \Z_\ell\to W(\overline{\F}_p)^\times
\]
such that $\tilde{I}_{0,\chi} = I_0$ and $\chi(g)^{2j}\neq 1$ for $j=1,\ldots,n$. As $g\not\in \Gal(L^{\infty_\ell}/F^{\cycl_\ell})$, only finitely many $\chi$ violate the second condition; as infinitely many $\chi$ satisfy the first, some suitable $\chi$ exists. There is an open neighborhood $U\subset \Gal(L^{\infty_\ell}/F)\setminus \Gal(L^{\infty_\ell}/F^{\cycl_\ell})$ of $g$ such that $\chi$ and the determinants
\[
\tilde{D}_1: A[\Gal(L^{\infty_\ell}/F)]\to A/I_1\ ,\ \tilde{D}_\chi: A[\Gal(L^{\infty_\ell}/F)]\to A/I_\chi
\]
are constant on $U$. In particular, the polynomials
\[
\tilde{P}_v(X)\mod I_1\ ,\ \tilde{P}_{v,\chi}(X)\mod I_\chi
\]
are constant on $U$. It is enough to see that $P_v(X)\mod I_0$ is constant on $U$. We do only the totally real case; the CM case is similar. Recall that if $\Frob_v\in U$, then
\[
\tilde{P}_v(X) = (1-X)P_v(X) P_v^\vee(X)\ ,\ \tilde{P}_{v,\chi}(X) = (1-X) P_v(X/\chi(g)) P_v^\vee(\chi(g)X)\ .
\]
Both are constant on $U$ modulo $\tilde{I}_\chi=I_1+I_\chi$. This implies that
\[
P_v(X)P_v^\vee(X)\ ,\ P_v(X/\chi(g)) P_v^\vee(\chi(g)X)
\]
are constant on $U$ modulo $\tilde{I}_\chi$. Let us forget that $Q_v = P_v^\vee$ is determined by $P_v$. Write
\[
P_v(X) = 1 + a_1(v) X + \ldots + a_n(v) X^n\ ,\ Q_v(X) = 1 + b_1(v) X + \ldots + b_n(v) X^n\ .
\]
By induction on $j$, we prove that $a_j(v)$ and $b_j(v)$ are constant on $U$ modulo $\tilde{I}_\chi$. Calculating the coefficient of $X^j$ in $P_v(X) Q_v(X)$ and $P_v(X/\chi(g)) Q_v(\chi(g)X)$ gives only contributions which are constant on $U$ modulo $\tilde{I}_\chi$, except possibly the sum $a_j(v)+b_j(v)$ in the first product, and $\chi(g)^{-j} a_j(v) + \chi(g)^j b_j(v)$ in the second product. Thus, $a_j(v) + b_j(v)$ and $\chi(g)^{-2j} a_j(v) + b_j(v)$ are constant on $U$ modulo $\tilde{I}_\chi$. Taking the difference, we find that $(1-\chi(g)^{-2j})a_j(v)$ is constant on $U$ modulo $\tilde{I}_\chi$, which implies that $a_j(v)$ is constant on $U$ modulo $\tilde{I}_\chi$, as $1-\chi(g)^{-2j}$ is a unit by assumption on $\chi$. Thus, $b_j(v) = (a_j(v) + b_j(v)) - a_j(v)$ is also constant on $U$ modulo $\tilde{I}_\chi$.

As both $a_j(v),b_j(v)\in A_0$, we find that they are constant modulo $A_0\cap \tilde{I}_\chi = \tilde{I}_{0,\chi} = I_0$, as desired.
\end{proof}

\begin{cor}\label{CorDetLocConst} There exists a unique continuous function
\[
g\mapsto P_g:G_{F,S}\to A_0/I_0[X]
\]
such that $P_{\Frob_v} = P_v$ for all $v\not\in S$.
\end{cor}

\begin{proof} Let $M/F$ be the extension for which $G_{F,S_{\ell,\ell^\prime}} = \Gal(M/F)$. Applying Lemma \ref{LemDetLocConst} for $\ell$ and $\ell^\prime$ individually, we see that there are continuous functions
\[
g\mapsto P_g: G_{F,S_{\ell,\ell^\prime}}\setminus \Gal(M/F^{\cycl_\ell})\to A_0/I_0[X]
\]
and
\[
g\mapsto P_g: G_{F,S_{\ell,\ell^\prime}}\setminus \Gal(M/F^{\cycl_{\ell^\prime}})\to A_0/I_0[X]\ .
\]
By uniqueness, they glue to a continuous function
\[
g\mapsto P_g: G_{F,S_{\ell,\ell^\prime}}\setminus \Gal(M/F^{\cycl_\ell}\cdot F^{\cycl_{\ell^\prime}})\to A_0/I_0[X]\ .
\]
But the map
\[
G_{F,S_{\ell,\ell^\prime}}\setminus \Gal(M/F^{\cycl_\ell}\cdot F^{\cycl_{\ell^\prime}})\to \Gal(L^{\infty_\ell}/F)
\]
is surjective, as $F^{\cycl_{\ell^\prime}}$ is linearly disjoint from $L^{\infty_\ell}$ (because one extension is pro-$\ell^\prime$, whereas the other is pro-prime-to-$\ell^\prime$). To check whether the continuous function
\[
g\mapsto P_g: \Gal(L^{\infty_\ell}/F)\setminus \Gal(L^{\infty_\ell}/F^{\cycl_\ell})\to A_0/I_0[X]
\]
extends continuously to some $g\in \Gal(L^{\infty_\ell}/F^{\cycl_\ell})$, one can lift $g$ to some
\[
\tilde{g}\in G_{F,S_{\ell,\ell^\prime}}\setminus \Gal(M/F^{\cycl_\ell}\cdot F^{\cycl_{\ell^\prime}})\ ,
\]
and use that the continuous function $g\mapsto P_g$ exists near $\tilde{g}$. This shows that there is a continuous function
\[
G_{F,S_\ell}\to \Gal(L^{\infty_\ell}/F)\to A_0/I_0[X]
\]
interpolating $P_{\Frob_v}$ for $v\not\in S_\ell$. Similarly, there is a continuous function
\[
G_{F,S_{\ell^\prime}}\to A_0/I_0[X]
\]
interpolating $P_{\Frob_v}$ for $v\not\in S_{\ell^\prime}$. By uniqueness, they give the same function on $G_{F,S_{\ell,\ell^\prime}}$, which will thus factor over $G_{F,S}$, and interpolate $P_{\Frob_v}$ for all $v\not\in S$.
\end{proof}

Thus, there exists a finite Galois extension $\tilde{F}/F$ unramified outside $S$ with Galois group $G$ and a function
\[
g\mapsto P_g: G\to A_0/I_0[X]
\]
such that $P_{\Frob_v} = P_v$ for all $v\not\in S$. By adjoining a primitive $p$-power root of $1$ to $\tilde{F}$, we may assume that there is also a map $g\mapsto q_g\in A_0/I_0$ interpolating $\Frob_v\mapsto q_v$. Also, in the CM case, we may assume that $\tilde{F}/F^+$ is Galois, so that there is map $g\mapsto g^c$ on conjugacy classes in $G$, given by the outer action of $\Gal(F/F^+) = \{1,c\}$. Choose some (new) rational prime $\ell\neq p$, $\ell\geq 3$, such that $\ell$ does not divide $[\tilde{F}:F]$; in particular, $\tilde{F}$ is linearly disjoint from $F^{\cycl_\ell}$.

\begin{lem}\label{LemUnivTwistExists1} There is an ideal $I\subset A[T^{\pm 1}]$ containing $I_0\cdot A[T^{\pm 1}]$ with $I^{4(d+1)} = 0$ and an $h$-dimensional determinant (i.e., a multiplicative $A$-polynomial law, homogeneous of degree $h$)
\[
\tilde{D}_I: A[G][V^{\pm 1}]\to A[T^{\pm 1}]/I
\]
such that for all $g\in G$, $k\in \Z$,
\[
\tilde{D}_I(1-XV^k g) = (1-X)P_g(X/T^k) P_g^\vee(T^k X)
\]
if $F$ is totally real, resp.
\[
\tilde{D}_I(1-XV^k g) = P_g(X/T^k) P_{g^c}^\vee(T^k q_g X)
\]
if $F$ is CM.
\end{lem}

\begin{proof} Embed $A[G][V^{\pm 1}]\hookrightarrow A[G\times \Z_\ell]$ by mapping $V$ to $(1,1)\in G\times \Z_\ell$, where $1\in G$ is the identity, and $1\in \Z_\ell$ is the tautological topological generator. One knows that for any character $\chi: \Z_\ell\to W(\overline{\F}_p)^\times$, one has a determinant
\[
A[G][V^{\pm 1}]\to A[G\times \Z_\ell]=A[\Gal(\tilde{F}\cdot F^{\cycl_\ell}/F)]\to A/I_\chi\ ,
\]
which, if $\tilde{D}_I$ exists, agrees with the composite of $\tilde{D}_I$ with $A[T^{\pm 1}]\to A/I_\chi$ sending $T$ to $\chi(1)$. However, the map
\[
A[T^{\pm 1}]\to \prod_{\begin{substack}{\chi: \Z_\ell\to W(\overline{\F}_p)^\times\\ \tilde{I}_{0,\chi} = I_0}\end{substack}} A
\]
is injective. Let $I\subset A[T^{\pm 1}]$ be the preimage of $\prod \tilde{I}_\chi$; then $I^{4(d+1)} = 0$ and $I_0\cdot A[T^{\pm 1}]\subset I$. One knows that the determinant
\[
\tilde{D}^\prime: A[G][V^{\pm 1}]\to \prod_{\begin{substack}{\chi: \Z_\ell\to W(\overline{\F}_p)^\times\\ \tilde{I}_{0,\chi} = I_0}\end{substack}} A/\tilde{I}_\chi
\]
exists, and that for all $g\in G$, $k\in \Z$,
\[
\tilde{D}^\prime(1-XV^k g)\in (A[T^{\pm 1}]/I)[X]\subset \prod_{\begin{substack}{\chi: \Z_\ell\to W(\overline{\F}_p)^\times\\ \tilde{I}_{0,\chi} = I_0}\end{substack}} A/\tilde{I}_\chi[X]\ .
\]
Thus, by \cite[Corollary 1.14]{ChenevierDet}, $\tilde{D}^\prime$ factors through a determinant $\tilde{D}: A[G][V^{\pm 1}]\to A[T^{\pm 1}]/I$, as desired.
\end{proof}

\begin{lem}\label{LemUnivTwistExists2} There exists an ideal $J\subset A$, $I_0\cdot A\subset J$, with $J^{4(d+1)} = 0$ and an $h$-dimensional determinant
\[
\tilde{D}: A[G][V^{\pm 1}]\to (A/J)[T^{\pm 1}]
\]
such that for all $g\in G$, $k\in \Z$,
\[
\tilde{D}(1-XV^k g) = (1-X)P_g(X/T^k) P_g^\vee(T^k X)
\]
if $F$ is totally real, resp.
\[
\tilde{D}(1-XV^k g) = P_g(X/T^k) P_{g^c}^\vee(T^k q_g X)
\]
if $F$ is CM.
\end{lem}

\begin{proof} Take $I\subset A[T^{\pm 1}]$ as in the last lemma, giving $\tilde{D}_I$. Let $a\in \Z$ be any integer, and look at the map $A[G][V^{\pm 1}]\to A[G][V^{\pm 1}]$ mapping $V$ to $V^a$. One gets an $h$-dimensional determinant
\[
\tilde{D}_{I,a}: A[G][V^{\pm 1}]\to A[G][V^{\pm 1}]\to A[T^{\pm 1}]/I\ .
\]
Let $I_a = \{f(T)\in A[T^{\pm 1}]\mid f(T^a)\in I\}$, an ideal of $A[T^{\pm 1}]$. Then the map $T\mapsto T^a$ induces an injection $A[T^{\pm 1}]/I_a\hookrightarrow A[T^{\pm 1}]/I$, and by checking on characteristic polynomials and using \cite[Corollary 1.14]{ChenevierDet}, one sees that $\tilde{D}_{I,a}$ factors through a determinant
\[
\tilde{D}_{I_a}: A[G][V^{\pm 1}]\to A[T^{\pm 1}]/I_a\ ,
\]
which satisfies the relations imposed on $\tilde{D}$. Let $I^\prime = \bigcap_{a\in \Z} I_a$. Then one has an injection
\[
A[T^{\pm 1}]/I^\prime\to \prod_a A[T^{\pm 1}]/I_a\ .
\]
By taking the product, one has a determinant with values in $\prod_a A[T^{\pm 1}]/I_a$; by checking on characteristic polynomials and using \cite[Corollary 1.14]{ChenevierDet} once more, one gets a determinant with values in $A[T^{\pm 1}]/I^\prime$. Let $J\subset A$ be the ideal generated by all coefficients of elements of $I^\prime$. Certainly, one gets a determinant with values in $(A/J)[T^{\pm 1}]$ by composition. Thus, to finish the proof, it suffices to see that $J^{4(d+1)} = 0$. Thus, take any elements $f_1,\ldots,f_{4(d+1)}\in I^\prime\subset A[T^{\pm 1}]$ and write
\[
f_i(T) = \sum_{j\in \Z} c_{i,j} T^j\ ,
\]
with only finitely many $c_{i,j}$ nonzero. Choose integers $a_1>>a_2>>\ldots>>a_{4(d+1)}$. One knows that $f_i(T^{a_i})\in I$, and $I^{4(d+1)} = 0$; thus
\[
0 = \prod_{i=1}^{4(d+1)} f_i(T^{a_i}) = \prod_{i=1}^{4(d+1)} (\sum_j c_{i,j} T^{a_ij})\ .
\]
If one has chosen the $a_i$ sufficiently generic, every power of $T$ will occur only once when factoring this out. This implies that any product $c_{1,j_1}\cdots c_{4(d+1),j_{4(d+1)}}$ is zero, showing that $J^{4(d+1)} = 0$, as desired.
\end{proof}

Finally, we are reduced to the following lemma on determinants, with $R=A/J$.
\end{proof}

\begin{lem} Let $G$ be a group, and $R$ some ring. For any $m\in \Z$, let a map
\[
g\mapsto P_g^{(m)}(X): G\to R[X]
\]
be given, taking values in polynomials of degree $n_m$ with constant coefficient $1$. We assume that $n_m=0$ for all but finitely many $m$. Let $n = \sum_{m\in \Z} n_m$, and assume that there is an $n$-dimensional determinant
\[
\tilde{D}: R[G][V^{\pm 1}]\to R[T^{\pm 1}]
\]
such that for all $g\in G$, $k\in \Z$,
\[
\tilde{D}(1-V^k Xg) = \prod_{m\in \Z} P_g^{(m)}(T^{km} X)\in R[T^{\pm 1}][X]\ .
\]
Then for all $m\in \Z$, there exists an $n_m$-dimensional determinant $D^{(m)}: R[G]\to R$ such that for all $g\in G$,
\[
D^{(m)}(1-X g) = P_g^{(m)}(X)\ .
\]
\end{lem}

\begin{rem} Intuitively, the lemma says the following, up to replacing `representation' by `determinant'. Assume you want to construct representations $\pi_m$, $m\in \Z$, of $G$, with prescribed characteristic polynomials. Assume that you know that for any character $\chi$ of $\Z$, the representation
\[
\bigoplus_{m\in \Z} \pi_m\otimes \chi^m
\]
of $G\times \Z$ exists; note that $R[G][V^{\pm 1}] = R[G\times \Z]$. Then all the representations $\pi_m$ exist.
\end{rem}

\begin{proof} We need the following lemma.

\begin{lem}\label{DecompositionLemma} Let $S$ be a (commutative) ring, and $Q\in S[T^{\pm 1}][[X]]$ be any polynomial such that $Q\equiv 1\mod X$. Then there is at most one way to write
\[
Q = \prod_{m\in \Z} Q_m
\]
with $Q_m\in S[[XT^m]]$, $Q_m\equiv 1\mod X$, almost all equal to $1$. Moreover, if $Q\in S[T^{\pm 1}][X]$, then all $Q_m\in S[XT^m]$.
\end{lem}

\begin{proof} Given any two such presentations $Q=\prod Q_m = \prod Q_m^\prime$, one may take the quotient, and thus reduce to the case $Q=1$. Let $k$ be minimal such that not all $Q_m$ are $\equiv 1\mod X^k$. Then $Q_m\equiv 1 + (XT^m)^k a_m\mod X^{k+1}$ for some $a_m\in S$, almost all $0$. But then
\[
1 = Q = \prod_{m\in \Z} Q_m\equiv 1 + X^k \sum_{m\in \Z} T^{mk} a_m\mod X^{k+1}\ .
\]
No cancellation can occur, so $a_m=0$ for all $m\in \Z$, contradiction.

For the final statement, one may replace $X$ by $XT^{m_0}$ for some $m_0$, so that we may assume that $Q_m = 1$ for $m<0$. In that case, all $Q_m\in S[[X,T]]$, so the same is true for $Q$, and we may reduce modulo $T$. Then one finds that $Q\equiv Q_0\mod T$, so in particular, $Q_0\in S[X]$. Let $Q^\prime = \prod_{m>0} Q_m$. We claim that $Q^\prime\in S[T^{\pm 1}][X]$; then an inductive argument will finish the proof.

Note that $Q^\prime\in S[[X,XT]]$, so in particular $Q^\prime\in S[T][[X]]$, with $Q^\prime Q_0\in S[T,X]$. Writing
\[
Q^\prime = \sum_{a\geq 0} Q^\prime_a(T) X^a\ ,
\]
with $Q^\prime_a(T)\in S[T]$, we see that for $a$ sufficiently large, $Q^\prime_a Q_0 = 0$. As $Q_0\in S[T]$ has constant coefficient $1$, this implies that $Q^\prime_a = 0$, so indeed $Q^\prime\in S[T,X]$.
\end{proof}

\begin{lem}\label{ExistsDm0} For any $R$-algebra $S$, there are unique multiplicative maps
\[
D^{(m)}_0: 1+US[G][[U]]\to 1 + US[[U]]
\]
such that for all $f(U)\in 1 + US[G][[U]]$,
\[
\tilde{D}(f(XV)) = \prod_{m\in \Z} D^{(m)}_0(f(XT^m))\in S[T^{\pm 1}][[X]]\ .
\]
It satisfies
\[
D^{(m)}_0(1 - U^a x g) = P_g^{(m)}(U^a x)
\]
for all $a\geq 1$, $x\in S$ and $g\in G$. Moreover, $D^{(m)}_0$ maps $1+US[G][U]$ into $1 + US[U]$.
\end{lem}

\begin{proof} Lemma \ref{DecompositionLemma} implies uniqueness of each value $D^{(m)}_0(f(U))$ individually. Moreover, uniqueness implies multiplicativity, by multiplicativity of $\tilde{D}$. For existence, note that the left-hand side $\tilde{D}(f(XV))$ is multiplicative in $f$. It follows that if the desired decomposition exists for two elements $f$, $f^\prime$, then also for their product. Moreover, the set of elements for which such a decomposition exists is $U$-adically closed. As any element $f\in 1+US[G][[U]]$ can (non-uniquely) be written as an infinite product
\[
f = \prod_{j=1}^\infty (1 - U^{a_j} x_j g_j)
\]
for certain $a_j\geq 1$, $a_j\to\infty$, $x_j\in S$ and $g_j\in G$, one reduces to the case that
\[
f = 1 - U^a x g
\]
with $a\geq 1$, $x\in S$ and $g\in G$. In that case,
\[
\tilde{D}(1 - (XV)^a x g) = \prod_{m\in \Z} P_g^{(m)}( (XT^m)^a x)
\]
by the defining equation of $\tilde{D}$. This gives the desired decomposition in this case, and proves the formula for $D^{(m)}_0(1 - U^a x g)$.

The final statement follows from the second half of Lemma \ref{DecompositionLemma}.
\end{proof}

Now, for any $R$-algebra $S$, we can define
\[
D^{(m)}(x) = D^{(m)}_0 ( 1 + U(x-1) ) |_{U=1}\in S
\]
for any $x\in S[G]$. This defines a polynomial map $D^{(m)}: R[G]\to R$. It satisfies
\[
D^{(m)}(1 - xg) = D^{(m)}_0(1 - Uxg)|_{U=1} = P_g^{(m)}(Ux)|_{U=1} = P_g^{(m)}(x)
\]
for $x\in S$, $g\in G$. We claim that $D^{(m)}$ is multiplicative, i.e.
\[
D^{(m)}((1-x)(1-y)) = D^{(m)}(1-x)D^{(m)}(1-y)
\]
for all $x,y\in S[G]$. The desired equation reads
\[
D^{(m)}_0(1-U(x+y)+Uxy)|_{U=1} = D^{(m)}_0(1-Ux)|_{U=1}D^{(m)}_0(1-Uy)|_{U=1}\ .
\]
Write $x=\sum_{g\in A} x_g g$, $y=\sum_{g\in A} y_g g$ for some finite subset $A\subset G$. We may reduce to the universal case $S=R[X_g,Y_g]_{g\in A}$, or also to $S=R[[X_g,Y_g]]_{g\in A}$. Thus it is enough to do it for all $S=R[[X_g,Y_g]]_{g\in A} / (X_g,Y_g)^n$. In other words, we may assume that the ideal $I\subset S$ generated by all $x_g$, $y_g$ is nilpotent. In that case, $x,y\in S[G]$ are nilpotent.

In particular, $1-Ux\in S[G][U]$ is invertible, with inverse $1+Ux+U^2x^2+\ldots$, where only finitely many terms occur, as $x\in S[G]$ is nilpotent. Similarly, $1-Uy\in S[G][U]$ is invertible. Using multiplicity of $D^{(m)}_0$, the desired equation reads
\[
D^{(m)}_0\left((1-U(x+y)+Uxy)(1+Ux+U^2x^2+\ldots)(1+Uy+U^2y^2+\ldots)\right)|_{U=1} = 1\ .
\]
Letting $f=(1-U(x+y)+Uxy)(1+Ux+U^2x^2+\ldots)(1+Uy+U^2y^2+\ldots)\in 1+US[G][U]$, one reduces multiplicativity of $D^{(m)}$ to the following lemma.

\begin{lem} Let $I\subset S$ be a nilpotent ideal, and $f\in 1 + US[G][U]$ such that $f\equiv 1\mod I$ and $f|_{U=1}=1$. Then
\[
D^{(m)}_0(f)|_{U=1} = 1\ .
\]
\end{lem}

\begin{proof} We claim that any such $f$ can be written as a product of terms
\[
(1-U^{a+1} z g) / (1 - U^a z g)
\]
for $a\geq 1$, $z\in I$ and $g\in G$. Note that, as before, the inverse of $1- U^a zg\in S[G][U]$ exists, as $z$ is nilpotent. Assume first that $I^2=0$. As $f-1\in U\cdot I[G][U]$ and $f|_{U=1}=1$, we have $f-1\in U(U-1)\cdot I[G][U]$, so we may write
\[
f=1 - \sum_{a\geq 1} \sum_{g\in G} (U^{a+1} - U^a) z_{g,a} g
\]
with $z_{g,a}\in I$. Using $I^2=0$, this rewrites as
\[
f=\prod_{a\geq 1,g\in G} (1-U^{a+1} z_{g,a} g) / (1 - U^a z_{g,a} g)\ ,
\]
as desired. In general, this shows that after dividing $f$ by terms of the form
\[
(1-U^{a+1} z g) / (1 - U^a z g)\ ,
\]
one gets an element $f^\prime$ with the same properties, and $f^\prime\equiv 1\mod I^2$. The nilpotence degree of $I^2$ is smaller than the nilpotence degree of $I$, so one gets the result by induction.

Using multiplicativity of $D^{(m)}_0$, it is now enough to prove that
\[
D^{(m)}_0\left((1-U^{a+1} z g) / (1 - U^a z g)\right)|_{U=1} = 1
\]
for all $a\geq 1$, $z\in I$ and $g\in G$. But by Lemma \ref{ExistsDm0}, 
\[
D^{(m)}_0(1-U^{a+1} zg)|_{U=1} = P_g^{(m)}(U^{a+1}z)|_{U=1} = P_g^{(m)}(z) = P_g^{(m)}(U^a z)|_{U=1} = D^{(m)}_0(1 - U^a zg)|_{U=1}\ ,
\]
finishing the proof of the lemma.
\end{proof}

It remains to verify that $D^{(m)}$ is homogeneous of degree $n_m$. For this, observe the following general lemma on determinants, which shows that homogeneity of some degree is automatic.

\begin{lem}\label{EverythingHomogeneous} Let $R[G]\to R$ be a multiplicative $R$-polynomial law. Then for some integer $N$ there is a decomposition $R=R_0\times\cdots\times R_N\times R_\infty$ of $R$ into direct factors such that for $0\leq d\leq N$ the induced multiplicative $R_d$-polynomial law $R_d[G]\to R_d$ is homogeneous of degree $d$, i.e. is a determinant of dimension $d$, and such that $R_\infty[G]\to R_\infty$ is constant $0$.
\end{lem}

\begin{proof} By restriction, we get a multiplicative $R$-polynomial law $R\to R$. It suffices to see that after a decomposition into direct factors, this is of the form $x\mapsto x^d$ for some integer $d\geq 0$, or constant $0$. Applying the polynomial law to $T\in R[T]$ gives an element $f(T)\in R[T]$, which by multiplicativity satisfies $f(UT) = f(U)f(T)$. Let
\[
f(T) = a_d T^d + a_{d+1}T^{d+1} + \ldots + a_N T^N\ ,
\]
where $a_d\neq 0$. Looking at the coefficient of $U^dT^d$ in $f(UT) = f(U)f(T)$ shows that $a_d^2 = a_d$, so after a decomposition into a direct product, we may assume that either $a_d=1$ or $a_d=0$. In the second case, we can continue this argument with a higher coefficient of $f$ to arrive eventually in the case $a_d=1$ (or $f(T) = 0$, in which case the polynomial law is constant $0$). Thus, assume that $a_d=1$. Looking at the coefficient of $U^d T^{d+i}$ in $f(UT) = f(U)f(T)$, we see that $0=a_{d+i}$ for all $i\geq 1$, thus $f(T) = T^d$, and the polynomial law is indeed given by $x\mapsto x^d$.
\end{proof}

Note that the degree of each $P_g^{(m)}(X)$ is at most $n_m$, but the product $\prod_{m\in \Z} P_g^{(m)}(X)$ has degree exactly $n = \sum_{m\in \Z} n_m$ (as it is the characteristic polynomial of $g$ in $\tilde{D}$); thus, each $P_g^{(m)}(X)$ has degree exactly $n_m$, and it follows that $D^{(m)}$ is homogeneous of degree $n_m$.
\end{proof}

\section{Conclusion}\label{SectionResults}

Finally, we can state our main result. Let $F$ be a totally real or CM field with totally real subfield $F^+\subset F$, $n\geq 1$ some integer, $p$ some rational prime, and $S$ a finite set of finite places of $F$, stable under complex conjugation, containing all places above $p$, and all places which are ramified over $F^+$. Let $G_{F,S}$ be the Galois group of the maximal extension of $F$ which is unramified outside $S$.

Let
\[
K\subset \prod_v \GL_n(\OO_{F_v})\subset \GL_n(\A_{F,f})
\]
be a compact open subgroup of the form $K=K_SK^S$, where $K_S\subset \prod_{v\in S} \GL_n(\OO_{F_v})$ is any compact open subgroup, and $K^S = \prod_{v\not\in S} \GL_n(\OO_{F_v})\subset \GL_n(\A_{F,f}^S)$. We get the locally symmetric space
\[
X_K = \GL_n(F)\backslash [(\GL_n(F\otimes_{\Q} \R)/\R_{>0} K_\infty)\times \GL_n(\A_{F,f})/K]\ ,
\]
where $K_\infty\subset \GL_n(F\otimes_{\Q} \R)$ is a maximal compact subgroup. If $K$ is not sufficiently small, we regard this as a `stacky' object, so that (by definition) $X_{K^\prime}\to X_K$ is a finite covering map of degree $[K:K^\prime]$, for any open subgroup $K^\prime\subset K$.

Moreover, fix an algebraic representation $\xi$ of $\Res_{\OO_F/\Z} \GL_n$ with coefficients in a finite free $\overline{\Z}_p$-module $M_\xi$. This defines a local system $\mathcal{M}_{\xi,K}$ of $\overline{\Z}_p$-modules on $X_K$, for any $K$ as above.

Let
\[
\mathbb{T}_{F,S} = \bigotimes_{v\not\in S} \mathbb{T}_v\ ,\ \mathbb{T}_v = \Z_p[\GL_n(F_v)//\GL_n(\OO_{F_v})]
\]
be the abstract Hecke algebra. By the Satake isomorphism, we have a canonical isomorphism
\[
\mathbb{T}_v[q_v^{1/2}]\cong \Z_p[q_v^{1/2}][X_1^{\pm 1},\ldots,X_n^{\pm 1}]^{S_n}\ ,
\]
where $q_v$ is the cardinality of the residue field of $F$ at $v$. We let $T_{i,v}\in \mathbb{T}_v[q_v^{1/2}]$ be the $i$-th symmetric polynomial in $X_1,\ldots,X_n$; then $q_v^{i(n+1)/2}T_{i,v}\in \mathbb{T}_v$. Define the polynomial
\[
P_v(X) = 1 - q_v^{(n+1)/2} T_{1,v} X + q_v^{2(n+1)/2} T_{2,v} X^2 - \ldots + (-1)^n q_v^{n(n+1)/2} T_{n,v} X^n\in \mathbb{T}_v[X]\ .
\]
Recall that there is a canonical action of $\mathbb{T}_{F,S}$ on
\[
H^i(X_K,\mathcal{M}_{\xi,K})\ .
\]

\begin{thm}\label{FinalMainThm} There exists an integer $N=N([F:\Q],n)$ depending only on $[F:\Q]$ and $n$, such that for any compact open subgroup $K=K_SK^S\subset \GL_n(\A_{F,f})$ as above, algebraic representation $\xi$, and any integers $i,m\geq 0$, the following is true. Let
\[
\mathbb{T}_{F,S}(K,\xi,i,m) = \im(\mathbb{T}_{F,S}\to \End_{\overline{\Z}_p/p^m}(H^i(X_K,\mathcal{M}_{\xi,K}/p^m)))\ .
\]
Then there is an ideal $I\subset \mathbb{T}_{F,S}(K,\xi,i,m)$ with $I^N = 0$ such that there is an $n$-dimensional continuous determinant $D$ of $G_{F,S}$ with values in $\mathbb{T}_{F,S}(K,\xi,i,m)/I$, satisfying
\[
D(1-X\Frob_v) = P_v(X)
\]
for all $v\not\in S$.
\end{thm}

\begin{proof} Fix a sufficiently small normal compact open subgroup $K^\prime\subset K$ such that $\mathcal{M}_{\xi,K^\prime}/p^m$ is trivial; the second condition can be ensured by requiring that the image of $K^\prime$ in $\prod_{v|p} \GL_n(F_v)$ is contained in $\{g\in \prod_{v|p} \GL_n(\OO_{F_v})\mid g\equiv 1\mod p^m\}$. One has the Hochschild-Serre spectral sequence
\[
H^i(K/K^\prime,H^j(X_{K^\prime},\mathcal{M}_{\xi,K^\prime}/p^m))\Rightarrow H^{i+j}(X_K,\mathcal{M}_{\xi,K}/p^m)\ .
\]
This reduces us to the case that $K$ is sufficiently small, and that $\xi$ is trivial. In that case, $\mathcal{M}_{\xi,K}/p^m$ is a direct sum of copies of $\Z/p^m\Z$.

Thus, we have to consider
\[
\mathbb{T}_{F,S}(K,i,m) = \im(\mathbb{T}_{F,S}\to \End_{\Z/p^m\Z}(H^i(X_K,\Z/p^m\Z)))\ .
\]
Using the Borel-Serre compactification $X_K^\BS$, we have the long exact sequence of $\mathbb{T}_{F,S}$-modules
\[
\ldots \to H_c^i(X_K,\Z/p^m\Z)\to H^i(X_K,\Z/p^m\Z)\to H^i(X_K^\BS\setminus X_K,\Z/p^m\Z)\to \ldots\ .
\]
It is an easy exercise to express $H^i(X_K^\BS\setminus X_K,\Z/p^m\Z)$ in terms of the locally symmetric spaces for $\GL_{n^\prime}/F$, with $n^\prime<n$ (cf. \cite[Section 3]{CaraianiLeHung} for more discussion of this point). Thus, by induction, the determinants exist for
\[
\im(\mathbb{T}_{F,S}\to \End_{\Z/p^m\Z}(H^i(X_K^\BS\setminus X_K,\Z/p^m\Z)))\ ,
\]
and in particular for
\[
\mathbb{T}_{F,S}(K,i,m,\partial) = \im(\mathbb{T}_{F,S}\to \End_{\Z/p^m\Z}(\im(H^i(X_K,\Z/p^m\Z)\to H^i(X_K^\BS\setminus X_K,\Z/p^m\Z))))\ ,
\]
modulo some nilpotent ideal $I_\partial\subset \mathbb{T}_{F,S}(K,i,m,\partial)$ of nilpotence degree bounded by $[F:\Q]$ and $n$. On the other hand, by Theorem \ref{ThmExGalRepr2}, there is a nilpotent ideal
\[
I_!\subset \mathbb{T}_{F,S}(K,i,m,!) = \im(\mathbb{T}_{F,S}\to \End_{\Z/p^m\Z}(\im(H_c^i(X_K,\Z/p^m\Z)\to H^i(X_K,\Z/p^m\Z))))
\]
of nilpotence degree bounded by $[F:\Q]$ and $n$, such that the determinants exist with values in $\mathbb{T}_{F,S}(K,i,m,!)/I_!$. But the kernel of the map
\[
\mathbb{T}_{F,S}(K,i,m)\to \mathbb{T}_{F,S}(K,i,m,!)\times \mathbb{T}_{F,S}(K,i,m,\partial)
\]
is a nilpotent ideal with square $0$. Thus, the kernel $I$ of
\[
\mathbb{T}_{F,S}(K,i,m)\to \mathbb{T}_{F,S}(K,i,m,!)/I_!\times \mathbb{T}_{F,S}(K,i,m,\partial)/I_\partial
\]
is a nilpotent ideal with nilpotence degree bounded by $[F:\Q]$ and $n$, and by \cite[Corollary 1.14]{ChenevierDet}, one finds that the determinant $D$ with values in $\mathbb{T}_{F,S}(K,i,m)/I$ exists. This finishes the proof.
\end{proof}

Let us state some corollaries, where the determinants give rise to actual representations. We start with the following result on classical automorphic representations, that has recently been proved by Harris-Lan-Taylor-Thorne, \cite{HLTT}. Recall that we have fixed an isomorphism $\C\cong \overline{\Q}_p$.

\begin{cor}\label{CorExGalReprChar0} Let $\pi$ be a cuspidal automorphic representation of $\GL_n(\A_F)$ such that $\pi_\infty$ is regular $L$-algebraic, and such that $\pi_v$ is unramified at all finite places $v\not\in S$. Then there exists a unique continuous semisimple representation
\[
\sigma_\pi: G_{F,S}\to \GL_n(\overline{\Q}_p)
\]
such that for all finite places $v\not\in S$, the Satake parameters of $\pi_v$ agree with the eigenvalues of $\sigma_\pi(\Frob_v)$.
\end{cor}

\begin{proof} Note that $\pi^\prime = \pi|\cdot|^{(n+1)/2}$ is regular $C$-algebraic, i.e. cohomological (cf. \cite[Theorem 3.13, Lemma 3.14]{ClozelMotifs}). Thus, there exists some algebraic representation $\xi$ of $\Res_{F/\Q} \GL_n$ with coefficients in $\C\cong \overline{\Q}_p$ (which can be extended to an algebraic representation of $\Res_{\OO_F/\Z} \GL_n$ with coefficients in $\overline{\Z}_p$, still denoted $\xi$) such that $\pi^\prime$ occurs in
\[
H^i(\tilde{X}_K,\mathcal{M}_{\xi,K})\otimes_{\overline{\Z}_p} \C
\]
for some sufficiently small level $K=K_SK^S$ and integer $i$. Here,
\[
\tilde{X}_K = \GL_n(F)\backslash [(\GL_n(F\otimes_{\Q} \R)/\R_{>0} K_\infty^\circ)\times \GL_n(\A_{F,f})/K]\ ,
\]
where $K_\infty^\circ\subset K_\infty$ is the connected component of the identity. Thus, $\tilde{X}_K = X_K$ if $F$ is CM, and is a $(\Z/2\Z)^{[F:\Q]}$-cover if $F$ is totally real. Twisting by a character $\A_F^\times\to \Z/2\Z$ with prescribed components at the archimedean places, one can arrange that $\pi^\prime$ occurs in $H^i(X_K,\mathcal{M}_{\xi,K})\otimes_{\overline{\Z}_p} \C$, which we shall assume from now on.

Let
\[
\mathbb{T}_{F,S}(K,\xi,i) = \im(\mathbb{T}_{F,S}\to \End_{\overline{\Z}_p}(H^i(X_K,\mathcal{M}_{\xi,K})))\ .
\]
The kernel of $\mathbb{T}_{F,S}\to \mathbb{T}_{F,S}(K,\xi,i)$ is contained in the kernel of
\[
\mathbb{T}_{F,S}\to \prod_{m\geq 1} \mathbb{T}_{F,S}(K,\xi,i,m)\ ,
\]
so by Theorem \ref{FinalMainThm} (and \cite[Example 2.32]{ChenevierDet}), there exists an $n$-dimensional continuous determinant $D$ of $G_{F,S}$ with values in $\mathbb{T}_{F,S}(K,\xi,i)/I$ for some nilpotent ideal $I$.\footnote{For this conclusion, it was necessary to know that the nilpotence degree is bounded independently of $m$; one gets also that $I^N = 0$.} Composing with the map $\mathbb{T}_{F,S}(K,\xi,i)\to \overline{\Q}_p$ corresponding to $\pi^\prime$, we get an $n$-dimensional continuous determinant $D_{\pi^\prime}$ of $G_{F,S}$ with values in $\overline{\Q}_p$, giving the desired continuous semisimple representation $\sigma_\pi$ by \cite[Theorem 2.12]{ChenevierDet} (continuity follows e.g. from \cite[Theorem 1]{TaylorPseudoRepr}).
\end{proof}

On the other hand, we can apply Theorem \ref{FinalMainThm} to characteristic $p$ cohomology.

\begin{cor}\label{CorExGalReprCharP} Let $\psi: \mathbb{T}_{F,S}\to \overline{\F}_p$ be a system of Hecke eigenvalues such that the $\psi$-eigenspace
\[
H^i(X_K,\mathcal{M}_{\xi,K}\otimes_{\overline{\Z}_p} \overline{\F}_p)[\psi]\neq 0\ .
\]
Then there exists a unique continuous semisimple representation
\[
\sigma_\psi: G_{F,S}\to \GL_n(\overline{\F}_p)
\]
such that for all finite places $v\not\in S$,
\[
\det(1-X\Frob_v|\sigma_\psi) = 1 - \psi(q_v^{(n+1)/2} T_{1,v}) X + \psi(q_v^{2(n+1)/2} T_{2,v}) X^2 - \ldots + (-1)^n \psi(q_v^{n(n+1)/2} T_{n,v}) X^n\ .
\]
\end{cor}

\begin{proof} This is immediate from Theorem \ref{FinalMainThm} and \cite[Theorem 2.12]{ChenevierDet}.
\end{proof}

\begin{cor}\label{CorExGalReprCharPDef} Let $\psi: \mathbb{T}_{F,S}\to \overline{\F}_p$ be as in Corollary \ref{CorExGalReprCharP}, and assume that $\sigma_\psi$ is irreducible. Let $\mathfrak{m}\subset \mathbb{T}_{F,S}$ be the maximal ideal which is the kernel of $\psi$, and let
\[
\mathbb{T}_{F,S}(K,\xi,i) = \im(\mathbb{T}_{F,S}\to \End_{\overline{\Z}_p}(H^i(X_K,\mathcal{M}_{\xi,K})))\ .
\]
Take $N=N([F:\Q],n)$ as in Theorem \ref{FinalMainThm}. Then there exists an ideal $I\subset \mathbb{T}_{F,S}(K,\xi,i)$ with $I^N =0$ and a unique continuous representation
\[
\sigma_{\mathfrak{m}}: G_{F,S}\to \GL_n(\mathbb{T}_{F,S}(K,\xi,i)_{\mathfrak{m}}/I)
\]
such that for all finite places $v\not\in S$,
\[
\det(1-X\Frob_v|\sigma_\psi) = 1 - q_v^{(n+1)/2} T_{1,v} X + q_v^{2(n+1)/2} T_{2,v} X^2 - \ldots + (-1)^n q_v^{n(n+1)/2} T_{n,v} X^n\ .
\]
Here, $\mathbb{T}_{F,S}(K,\xi,i)_{\mathfrak{m}}$ denotes the localization of $\mathbb{T}_{F,S}(K,\xi,i)$ at $\mathfrak{m}$.
\end{cor}

\begin{rem} One obtains (with the same proof) a slightly stronger result, replacing $\mathbb{T}_{F,S}(K,\xi,i)$ by
\[
\mathbb{T}_{F,S}(K,\xi,i)^\prime = \im(\mathbb{T}_{F,S}\to \bigoplus_m \End_{\overline{\Z}_p/p^m}(H^i(X_K,\mathcal{M}_{\xi,K}/p^m)))\ ,
\]
or
\[
\mathbb{T}_{F,S}(K,\xi) = \im(\mathbb{T}_{F,S}\to \bigoplus_{i,m} \End_{\overline{\Z}_p/p^m}(H^i(X_K,\mathcal{M}_{\xi,K}/p^m)))\ .
\]
Such results were conjectured by Calegari and Geragthy, \cite[Conjecture B]{CalegariGeragthy}. In fact, this proves the existence of the Galois representations of \cite[Conjecture B]{CalegariGeragthy} (modulo a nilpotent ideal of bounded nilpotence degree), but does not establish all their expected properties.
\end{rem}

\begin{proof} There is an ideal $I\subset \mathbb{T}_{F,S}(K,\xi,i)$ with $I^N =0$ such that there is an $n$-dimensional continuous determinant with values in $\mathbb{T}_{F,S}(K,\xi,i)_{\mathfrak{m}}/I$, reducing to (the determinant associated with) $\sigma_\psi$ modulo $\mathfrak{m}$. As by assumption, $\sigma_\psi$ is irreducible, the result follows from \cite[Theorem 2.22 (i)]{ChenevierDet}.
\end{proof}

\begin{rem}\label{UnconditionalRem} As mentioned previously, these results are based on the work of Arthur, \cite{ArthurBook} (resp. Mok, \cite{Mok}), which are still conditional on the stabilization of the twisted trace formula. Let us end by noting that our results are unconditional under slightly stronger hypothesis. Namely, from the result of Shin, \cite{ShinGoldringApp} (cf. also the result in the book of Morel, \cite[Corollary 8.5.3]{MorelBook}), all results stated in this section are unconditional under the following assumptions:
\begin{altenumerate}
\item[{\rm (i)}] The field $F$ is CM, and contains an imaginary-quadratic field.
\item[{\rm (ii)}] The set $S$ comes via pullback from a finite set $S_{\Q}$ of finite places of $\Q$, which contains $p$ and all places at which $F/\Q$ is ramified.
\end{altenumerate}
In particular, if $F$ is imaginary-quadratic, then the results are unconditional as stated. Note that Shin's result is stated in terms of unitary similitude groups. However, Theorem \ref{PerfShHodge} (and all results in Chapter \ref{AutomorphicChapter} deduced from it) stays true verbatim for usual Shimura varieties of Hodge type (with the same proof), so that one can apply it to the Shimura varieties associated with unitary similitude groups. Then the argument of Chapter \ref{GaloisReprChapter} goes through as before.

Using a patching argument (cf. proof of \cite[Theorem VII.1.9]{HarrisTaylor}), Corollary \ref{CorExGalReprChar0} follows for general totally real or CM fields, but still under the assumption that $S$ comes via pullback from a finite set $S_{\Q}$ of finite places of $\Q$ which contains all places at which $F/\Q$ is ramified.
\end{rem}

\bibliographystyle{abbrv}
\bibliography{Torsion}

\end{document}